\documentclass[11pt,oneside,dvipsnames]{amsart}
\usepackage{stackrel}
\usepackage[french, english]{babel}
\usepackage{a4wide}
\usepackage[utf8,latin1]{inputenc}
\usepackage{times}
\usepackage{adjustbox}
\usepackage{float}
\usepackage[cyr]{aeguill}
\usepackage{amsmath,amscd}
\usepackage{amssymb}
\usepackage{mathrsfs}
\usepackage{amsthm}
\usepackage{geometry}
\usepackage{graphicx}
\usepackage{epstopdf}
\usepackage{amsfonts}
\usepackage{amsmath}
\usepackage{amsmath,mathtools}
\usepackage{amsmath,graphicx}
\usepackage{amscd}
\usepackage{latexsym}
\usepackage{verbatim}
\usepackage{comment}
\usepackage{tikz-cd}
\usetikzlibrary{matrix,arrows}
\usepackage{tikz}
\usetikzlibrary{tikzmark}
\usepackage{tensor}
\usepackage{leftidx}
\usepackage[retainorgcmds]{IEEEtrantools}
\usepackage[title]{appendix}
\usepackage[mathscr]{euscript}
\usepackage[colorinlistoftodos]{todonotes}
\usepackage{bbm}
\usepackage{url}
\usepackage{upgreek}
\usepackage{mathtools, nccmath}
\usepackage[utf8]{inputenc}
\usepackage{multirow}
\usepackage{rotating}
\usepackage{marginnote}
\usepackage{hyperref} 

\usepackage[all]{xy}
\usepackage[T1]{fontenc}

\usepackage{hyperref}
\usepackage{array}
\usepackage{tabularx}
\usepackage{yfonts}
\usepackage{setspace}  
\usepackage{geometry}
\makeatletter
\renewcommand\subsubsection{\@startsection{subsubsection}{3}{\z@}%
                                     {-3.25ex\@plus -1ex \@minus -.2ex}%
                                     {-0.5em}% <---changed
                                     {\normalfont\normalsize\bfseries}}
\makeatother

\theoremstyle{plain}
\newtheorem{Thm}{Theorem}
\newtheorem{Prop}[Thm]{Proposition}
\newtheorem{Lem}[Thm]{Lemma}
\newtheorem{Cor}[Thm]{Corollary}

\theoremstyle{definition}
\newtheorem{Conv}[Thm]{Convention}

\newtheorem{Nota}[Thm]{Notation}
\newtheorem{Rem}[Thm]{Remark}

\newtheorem{Defn}[Thm]{Definition}
\newtheorem{Eg}[Thm]{Example}
\newtheorem*{Pro}{\textit{Problem}}
\newtheoremstyle{named}{}{}{}{}{\bfseries}{.}{.5em}{\thmnote{#3}#1}
\theoremstyle{named}
\newtheorem*{namedtheorem}{}

\DeclareMathOperator{\ad}{ad}

\DeclareMathOperator{\Aut}{Aut}

\DeclareMathOperator{\ch}{char}

\DeclareMathOperator{\diag}{diag}
\DeclareMathOperator{\End}{End}

\DeclareMathOperator{\GL}{GL}

\DeclareMathOperator{\Hom}{Hom}
\DeclareMathOperator{\IC}{\textbf{IC}}
\DeclareMathOperator{\Id}{Id}
\DeclareMathOperator{\Ima}{Im}

\DeclareMathOperator{\Irr}{Irr}

\DeclareMathOperator{\isom}{\!\!\smash{\begin{array}{c}\sim\\[-1em]
\rightarrow\end{array}}\!\!}

\DeclareMathOperator{\isomLR}{\!\!\smash{\begin{array}{c}\sim\\[-1em]
\leftrightarrow\end{array}}\!\!}
\DeclareMathOperator{\lisomLR}{\!\!\smash{\begin{array}{c}\sim\\[-1em]
\longleftrightarrow\end{array}}\!\!}
\DeclareMathOperator{\Ind}{Ind}

\DeclareMathOperator{\ladic}{\bar{\mathbb{Q}}_{\ell}}
\DeclareMathOperator{\Lie}{Lie}

\DeclareMathOperator{\Ort}{O}

\DeclareMathOperator{\pr}{pr}
\DeclareMathOperator{\pt}{pt}

\DeclareMathOperator{\Res}{Res}

\DeclareMathOperator{\sgn}{sgn}
\DeclareMathOperator{\Sh}{Sh}
\DeclareMathOperator{\SL}{SL}
\DeclareMathOperator{\SO}{SO}
\DeclareMathOperator{\Sp}{Sp}

\DeclareMathOperator{\St}{St}

\DeclareMathOperator{\tC}{\mathbf{1}}
\DeclareMathOperator{\tr}{Tr}

\DeclareMathOperator{\lb}{\!<\!}
\DeclareMathOperator{\rb}{\!>\!}

\newcommand*\circled[1]{
  \tikz[baseline=(char.base)]\node[shape=circle,draw,inner sep=0.2pt,font=\tiny,minimum size=8pt] (char) {#1};}

\newcommand{\hooklongrightarrow}{\lhook\joinrel\longrightarrow}

\makeatletter
\newcommand*{\rom}[1]{\expandafter\@slowromancap\romannumeral #1@}
\makeatother

\newcommand{\rn}[2]{%% "rn": "remember node"
    \tikz[remember picture,baseline=(#1.base)]\node [inner sep=0] (#1) {$#2$};%
}

\usepackage{subfig}	 % \chaptionof

\title{The Character Table of $\mathbf{\GL_n(\mathbb{F}_q)\rtimes\lb\sigma\rb}$}
\author{Cheng Shu}
\usepackage{pxfonts}
\usepackage{indentfirst}
\usepackage{setspace}  
\usepackage{hyperref}
\colorlet{ivory}{Apricot!30!}
\colorlet{space}{black!85!}
\definecolor{bgc}{RGB}{29, 44, 46}
\definecolor{txt}{RGB}{223, 222, 189}
\definecolor{cmd}{RGB}{206, 151, 88}
\begin{document}
\let\bs\boldsymbol
\dedicatory{Dedicated to Jasmine}
\maketitle
\begin{abstract}
We compute the character table of $\GL_n(\mathbb{F}_q)\rtimes\lb\sigma\rb$, $\sigma$ being an order 2 exterior automorphism. We begin by giving the parametrisation of $\sigma$-stable irreducible characters of $\GL_n(q)$ and that of the conjugacy classes contained in $\GL_n(q)\sigma$. Our main theorem is a formula expressing the extension of a $\sigma$-stable character of $\GL_n(q)$ to $\GL_n(q)\sigma$ as a linear combination of induced cuspidal functions. This formula is built upon a key result of J-L. Waldspurger concerning character sheaves. Using the main theorem, the determination of the character table is then reduced to computations in various Weyl groups and the generalised Green functions of classical groups. Finally, we explicitly determine the table for $n=2$ and $3$. %and certain values for $n=4$. 
\end{abstract}

\tableofcontents
\addtocontents{toc}{\protect\setcounter{tocdepth}{-1}}
\section*{Introduction}
\subsection*{General Notations}\hfill

We fix an odd prime number $p$ and denote by $q$ a fixed power of $p$, then we fix a prime number $\ell$ prime to $p$. Let $k$ be an algebraic closure of the finite field $\mathbb{F}_q$ with $q$ elements. Sheaves on algebraic varieties over $k$ are understood to be sheaves of $\ladic$-vector spaces in the \'etale topology. Denote by $N_{\mathbb{F}_{q^r}/\mathbb{F}_{q^s}}:\mathbb{F}_{q^r}\rightarrow\mathbb{F}_{q^s}$ the norm map whenever $s$ divides $r$. If $G/k$ is any linear algebraic group defined over $\mathbb{F}_q$ with Frobenius endomorphism $F$, we denote by $G^F$ the set of fixed points of $F$, which is a finite group. We may also denote this group by $G(q)$ if $F$ is clear from the context. We only consider the $\ladic$-representations of the finite groups $G^F$, and notably $\GL_n(q)$ for some non negative integer $n$. We will always assume that $q>n$. The identity component of $G$ will be denoted by $G^{\circ}$. The centre of $G$ will be denoted by $Z_G$. For any closed subsets $X$, $H$, $H'\subset G$, we write $C_X(H)=\{x\in X\mid xh=hx\text{, $\forall$ $h\in H$}\}$, $N_X(H)=\{x\in X\mid xH=Hx\}$ and $N_X(H,H')=\{x\in X\mid xH=Hx,~xH'=H'x\}$. When $H=\{h\}$ consists of a single element, the centraliser $C_X(h)$ is also sometimes denoted by $X^h$. 

Throughout the article, we will denote by $\mathfrak{i}$ a fixed square root of $(-1)$ in $k$. For any matrix $A$, its transpose-inverse is denoted by $A^{-t}$. The multiplicative 2-element group will be denoted by $\bs\mu_2$, and its elements are written as $\pm 1$, or simply $\pm$. We will denote by $\eta$ the nontrivial irreducible character of $\bs\mu_2$. 

Since we will be primarily interested in a non-connected group $G$ with two connected components, the hypothesis on $p$ implies that the unipotent elements of $G$ are all contained in $G^{\circ}$. If $G^1$ is the connected component of $G$ other than $G^{\circ}$ with some fixed element $\sigma\in G^1$, we will often write the Jordan decomposition of an element of $G^1$ as $s\sigma u$, with $s\sigma$ being semi-simple and $u$ unipotent. Note that in this expression $s\sigma$ should be regarded as a whole and $s$ may well not be semi-simple.

%%%%%%%%%%%%%%%%
\subsection*{Irreducible Characters of $\GL_n(q)$}\hfill

Let $G=\GL_n$ over $k$ and let $F$ be the Frobenius endomorphism associated with some $\mathbb{F}_q$-structure of $G$, then the finite group $G^F=G(q)$ is denoted by $\GL_n(q)$ or $\GL^-_n(q)$ according to whether the action of $F$ on the Dynkin diagram is trivial or not.  The character table of $\GL_n(q)$ has been well known  since the work of Green \cite{Gr}. Instead of the combinatorial point of view of Green, we present below a parametrisation of the irreducible characters due to Lusztig and Srinivasan, which is fit for the problem of extending characters to $\GL_n(q)\rtimes\lb\sigma\rb$ (see below).

For each $F$-stable Levi subgroup $L$, we denote by $\Irr_{reg}(L^F)$ the set of regular linear characters of $L^F$ (See \cite[\S 3.1]{LS}), and denote by $\Irr(W_L)^F$ the set of $F$-stable irreducible characters of the Weyl group $W_L=W_L(T)$, with $T\subset L$ being an $F$-stable maximal torus. We take $\theta\in\Irr_{reg}(L^F)$ and $\varphi\in\Irr(W_L)^F$. For each $\varphi$, we denote by $\tilde{\varphi}$ an extension of $\varphi$ to $W_L\rtimes\lb F\rb$. We put 
\begin{equation}\label{eqthm1}
R^{G}_{\varphi}\theta=\epsilon_G\epsilon_L|W_L|^{-1}\sum_{w\in W_L}\tilde{\varphi}(wF)R_{T_w}^{G}\theta,
\end{equation}
where for any linear algebraic group $H$, $\epsilon_H:=(-1)^{rk_H}$ and $rk_H$ is the $\mathbb{F}_q$-rank of $H$, and $R^G_T\theta$ is the Deligne-Lusztig induction (\cite{DL}) of $(T,\theta)$ .
\begin{Thm}(Lusztig, Srinivasan, \cite[Theorem 3.2]{LS})\label{LS}
Let $G=\GL_n^{\pm}(q)$. For some choice of $\tilde{\varphi}$, the virtual character $R^G_{\varphi}\theta$ is an irreducible character of  $G^F$. Moreover, all irreducible characters of $G^F$ are of the form $R^G_{\varphi}\theta$ for a triple $(L,\varphi,\theta)$. The characters associated to the triples $(L,\varphi,\theta)$ and $(L',\varphi',\theta')$ are distinct if and only if one of the following conditions is satisfied
\begin{itemize}
\item[-] $(L,\theta)$ and $(L',\theta')$ are not $G^F$-conjuguate;
\item[-] $(L,\theta)=(L',\theta')$ and $\varphi\ne\varphi'$.
\end{itemize}
\end{Thm}
Therefore, the calculation of the values of the irreducible characters of $\GL_n(q)$ is reduced to the calculation of the values of Deligne-Lusztig characters, i.e. virtual characters of the form $R_T^G\theta$.

\subsection*{Clifford Theory}\hfill

Let $\sigma$ be an automorphism of order 2 of $\GL_n$ that is compatible with the given Frobenius. It defines a semi-direct product $\GL_n(q)\rtimes\mathbb{Z}/2\mathbb{Z}$. This group will be denoted by $\GL_n(q)\rtimes\lb\sigma\rb$ (or simply $\GL_n(q)\lb\sigma\rb$) in order to specify the action of $1\in\mathbb{Z}/2\mathbb{Z}$. We will assume that $\sigma$ is an exterior automorphism. Regarded as an element of this non-connected group, $\sigma=(\Id,1)$ satisfies $\sigma^2=1$ and $\sigma g\sigma^{-1}=\sigma(g)$, for all $g\in\GL_n(q)$.

The representations of $\GL_n(q)\rtimes\lb\sigma\rb$ are related to the representations of $\GL_n(q)$ by the Clifford theory in the following way. Let $H$ be a finite group and let $N$ be a normal subgroup of $H$ such that $H/N\simeq \mathbb{Z}/r\mathbb{Z}$ with $r$ prime, and let let $\chi$ be an irreducible character of $H$. We denote by $\chi_N$ the restriction of $\chi$ to $N$. Then
\begin{itemize}
\item[-] Either $\chi_N$ is irreducible;
\item[-] Or $\chi_N=\bigoplus_i^r\theta_i$, where $\theta_i\in\Irr(N)$ are some distinct irreducible characters.
\end{itemize}
Moreover, the $\theta_i$'s form an orbit under the action of $H/N$ on $\Irr(N)$. Conversely, $\chi_N\in\Irr(N)$ extends to an irreducible character of $H$ if and only if it is invariant under the action of $H/N$ by conjugaction. If $\chi$ is such an extension, we obtain all other extensions by multiplying $\chi$ by a character of $H/N$.

Denote by $\Irr(\GL_n(q))^{\sigma}$ the set of $\sigma$-stable irreducible characters, i.e. those satisfying $\chi=\leftidx{^{\sigma}\!}{\chi}:=\chi\circ\sigma$. The irreducible characters of $\GL_n(q)\rtimes\lb\sigma\rb$ are either an extension of a character $\chi\in\Irr(\GL_n(q))^{\sigma}$ or an extension of $\chi\oplus\leftidx{^{\sigma}\!}{\chi}{}$ with $\chi$ a non $\sigma$-stable irreducible character of $\GL_n(q)$. Note that the extension of $\chi\oplus\leftidx{^{\sigma}\!}{\chi}{}$ for $\chi$ non $\sigma$-stable vanishes on the component $\GL_n(q)\sigma$, whereas two extensions of $\chi\in\Irr(\GL_n(q))^{\sigma}$ differ by a sign on $\GL_n(q)\sigma$. Once we fix an extension $\tilde{\chi}$ for each $\chi\in\Irr(\GL_n(q))^{\sigma}$, it remains for us to calculate the restriction of $\tilde{\chi}$ on $\GL_n(q)\sigma$. If no confusion arises, we will also say that $\tilde{\chi}|_{\GL_n(q)\sigma}$ is an extension of $\chi$ to $\GL_n(q)\sigma$.

The conjugacy classes of $\GL_n(q)\lb\sigma\rb$ consist of the conjugacy classes of $\GL_n(q)$ that are stable under $\sigma$, the unions of pairs of conjugacy classes of the form $(C,\sigma(C))$ with $C\subset\GL_n(q)$ non $\sigma$-stable, and the conjugacy classes contained in $\GL_n(q)\sigma$. From the equality
\begin{equation}
\begin{split}
&\#\{\text{$\sigma$-stable classes}\}+\frac{1}{2}\#\{\text{non $\sigma$-stable classes}\}+\#\{\text{classes in $\GL_n(q)\sigma$}\}\\
=&\#\{\text{classes of $\GL_n(q)\lb\sigma\rb$}\}\\
=&\#\{\text{irreducible characters of $\GL_n(q)\lb\sigma\rb$}\}\\
=&2\#\{\text{$\sigma$-stable characters}\}+\frac{1}{2}\#\{\text{non $\sigma$-stable characters}\}
\end{split}
\end{equation}
and from the fact that $\#\{\text{$\sigma$-stable classes}\}=\#\{\text{$\sigma$-stable characters}\}$, we deduce that
\begin{equation}
\#\{\text{classes contained in $\GL_n(q)\sigma$}\}=\#\{\text{$\sigma$-stable characters}\}.
\end{equation}
So the table that we are going to calculate is a square table, its lines and columns being indexed by the $\sigma$-stable irreducible characters of $\GL_n(q)$ and the conjugacy classes in $\GL_n(q)\sigma$ respectively. However, there is no natural bijection between the classes and the characters.

\subsection*{Deligne-Lusztig Induction for $\GL_n(q)\rtimes\lb\sigma\rb$}\hfill

Let $\sigma$ be as above and let $\rho:\GL_n(q)\rightarrow V$ be a $\sigma$-stable irreducible representation (meaning that $\rho$ and $\rho\circ\sigma$ are isomorphic). Defining an extension of $\rho$, say $\tilde{\rho}$, is to define an action of $\sigma$ on $V$ in such a way that $\tilde{\rho}(\sigma)^2=\Id$ and that $\tilde{\rho}(\sigma)\rho(g)\tilde{\rho}(\sigma)^{-1}=\rho(\sigma(g))$ for all $g\in\GL_n(q)$. Except in some particular cases, we do not know how to do this. However, when $\sigma$ is \textit{quasi-central}, we have a natural action of $\sigma$ on the Deligne-Lusztig varieties $X_w$ associated to $w\in W^{\sigma}$, the subgroup of $\sigma$-fixed elements of $W=W_G(T_0)$, with $T_0$ being a $\sigma$-stable and $F$-stable maximal torus of $G$. This allows us to define the extensions of the Deligne-Lusztig characters $R^G_{T_w}1$ to $\GL_n(q)\lb\sigma\rb$, where $T_w$ is an $F$-stable maximal torus corresponding to the $F$-conjugacy class of $w$. By expressing a unipotent character of $\GL_n(q)$ as a linear combination of these Deligne-Lusztig characters, we obtain an extension of this unipotent character. More concretely, if we take an $F$-stable and $\sigma$-stable Borel subgroup $B_0\subset G$, the variety $X_w$ consists of the Borel subgroups $B$ such that $(B,F(B))$ are conjugate to $(B_0,\dot{w}B_0\dot{w}^{-1})$ by $G$, where $\dot{w}$ is a representative of $w\in W^{\sigma}$ in $G$ which can be chosen to be $\sigma$-stable (one needs $\sigma$ to be quasi-central here). Then the  action of $\sigma$ on $X_w$ is just $B\mapsto \sigma(B)$, which induces an action on the cohomology. The character $R^G_{T_w}1$ thus extends to the function $$g\sigma\mapsto \tr(g\sigma|H^{\ast}_c(X_w,\ladic)),$$ denoted by $R^{G\sigma}_{T_w\sigma}1$. This is a particular case of the Deligne-Lusztig induction for non-connected reductive groups developed by Digne and Michel \cite{DM94}. More generally, given an $F$-stable and $\sigma$-stable Levi factor of a $\sigma$-stable parabolic subgroup, we have the maps $R^{G\sigma}_{L\sigma}$, that sends $L^F$-invariant functions on $L^F\sigma$ to $G^F$-invariant functions on $G^F\sigma$.

Each irreducible character $\chi$ of $\GL_n(q)$ is induced from an irreducible character $\chi_L$ of $L^F$, where $L$ is an $F$-stable Levi subgroup as in the setting of Theorem \ref{LS}. If $L$ is moreover an $\sigma$-stable factor of some $\sigma$-stable parabolic subgroup, and $\chi_L$ is a $\sigma$-stable character of $L^F$, then $\chi$ is also $\sigma$-stable. Suppose that we know how to calculate $\tilde{\chi}_L$, an extension of $\chi_L$ to $L^F\sigma$, then the character formula will allow us to calculate the values of $R^{G\sigma}_{L\sigma}\tilde{\chi}_L$, which coincides with the values of an extension of $\chi$ to $G^F\sigma$.

\subsection*{Quadratic-Unipotent Characters}\hfill

However, there exist some $\sigma$-stable characters of $G^F$ that can not be obtained by the above procedure. Let us look at some examples for $n=2$, $3$ and $4$. We take as $\sigma$ the automorphism $g\mapsto \mathscr{J}_ng^{-t}\mathscr{J}_n^{-1}$ with $g\in\GL_n(k)$, where

\begin{equation}
\mathscr{J}_2=
\left(
\begin{array}{cc}
& 1\\
-1& 
\end{array}
\right),~\mathscr{J}_3=
\left(
\begin{array}{ccc}
& & 1\\
&1& \\
1& &
\end{array}
\right).
\end{equation}
If $G=\GL_2(k)$, and $T$ is the maximal torus consisting of diagonal matrices, $1$ the trivial character of $\mathbb{F}_q^{\ast}$, $\eta$ the order 2 irreducible character of $\mathbb{F}_q^{\ast}$, and if we denote by $\theta$ the character $(1,\eta)$ of $T^F\cong \mathbb{F}_q^{\ast}\times\mathbb{F}_q^{\ast}$, then one can verify that $R^G_T\theta$ is an $\sigma$-stable irreducible character while $\theta\in\Irr(T^F)$ is not $\sigma$-stable.

Besides, a priori, the map $R^{G\sigma}_{L\sigma}$ is not defined for $L^F\lb\sigma\rb$, but for the normaliser $N_{G\lb\sigma\rb}(L,P)$ (the set of elements that simultaneously normalise $L$ and $P$, with $L$ being a Levi factor of $P$). If $L$ is a $\sigma$-stable Levi factor of a $\sigma$-stable parabolic subgroup $P$, then $N_{G\lb\sigma\rb}(L,P)=L\lb\sigma\rb$, otherwise, the two groups are not the same. In fact, what really matters is whether $N_{G\lb\sigma\rb}(L,P)$ meets the connected component $G\sigma$. Now we take $L= \GL_2(k)\times k^{\ast}$ and a character with semi-simple part $(\Id,\eta)$ with respect to this direct product. In this case, $L$ itself is not $\sigma$-stable. In fact, it is not conjugate to any $\sigma$-stable Levi subgroup. So $N_{G\lb\sigma\rb}(L,P)\subset G$, regardless of the choice of $P$.

The above examples are typical. Let $L=G_1\times G_2$ be a Levi subgroup of $G$, where $G_1\cong\GL_m(k)$ and $G_2\cong\GL_{n-m}(k)$. Let $\chi_1$(resp. $\chi_2$) be a unipotent irreducible character of $G_1^F$(resp. $G_2^F$). The character $\chi_L:=\chi_1\boxtimes\chi_2\eta$ (or $\chi_1\eta\boxtimes\chi_2$) always induces a $\sigma$-stable irreducible character of $G^F$, where we regard $\eta$ as a central character of $G_1^F$ or $G_2^F$. But $L$ does not fit into Deligne-Lusztig theory for non-connected groups: either $L$ is not conjugate to a $\sigma$-stable Levi factor of any $\sigma$-stable parabolic subgroup, or $\chi_L$ is not a $\sigma$-stable character of $L^F$. The irreducible characters of $\GL_n(q)$ that are induced from $\chi_L$ as above are called \textit{quadratic-unipotent}. They are parametrised by the 2-partitions of $n$. Their extensions to $\GL_n(q)\lb\sigma\rb$ have been computed by J.-L.Waldspurger by using character sheaves for non-connected groups. The result is as follows.

Let $(\mu_+,\mu_-)$ be a 2-partition of $n$, to which is associated the data $(\varphi_+,\varphi_-,h_1,h_2)$, where $\varphi_+$ (resp. $\varphi_-$) is an irreducible character, determined by the 2-quotient of $\mu_+$ (resp. $\mu_-$), of the Weyl group $\mathfrak{W}_+$ (resp. $\mathfrak{W}_-$) of type $\text{C}_{N_+}$ (resp. $\text{C}_{N_-}$), while $h_1$ and $h_2$ are two non negative integers related to the 2-cores of $\mu_+$ and $\mu_-$. We have $n=2N_++2N_-+h_1(h_1+1)+h_2^2$.

\begin{Thm}[Waldspurger]\label{Walds}
The extension of the quadratic-unipotent character of  $\GL_n(q)$ associated to $(\mu_+,\mu_-)$ is given up to a sign by
\begin{equation}
R^{G\sigma}_{\varphi}:=\frac{1}{|\mathfrak{W}_+|}\frac{1}{|\mathfrak{W}_-|}\sum_{\substack{w_+\in \mathfrak{W}_+\\ w_-\in \mathfrak{W}_-}}\varphi_+(w_+)\varphi_-(w_-)R^{G\sigma}_{L_{\mathbf{w}}\sigma}\phi_{\mathbf{w}}.
\end{equation}
\end{Thm}
In the above expression, $L_{\mathbf{w}}$ is a $\sigma$-stable and $F$-stable Levi factor, isomorphic to $\GL_{h_1(h_1+1)+h_2^2}(k)\times T_{w_+}\times T_{w_-}$, of a $\sigma$-stable parabolic subgroup, and $\phi_{\mathbf{w}}$ is some kind of twisted tensor product of the characters $\phi(h_1,h_2)$, $\tilde{1}$ and $\tilde{\eta}$, where $\phi(h_1,h_2)$ is a cuspidal function on $\GL_{h_1(h_1+1)+h_2^2}(q)\sigma$, $\smash{\tilde{1}}$(resp. $\smash{\tilde{\eta}}$) is the extension of the linear character $1$(resp. $\eta$) of $T_{w_+}^F$(resp. $T_{w_-}^F$), the latter being defined via the product of norm maps. 

Suppose that $\mu_-$ is the empty partition and so we have a unipotent character. If $\mu_+$ has trivial 2-core or the 2-core $(1)$ according to the parity of $n$, then $L_{\mathbf{w}}$ becomes a $\sigma$-stable maximal torus and $\mathfrak{W}_+$ is isomorphic to $W^{\sigma}_G$, the $\sigma$-fixed subgroup of the Weyl group of $G$. The formula in this particular case resembles (\ref{eqthm1}) and was obtained by Digne and Michel \cite{DM94} without using character sheaves.

\subsection*{Parametrisation of $\sigma$-Stable Characters of $\GL_n(q)$}\hfill

A general $\sigma$-stable irreducible character is the product of a quadratic-unipotent component and a component that looks like  induced from an $\sigma$-stable Levi factor of a $\sigma$-stable parabolic subgroup. We refer to Proposition \ref{IrrLIstable} and Proposition \ref{IrrMstable} for the details. A brief summary is given below. 

Suppose that $\chi$ is a $\sigma$-stable irreducible character corresponding to $(M,\theta,\varphi)$ as in Theorem \ref{LS}. We write $\theta=(\alpha_i)_i$ with respect to the the decomposition of $M^F$ into a product of some $\GL_{n_i}(q^{r_i})$'s, where $\alpha_i$ are some characters of $\mathbb{F}_{q^{r_i}}^{\ast}$ and we have omitted the determinant map from the notation. Here we observe a restriction on the $\alpha_i$'s that is necessary for $\chi$ to be $\sigma$-stable. It is easy to see that the action of $\sigma$ sends $\chi$ to the character associated to $(\sigma(M),\sigma_{\ast}\theta,\sigma_{\ast}\varphi)$. According to the parametrisation of the irreducible characters of $\GL_n(q)$, there exists some $$g\in N_{G^F}(\sigma(M),M)=\{x\in G^F\mid x\sigma(M)x^{-1}=M\}$$ such that $\sigma_{\ast}\theta=\ad^{\ast}g\theta$. Note that the value of $\alpha_i$ only depends on the determinant of the corresponding factor. The action of $\sigma$ inverts the determinant while the conjugation by $g$ does not change the determinant. We can then conclude that for each $\alpha_i$, its inverse $\alpha_i^{-1}$ is also a factor of $\theta$. The factors satisfying $\alpha_i^{-1}=\alpha_i$ form the quadratic-unipotent part.

The parametrisation is as follows. Fix a $\sigma$-stable maximal torus $T$ contained in a $\sigma$-stable Borel subgroup $B\subset G$ and identify the simple roots with the vertices of the Dynkin diagram of $G$. If $I$ is a $\sigma$-stable subdiagram of the Dynkin diagram of $G$, then it defines a $\sigma$-stable Levi factor $L_I$ of some $\sigma$-stable standard  parabolic subgroup containing $B$. We fix an isomorphism between $L_I$ and a direct product of some smaller general linear groups. It has at most one $\sigma$-stable direct factor, denoted by $L_0$, and so $L_I\cong L_0\times L_1$ so that $\sigma$ non-trivially permutes the direct factors of $L_1$. All $\sigma$-stable standard Levi subgroup corresponds to such an $I$. We associate a quadratic-unipotent character $\chi_0$ to $L_0^F$, and a pair of characters, with semi-simple parts $\alpha_i$ and $\alpha_i^{-1}$ respectively, to each pair of direct factors of $L_1$ that are exchanged by $\sigma$. By defining the unipotent parts of the character in a way compatible with the action of $\sigma$, we obtain a $\sigma$-stable character of $L_1^F$, denoted by $\chi_1$. Then
$R^G_L(\chi_0\boxtimes\chi_1)$ is a $\sigma$-stable irreducible character of $G^F$ (for suitable $\alpha_i$'s), and all $\sigma$-stable irreducible characters of $\GL_n(q)$ are obtained this way. 

We will calculate the extension of $\chi_L=\chi_0\boxtimes\chi_1$ to $L^F\sigma$, and then apply the map $R^{G\sigma}_{L\sigma}$. Note that if we regard $\chi$ as induced from $(M,\theta,\varphi)$ following Theorem \ref{LS}, then $M$ is not necessarily $\sigma$-stable.

\subsection*{Extensions of $\sigma$-Stable Characters}\hfill

The extension of the quadratic-unipotent part given by Waldspurger's theorem, the problem can be reduced to the following.

\begin{Pro}
Put $L_1=G_0\times G_0$, $G_0=\GL_m(k)$, and let $\sigma_0$ be an automorphism of $G_0$ of order 2. Denote by $F_0$ the Frobenius of $\GL_m(k)$ that sends each entry to its q-th power. Define an automorphism $\sigma$ of $L_1$ by 
\begin{equation}
(g,h)\longmapsto (\sigma_0(h),\sigma_0(g)),
\end{equation}
and a Frobenius $F$ by
\begin{itemize}
\item Linear Case:
$(g,h)\longmapsto (F_0(g),F_0(h))$, 
\item Unitary Case:
$(g,h)\longmapsto (F_0(h),F_0(g))$.
\end{itemize}
The problem is to decompose the extension of a $\sigma$-stable irreducible character of $L_1^F$ to $L_1^F\lb\sigma\rb$ as a linear combination of Deligne-Lusztig characters.
\end{Pro}
Let us first look at the linear case. We have $L_1^F=G_0^{F_0}\times G_0^{F_0}$. Let $\chi$ be a unipotent character of $G_0^{F_0}$. Then $\chi\boxtimes\chi\in\Irr(L_1^F)$ is $\sigma$-stable. In order to calculate its extension, we separate $\sigma$ into two automorphisms, one sending $(g,h)$ to $(\sigma_0(g),\sigma_0(h))$, the other one, denoted by $\tau$,  sending $(g,h)$ to $(h,g)$. Denote by $\tilde{\chi}$ an extension of $\chi$ to $G_0^{F_0}\lb\sigma_0\rb$. Consider the $\tau$-stable character $\tilde{\chi}\boxtimes\tilde{\chi}$ of $G_0^{F_0}\lb\sigma_0\rb\times G_0^{F_0}\lb\sigma_0\rb$. Its extension to $(G_0^{F_0}\lb\sigma_0\rb\times G_0^{F_0}\lb\sigma_0\rb)\rtimes\lb\tau\rb$ restricts to an irreducible character $\bar{\chi}$ of $L_1^F\lb\sigma\rb$, regarded as a subgroup of $(G_0^{F_0}\lb\sigma_0\rb\times G_0^{F_0}\lb\sigma_0\rb)\rtimes\lb\tau\rb$. This gives an extension of $\chi\otimes\chi$. Some linear algebra calculation shows that $\bar{\chi}((g,h)\sigma)=\chi(g\sigma_0(h))$. The latter is the value of a character of $\GL_n(q)$, the decomposition of which into Deligne-Lusztig characters is known.

The unitary case is a little more complicated and relies on the result of the linear case. In this case, $L_1^F\cong G_0^{F_0^2}$, and the action of $\sigma$ on $G_0^{F_0^2}$ is given by $g\mapsto \sigma_0F_0(g)$, which can be thought of as another Frobenius endomorphism. That is where the Shintani descent intervenes, which relates the functions on $G_0^{F^2_0}\sigma_0F_0$ to the functions on $G_0^{\sigma_0F_0}.F_0^2$. Note that $(\sigma_0F_0)^2=F_0^2$ acts trivially on $G_0^{\sigma_0F_0}$. We know how to calculate the characters of $G_0^{\sigma_0F_0}\cong\GL^{-}_m(q)$, which extends trivially to $G_0^{\sigma_0F_0}.F_0^2$. Thus, we obtain the extension of a character to $L_1^F\sigma$.  The result is as follows.

Let $\theta_1$ be a $\sigma$-stable linear character of $L_1^F$ and let $\varphi\in\Irr(W^{\sigma}_{L_1})^F$. Note that a $\sigma$-stable linear character extends trivially to $L_1^F\sigma$, and that $W^{\sigma}_{L_1}$ is in fact a product of symmetric groups.

\begin{Thm}
Let $\chi_{L_1}$ be a $\sigma$-stable irreducible character of $L_1^F$ defined by $(\theta,\varphi)$. Then, for some choice of $\tilde{\varphi}$, the extension of $\chi_{L_1}$ to $L_1^F\sigma$ is given up to a sign by
\begin{equation}
R^{L_1\sigma}_{\varphi}\tilde{\theta}_1=|W_{L_1}^{\sigma}|^{-1}\sum_{w\in W_{L_1}^{\sigma}}\tilde{\varphi}(wF)R_{T_w\sigma}^{L_1\sigma}\tilde{\theta}_1.
\end{equation}
\end{Thm}

Combined with the theorem of Waldspurger, it gives the main theorem below.

According to the parametrisation of $\sigma$-stable characters, each  $\chi\in\Irr(\GL_n(q))^{\sigma}$ is of the form $R^G_L(\chi_0\boxtimes\chi_1)$, where $L\cong L_0\times L_1$ is a $\sigma$-stable and $F$-stable Levi factor of some $\sigma$-stable parabolic subgroup, $\chi_0$ is a quadratic-unipotent character of $L_0^F$ and $\chi_1$ is a $\sigma$-stable irreducible character of $L_1^F$ whose semi-simple part and  unipotent part are defined by $\theta_1\in\Irr_{reg}(L_1^F)$ and $\varphi\in\Irr(W_{L_1}^{\sigma})$ respectively. The notations of Theorem \ref{Walds} are used in the following theorem.

\begin{namedtheorem}[Main Theorem]
For some choice of $\tilde{\varphi}$, the extension of $\chi$ is up to a sign given by
\begin{equation}
\tilde{\chi}|_{G^F\sigma}=|W^{\sigma}_{L_1}\times\mathfrak{W}_+\times\mathfrak{W}_-|^{-1}\!\!\!\!\!\!\!\!\!\!\!\!\!\!\!\!\!\!\!\!\sum_{(w,w_+,w_-)\in W^{\sigma}_{L_1}\times\mathfrak{W}_+\times\mathfrak{W}_-}\!\!\!\!\!\!\!\!\!\!\!\!\!\!\!\!\tilde{\varphi}(wF)\varphi_+(w_+)\varphi_-(w_-)R^{G\sigma}_{(T_w\times L_{\mathbf{w}})\sigma}(\tilde{\theta}_1\widetilde{\boxtimes}\phi_{\mathbf{w}}).
\end{equation}
\end{namedtheorem}
The symbol $\widetilde{\boxtimes}$, which we define in \S 9.1, is some kind of tensor product on the component $(T_w\times L_{\mathbf{w}})\sigma$.

\subsection*{Computation of the Character Table}\hfill

We notice that in $R^{L_0\sigma}_{L_{\mathbf{w}}\sigma}\phi_{\mathbf{w}}$ appears the generalised Green functions defined on the centraliser of a semi-simple element in $L_0^F\sigma$, which is in general a product of $\GL^{\pm}_m(q)$, $\Sp_{2m}(q)$, $\SO_{2m+1}(q)$ and $\SO^{\pm}_{2m}(q)$, where the negative sign means that the Frobenius is twisted by a graph automorphism of order 2. The Green functions of finite classical groups were first tackled in \cite{S83}. In \cite{L85V}, Lusztig gives an algorithm of calculating the generalised Green functions of classical groups. The values of the (generalised) Green functions have since been computed by various people, and we will not make explicit theirs values except in the examples.

In the implementation of the computation using the above formula, there is one more technical issue: determination of the set
\begin{equation*}
\begin{split}
A(s\sigma,\tau,h_1,h_2)^F=&\{h\in G^F\mid hs\sigma h^{-1}\in L_{\mathbf{w}}\sigma \text{ is isolated with $C_{L_{\mathbf{w}}}(hs\sigma h^{-1})$ isomorphic }\\&\text{ to the product of $\Sp_{h_1(h_1+1)}(k)\times\Ort_{h^2_2}(k)$ and a torus}\}
\end{split}
\end{equation*}
which is involved in the character formula for $R^{L_0\sigma}_{L_{\mathbf{w}}\sigma}\phi_{\mathbf{w}}$. In Section \ref{ComputeA}, we reduce the problem to some computations in various Weyl groups, which should be finished case by case. 

\subsection*{Organisation of the Article}\hfill

Due to the potential interest to a general audience, also due to the fact that collecting the necessary information from the literature for implementing the computation is already an effort, we include substantial review of established results in this article. 

In Section 1, we recall many results on finite classical groups that will be used in the computation of the character table. In Section 2 and 3, we recall some results on non-connected linear algebraic groups and the theory of Deligne-Lusztig induction for the finite groups that arise from non-connected groups. These results will be used throughout the article. In Section 4, we specialise to the group $\GL_n(q)\rtimes\lb\sigma\rb$ and give some explicit information on this group. In Section 5 and 6, we give the parametrisation of $\sigma$-stable irreducible characters of $\GL_n(q)$ and the conjugacy classes contained in $\GL_n(q)\sigma$. The parametrisations are give in terms of some very explicit combinatorial data. In section 7, we collect some information on the eigenvalues of the Frobenius endomorphism on intersection cohomology groups of Deligne-Lusztig varieties and recall some results on Shintani descent. In Section 8 we recall some general results on character sheaves and give some examples in the case that we are interested in. The contents of Section 7 and 8 will only be used in Section 9, where we essentially prove the main theorem, which is stated in Section 11. In Section 10, we tackle a specific problem in the computation of Deligne-Lusztig inductions, namely the determination of the set $\{h\in \GL_n(q)\mid hs\sigma h^{-1}\in M\sigma\}$ for some semi-simple element $s\sigma$ and some Levi subgroup $M$. We give a procedure to reduce the problem to some computations in Weyl groups. This allows us to write the formula in the main theorem in terms of purely combinatorial data. In the last section, we explicitly determine the complete character table of $\GL_n(q)\rtimes\lb\sigma\rb$ for $n=2$ and $3$.
 
\subsection*{Acknowledgement.}
This article is part of my thesis prepared at Universit\'e de Paris. I thank Olivier Brunat for bringing to my attention Shintani descent. I thank Fran\c cois Digne for carefully reading and giving numerous comments on an earlier version of the article. I thank my advisor Emmanuel Letellier for carefully reading the first four sections and giving many comments. I thank Jean Michel for answering some questions and reading a draft version of this article. I thank the anonymous referee for pointing out some mistakes and many typos.

\setcounter{tocdepth}{1}
\numberwithin{Thm}{subsection}
\numberwithin{equation}{subsubsection}
\addtocontents{toc}{\protect\setcounter{tocdepth}{1}}
%%%%%%%%%%%%%%%%%%% Partie I

%%%%%%%%%%%%%%%%%%
\section{Finite Classical Groups}

\subsection{Partitions}
\subsubsection{}
We denote by $\mathcal{P}_n$ the set of all partitions of the integer $n\ge 0$ and by $\mathcal{P}$ the union $\sqcup_n\mathcal{P}_n$. A partition is written as $\lambda=(\lambda_1\ge\lambda_2\ge\cdots)$, a decreasing sequence of positive integers, or as $\lambda=(1^{m_1},2^{m_2},\ldots)$ where $m_i$ is the multiplicity of $i$ that appears in $\lambda$. Each $\lambda_i$ is called a \textit{part} of $\lambda$. We denote by $|\lambda|$ the size of $\lambda$ and $l(\lambda)$ the length of $\lambda$. Denote by $\mathcal{P}_n(2)$ the set of all $2$-partitions of $n$, i.e. the couples of partitions $(\lambda^{(0)},\lambda^{(1)})\in\mathcal{P}\times\mathcal{P}$ satisfying $|\lambda^{(0)}|+|\lambda^{(1)}|=n$.

\subsubsection{}\label{kappa}
A partition $\lambda=(1^{m_1},2^{m_2},\ldots)$ is called \textit{symplectic} if $m_i$ is even for every odd $i$. To each symplectic partition $\lambda$ we associate a finite set $\bs{\kappa}(\lambda):=\{i\text{ even}|m_i>0\}$ and put $\kappa(\lambda)=|\bs{\kappa}(\lambda)|$. We denote by $\mathcal{P}^{sp}_n\subset\mathcal{P}_n$ the subset of symplectic partitions. A partition $\lambda=(1^{m_1},2^{m_2},\ldots)$  is called \textit{orthogonal} if $m_i$ is even for every even $i$. To each orthogonal partition $\lambda$ we associate a finite set $\bs{\kappa}(\lambda):=\{i\text{ odd}|m_i>0\}$ and put $\kappa(\lambda)=|\bs{\kappa}(\lambda)|$. We denote by $\mathcal{P}^{ort}_n\subset\mathcal{P}_n$ the subset of orthogonal partitions. We may omit $\lambda$ from the notations $\bs{\kappa}(\lambda)$ and $\kappa(\lambda)$ if no confusion arises. The orthogonal partitions with $\kappa=0$ are called \textit{degenerate}. The subset of non degenerate partitions of $n$ is denoted by $\mathcal{P}^{ort,nd}_{n}$ and that of degenerate partitions is denoted by $\mathcal{P}^{ort,d}_{n}$.

\subsubsection{}\label{quotcore}
Given a partition $\lambda$ of size $n$, we take $r\ge l(\lambda)$, and we put $\delta_r=(r-1,r-2,\ldots,1,0)$. Let $(2y_1>\cdots>2y_{l_0})$ and $(2y_1'+1>\cdots>2y'_{l_1}+1)$ be the even parts and the odd parts of $\lambda+\delta_r$, where the sum is taken term by term and $\lambda$ is regarded as an decreasing sequence of integers $(\lambda_i)_i$ with $\lambda_i=0$ for $i>l(\lambda)$. Denote by $\smash{\lambda^{(0)}}$ the partition defined by $\smash{\lambda^{(0)}_k}=y_k-l_0+k$ and denote by $\smash{\lambda^{(1)}}$ the partition defined by $\smash{\lambda^{(1)}_k}=y'_k-l_1+k$. Then $(\lambda^{(0)},\lambda^{(1)})_r$ is a 2-partition that depends on $r$. Changing the value of $r$ will permutes $\lambda^{(0)}$ and $\lambda^{(1)}$. The \textit{2-quotient}  of $\lambda$ is then the unordered pair of partitions $(\lambda^{(0)},\lambda^{(1)})$. 

Denote by $\lambda'$ the partition that has as its parts the numbers $2s+t$, $0\le s\le l_t-1$, $t=0,1$. We have $l(\lambda')=l(\lambda)$. The \textit{2-core} of $\lambda$ is the partition defined by $(\lambda'_k-l(\lambda)+k)_{1\le k\le l(\lambda)}$. It is independent of $r$ and, if non trivial, is necessarily of the form $(d,d-1,\ldots,2,1)$, for some $d\in\mathbb{Z}_{>0}$. Fixing $r$, the above constructions give a bijection $$\left\{\begin{array}{c}\text{The partitions of $n$}\\\text{with 2-core $(d,d-1,\ldots,2,1)$}\end{array}\right\}\longleftrightarrow\left\{\text{The 2-partitions of $\frac{1}{2}(n-d(d+1)/2)$}\right\}.$$

\subsection{Weyl Groups}
Some basic facts about Weyl groups of type $\text{B}_m$, $\text{C}_m$ and $\text{D}_m$.
\subsubsection{}\label{couronne}
Fix a positive integer $m$. Denote by $w_0$ the permutation $(1,-1)(2,-2)\cdots(m,-m)$ of the set $$I=\{1,\ldots,m,-m,\ldots,-1\}.$$ The set of permutations of $I$ that commute with $w_0$ is identified with $(\mathbb{Z}/2\mathbb{Z})^m\rtimes\mathfrak{S}_m$. This is the Weyl group of type $B_m$ and $C_m$, which will be denoted by $\mathfrak{W}_m$. Here we identify $\mathbb{Z}/2\mathbb{Z}$ with $\bs\mu_2$. An element of $\mathfrak{W}_m$ can be written as
\begin{equation}
\begin{split}
w&=((\epsilon_1,\ldots,\epsilon_m),\tau)\in(\mathbb{Z}/2\mathbb{Z})^m\rtimes\mathfrak{S}_m,\\
%w'&=((\epsilon'_1,\ldots,\epsilon'_n),\mu')\in(\mathbb{Z}/2\mathbb{Z})^n\rtimes\mathcal{S}_n,
\end{split}
\end{equation}
with $((1,\ldots,1,\epsilon_i,1,\ldots,1),1)$, $\epsilon_i=-1$ being the permutation $(i,-i)$, and $((1,\ldots,1),\tau)$ being the permutation $$i\mapsto \tau(i),\quad-i\mapsto\tau(-i)=-\tau(i).$$

\subsubsection{}\label{wreath}
The permutation $\tau$ is decomposed into cycles $\tau=c_{I_1}\cdots c_{I_l}$, where the disjoint subsets $I_r\subset\{1,\ldots,m\}$ form a partition of $\{1,\ldots,m\}$ and $c_{I_r}$ is a circular permutation of the indices in $I_r$. The permutation $\tau$ determines a partition $(\tau_1,\ldots,\tau_l)$ of $m$, also denoted by $\tau$, with $\tau_r=|I_r|$ for any $1\le r\le l$.

For all $1\le r\le l$, put $\bar{\epsilon}_r=\prod_{k\in I_r}\epsilon_k$ and $\underline{\epsilon}_r=(\epsilon_k)_{k\in I_r}$. Define the permutations
\begin{equation}
\tau^{(0)}=\prod_{\bar{\epsilon}_r=1}c_{I_r},\quad
\tau^{(1)}=\prod_{\bar{\epsilon}_r=-1}c_{I_r},
\end{equation}
so that $\tau=\tau^{(0)}\tau^{(1)}$. Also denote by $\tau^{(0)}=\smash{(\tau^{(0)}_r)}$ and $\tau^{(1)}=\smash{(\tau^{(1)}_r)}$ the associated partitions. We then have a 2-partition $(\tau^{(0)},\tau^{(1)})$, which determines the conjugacy class of $w$. We sometimes call it a signed partition of $m$. The conjugacy classes and irreducible characters of $\mathfrak{W}_m$ are both parametrised by the 2-partitions of size 
$m$.  

\subsubsection{}\label{W^D}
The Weyl group of type $D_m$, denoted by $\mathfrak{W}^D_m$, is the subgroup of $\mathfrak{W}_m$ consisting of the elements $((\epsilon_1,\ldots,\epsilon_m),\tau)$ such that $\prod\epsilon_i=1$. For any $a$, let $$\sgn:\mathfrak{W}_a\rightarrow \{\pm 1\}$$ be the map whose kernel is $\mathfrak{W}_a^D$. The parametrisation of the conjugacy classes of $\mathfrak{W}^D_m$ is given as follows. (See \cite[Proposition 25]{Ca}) Let $\tau$ be a signed partition. If each part of $\tau$ is even and the $\bar{\epsilon}_r$'s are all equal to $1$, then the conjugacy class of $\mathfrak{W}_m$ corresponding to $\tau$ splits into two classes of $\mathfrak{W}^D_m$. Otherwise, this conjugacy class restricts to one single class of $\mathfrak{W}^D_m$ if the restriction is non-empty.

\subsubsection{}
Let $m_0$ and $m$ be some non negative integers and put $n=2m+m_0$. Let $L$ be a Levi subgroup of $G=\Sp_n(k)$ (resp. $\SO_n(k)$) which is isomorphic to $\Sp_{m_0}(k)\times(k^{\ast})^m$ (resp. $\SO_{m_0}(k)\times(k^{\ast})^m$), then $N_G(L)/L$ is isomorphic to $\mathfrak{W}_m$ unless $m_0=0$ and $G$ is the orthogonal group, in which case $N_G(L)/L\cong \mathfrak{W}_m^D$.

\subsection{Finite Classical Groups}
Details concerning the followings facts can be found in \cite[\S 3.3, \S 3.4, \S 3.5, \S 3.6, \S 3.7]{Wi} and \cite[\S 2.6]{Wall}.

\subsubsection{}
Let $V$ be an $n$ dimensional vector space over $\mathbb{F}_q$.  Here we recall the classification of non degenerate symmetric bilinear forms over $V$ when $\mathbb{F}_q$ has odd characteristic. The non degenerate Hermitian forms and non degenerate skew symmetric bilinear forms are unique up to equivalence.  

Fix a non-square $\lambda$ in $\mathbb{F}_q$. Each non degenerate symmetric bilinear form $Q$ on $V$ is either equivalent to the diagonal form $\diag(1,\ldots,1)$ or the diagonal form $\diag(1,\ldots,1,\lambda)$, so there are exactly two equivalence classes and they are distinguished by the determinant. If $n=2N$ is even, we say that the form is split if there is a totally isotropic subspace of dimension $N$, and non split otherwise. In the case of even dimension, the diagonal form $\diag(1,\ldots,1)$ is non split exactly when $q\equiv 3\mod 4$ and $N$ is odd. We regard $\eta$ as a character of $\mathbb{F}_q^{\ast}\setminus(\mathbb{F}_q^{\ast})^2$. Put $$\eta(Q):=\eta\left((-1)^{[n/2]}\det(Q)\right).$$ Then in the even dimensional case $\eta(Q)=+1$ if and only if $Q$ is split. In the odd dimensional case, we can also use $\eta(Q)$ to distinguish the two forms.

A symmetric bilinear form $Q$ defines an orthogonal group as a subgroup of $\GL(V)=\GL_n(q)$. If $n$ is even, we say that the group is split if $Q$ is split, in which case it is denoted by $\Ort_n^+(q)$, and that the group is non split otherwise, in which case it is denoted by $\Ort_n^-(q)$. If $n$ is odd, we use the same notations, with the sign determined by the value of $\eta(Q)$. The special orthogonal groups are then denoted by $\SO^+_n(q)$ and $\SO^-_n(q)$ accordingly. If $n$ is odd, then the two non equivalent forms define isomorphic orthogonal groups. In general, the '$+$' for the split groups will be omitted from the notation. We will however keep '$\pm$' in the notation of odd orthogonal groups whenever it is necessary to distinguish the defining forms. The subgroup of $\GL_n(q^2)$ defined by a non degenerate Hermitian form will be denoted by $\GL_n^-(q)$.

The notations and terminology introduced above come from the language of algebraic groups. Suppose that $G$ is one of the groups $\GL_n(k)$, $\Sp_n(k)$, $\SO_n(k)$ and $\Ort_n(k)$, and $F$ denotes the Frobenius endomorphism. If the Frobenius $F$ is split, then the associated finite group $G^F$ is one of the split groups $\GL_n(q)$, $\Sp_n(q)$, $\SO_n(q)$ and $\Ort_n(q)$, and if $F$ induces an involution of the Dynkin diagram, then the associated finite group is one of the non split groups $\GL^{-}_n(q)$, $\SO^{-}_n(q)$ and $\Ort^{-}_n(q)$ (only when $n$ is even for the orthogonal groups). We will most often work with algebraic groups instead of linear algebra, except when the parametrisation of unipotent classes is in question.

\subsubsection{}
The cardinality of finite classical groups is as follows. 
\begingroup
\allowdisplaybreaks
\begin{align}
|\GL_n(q)|&=q^{n(n-1)/2}\prod_{i=1}^n(q^i-1),\\
|\GL_n^{-}(q)|&=q^{n(n-1)/2}\prod_{i=1}^n(q^i-(-1)^i),\\
|\Sp_{2m}(q)|&=q^{m^2}\prod_{i=1}^m(q^{2i}-1),\\
|\SO_{2m+1}(q)|&=q^{m^2}\prod_{i=1}^m(q^{2i}-1),\\
|\SO_{2m}(q)|&=q^{m(m-1)}(q^m-1)\prod_{i=1}^{m-1}(q^{2i}-1),\\
|\SO^{-}_{2m}(q)|&=q^{m(m-1)}(q^m+1)\prod_{i=1}^{m-1}(q^{2i}-1).
\end{align}
\endgroup

\subsection{Centralisers of Unipotent Elements}\label{uni-centraliser}
\subsubsection{}
Suppose that $G^F$ is isomorphic to one of the groups: $\Sp_n(q)$, $\Ort_n(q)$, $\Ort_n^{-}(q)$. We regard $G^F$ as a subgroup of $\GL(V)$ defined by some bilinear form $Q(-,-)$ on the $n$-dimensional vector space $V$ over $\mathbb{F}_q$. Let $\epsilon\in\{\pm\}$ be such that $\epsilon=+$ if the form is symmetric, and $\epsilon=-$ if it is skew symmetric. According to \cite{SS}, each nilpotent element $x\in \mathfrak{g}^F:=\Lie(G)^{F}$ has a normal form described as follows.
\begin{Thm}(\cite[\S 2.19]{SS})\label{SS2.19}
Let $x\in \mathfrak{g}^F$ be any nilpotent element. There exists a basis of $V$ $$\{e_k^{(a)},f_l^{(b)},g_l^{(b)}\mid 1\le k\le s,~1\le l\le t,~a\in\mathbb{Z}/d_k\mathbb{Z},~b\in\mathbb{Z}/\delta_l\mathbb{Z}\},$$ where $s$, $t$, $d_k$ and $\delta_l$ are some integers, such that 
\begin{itemize}
\item[(i)] The action of $x$ on $V$ is given by
\begin{equation*}
e_k^{(a)}\longmapsto e_k^{(a+1)},\quad
f_l^{(b)}\longmapsto f_l^{(b+1)},\quad
g_l^{(b)}\longmapsto g_l^{(b+1)},
\end{equation*}
for any $k$, $l$, $a$ and $b$;
\item[(ii)] With respect to this basis, the bilinear form $Q$ is given by 
\begin{equation*}
\begin{split}
Q(e_k^{(a)},e_k^{(d_k-a-1)})&=(-1)^a\lambda_k,\text{ for some $\lambda_k\in\mathbb{F}_q^{\ast}$},\\
Q(f_l^{(b)},g_l^{(\delta_l-b-1)})&=\epsilon Q(g_l^{(\delta_l-b-1)},f_l^{(b)})=(-1)^b,
\end{split}
\end{equation*}
with the values on all other pairs of the basis vectors being zero.
\end{itemize}
\end{Thm}
\begin{Cor}(\cite[Corollary 2.20]{SS})\label{SS2.20}
We have,
\begin{equation}
(-1)^{d_k}=(-1)^{\delta_l-1}=-\epsilon.
\end{equation}
\end{Cor}

These results are also valid over $\bar{\mathbb{F}}_q$.

\subsubsection{}\label{cent-uni}
The structure of the centraliser of $x$ follows from the above theorem. For any $1\le k\le s$ and any $1\le l\le t$, let $W_k\subset V$ be the subspace spanned by the basis vectors $e_k^{(a)}$, $a\in\mathbb{Z}/d_k\mathbb{Z}$, let $W'_l$ be the subspace spanned by $f_l^{(b)}$, $b\in\mathbb{Z}/\delta_l\mathbb{Z}$, and let $W''_l$ the subspace spanned by $g_l^{(b)}$, $b\in\mathbb{Z}/\delta_l\mathbb{Z}$. Then $V=\oplus_kW_k\bigoplus\oplus_l W'_l\bigoplus\oplus_lW''_l$. For any $d\in\mathbb{Z}_{>0}$, let $m_d$ be the number of $d_k$ that are equal to $d$, and for $\delta\in\mathbb{Z}_{>0}$, $m_{\delta}$ is similarly defined. Let $W_d$ be a copy of one of the $W_k$ such that $d_k=d$ and let $W'_{\delta}$ be a copy of one of the $W'_l$ such that $\delta_l=\delta$. Then there are some $m_d$ (resp. $2m_{\delta}$) dimensional vector spaces $U_d$ (resp. $U'_{\delta}$) such that $$V\cong (\oplus_dU_d\otimes W_d)\bigoplus(\oplus_{\delta}U'_{\delta}\otimes W'_{\delta}),$$ and $x$ acts as the maximal Jordan block on each factors $W_d$ and $W'_{\delta}$.

\textit{Symplectic case, $\epsilon=-$}. The restriction of the symplectic form to $U_d\otimes W_d$ is the product of forms on the two factors. On $W_d$, it is the skew symmetric (note that $W_d$ has even dimension) skew diagonal form with alternating entries $(-1)^a$, and on $U_d$, it is the diagonal form with entries $\lambda_k$, where $k$ is such that $d_k=d$. Similarly, we have a product form on each $U'_{\delta}\otimes W'_{\delta}$. On $W'_{\delta}$, it is the symmetric (note that $W'_{\delta}$ has odd dimension) skew diagonal form with alternating entries $(-1)^b$, and on $U'_{\delta}$ it is a skew symmetric form. Now, for any $i\in\mathbb{Z}_{>0}$, let $m_i=m_d$ if $i=d$ and $m_i=2m_{\delta}$ if $i=\delta$. Then $x$ commutes with the subgroup $$R\cong\prod_{\text{$i$ even}}\Ort(U_i)\times\prod_{\text{$i$ odd}}\Sp(U'_i).$$ 

\textit{Orthogonal case, $\epsilon=+$}. By similar arguments, we see that $x$  commutes with the subgroup $$R\cong\prod_{\text{$i$ odd}}\Ort(U_i)\times\prod_{\text{$i$ even}}\Sp(U'_i).$$

In either case, the factor $\Ort(U_i)$ is defined by the diagonal form with entries $\lambda_k$ such that $d_k=i$. We will denote by $Q_i$ this bilinear form. 

\begin{Thm}(\cite[\S 2.23, \S 2.25]{SS})\label{C=VR}
Let $u\in G^F$ be the unipotent element corresponding to $x$. The centraliser of $u$ is the semi-direct product of its unipotent radical and $R$ given above.
\end{Thm}

The integers $m_i$, $i\in\mathbb{Z}_{>0}$, define a partition $\lambda=(1^{m_1},2^{m_2},\ldots)$, which is either symplectic or orthogonal. Let $\kappa$ and $\bs\kappa$ be the $\kappa(\lambda)$ and $\bs\kappa(\lambda)$ defined in \S \ref{kappa}. For any $i\in\bs\kappa$, let $e_i=\eta(Q_i)$. Then $R$ may be alternatively written as $$R\cong\prod_{\text{$i$ odd}}\Sp_{m_i}(q)\times\prod_{\text{$i$ even}}\Ort_{m_i}^{e_i}(q), \text{ if $G^F=\Sp_n(q)$,}$$ $$R\cong\prod_{\text{$i$ odd}}\Ort_{m_i}^{e_i}(q)\times\prod_{\text{$i$ even}}\Sp_{m_i}(q), \text{ if $G^F=\Ort^{\pm}_n(q)$.}$$

Passing to the the algebraic closure, we see that $C_G(u)/C_G(u)^{\circ}\cong(\mathbb{Z}/2\mathbb{Z})^{\kappa}$ in both cases. 
\begin{Rem}(\cite[\S 2.19, \S 2.23, \S 2.25]{SS})
In the case $G^F\cong\GL_n^{-}(q)$, let $u\in G^F$ be a unipotent element whose Jordan blocks are given by a partition $\lambda=(1^{m_1},2^{m_2},\ldots)$. Put $$R=\prod_i\GL_{m_i}^{-1}(q).$$ Then $R$ is the reductive quotient of the centraliser of $u$. Passing to the algebraic closure, this implies that the centraliser of $u$ is connected.
\end{Rem}

\subsection{Parametrisation of Unipotent Classes}
The parametrisation of unipotent conjugacy classes of finite classical groups is well known. We refer to (\cite[Chapter 3, Chapter 7]{LiSe}) for a more complete survey.

\subsubsection{}
Corollary \ref{SS2.20} means, in the case of symplectic groups, the $d_k$'s are even and the $\delta_l$'s are odd, while in the case of orthogonal groups, the $d_k$'s are odd and the $\delta_l$'s are even. This imposes a restriction on the Jordan types of $u$ that can appear in each case. In fact, this is the only restriction, and we have the following results.

The unipotent classes of $\GL_n(k)$ are parametrised by $\mathcal{P}_n$, with the sizes of Jordan blocks given by the corresponding partition. The unipotent classes of $\Sp_n(k)$ are parametrised by $\mathcal{P}^{sp}_n$. These are represented by the Jordan matrices in $\GL_n(k)$ that belong to $\Sp_n(k)$. The unipotent classes of $\textnormal{O}_n(k)$ are parametrised by $\mathcal{P}^{ort}_n$. These are represented by the Jordan matrices in $\GL_n(k)$ that belong to $\Ort_n(k)$. A unipotent class of $\Ort_n(k)$ is called \textit{degenerate} if the corresponding partition is degenerate. A unipotent class of $\Ort_n(k)$ splits into two $\SO_n(k)$-classes if and only if it is degenerate, and restricts to one single $\SO_n(k)$-class otherwise.

\subsubsection{}
To each unipotent element $u\in G^F$ we have associated a partition $\lambda$ encoding the sizes of the Jordan blocks, and a $\kappa$-tuple of signs $(e_i)_{i\in\bs\kappa}$ encoding the equivalence classes of the forms defining the reductive quotient of the centraliser of $u$. These are invariants of, and uniquely determine, the $G^F$-conjugacy class of $u$. With the parametrisation of $G$-conjugacy classes described in the previous paragraph, the parametrisation of $G^F$-classes is reduced to determining which $\kappa$-tuple can occur for a given partition $\lambda$. By Lang-Steinberg theorem, the $G^{\circ F}$-conjugacy classes contained in the $G$-conjugacy class of $u$ are in bijection with the $F$-conjugacy classes in $C_{G^{\circ}}(u)/C_{G^{\circ}}(u)^{\circ}$. This group is isomorphic to $(\mathbb{Z}/2\mathbb{Z})^{\kappa}$ if $G=\Sp_n$, $(\mathbb{Z}/2\mathbb{Z})^{\kappa-1}$ if $G=\SO_n$ and $\kappa>0$, and is trivial if $G=\GL_n$.

\begin{Thm}
Let $u\in G^F$ be a unipotent element, let $C$ be its $G$-conjugacy class, which corresponds to a partition $\lambda$, and let $\kappa=\kappa(\lambda)$. 
\begin{itemize}
\item[(i)] If $G^F=\Sp_n(q)$, then there are $2^{\kappa}$ $G^F$-conjugacy classes contained in $C$;
\item[(ii)] If $G^F=\Ort^{\pm}_n(q)$, and $\kappa>0$, then there $2^{\kappa-1}$ $G^F$-conjugacy classes contained in $C$, and each such conjugacy class is an $\SO_n^{\pm}(q)$-class;
\item[(iii)] If $G^F=\Ort^{+}_n(q)$, and $\kappa=0$, then there is one $G^F$-conjugacy class contained in $C$, which splits into two $\SO^+_n(q)$-conjugacy classes; If $G^F=\Ort^{-}_n(q)$, and $\kappa=0$, then there is no $G^F$-conjugacy class contained in $C$.
\end{itemize}
\end{Thm}
In particular, if $G^F=\Sp_n(q)$, then all $\kappa$-tuples $(e_i)_{i\in\bs\kappa}$ for a given partition $\lambda$ must occur, if $G^F=\Ort_n^{\pm}(q)$ and $\kappa>0$, then the $\kappa$-tuples that occur are subject to a particular constraint. Indeed, in the basis given by Theorem \ref{SS2.19}, one checks easily that $$\det(Q)=\prod_i\det(-Q_i)^i,$$ which implies 
\begin{equation}\label{eta-constraint}
\eta(-1)^{[\kappa'/2]}\eta(Q)=\prod_i\eta(Q_i),
\end{equation} where $\kappa'=|\{i\in\bs\kappa\mid m_i\text{ is odd}\}|$ and $Q$ is the defining form of $\Ort_n^{\pm}(q)$.

\subsubsection{}\label{unirat}
Summarizing the above discussions, the unipotent conjugacy classes in $G^F$ are in bijection with:
$$
\bs{\Psi}^{sp}_n:=\bigsqcup_{\lambda\in\mathcal{P}^{sp}_{n}}(\mathbb{Z}/2\mathbb{Z})^{\kappa(\lambda)},\text{ if $G^F=\Sp_n(q)$;}
$$
$$
\bs{\Psi}^{ort,+}_n:=\bigsqcup_{\lambda\in\mathcal{P}^{ort,nd}_{n}}(\mathbb{Z}/2\mathbb{Z})^{\kappa(\lambda)-1}\sqcup\bigsqcup_{\lambda\in\mathcal{P}^{ort,d}_{n}}(\{\pt\}\sqcup\{\pt\}),\text{ if $G^F=\SO^+_n(q)$;}
$$ 

$$
\tilde{\bs{\Psi}}{}^{ort,+}_n:=\bigsqcup_{\lambda\in\mathcal{P}^{ort,nd}_{n}}(\mathbb{Z}/2\mathbb{Z})^{\kappa(\lambda)-1}\sqcup\bigsqcup_{\lambda\in\mathcal{P}^{ort,d}_{n}}\{\pt\},\text{ if $G^F=\Ort^+_n(q)$;}
$$ 
$$
\bs{\Psi}^{ort,-}_n=\tilde{\bs{\Psi}}{}^{ort,-}_n:=\bigsqcup_{\lambda\in\mathcal{P}^{ort,nd}_{n}}(\mathbb{Z}/2\mathbb{Z})^{\kappa(\lambda)-1},\text{ if $G^F=\SO^-_n(q)$ or $\Ort^-_n(q)$.}
$$
Write $\bs\Psi^{sp}=\bigsqcup_{n\in\mathbb{Z}_{\ge 0},\text{ $n$ even}}\bs\Psi^{sp}_n$, $\bs\Psi^{ort,\eta}=\sqcup_{n\in\mathbb{Z}_{\ge 0}}\bs\Psi^{ort,\eta}_n$ and $\tilde{\bs\Psi}{}^{ort,\eta}=\sqcup_{n\in\mathbb{Z}_{\ge 0}}\tilde{\bs\Psi}{}^{ort,\eta}_n$, where $n$ must be even if $\eta=-$.

\subsubsection{}\label{formules}
Let $u$ and $\lambda$ be as above, then the last piece of information we need about the centraliser of $u$ is some dimension formulae:
\begingroup
\allowdisplaybreaks
\begin{align}
\dim C_{\GL}(u)&=\sum_i im_i^2+2\sum_{i<j}im_im_j,\\
\dim C_{\Sp}(u)&=\frac{1}{2}\sum_iim_i^2+\sum_{i<j}im_im_j+\frac{1}{2}\sum_{\text{$i$ odd}}m_i,\\
\dim C_{\SO}(u)&=\frac{1}{2}\sum_iim_i^2+\sum_{i<j}im_im_j-\frac{1}{2}\sum_{\text{$i$ odd}}m_i.
\end{align}
\endgroup

Denote by $V(u)$ the corresponding unipotent radicals of $C_G(u)$, we have:
\begingroup
\allowdisplaybreaks
\begin{align}
\dim V(u)&=\sum_i (i-1)m_i^2+2\sum_{i<j}im_im_j,\text{ if $G=\GL_n(k)$,}\\
\dim V(u)&=\frac{1}{2}\sum_i(i-1)m_i^2+\sum_{i<j}im_im_j+\frac{1}{2}\sum_{\text{$i$ even}}m_i,\text{ if $G=\Sp_n(k)$,}\\
\dim V(u)&=\frac{1}{2}\sum_i(i-1)m_i^2+\sum_{i<j}im_im_j+\frac{1}{2}\sum_{\text{$i$ even}}m_i,\text{ if $G=\SO_n(k)$.}
\end{align}
\endgroup

%%%%%%%%% Wrong section
\begin{comment}
\subsubsection{}\label{unirat}
The symbols in the sinilarity class associated to the symplectic partition $\lambda$ are in bijection with $(\mathbb{Z}/2\mathbb{Z})^{\kappa}$, and so in bijection with the $G^F$-conjugacy classes contained in the $G$-class corresponding to $\lambda$. The symbols in the similarity class associated to the non degenerate orthogonal partition $\lambda$ are in bijection with $(\mathbb{Z}/2\mathbb{Z})^{\kappa-1}$, and so in bijection with the $G^F$-conjugacy classes contained in the $G$-class corresponding to $\lambda$, for $G=\SO_n(k)$. The similarity class associated to a degenerate orthogonal partition $\lambda$ consists of a single element, which corresponds to the two degenerate classes of $\SO_n(q)$, or the one conjugacy class of $\Ort_n(q)$, corresponding to $\lambda$. For $\SO^{-}_n(q)$, there is no $G^F$-conjugacy class corresponding to a degenerate partition, so the degenerate symbols do not correspond to any conjugacy classes. 
\end{comment}

\subsection{Cuspidal Local Systems}
\subsubsection{}
Let $G$ be a connected reductive group. We refer to \cite[Definition 2.4]{L84a} for the definition of cuspidal pair, which consists of a unipotent $G$-conjugacy class $C$ and an irreducible $G$-equivariant local system $\mathcal{E}$ on $C$. If $(C,\mathcal{E})$ is a cuspidal pair, then $\mathcal{E}$ is called a cuspidal local system. 

\begin{Thm}\label{cusploc}(\cite[Corollary 12.4]{L84a}, \cite[Corollary 13.4]{L84a})
\begin{itemize}
\item[(i)] Let $G=\Sp_{2n}(k)$. There exists a cuspidal pair on $G$ if and only if $2n=d(d+1)$ for some non negative integer $d$. If this is the case, then the cuspidal pair is unique. The underlying conjugacy class of the cuspidal pair corresponds to the symplectic partition $(2d,2d-2,\ldots,4,2)$. 

\item[(ii)] Let $G=\SO_N(k)$. There exists a cuspidal pair on $G$ if and only if $N=d^2$ for some non negative integer $d$. If this is the case, then the cuspidal pair is unique. The underlying conjugacy class of the cuspidal pair corresponds to the orthogonal partition $(2d-1,2d-3,\ldots,3,1)$.
\end{itemize}
\end{Thm}

%%%%%%%%%%%%%%%%%% Groupes Non Connexes
\section{Non-Connected Algebraic Groups}
\subsection{Quasi-Semi-Simple Elements}\label{QSS}
We say that a not necessarily connected algebraic group $G$ is reductive if $G^{\circ}$ is reductive. In this section we denote by $G$ such a group and denote by $F$ the Frobenius endomorphism.

\subsubsection{}\label{barT}
An automorphism of $G^{\circ}$ is \textit{quasi-semi-simple} if it leaves stable a pair consisting of a maximal torus and a Borel subgroup containing it. An element of $G$ is \textit{quasi-semi-simple} if it induces by conjugation a quasi-semi-simple automorphism of $G^{\circ}$. Let ($T^{\circ}$, $B^{\circ}$) be a pair consisting of a maximal torus and a Borel subgroup containing it. Put $B=N_{G}(B^{\circ})$ and $T=N_G(B^{\circ},T^{\circ})$ to be the normalisers. By definition, an element of $G$ is quasi-semi-simple if and only if it belongs to $T$ for some $B^{\circ}$ and $T^{\circ}$.

Every semi-simple element is quasi-semi-simple (\cite[Theorem 7.5]{St}). Every element of $G$ normalises some Borel subgroup of $G^{\circ}$. Let $s\in G$ be a quasi-semi-simple element, every $s$-stable (for the conjugation) Borel subgroup contains some $s$-stable maximal torus. Every $s$-stable parabolic subgroup of $G^{\circ}$ contains some $s$-stable Levi factor (\cite[Proposition 1.11]{DM94} ). However, an $s$-stable Levi subgroup of $G^{\circ}$ is not necessarily an Levi factor of some $s$-stable parabolic subgroup.

Let $G^1\ne G^{\circ}$ be a connected component of $G$, and let $s\in G^1$ be a quasi-semi-simple element. Fix an $s$-stable maximal torus $T^{\circ}$ contained in some $s$-stable Borel subgroup of $G^{\circ}$. The quasi-semi-simple $G^{\circ}$-conjugacy classes in $G^1$ are then described as follows.
\begin{Prop}(\cite{DM18}[Proposition 1.16])\label{DM1.16}
Every quasi-semi-simple $G^{\circ}$-conjugacy class in $G^1$ has a representative in $C_{T^{\circ}}(s)^{\circ}.s$. Two elements $ts$ and $t's$ with $t$, $t'\in C_{T^{\circ}}(s)^{\circ}$, represent the same class if and only if $t$ and $t'$, when passing to the quotient $T^{\circ}/(T^{\circ},s)$, belong to the same $W^s$-orbit, where $(T^{\circ},s)$ is the commutator, which is preserved by $W^s:=\{w\in W_{G^{\circ}}(T^{\circ})\mid sws^{-1}=w\}$.
\end{Prop}

\subsubsection{}
Let $L^{\circ}$ be a Levi factor of some parabolic subgroup $P^{\circ}\subset G^{\circ}$. Put $P=N_{G}(P^{\circ})$ and $L=N_{G}(P^{\circ},L^{\circ})$ to be the normalisers. According to  \cite[Lemma 6.2.4]{Spr}, $P$ is a parabolic subgroup of $G$, in the sense that $G/P$ is proper. Suppose that the Levi decomposition of $P^{\circ}$ is given by $P^{\circ}=U\rtimes L^{\circ}$, where $U$ is the unipotent radical of $P^{\circ}$, then $P=U\rtimes L$. (See \cite[Proposition 1.5]{DM94}) In particular, $L$ is a Levi factor of $P$. (For an arbitrary linear algebraic group $G$, a Levi factor $H$ of $G$ is a closed subgroup such that $G=R_u(G)\rtimes H$.)

\subsubsection{}
If $s\in G$ is a quasi-semi-simple element, then the connected centraliser $H=C_G(s)^{\circ}$ is reductive. (\cite[\S 1.17]{Spa}) If the pair ($T^{\circ}$, $B^{\circ}$) consists of an $s$-stable maximal torus and an $s$-stable Borel subgroup of $G^{\circ}$ containing it, then $C_{B^{\circ}}(s)^{\circ}$ is a Borel subgroup of $H$ containing the maximal torus $C_{T^{\circ}}(s)^{\circ}$ (\cite[Th\'eor\`eme 1.8(iii)]{DM94}). More generally, we have
\begin{Prop}\label{PLsigma}
For $s$ and $H$ as above, we have
\begin{itemize}
\item[(i)]
If ($L^{\circ}$, $P^{\circ}$) is a pair consisting of an $s$-stable Levi subgroup and an $s$-stable parabolic subgroup containing it as a Levi factor, then $C_{P^{\circ}}(s)^{\circ}$ is a parabolic subgroup of $H$, with $C_{L^{\circ}}(s)^{\circ}$ as a Levi factor.
\item[(ii)]
If $L'$ is the Levi factor of a parabolic subgroup $P'\subset H$, then there exists an $s$-stable parabolic subgroup $P\subset G^{\circ}$ such that $C_{P}(s)^{\circ}=P'$, and an $s$-stable Levi factor $L$ of $P$ such that $C_{L}(s)^{\circ}=L'$.
\end{itemize}
\end{Prop}
\begin{Rem}
The groups $P$ and $L$ in this proposition are not necessarily unique in general. See however Proposition \ref{Corrparalevi}.
\end{Rem}

\begin{proof}
The first part is \cite[Proposition 1.11]{DM94}. Given $L'$ and $P'$, there exists a cocharacter of $H$, say $\lambda$, such that $L'=L'_{\lambda}$ and $P'=P'_{\lambda}$, where $L'_{\lambda}$ and $P'_{\lambda}$ are the Levi subgroup and parabolic subgroup associated to $\lambda$. Regarded as a cocharacter of $G^{\circ}$, it defines a Levi subgroup and a parabolic subgroup $L_{\lambda}\subset P_{\lambda}$ of $G^{\circ}$. They are $s$-stable since the image of $\lambda$ commutes with $s$. It is clear that $L'_{\lambda}=L_{\lambda}\cap H$ and $P'_{\lambda}=P_{\lambda}\cap H$. 
\end{proof}

\begin{Prop}\label{K'etK}
Given $L'$ as in the preceding proposition, we write $L=C_{G^{\circ}}(Z^{\circ}_{L'})$. It is an $s$-stable Levi factor of some $s$-stable parabolic subgroup of $G^{\circ}$, such that $L'=C_L(s)^{\circ}$. If $M\subset G^{\circ}$ is an $s$-stable Levi factor of some $s$-stable parabolic subgroup, such that $L'\subset M$, then $L\subset M$. 
\end{Prop}
The proof is completely analogous to \cite[\S 2.1]{L03I} where the assertion is proved for $s$ semi-simple.
\begin{proof}
We can find a cocharacter $\chi:k^{\ast}\rightarrow Z^{\circ}_{L'}$ such that $L=C_{G^{\circ}}(\chi(k^{\ast}))$. As in the preceding proposition, we see that $L$ is an $s$-stable Levi factor of some $s$-stable parabolic subgroup. Since $L'=C_H(Z^{\circ}_{L'})$, we have $C_L(s)^{\circ}=L'$. 

Note that $M':=(M\cap H)^{\circ}=C_M(s)^{\circ}$ is a Levi subgroup of $H$, and that $L'$ is a Levi subgroup of $H$ contained in $M'$. Since $(Z_M^{\circ}\cap H)^{\circ}\subset Z_{M'}$, whence $(Z_M^{\circ}\cap H)^{\circ}\subset Z_{L'}^{\circ}$, whence $C_G(Z_{L'}^{\circ})\subset C_G((Z_M^{\circ}\cap H)^{\circ})$. According to \cite[\S 1.10]{L03I}, $C_G((Z_M^{\circ}\cap H)^{\circ})^{\circ}=M$, so $L\subset M$. 
\end{proof}

\begin{Rem}
In particular, if $T'\subset C_G(s)^{\circ}$ is a maximal torus, then $T:=C_{G^{\circ}}(T')$ is the unique maximal torus of $G^{\circ}$ containing $T'$. It is $s$-stable and contained in an $s$-stable Borel subgroup, and we have $C_T(s)^{\circ}=T'$. 
\end{Rem}
\begin{Rem}\label{K'etKisole}
Let $M$ be an $s$-stable Levi factor of some $s$-stable parabolic subgroup $Q\subset G^{\circ}$ such that $C_M(s)^{\circ}=L'$. Suppose in the proof of the above proposition that the equality $(Z_M^{\circ}\cap H)^{\circ}=Z_{L'}^{\circ}$ holds, i.e. the $s$-fixed part of the centre of $M$ coincides with the centre of the $s$-fixed part of $M$. Then we have $M=L$ by \cite[\S 1.10]{L03I}. We will see in Proposition \ref{L3S2.2} that this equality can be satisfied only if $s$ is an \textit{isolated} element of $N_G(Q)\cap N_G(M)$.
\end{Rem}
\begin{Rem}\label{L'Lnorm}
It follows from the definition of $L$ that if an element of $G$ normalises $L'$, then it normalises $L$.
\end{Rem}

\subsubsection{}%(\cite{DM2} D\'efinition/Th\'eor\`eme 1.15)
A quasi-semi-simple automorphism $\sigma$ of $G^{\circ}$ is \textit{quasi-central} if it satisfies the following condition. 
\begin{quote}
There exists no quasi-semi-simple automorphism of the form $\sigma'=\sigma\circ\ad g$ with $g\in G^{\circ}$ such that $\dim C_G(\sigma)^{\circ}<\dim C_G(\sigma')^{\circ}$. 
\end{quote}
A quasi-semi-simple element of $G$ is \textit{quasi-central} if it induces by conjugation a quasi-central automorphism of $G^{\circ}$.

A quasi-semi-simple element $\sigma\in G$ is quasi-central if and only if there exists a $\sigma$-stable maximal torus $T$ contained in a $\sigma$-stable Borel subgroup of $G^{\circ}$ such that every $\sigma$-stable element of $N_{G^{\circ}}(T)/T$ has a representative in $C_G(\sigma)^{\circ}$(\cite[Th\'eor\`eme 1.15]{DM94}). Considering the natural map $N_{C_G(\sigma)^{\circ}}(C_T(\sigma)^{\circ})\rightarrow N_{G^{\circ}}(T)$, this simply means that $W_{G^{\circ}}(T)^{\sigma}\cong W_{C_G(\sigma)^{\circ}}(C_T(\sigma)^{\circ})$. 

If $\sigma$ is quasi-central, we will often denote $C_H(\sigma)$ by $H^{\sigma}$.

\subsubsection{}\label{qisole}
Let $g=g_sg_u$ be the Jordan decomposition of an element of $G$. Write $L'(g)=C_G(g_s)^{\circ}$ and $L(g)=C_{G^{\circ}}(Z^{\circ}_{L'})$. We say that $g$ is isolated in $G$ if $L(g)=G^{\circ}$. The conjugacy class of an isolated element will be called isolated. The isolated elements can be characterised as follows.
\begin{Prop}\cite[\S 2.2]{L03I}\label{L3S2.2}
Let $g\in G$, and put $L'=L'(g)$ and $L=L(g)$. Then the following assertions are equivalent.
\begin{itemize}
\item[(i)] $L=G^{\circ}$;
\item[(ii)] $Z^{\circ}_{L'}=C_{Z_{G^{\circ}}}(g)^{\circ}$;
\item[(iii)] There is no $g_s$-stable proper parabolic subgroup $Q\subset G^{\circ}$ with $g_s$-stable Levi factor $M$ such that $L'\subset M$.
\end{itemize}
\end{Prop}

This definition of isolated element agrees with the definitions in the literatures, due to the following result, which is not obvious.
\begin{Prop}(\cite[Proposition 18.2]{L03IV})\label{IV18.2}
Let $s\in G$ be a semi-simple element and $u\in G$ a unipotent such that $su=us$. Then $su$ is isolated in $G$ (for the definition in \cite[\S 2]{L03I}) if and only if $s$ is isolated in $G$.
\end{Prop}

Therefore, the definition of isolated semi-simple elements coincides with \cite[Definition 3.1]{DM18}, where one fixes a maximal torus $T$, a Borel subgroup $B\subset G^{\circ}$ containing $T$ and a quasi-central element $\sigma\in N_G(T,B)$, and says that $t\sigma\in T\sigma$ is isolated if $C_G(t\sigma)^{\circ}$  is not contained in a $\sigma$-stable Levi factor $M$ of a $\sigma$-stable proper parabolic subgroup $Q$ of $G^{\circ}$. Note that in this definition, $M$ necessarily contains $T$ because $C_M(t\sigma)^{\circ}$ contains $C_T(t\sigma)^{\circ}$. Let us also recall that (\textit{cf.} \cite[Definition 3.12]{DM18}), an element $t\sigma$ is \textit{quasi-isolated} if $C_G(t\sigma)$  is not contained in a $\sigma$-stable Levi factor of a $\sigma$-stable proper parabolic subgroup of $G^{\circ}$.

\subsection{Parabolic Subgroups and Levi Subgroups}
\subsubsection{}\label{Corrparalevi}
Recall that in the setting of Proposition \ref{PLsigma}, one does not have a bijection in general.
\begin{Prop}(\cite[Corollaire 1.25]{DM94})
Let $\sigma$ be a quasi-central automorphism of $G^{\circ}$.
\begin{itemize}
\item[(1)] The map $P\mapsto(P^{\sigma})^{\circ}$  defines a bijection between the $\sigma$-stable parabolic subgroups of $G^{\circ}$ and the parabolic subgroups of $(G^{\sigma})^{\circ}$.
\item[(2)] Then map $L\mapsto(L^{\sigma})^{\circ}$  defines a bijection between the $\sigma$-stable Levi factors of $\sigma$-stable parabolic subgroups of $G^{\circ}$ and the Levi subgroups of $(G^{\sigma})^{\circ}$.
\end{itemize}
\end{Prop}
Considering the fact that $W_{G^{\circ}}(T)^{\sigma}=W_{(G^{\sigma})^{\circ}}((T^{\sigma})^{\circ})$, the bijection is obtained at the level of Weyl groups.

\subsubsection{}
The following propositions will be useful.
\begin{Prop}(\cite[Proposition 1.6]{DM94})\label{DM2Prop1.6}
Let $\sigma$ be a quasi-semi-simple element of $G$ and let $(L^{\circ},P^{\circ})$ be a pair consisting of a Levi subgroup of $G^{\circ}$ and a parabolic subgroup that contains it as a Levi factor. If the $G^{\circ}$-conjugacy class of $(L^{\circ},P^{\circ})$ is $\sigma$-stable, then there exists $x\in G^{\circ}$  such that $(xL^{\circ}x^{-1},xP^{\circ}x^{-1})$ is $\sigma$-stable.
\end{Prop}

\begin{Prop}(\cite[Proposition 1.38]{DM94})\label{1.38}
Let $\sigma$ be an $F$-stable quasi-central element of $G$ and let $(L^{\circ},P^{\circ})$ be a pair consisting of an $F$-stable Levi factor and a parabolic subgroup containing it as a Levi factor. If the $G^{\circ F}$-conjugacy class of $(L^{\circ},P^{\circ})$ is $\sigma$-stable, then there exists $x\in G^{\circ F}$ such that $(xL^{\circ}x^{-1},xP^{\circ}x^{-1})$ is $\sigma$-stable.
\end{Prop}

Let $L^{\circ}$ be a Levi factor of some parabolic subgroup $P^{\circ}$ of $G^{\circ}$, put $L=N_G(L^{\circ},P^{\circ})$. Let $G^1$ be a connected component of $G$. It acts by conjugation on the $G^{\circ}$-conjugacy classes of the pairs $(L^{\circ},P^{\circ})$. Then $L$ meets $G^1$ if and only if the class of $(L^{\circ},P^{\circ})$ is stable for this action. According to the above propositions, there is a conjugate of $(L^{\circ},P^{\circ})$ that is $\sigma$-stable. This means that $L$ contains $\sigma$ and so $(L^{\sigma})^{\circ}$ is a Levi subgroup of $(G^{\sigma})^{\circ}$. 

\begin{Prop}(\cite[Proposition 1.40]{DM94})\label{1.40}
Assume that $\sigma\in G$ is quasi-central, $F$-stable, and $G/G^{\circ}$ is generated by the component of $\sigma$. Then the $G^F$-conjugacy classes of the $F$-stable groups $L=N_G(L^{\circ},P^{\circ})$ meeting the connected component $G^{\circ}\sigma$ are in bijection with the $((G^{\sigma})^{\circ})^F$-conjugacy classes of the $F$-stable Levi subgroups of $(G^{\sigma})^{\circ}$ in the following manner. Each $L$ has a $G^F$-conjugate $L_1$ containing $\sigma$, and the bijection associates the $((G^{\sigma})^{\circ})^F$-class of $((L_1)^{\sigma})^{\circ}$  to the $G^F$-class of $L$.
\end{Prop}
This gives in particular the classification of the $G^F$-conjugacy classes of the $F$-stable groups of the form $T=N_G(T^{\circ},B^{\circ})$.

%%%%%%%%%%%%%%%%%%% Induction de Deligne-Lusztig
\section{Generalised Deligne-Lusztig Induction}
\subsection{Induction for Connected Groups}
We recall some generalities on the Deligne-Lusztig induction for connected reductive groups. In this section we assume $G$ to be connected. If $X$ is an algebraic variety over $k$, we denote by $H^i_c(X)$ the i-th cohomology group with compact support with coefficient in $\ladic$, and we denote by $H^{\ast}_c(X)=\bigoplus (-1)^iH^i_c(X)$ the virtual vector space. For a finite group $H$, denote by $\mathcal{C}(H)$ the set of the invariant $\ladic$-valued functions on $H$.
\subsubsection{}
Let $L$ be an $F$-stable Levi factor of some parabolic subgroup $P\subset G$ not necessarily $F$-stable. The Levi decomposition writes $P=LU$. Put $\mathcal{L}_{G}^{-1}(U)=\{x\in G|x^{-1}F(x)\in U\}$.  Then $G^F$ acts on $\mathcal{L}_G^{-1}(U)$ by left multiplication and $L^F$ acts by right multiplication. This induces a $G^F\times (L^{F})^{op}$-module structure on $H^i_c(\mathcal{L}^{-1}_G(U))$ for any $i$. Let $\theta\in\mathcal{C}(L^F)$, then the Deligne-Lusztig induction of $\theta$, denoted by $R^G_L\theta$, is the invariant function on $G^F$ defined by
\begin{equation}
R^G_{L}\theta(g)=|L^F|^{-1}\sum_{l\in L^F}\theta(l^{-1})\tr((g,l)|H^{\ast}_c(\mathcal{L}^{-1}(U))),\quad\text{for any }g\in G^F.
\end{equation}
It does not depend on the choice of $P$ if $q\ne 2$ (\textit{cf.} \cite[\S 9.2]{DM20}). The functions of the form $R^G_L\theta$ with $L$ being a maximal torus are called Deligne-Lusztig characters.

\subsubsection{}\label{formucarconn}
The Green function is defined on the subset of unipotent elements in the following manner.
\begin{equation}
\begin{split}
Q_{L}^{G}(-,-):G^F_u\times L^F_u&\longrightarrow \mathbb{Z}\\
(u,v)&\longmapsto
\tr((u,v)|H^{\ast}_c(\mathcal{L}_G^{-1}(U))).
\end{split}
\end{equation}
The calculation of the Deligne-Lusztig induction is reduced to the calculation of the Green functions. If $g=su$ is the Jordan decomposition of $g\in G^F$, we have the character formula (\cite[Proposition 10.1.2]{DM20}), 
\begin{equation}\label{FdCconn}
R^{G}_{L}\theta(g)=|L^F|^{-1}|C_G(s)^{\circ F}|^{-1}\sum_{\{h\in G^F|s\in \leftidx{^h}{L}{} \}}\sum_{\{v\in C_{\leftidx{^h}{L}{}}(s)^{\circ F}_u\}}Q^{C_G(s)^{\circ}}_{C_{\prescript{h}{}L}(s)^{\circ}}(u,v^{-1})\prescript{h}{}\theta(sv),
\end{equation}
where $\prescript{h}{}L=h^{-1}Lh$ and $\prescript{h}{}\theta(sv)=\theta(hsvh^{-1})$. The Green functions are sometimes normalised in such a way that the factor $|L^F|^{-1}$ is removed from the above formula.

\subsubsection{}
We will need the following simple lemmas. Let $\sigma$ be an automorphism of $G$ that commutes with $F$. If $L\subset G$ is an $F$-stable Levi subgroup, then we also denote by $\sigma$ the isomorphism $L^F\rightarrow \sigma(L)^F$ and the isomorphism $W_L(T)\rightarrow W_{\sigma(L)}(\sigma(T))$ for an $F$-stable maximal torus $T$.
\begin{Lem}\label{sigmaRGM}
Let $M\subset G$ be an $F$-stable Levi subgroup and let $\theta$ be a character of $M^F$. Then the character $(R^G_M\theta)\circ\sigma^{-1}$ is equal to $R^G_{\sigma(M)}\sigma_{\ast}\theta$.
\end{Lem}
\begin{proof}
Let $Q$ be a parabolic subgroup containing $M$ such that $Q=MU_Q$. Then $\sigma(\mathcal{L}^{-1}(U_Q))=\mathcal{L}^{-1}(\sigma(U_Q))$ as $F$ commutes with $\sigma$, and so $$
\tr((\sigma(g),\sigma(l))|H^{\ast}_c(\mathcal{L}^{-1}(\sigma(U_Q)))=
\tr((g,l)|H^{\ast}_c(\mathcal{L}^{-1}(U_Q))$$ for any $g\in G^F$ and $l\in M^F$. The assertion then follows from the definition of $R^G_M\theta$.
\end{proof}
\begin{Lem}\label{sigmachi}
Assume that $\chi\in\Irr(\GL_n(q))$ is of the form $R^G_{\varphi}\theta$ for a triple $(M,\varphi,\theta)$ as in Theorem \ref{LS}. Then , the character $\chi\circ\sigma^{-1}$ is of the form $R^G_{\sigma_{\ast}\varphi}\sigma_{\ast}\theta$ for a triple $(\sigma(M),\sigma_{\ast}\varphi,\sigma_{\ast}\theta)$.
\end{Lem}
\begin{proof}
Since $F$ commutes with $\sigma$, we can define $T_{\sigma(w)}$ to be $\sigma(T_w)$ and so by the definition of $R^G_{\varphi}\theta$, the lemma follows from the equality $R^G_{\sigma(T_w)}\sigma_{\ast}\theta(g)=R^G_{T_w}\theta(\sigma^{-1}(g))$. 
\end{proof}

\subsection{Induction for Non Connected Groups}We will assume that $G/G^{\circ}$ is cyclic, and fix $\sigma\in G$ quasi-central and $F$-stable such that $G/G^{\circ}$ is generated by the component of $\sigma$. We may write $G=G^{\circ}\lb\sigma\rb$.
\subsubsection{}
Given an $F$-stable Levi factor $L^{\circ}$ of some parabolic subgroup $P^{\circ}$ not necessarily $F$-stable that is decomposed as $P^{\circ}=L^{\circ}U$, we put $L$ and $P$ to be the normalisers defined in \S \ref{barT}. Put $\mathcal{L}_{G}^{-1}(U)=\{x\in G|x^{-1}F(x)\in U\}$ and $\mathcal{L}_{G^{\circ}}^{-1}(U)=\{x\in G^{\circ}|x^{-1}F(x)\in U\}$. Then $G^F\times (L^F)^{op}$ acts on $\mathcal{L}_{G}^{-1}(U)$, and $H^{\ast}_c(\mathcal{L}_{G}^{-1}(U))$ is thus a $G^F\times (L^F)^{op}$-module. For $\theta\in\mathcal{C}(L^F)$, the (generalised) Deligne-Lusztig induction of $\theta$ is defined by,
\begin{equation}
R^G_{L}\theta(g)=|L^F|^{-1}\sum_{l\in L^F}\theta(l^{-1})\tr((g,l)|H^{\ast}_c(\mathcal{L}_{G}^{-1}(U))),\quad\text{for any }g\in G^F.
\end{equation}
It does not depend on the choice of $P^{\circ}$ if $q\ne 2$.\footnote{As is pointed out by F. Digne, for $G=G^{\circ}\lb\sigma\rb$ and $\sigma$ inducing a non trivial automorphism, this is reduced to the case of connected groups. (See \cite[\S 9.2]{DM20}.)}  According to Proposition \ref{1.40}, the generalised Deligne-Lusztig induction are parametrised by the pairs $(L^{\circ},P^{\circ})$ consisting of an $F$-stable and $\sigma$-stable Levi factor of some $\sigma$-stable parabolic subgroup. Since only those $L$ that meets $G^{\circ}\sigma$ interest us, we can assume that $L=L^{\circ}\lb\sigma\rb$. The restriction of $R^G_L$ to $L^{\circ F}\sigma$ is a map $\mathcal{C}(L^{\circ F}\sigma)\rightarrow\mathcal{C}(G^{\circ F}\sigma)$, that we denote by $R^{G^{\circ}\sigma}_{L^{\circ}\sigma}$. To simplify the terminology, we may also call these maps Deligne-Lusztig inductions.

\subsubsection{}
The following lemma shows that the induction thus defined is compatible with that defined for connected groups.
\begin{Lem}\label{IndRes}
We keep the notations as above. Denote by $\Res$ the usual restriction of functions. If $G^F=L^F.G^{\circ F}$, then,
\begin{equation}
\Res^{G^{F}}_{G^{\circ F}}\circ R^G_{L}=R^{G^{\circ}}_{L^{\circ}}\circ\Res^{L^F}_{L^{\circ F}}.
\end{equation}
\end{Lem}
\begin{proof}
See \cite[Corollaire 2.4 (i)]{DM94}.
\end{proof}
\subsubsection{}\label{2-var-Green}
The Green function is defined by
\begin{equation}
\begin{split}
Q_{L}^{G}(-,-):G^F_u\times L^F_u&\longrightarrow \mathbb{Z}\\
(u,v)&\longmapsto
\begin{cases}
0 & \text{if $uv\notin G^{\circ}$}\\
\tr((u,v)|H^{\ast}_c(\mathcal{L}_{G^{\circ}}^{-1}(U))) & \text{otherwise}.
\end{cases}
\end{split}
\end{equation}
Note that $\mathcal{L}_{G^{\circ}}^{-1}(U)$ is the usual Deligne-Lusztig variety.
\begin{Prop}(Character Formula, \cite[Proposition 2.6]{DM94})\label{FormdeCara}
Let $L$ be the normaliser of the pair $(L^{\circ}\subset P^{\circ})$  as above, and let $\theta$ be a character of $L^F$. Then for any $g\in G^F$ with Jordan decomposition $g=su$,
\begin{equation}\label{FdCnonconn}
R^{G}_{L}\theta(g)=|L^F|^{-1}|C_G(s)^{\circ F}|^{-1}\sum_{\{h\in G^F|s\in ^hL \}}\sum_{\{v\in C_{^hL}(s)^F_u\}}Q^{C_G(s)^{\circ}}_{C_{\prescript{h}{}L}(s)^{\circ}}(u,v^{-1})\prescript{h}{}\theta(sv).
\end{equation}
\end{Prop}
This formula will only be used in the following form.
\begin{Prop}(\cite[Proposition 2.10]{DM94})\label{2.10}
We keep the notations as above, except that $\theta$ is now a $\sigma$-stable character of $L^{\circ F}$. Denote by $\smash{\tilde{\theta}}$ an extension of $\theta$ to $L^{\circ F}\sigma$, and let $su$ be the Jordan decomposition of $g\sigma$, $g\in G^{\circ F}$. Then,
\begin{equation}
R^{G^{\circ}\sigma}_{L^{\circ}\sigma}\tilde{\theta}(g\sigma)=|L^{\circ F}|^{-1}|C_G(s)^{\circ F}|^{-1}\sum_{\{h\in G^{\circ F}|s\in ^hL \}}\sum_{\{v\in C_{^hL}(s)^F_u\}}Q^{C_G(s)^{\circ}}_{C_{\prescript{h}{}L}(s)^{\circ}}(u,v^{-1})\leftidx{^h}{\tilde{\theta}}{}(sv).
\end{equation}
\end{Prop}

The Green functions are sometimes normalised in such a way that the above two formulas should be multiplied by $|C_{\leftidx{^h\!}{L}{}}(s)^{\circ F}|$.

\subsection{Uniform Functions}

\subsubsection{}
The irreducible characters of $G^F$ for connected $G$ can be expressed as linear combinations of the Deligne-Lusztig inductions of cuspidal functions on various $L^F$. For $G=\GL_n^{\epsilon}(\mathbb{F}_q)$, we have Theorem \ref{LS} in Introduction. It shows that in this particular case one only needs the functions $R^G_T\theta$ induced from $F$-stable tori, and the transition matrix is given by the characters of the Weyl group of $G$. In general, the transition matrix could be more complicated and the functions induced from the characters of tori are not sufficient. The invariant functions on $G^F$ that are linear combinations of the $R^G_T\theta$'s are called \textit{uniform functions}.

Recall that for each $w\in W_G$, the Weyl group of $G$, and some $\dot{w}\in G$ representing $w$, one can find $g\in G$ such that $g^{-1}F(g)=\dot{w}$, then $T_w:=gTg^{-1}$ is an $F$-stable maximal torus such that $T_w^F\cong T^{F_w}$ with $F_w:=\ad\dot{w}\circ F$. Then $w\mapsto T_w$ defines a bijection between the $F$-conjugacy classes of $W_G$ and the $G^F$-conjugacy classes of $F$-stable maximal tori.

\subsubsection{}\label{nonconnunif}
Now assume that $G$ is non connected. Let $T\subset G^{\circ}$ be an $F$-stable and $\sigma$-stable maximal torus contained in a $\sigma$-stable Borel subgroup. If $\theta\in\Irr(T^F)$ extends into $\tilde{\theta}\in \Irr(T^F\lb\sigma\rb)$, then $R^{G^{\circ}\sigma}_{T\sigma}\tilde{\theta}$ belongs to $\mathcal{C}(G^{\circ F}\sigma)$, the set of $\ladic$-valued functions on $G^{\circ F}\sigma$ that are invariant under $G^{\circ F}$. A function in $\mathcal{C}(G^{\circ F}\sigma)$ is called \textit{uniform} if it is a linear combination of functions of the form $R^{G^{\circ}\sigma}_{T\sigma}\tilde{\theta}$.

\subsubsection{}\label{uniforme}
Let $\tilde{\chi}$ be an irreducible character of $G^F$. It is called \textit{unipotent} if $\chi:=\Res^{G^F}_{G^{\circ F}}\tilde{\chi}$  contains a unipotent character as a direct summand. In this case, $\chi$ is a sum of unipotent characters, as its summands are $G^F$-conjugate. Denote by $$\mathcal{E}(G^{\circ F}\sigma,(1))=\{\tilde{\chi}|_{G^{\circ F}\sigma}\mid\tilde{\chi}\text{ is a unipotent character}\}.$$ An element of $\mathcal{C}(G^{\circ F}\sigma)$ is called \textit{unipotent} if it is a linear combination of some elements of $\mathcal{E}(G^{F}\sigma,(1))$, and we denote by $\mathcal{C}(G^{\circ F}\sigma,(1))$ this subspace. It is clear that the characters $R^{G^{\circ}\sigma}_{T_w\sigma}1$ are unipotent functions, and they are parametrised by the $F$-conjugacy classes of $W^{\sigma}$. An element of $\mathcal{C}(G^{\circ F}\sigma)$ is called \textit{uniform-unipotent} if it is a linear combination of the functions $R^{G^{\circ}\sigma}_{T_w\sigma}1$.

\subsubsection{}
A natural question is to identify those elements of $\mathcal{E}(G^{\circ F}\sigma,(1))$ that are uniform. We have,
\begin{Thm}(\cite[Th\'eor\`eme 5.2]{DM94})
Let $G=\GL_n^{\epsilon}(k)$ and let $W_G$ be the Weyl group defined by some $F$-stable and $\sigma$-stable maximal torus of $G$, and so $F$ acts trivially on the $\sigma$-fixed part $W_G^{\sigma}$. For any $\varphi\in\Irr(W_G^{\sigma})$, put
\begin{equation}
R^{G\sigma}_{\varphi}1:=|W_{G}^{\sigma}|^{-1}\sum_{w\in W_{G}^{\sigma}}\varphi(w)R_{T_w\sigma}^{G\sigma}1.
\end{equation}
Then, $R^{G\sigma}_{\varphi}1$ is an extension of a principal series unipotent representation of $G(q)$.
\end{Thm}
This gives all uniform-unipotent functions on $\GL_n^{\epsilon}(\mathbb{F}_q)\sigma$. It follows that an element of $\mathcal{E}(G^{\circ F}\sigma,(1))$ is either uniform or orthogonal to the space of uniform(-unipotent) functions.

\subsubsection{}
(\cite[Proposition 6.4]{DM15}) The characteristic functions of quasi-semi-simple classes are uniform. Consequently, all non uniform characters vanish on the quasi-semi-simple classes.

%%%%%%%%%%%%%%%%%% Le Groupe
\section{The Group $\GL_n(k)\rtimes\mathbb{Z}/2\mathbb{Z}$}
In what follows, we write $G=\GL_n(k)$.
\subsection{Automorphisms of $\GL_n(k)$}
\subsubsection{}\label{autostand}
Let $\mathscr{J}_n$ be the matrix
\begin{equation}
(\mathscr{J}_n)_{ij}=\delta_{i,n+1-j}
=
\left(
\begin{array}{ccc}
& & 1\\
&\reflectbox{$\ddots$}&\\
1& &
\end{array}
\right)
\end{equation}
Put $t_0=\diag(a_1,\ldots,a_n)$ with $a_i=1$ if $i\le[(n+1)/2]$ and $a_i=-1$ otherwise. Put $\mathscr{J}'_n=t_0\mathscr{J}_n$. Define the matrices
\begin{equation}
\mathscr{J}=
\begin{cases}
\mathscr{J}'_n & \text{if $n$ is even}\\
\mathscr{J}_n & \text{if $n$ is odd}
\end{cases},
\quad
\mathscr{J}'=
\begin{cases}
\mathscr{J}_n & \text{if $n$ is even}\\
\mathscr{J}'_n & \text{if $n$ is odd}
\end{cases}
\end{equation}

The automorphism $\sigma\in\Aut(\GL_n(k))$ that sends $g$ to $\mathscr{J}g^{-t}\mathscr{J}^{-1}$ will be called the standard automorphism. We will denote by $\sigma'$ the automorphism defined by replacing $\mathscr{J}$ with $\mathscr{J}'$ in the definition of $\sigma$. They are quasi-semi-simple automorphisms because the maximal torus consisting of the diagonal matrices and the Borel subgroup consisting of the upper triangular matrices are stable under the action of $\sigma$ and $\sigma'$. Moreover, $\sigma$ is a quasi-central involution regardless of the parity of $n$, while $\sigma'$ is not an involution if $n$ is odd and is not quasi-central if $n$ is even.

\subsubsection{}
The classification of the involutions and the quasi-central automorphisms is well known.(\cite[Lemma 2.9]{LiSe} and \cite[Proposition 1.22]{DM94}).

The conjugacy classes of involutions are described as follows.
\begin{itemize}
\item[-]If $n=2m+1$, the exterior involutions (exterior automorphisms of order 2) are all $G$-conjugate and their centralisers are of type $B_m$.
\item[-]If $n=2m>2$, there are two $G$-conjugacy classes of exterior involutions, with centralisers of type $C_{m}$ and $D_{m}$ respectively. If $n=2$, the connected centralisers are $\SL_2(k)$ and  $k^{\ast}$ respectively.
\end{itemize}

The conjugacy classes of quasi-central automorphisms are described as follows.
\begin{itemize}
\item[-]If $n=2m+1$, there are two classes of exterior quasi-central automorphisms, with centralisers of type $B_m$ and of type $C_m$ respectively.
\item[-]If $n=2m$, there is one single class of exterior quasi-central automorphisms, with centraliser of type $C_m$.
\end{itemize}
Explicitly,
\begin{equation}
\text{ if $n=2m$,}
\begin{cases}
(G^{\sigma})^{\circ}\cong\Sp_{2m}(k)\\
(G^{\sigma'})^{\circ}\cong\SO_{2m}(k)
\end{cases},\quad
\text{ if $n=2m+1$,}
\begin{cases}
(G^{\sigma})^{\circ}\cong\SO_{2m+1}(k)\\
(G^{\sigma'})^{\circ}\cong\Sp_{2m}(k)
\end{cases}.
\end{equation}
Put $t=\diag(a_1,\ldots,a_m,a_{m+1},\ldots,a_{2m})$ or $\diag(a_1,\ldots,a_m,1,a_{m+1},\ldots,a_{2m})$ depending on the parity of $n$, with $a_i=\mathfrak{i}$ for $i\le m$ and $a_i=-\mathfrak{i}$ for $i>m$, we have
\begin{equation}\label{t}
\begin{cases}
(G^{t\sigma})^{\circ}=\SO_{2m}(k) &\text{ if $n=2m$,}\\
(G^{t\sigma})^{\circ}=\Sp_{2m}(k) &\text{ if $n=2m+1$.}
\end{cases}
\end{equation}

We say that an automorphism is of symplectic type or orthogonal type according to the type of its centraliser.

\subsubsection{}\label{2facQC}
We will encounter another type of quasi-central automorphism. Let $\tau_0$ be an involution of $G$, not necessarily an exterior automorphism. Define an automorphism $\tau$ of $G\times G$ by $(g,h)\mapsto (\tau_0(h),\tau_0(g))$. It is easy to see from the definition that $\tau$ is quasi-central and $C_{G\times G}(\tau)\cong G$.

\subsection{The Group $\bar{G}$}
\subsubsection{}
Let $H$ be an abstract group and let $\tau$ be an automorphism of finite order of $H$. By a $\tau$-conjugacy class of $H$, we mean an orbit in $H$ under the action $h:x\mapsto hx\tau(h^{-1})$, $x$, $h\in H$. By a $\tau$-class function on $H$, we mean a function that is constant on the $\tau$-conjugacy classes. We denote by $\mathcal{C}(H.\tau)$ the set of $\tau$-class functions.

On the other hand, the conjugacy classes of $H\rtimes\lb\tau\rb$ contained in $H.\tau$ are identified with the $H$-conjugacy classes in $H.\tau$, as $\tau(h)\tau=h^{-1}(h\tau)h$, which are in turn identified with the $\tau$-conjugacy classes of $H$. This justifies the notation $\mathcal{C}(H.\tau)$.

\subsubsection{}\label{barG}
The semi-direct product $\bar{G}:=G\rtimes\mathbb{Z}/2\mathbb{Z}$  is defined by an exterior involution. If $n$ is odd, there is only one class of involutions. It is the class of $\sigma$. If $n$ is even, there are two classes of involutions, and the resulting semi-direct products are not isomorphic. Whenever it is necessary to distinguish the two semi-direct products, we will denote by $\leftidx{^s\!}{\bar{G}}$ the one defined by the symplectic type automorphism, i.e. by $\sigma$, and denote by $\leftidx{^o\!}{\bar{G}}$ the one defined by the orthogonal type automorphisms, i.e. by $\sigma'$. We will often denote by $\sigma$ the element $(e,1)\in\bar{G}$, although $\sigma$ is not acturally an element of $\bar{G}$.

By definition, $\sigma$ is a quasi-central element in $G\sigma$.

\subsubsection{}\label{GsGo}
The character tables of $\leftidx{^s\!}{\bar{G}}^F$ and of $\leftidx{^o\!}{\bar{G}}^F$ are related in the following manner.

The $G^F$-conjugacy classes in $\leftidx{^s\!}{\bar{G}}^F\setminus G^F$ are in bijection with the $\sigma$-conjugacy classes in $G^F$, which are in bijection with the $t_0\sigma'$-conjugacy classes in $G^F$ (See \S \ref{autostand} for $t_0$), which are in bijection with the $G^F$-conjugacy classes in $\leftidx{^o\!}{\bar{G}}^F\setminus G^F$. Under this bijection, for any $g\in G^F$,  the $G^F$-class of $g\sigma\in\leftidx{^s\!}{\bar{G}}^F$ corresponds to the $G^F$-class of $gt_0\sigma'\in\leftidx{^o\!}{\bar{G}}^F$.

Since $\sigma$ and $\sigma'$ differ by an inner automorphism, the set of $\sigma$-stable characters coincides with that of $\sigma'$-stable characters. However, the extension of a $\sigma$-stable character to $\bar{G}^F$ behaves differently for $\leftidx{^s\!}{\bar{G}}^F$ and $\leftidx{^o\!}{\bar{G}}^F$ with respect to the above bijection of conjugacy classes. Let $\rho:G^F\rightarrow\GL(V)$ be a $\sigma$-stable irreducible representation of $G^F$. To find an extension $\tilde{\rho}\in\Irr(\leftidx{^s\!}{\bar{G}}^F)$ of $\rho$ is to define $\tilde{\rho}(\sigma)$ in such a way that $\tilde{\rho}(\sigma)^2=\rho(\sigma^2)=\Id$ and $\tilde{\rho}(\sigma)\rho(g)\tilde{\rho}(\sigma)^{-1}=\rho(\sigma(g))$ for all $g\in G^F$. Suppose that we have defined such an extension, and we would like to define $\tilde{\rho}'\in\Irr(\leftidx{^o\!}{\bar{G}}^F)$ by $\tilde{\rho}'(t_0\sigma')=\tilde{\rho}(\sigma)$. For $\leftidx{^o\!}{\bar{G}}^F$, if $\rho(-1)\ne \Id_V$, the equality $\tilde{\rho}(\sigma)^2=\Id$ would be violated. Consequently, we define instead $\tilde{\rho}'(t_0\sigma')=\tilde{\rho}(\sigma)\sqrt{\rho(-1)}$, where $\rho(-1)$ has as value $\pm\Id$ regarded as a scalar. Replacing $\sqrt{\rho(-1)}$ by $-\sqrt{\rho(-1)}$ defines another extension of $\rho$ to $\leftidx{^o\!}{\bar{G}}^F$. We denote by $\tilde{\chi}$ and $\tilde{\chi}'$ the characters of $\tilde{\rho}$ et $\tilde{\rho}'$ respectively. Then, for all $g\in G^F$, 
\begin{equation}
\tilde{\chi}'(gt_0\sigma')=\tilde{\chi}(g\sigma)\sqrt{\rho(-1)}
\end{equation}
up to a sign.

\begin{Conv}\label{convsigma}
Because of the above discussion, we will also denote by $\sigma$ the element $t_0\sigma'\in\leftidx{^o\!}{\bar{G}}^F$. We will later parametrise the conjugacy classes in $\leftidx{^o\!}{\bar{G}}^F\setminus G^F$ with respect to $\sigma$ (following Proposition \ref{DM1.16}).
\end{Conv}

\subsection{Quasi-Semi-Simple Elements}

\subsubsection{}
We have said in \S \ref{QSS} that all semi-simple elements are quasi-semi-simple.  For $\bar{G}$, we have
\begin{Lem}
An element of $\bar{G}$ is quasi-semi-simple if and only if it is semi-simple.
\end{Lem}
More generally, if $G$ is a reductive algebraic group and $G/G^{\circ}$ is semi-simple, then all quasi-semi-simple elements are semi-simple. (See \cite[Remarque 2.7]{DM94}) In positive characteristic, this is to require that $\ch k\nmid |G/G^{\circ}|$. We give a short proof below in the particular case of $\bar{G}$.
\begin{proof}
It suffices to show that each quasi-semi-simple element $s\sigma\in G\sigma$ is semi-simple. We see that $s\sigma$ is semi-simple if and only if $(s\sigma)^2=s\sigma(s)\sigma^2$ is semi-simple, as we have assumed $\ch k$ to be odd. Let $(T,B)$ be a pair consisting of a maximal torus and a Borel subgroup containing it, both normalised by $\sigma$. Then every quasi-semi-simple element is conjugate to an element of $(T^{\sigma})^{\circ}\sigma$ (Proposition \ref{DM1.16}), and its square, lies in $T$, and so is semi-simple.
\end{proof}

\begin{Rem}
That $s\sigma$ is semi-simple does not imply that $s$ is so. Let us fix $T$ and $B$ as above, and write $B=TU$, where $U$ is the unipotent radical of $B$. If we take $u\in U$, then $u\sigma(u^{-1})$ is a unipotent element, whereas $u\sigma u^{-1}$ is semi-simple.
\end{Rem}
\begin{Rem}
There is no unipotent element in $G\sigma$ because an odd power of $s\sigma$ lies in $G\sigma$.
\end{Rem}

\subsubsection{Isolated Elements}\label{isolelts}
Define the diagonal matrix
\begin{equation}
t(j)=\diag(\overbrace{\mathfrak{i},\ldots,\mathfrak{i}}^j,1,\ldots,1,\overbrace{-\mathfrak{i},\ldots,-\mathfrak{i}}^j)\in\GL_n(k)
\end{equation} 
The elements $t(j)\sigma$, $0\le j\le[n/2]$, are the representatives of the isolated elements (\S\ref{qisole}), except when $n$ is even and $j=1$, in which case $t(j)\sigma$ is quasi-isolated (\cite[Proposition 4.2]{DM18}). We have,
\begin{equation}
\begin{split}
C_G(t(j)\sigma)\cong \Ort_{2j}(k)\times\Sp_{n-2j}(k), \text{ if $n$ is even;} \\
C_G(t(j)\sigma)\cong\Sp_{2j}(k)\times\Ort_{n-2j}(k), \text{ if $n$ is odd.}
\end{split}
\end{equation}
In particular, when $n$ is even, $t(j)\sigma$ is quasi-central only if $j=0$, and when $n=2m+1$, $t(j)\sigma$ is quasi-central only if $j=m$. Note that our choice of $\sigma$ for odd $n$ is different from that of \cite{DM18}. In this article, we will not encounter quasi-isolated elements that are not isolated.

\subsection{Parabolic Subgroups and Levi Subgroups}\label{notationgene}
\subsubsection{}
The parabolic subgroups $P\subset G$ such that $N_{\bar{G}}(P)$ meets $G\sigma$ are described as follows.

Let $P\subset G$ be a parabolic subgroup. The normaliser $N_{\bar{G}}(P)$ meets $G\sigma$ if and only if there exists $g\in G$ such that $g\sigma$ normalises $P$, if and only if the $G$-conjugacy classes of $P$ is $\sigma$-stable. According to Proposition \ref{DM2Prop1.6}, this means that there exists a $G$-conjugate of $P$ that is $\sigma$-stable. We can then assume that $P$ is $\sigma$-stable. Then, there exist a maximal torus and Borel subgroup containing this torus, both being $\sigma$-stable and contained in $P$. A standard parabolic subgroup with respect to $T$ and $B$ is $\sigma$-stable if and only if the Dynkin subdiagram that defines it is $\sigma$-stable. We conclude that if $T$ is the maximal torus of the diagonal matrices and $B$ is the Borel subgroup of the upper triangular matrices, then every standard parabolic subgroup with respect to $T\subset B$ such that $N_{\bar{G}}(P)$ meets $G\sigma$ has as its unique Levi factor containing $T$ the group of the form (\ref{LI}).

%%%%%%%%%%%%%%%%%% Classification des Caract\`eres
\section{Parametrisation of Characters}\label{SecParaChar}
\subsection{$F$-Stable Levi Subgroups}
Recall the parametrisation of the $F$-stable Levi subgroups of $G=\GL_n(k)$. 

\subsubsection{Notations}\label{standLevinotation}
Denote by $T\subset G$ the maximal torus consisting of the diagonal matrices and denote by $B\subset G$ the Borel subgroup consisting of the upper triangular matrices. The Frobenius $F$ of $G$ sends each entry of an matrix to its $q$-th power. Denote by $\Phi$ the root system defined by $T$ and $\Delta\subset\Phi$ the set of simple roots determined by $B$. Denote by $W:=W_G(T)=N_G(T)/T$ the Weyl group of $G$. The Frobenius acts trivially on $W$. Given a subset $I\subset \Delta$, we denote by $P_I$ the standard parabolic subgroup defined by $I$, and $L_I$ the unique Levi factor of $P_I$ containing $T$. A Levi subgroup of the form $L_I$ is called a standard Levi subgroup. Every Levi subgroup is conjugated to a standard Levi subgroup

\begin{Prop}\label{G^FLevi}
The set of the $G^F$-conjugacy classes of the $F$-stable Levi subgroups of $G=\GL_n(k)$ is in bijection with the set of the unordered sequence of pairs of positive integers $(r_1,d_1)\cdots(r_s,d_s)$, satisfying $\sum_ir_id_i=n$.
\end{Prop}
We sketch the proof in order to introduce some notations that will be used later.
\begin{proof}[sketch of proof]
The $G$-conjugacy classes of the Levi subgroups are in bijection with the equivalence classes of the subsets $I\subset\Delta$. Two subsets $I$ and $I'$ are equivalent if there is an element $w\in W$ such that $I'=wI$. It suffices for us to fix $I\subset \Delta$ and only consider the $G$-conjugacy class of $L_I$. The $G^F$-conjugacy classes of the $G$-conjugates of $L_I$ are in bijection with the $F$-conjugacy classes of $N_G(L_I)/L_I$. Note that $F$ acts trivially on $N_G(L_I)/L_I$. 

Let $\Gamma_I$ be a finite set such that $|\Gamma_I|=\dim Z_{L_I}$ and we fix an isomorphism $Z_{L_I}\cong(k^{\ast})^{\Gamma_I}$. There are positive integers $n_i$, $i\in\Gamma_I$, such that $L_I\cong \prod_{i\in\Gamma_I}\GL_{n_i}$. For any $r\in\mathbb{Z}_{>0}$, put $\Gamma_{I,r}:=\{i\in\Gamma_I\mid n_i=r\}$ and put $N_r=|\Gamma_{I,r}|$. Then the equivalence class of $I\subset\Delta$ is specified by the sequence $(N_r)_{r>0}$. It is known that $N_G(L_I)/L_I\cong \prod_r\mathfrak{S}(\Gamma_{I,r})$, where $\mathfrak{S}(\Gamma_{I,r})$ denotes the permutation group of the set $\Gamma_{I,r}$. The conjugacy classes of $N_G(L_I)/L_I$ are then parametrised by the partitions of $N_r$, $r\in\mathbb{Z}_{>0}$. The desired sequence $(r_1,d_1)\cdots(r_s,d_s)$ will then be defined in such a way that $(d_i)_{i\in\{1\le i\le s\mid r_i=r\}}$ forms the corresponding partition of $N_r$.

Let $nL_I\in N_G(L_I)/L_I$ be a class representing a $G^F$-class of $F$-stable Levi subgroup. Then $n$ acts on $Z_{L_I}$ by a permutation $w$ of $\Gamma_I$. For each $r\in\mathbb{Z}_{>0}$, denote by $\Lambda_{I,r}(n)$ the set of the orbits in $\Gamma_{I,r}$ under the action of $n$ and write $\Lambda_I(n)=\cup_r \Lambda_{I,r}(n)$. For each $i\in \Lambda_I(n)$, put $r_i=r$ if $i\in \Lambda_{I,r}(n)$ and define $d_i$ to be the cardinality of $i$. The $d_i$'s, for $i\in \Lambda_{I,r}(n)$, form a partition of $|\Gamma_{I,r}|$. The integer $s$ in the statement of the proposition is equal to $|\Lambda_I(n)|$.
\end{proof}

If $L$ is an $F$-stable Levi subgroup corresponding to $(n_1,d_1)\cdots(n_s,d_s)$, then $(L,F)$ is isomorphic to a standard Levi subgroup
\begin{equation}
L_I\cong\prod_i\GL_{n_i}(k)^{d_i},
\end{equation}
equipped with a Frobenius acting on each factor $\GL_{n_i}(k)^{d_i}$ in the following manner,
\begin{equation}
\begin{split}
\GL_{n_i}(k)^{d_i}&\longrightarrow\GL_{n_i}(k)^{d_i}\\
(g_1,g_2,\ldots,g_{d_i})&\longmapsto(F_0(g_{d_i}),F_0(g_1),\ldots,F_0(g_{d_i-1})),
\end{split}
\end{equation}
where $F_0$ is the Frobenius of $\GL_{n_i}(k)$ that sends each entry to its q-th power.

\subsection{$\sigma$-Stable Characters}
Fix $T$ and $B$ as above, then standard Levi subgroups and standard parabolic subgroups are always taken with respect to $T\subset B$. In this section, we will consider the automorphism $\sigma$ defined in \S \ref{autostand}. The operation $\chi\mapsto\chi\circ\sigma^{-1}$ defines an involution of $\Irr(\GL_n(q))$. Denote by $\leftidx{^{\sigma}\!}{}{\chi}$ the image of $\chi$ under this involution. 

\subsubsection{Quadratic-Unipotent Characters}
Given a 2-partition $(\mu_+,\mu_-)\in\mathcal{P}_n(2)$, we define an irreducible character $\chi_{(\mu_+,\mu_-)}\in\Irr(\GL_n(q))$ as follows. Put $n_+=|\mu_+|$ and $n_-=|\mu_-|$  and so $n_++n_-=n$. Let $M\subset\GL_n(k)$ be a standard Levi subgroup isomorphic to $\GL_{n_+}(k)\times\GL_{n_-}(k)$, i.e. 
\begin{equation}\label{q.u.M}
  \renewcommand{\arraystretch}{1.2}
M= \left(
  \begin{array}{ c c | c c }
     & & \multicolumn{2}{c}{\raisebox{-.6\normalbaselineskip}[0pt][0pt]{$0$}}  \\
     \multicolumn{2}{c|}{\raisebox{.6\normalbaselineskip}[0pt][0pt]{$\GL_{n_+}$}} &  & \\
    \cline{1-4}
     & & & \\
     \multicolumn{2}{c|}{\raisebox{.6\normalbaselineskip}[0pt][0pt]{$0$}} & \multicolumn{2}{c}{\raisebox{.6\normalbaselineskip}[0pt][0pt]{$\GL_{n_-}$}}
     \end{array}
  \right).
\end{equation}
Denote by $W_M:=W_M(T)$ the Weyl group of $M$. It is isomorphic to $\mathfrak{S}_{n_+}\times\mathfrak{S}_{n_-}$. The Frobenius $F$ acts trivially on $W_M$. With respect to the isomorphism $M^F\cong\GL_{n_+}(q)\times\GL_{n_-}(q)$, we define a linear character $\theta\in\Irr(M^F)$ to be $(1\circ\det,\eta\circ\det)$, where $1$ is the trivial character of $\mathbb{F}_q^{\ast}$ and $\eta$ is the order 2 irreducible character of $\mathbb{F}_q^{\ast}$. It is regular in the sense of \cite[\S 3.1]{LS}. We define $\varphi\in\Irr(W_M)$  by a 2-partition $(\mu_+,\mu_-)$ in such a way that the factor corresponding to $\mathfrak{S}_{n_+}$ is defined by the partition $\mu_+$, and the other by $\mu_-$. According to Theorem \ref{LS}, the triple $(M,\theta,\varphi)$  gives an irreducible character of $\GL_n(q)$, which we denote by $\chi_{(\mu_+,\mu_-)}$.  The irreducible characters of $\GL_n(q)$ thus obtained are called quadratic-unipotents.
\begin{Lem}
For any $(\mu_+,\mu_-)\in\mathcal{P}_n(2)$, the character $\chi_{(\mu_+,\mu_-)}$ is $\sigma$-stable.
\end{Lem}
\begin{proof}
If $\chi_{(\mu_+,\mu_-)}$ is of the form $R^G_{\varphi}\theta$ for a triple $(M,\theta,\varphi)$, then $\leftidx{^{\sigma}\!}{\chi}_{(\mu_+,\mu_-)}$ is of the form $R^G_{\sigma_{\ast}\varphi}\sigma_{\ast}\theta$ for the triple $(\sigma(M),\sigma_{\ast}\theta,\sigma_{\ast}\varphi)$ according to Lemma \ref{sigmachi}. Explicitly, $\sigma(M)$ is the Levi subgroup (\ref{q.u.M}) with $\GL_{n_+}$ and $\GL_{n_-}$ exchanged, $\sigma_{\ast}\theta$ associates to the factor $\GL_{n_+}$ the trivial character of $\mathbb{F}_q^{\ast}$ and $\eta$ to the other factor, and $\sigma_{\ast}\varphi$ associates to the factor $\GL_{n_+}$ the character of $\mathfrak{S}_{n_+}$  corresponding to $\mu_+$ and to the other factor the character corresponding to $\mu_-$. Then the conjugation by $\mathscr{J}_n$ sends $(\sigma(M),\sigma_{\ast}\theta,\sigma_{\ast}\varphi)$ to $(M,\theta,\varphi)$,  and so $\leftidx{^{\sigma}\!}{\chi}_{(\mu_+,\mu_-)}=\chi_{(\mu_+,\mu_-)}$ according to Theorem \ref{LS}
\end{proof}

\subsubsection{}
Now we construct some more general $\sigma$-stable irreducible characters. If $I\subset \Delta$ is a $\sigma$-stable subset, then it defines a $\sigma$-stable standard Levi subgroup. Every $\sigma$-stable standard Levi subgroup is of this form. We are going to use the notations of \S \ref{standLevinotation} and of Proposition \ref{G^FLevi}. Denote by $0$ the unique element of $\Gamma_I$ fixed by $\sigma$. For any $i\in\Gamma_I$, denote by $i^{\ast}$  its image under $\sigma$. The standard Levi subgroup $L_I$ is decomposed as $\prod_{i\in\Gamma_I}L_i$. For any $i\in\Gamma_I$, write $n_i=r$ if $i\in\Gamma_{I,r}$, and so $L_i\cong \GL_{n_i}(k)$. Schematically, $L_I$ is equal to 
\begin{equation}\label{LI}
  \renewcommand{\arraystretch}{1.3}
\left(\!
  \begin{array}{cccccccc}
    \multicolumn{1}{c|}{L_{i_1}} & & & & & & &\\
    \cline{1-1}
    & \ddots & & & & & &\\
    \cline{3-3}
    & & \multicolumn{1}{|c|}{L_{i_s}\!} &  \multicolumn{2}{c}{} & & & \\
    \cline{3-5}
    & & & \multicolumn{2}{|c|}{} & \multirow{2}*{} & & \\
    & & & \multicolumn{2}{|c|}{\raisebox{.6\normalbaselineskip}[0pt][0pt]{$~\GL_{n_0}~$}} & & & \\
    \cline{4-6}
    & & & & & \multicolumn{1}{|c|}{\!L_{i_s^{\ast}}\!} & & \\
    \cline{6-6}
    & & & & & & \ddots & \\
    \cline{8-8}
    & & & & & & &\multicolumn{1}{|c}{L_{i_1^{\ast}}\!\!}\\
     \end{array}
  \right).
\end{equation}

\subsubsection{}\label{constIrr1}
Recall that $W_G(L_I):=N_G(L_I)/L_I$ is isomorphic to $\prod_r\mathfrak{S}_{N_r}$. Write $N_r'=N_r/2$ if $n_0\ne r$ and write $N'_r=(N_r-1)/2$ if $n_0=r$. Then, $W_G(L_I)^{\sigma}\cong\prod_r \mathfrak{W}_{N'_r}^C$. Regarded as a block permutation matrix\footnote{By \textit{block permutation matrix}, we always mean a block matrix with each block equal to an identity matrix or a zero matrix, which itself is also a permutation matrix. Of course this is more restrictive than a block matrix that is also a permutation matrix}, an element of $\smash{\mathfrak{W}_{N'_r}^C}$ typically acts on $\prod_{\{i\mid n_i=r\}}L_i$ in the following manners,
\begin{equation}\label{ActionSymetrique}
  \renewcommand{\arraystretch}{1.3}
\left(\!
  \begin{array}{cccccc}
    \multicolumn{1}{c|}{\rn{1}{L_{i_1}}\!\!} & & & & & \\
    \cline{1-1}
    & \rn{2}{\ddots} & & & & \\
    \cline{3-3}
    & & \multicolumn{1}{|c|}{\rn{3}{L_{i_t}}\!\!} &  \multicolumn{2}{c}{} & \\
    \cline{3-4}
    & & &  \multicolumn{1}{|c|}{\rn{4}{\!L_{i_t^{\ast}}}\!\!} & & \\
    \cline{4-4}
    & & & & \rn{5}{\ddots} & \\
    \cline{6-6}
    & & & & & \multicolumn{1}{|c}{\rn{6}{\!L_{i_1^{\ast}}}\!\!}\\
     \end{array}
  \right),\quad
  \left(\!
  \begin{array}{cccccc}
    \multicolumn{1}{c|}{\rn{1'}{L_{i_1}}\!\!} & & & & & \\
    \cline{1-1}
    & \rn{2'}{\ddots} & & & & \\
    \cline{3-3}
    & & \multicolumn{1}{|c|}{\rn{3'}{L_{i_t}}\!\!} &  \multicolumn{2}{c}{} & \\
    \cline{3-4}
    & & &  \multicolumn{1}{|c|}{\rn{4'}{\!L_{i_t^{\ast}}}\!\!} & & \\
    \cline{4-4}
    & & & & \rn{5'}{\ddots} & \\
    \cline{6-6}
    & & & & & \multicolumn{1}{|c}{\rn{6'}{\!L_{i_1^{\ast}}}\!\!}\\
     \end{array}
  \right),
%\left(
%  \begin{array}{ c c c }
%     1& & \\
%     & 1&   \\
%     & & 1 \\
%     \end{array}
%  \right).
\begin{tikzpicture}[overlay,remember picture]
 \node [below= -0.1 of 1] (1a) {};
 \node [right= 0 of 1] (1b) {};
 \node [left=-0.2 of 2] (2a') {};
 \node [below=-0.15 of 2a'] (2a) {};
 \node [below= -0.2 of 2] (2b') {};
 \node [left= -0.15 of 2b'] (2b) {};
 \node [left= 0 of 3] (3a') {};
 \node [above= -0.1 of 3a'] (3a) {};
 \node [above= 0.1 of 3] (3b) {};
 
 \node [below= -0.2 of 4] (4a) {};
 \node [right= 0 of 4] (4b') {};
 \node [above= -0.1 of 4b'] (4b) {};
 \node [left=-0.2 of 5] (5a') {};
 \node [below=-0.15 of 5a'] (5a) {};
 \node [below= -0.2 of 5] (5b') {};
 \node [left= -0.2 of 5b'] (5b) {};
 \node [left= 0 of 6] (6a') {};
 \node [above= -0.1 of 6a'] (6a) {};
 \node [above= 0.1 of 6] (6b) {};
 
 \node [below= -0.1 of 1'] (1'a) {};
 \node [right= 0 of 1'] (1'b) {};
 \node [left=-0.2 of 2'] (2'a') {};
 \node [below=-0.15 of 2'a'] (2'a) {};
 \node [below= -0.2 of 2'] (2'b') {};
 \node [left= -0.15 of 2'b'] (2'b) {};
 \node [left= 0 of 3'] (3'a') {};
 \node [above= -0.1 of 3'a'] (3'a) {};
 \node [below= 0 of 3'] (3'b) {};
 
 \node [right= -0.1 of 4'] (4'a) {};
 \node [above= 0 of 4'] (4'b) {};
 \node [above= -0.45 of 5'] (5'a') {};
 \node [right= -0.2 of 5'a'] (5'a) {};
 \node [right= -0.2 of 5'] (5'b') {};
 \node [below= -0.2 of 5'b'] (5'b) {};
 \node [left= 0 of 6'] (6'a) {};
 \node [above= 0 of 6'] (6'b) {};
 
        \draw [->] (1a) to [out=265,in=190, looseness=1.3] (2a);
        \draw [->] (2b) to [out=265,in=190, looseness=1.3] (3a);
        \draw [->] (3b) to [out=85,in=10, looseness=1.3] (1b);
        
        \draw [<-] (4a) to [out=265,in=190, looseness=1.3] (5a);
        \draw [<-] (5b) to [out=260,in=190, looseness=1.3] (6a);
        \draw [<-] (6b) to [out=85,in=2, looseness=1.3] (4b);
        
        \draw [->] (1'a) to [out=265,in=190, looseness=1.3] (2'a);
        \draw [->] (2'b) to [out=265,in=190, looseness=1.3] (3'a);
        \draw [->] (3'b) to [out=265,in=185, looseness=1.3] (6'a);
        
        \draw [<-] (4'a) to [out=10,in=80, looseness=1.3] (5'a);
        \draw [<-] (5'b) to [out=15,in=80, looseness=1.3] (6'b);
        \draw [->] (4'b) to [out=85,in=5, looseness=1.3] (1'b);
 \end{tikzpicture}
\end{equation}
corresponding to a cycle of positive sign and a cycle of negative sign respectively.

If $w$ is an element of $W_G(L_I)^{\sigma}$, then a block permutation matrix\footnote{If $n$ is even, then certain blocks of this matrix should be modified by $-1$ due to our choice of $\sigma$.} that represents it, denoted by $\dot{w}$, obviously can be chosen to be $\sigma$-stable, and so there is some $g\in (G^{\sigma})^{\circ}$ such that $g^{-1}F(g)=\dot{w}$. Put $M_{I,w}:=gL_Ig^{-1}$. It is an $F$-stable and $\sigma$-stable Levi factor of a $\sigma$-stable parabolic subgroup, and there is an isomorphism $\ad g:L_I^{F_w}\cong M_{I,w}^F$. A character of $M^F_{I,w}$ is $\sigma$-stable if and only if it is identified with a $\sigma$-stable character of $L_I^{F_w}$ by $\ad g$ since $g$ is $\sigma$-stable. Let us fix such a $w$ and construct some $\sigma$-stable irreducible characters of $L_I^{F_w}$.

Write $\Lambda=\Lambda_I(\dot{w})$, following the proof of Proposition \ref{G^FLevi}. The action of $\sigma$ on $\Gamma_I$ induces an action on $\Lambda$ as $w$ commutes with $\sigma$, which allows us to use the notation $i^{\ast}$ for $i\in\Lambda$. If $i$, $j\in\Gamma_I$ belong to the same orbit of $w$, then $n_i=n_j$, which justfies the notation $n_i$ for $i\in\Lambda$. Also note that $n_i=n_{i^{\ast}}$. For any $i\in\Lambda$, denote by $d_i$ the cardinality of the orbit $i$. Write
\begin{equation}\label{Lambda1chi}
\Lambda_1=\{i\in \Lambda |i^{\ast}\ne i\}/\sim
\end{equation}
where the equivalence identifies $i$ to $i^{\ast}$. We will say that $i$ belongs to $\Lambda_1$ if $i\ne i^{\ast}$ and the equivalence class of $i$ belongs to $\Lambda_1$. Write
\begin{equation}\label{Lambda2chi}
\Lambda_2=\{i\in \Lambda |i=i^{\ast}\}\setminus\{0\}.
\end{equation}
We can choose an isomorphism
\begin{equation}\label{decdeLI}
L_I\cong\prod_{i\in \Lambda_1}(\GL_{n_i}\times\GL_{n_{i^{\ast}}})^{d_i}\times\prod_{i\in \Lambda_2}(\GL_{n_i}\times\GL_{n_{i^{\ast}}})^{d_i/2}\times\GL_{n_0}
\end{equation}
such that $F_w$ and $\sigma$ act on $L_I$ in the following manner.

The action of $F_w$ is given by:
\begin{IEEEeqnarray}{lRCl}
i=0:&\GL_{n_0}&\longrightarrow&\GL_{n_0}\label{Fn_0}\\
&A&\longmapsto &F_0(A);\nonumber\\
i\in\Lambda_1:&(\GL_{n_i}\times\GL_{n_{i^{\ast}}})^{d_i}&\longrightarrow& (\GL_{n_i}\times\GL_{n_{i^{\ast}}})^{d_i}\label{Flineaire}\\
&(A_1,B_1,A_2,B_2\ldots,A_{d_i},B_{d_i})&\longmapsto &\nonumber\\
&&&\!\!\!\!\!\!\!\!\!\!\!\!\!\!\!\!\!\!\!\!\!\!\!\!\!\!\!\!\!\!\!\!\!\!\!\!\!\!\!\!\!\!\!\!\!\!\!\!\!\!\!\!\!\!\!\!\!\!\!\!\!\!\!\!(F_0(A_{d_i}),F_0(B_{d_i}),F_0(A_1),F_0(B_1)\ldots F_0(A_{d_i-1}),F_0(B_{d_i-1}));\nonumber\\
i\in\Lambda_2:&(\GL_{n_i}\times\GL_{n_{i^{\ast}}})^{d_i/2}&\longrightarrow &(\GL_{n_i}\times\GL_{n_{i^{\ast}}})^{d_i/2}\label{Funitaire}\\
&(A_1,B_1,A_2,B_2\ldots,A_{d_i/2},B_{d_i/2})&\longmapsto & \nonumber\\
&&&\!\!\!\!\!\!\!\!\!\!\!\!\!\!\!\!\!\!\!\!\!\!\!\!\!\!\!\!\!\!\!\!\!\!\!\!\!\!\!\!\!\!\!\!\!\!\!\!\!\!\!\!\!\!\!\!\!\!\!\!\!\!\!\!
(F_0(B_{d_i/2}),F_0(A_{d_i/2}),F_0(A_1),F_0(B_1)\ldots F_0(A_{d_i/2-1}),F_0(B_{d_i/2-1})).\nonumber
\end{IEEEeqnarray}
where $F_0$ is the Frobenius of $\GL_r(k)$ that sends each entry to its q-th power, for any $r$.

The action of $\sigma$ is given by:
\begin{IEEEeqnarray}{lRCl}
i=0:&\GL_{n_0}&\longrightarrow&\GL_{n_0}\label{sigman_0}\\
&A&\longmapsto &\sigma_0(A);\nonumber\\
i\ne 0:&\GL_{n_i}\times\GL_{n_{i^{\ast}}}&\longrightarrow&\GL_{n_i}\times\GL_{n_{i^{\ast}}}\label{sigmatransposition}\\
&(A,B)&\longmapsto& (\sigma_i(B),\sigma_i(A))\nonumber
\end{IEEEeqnarray}
where $\sigma_0$ is the standard automorphism of $\GL_{n_0}$ (\S \ref{autostand}) and $\sigma_i$ is the automorphism of $\GL_{n_i}$  that sends $g$ to $\smash{\mathscr{J}_{n_i}g^{-t}\mathscr{J}^{-1}_{n_i}}$  regardless of the parity of $n_i$ (\S \ref{autostand}). We have
\begin{equation}\label{decdeLI}
L^{F_w}_I\cong\prod_{i\in \Lambda_1}(\GL_{n_i}(q^{d_i})\times\GL_{n_{i^{\ast}}}(q^{d_i}))\times\prod_{i\in \Lambda_2}\GL_{n_i}(q^{d_i})\times\GL_{n_0}(q)
\end{equation}
and $\sigma$ acts on it in the following manner,
\begin{IEEEeqnarray}{lRCl}
i=0:&\GL_{n_0}(q)&\longrightarrow&\GL_{n_0}(q)\\
&A&\longmapsto &\sigma_0(A);\nonumber\\
i\in\Lambda_1:&\GL_{n_i}(q^{d_i})\times\GL_{n_{i^{\ast}}}(q^{d_i})&\longrightarrow&\GL_{n_i}(q^{d_i})\times\GL_{n_{i^{\ast}}}(q^{d_i})\\
&(A,B)&\longmapsto& (\sigma_i(B),\sigma_i(A));\nonumber\\
i\in\Lambda_2:&\GL_{n_i}(q^{d_i})&\longrightarrow&\GL_{n_i}(q^{d_i})\\
&A&\longmapsto &\sigma_iF_0^{d_i/2}(A).\nonumber
\end{IEEEeqnarray}

Denote by $L_1$ the product of the direct factors of $L_I$ except $L_0$. If no confusion arises, we may also denote by $F_w$ and $\sigma$ their restrictions on $L_1$. With respect to the decomposition of $L_I^{F_w}$ as above, a linear character $\theta_1$ of $L_1^{F_w}$ can be written as 
\begin{equation}
\prod_{i\in\Lambda_1}(\alpha_i,\alpha_{i^{\ast}})\prod_{i\in\Lambda_2}\alpha_i\in\prod_{i\in\Lambda_1}(\Irr(\mathbb{F}_{q^{d_i}}^{\ast})\times\Irr(\mathbb{F}^{\ast}_{q^{d_i}}))\times\prod_{i\in\Lambda_2}\Irr(\mathbb{F}_{q^{d_i}}^{\ast}).
\end{equation}

The set $\Irr(W_{L_1})^{F_w}$ is in bijection with the irreducible characters of 
\begin{equation}
\prod_{i\in\Lambda_1}(\mathfrak{S}_{n_i}\times\mathfrak{S}_{n_{i^{\ast}}})\times\prod_{i\in\Lambda_2}\mathfrak{S}_{n_i}.
\end{equation}
Such a character can be written as $\varphi_1=\prod_{i\in\Lambda_1}(\varphi_i,\varphi_{i^{\ast}})\prod_{i\in\Lambda_2}\varphi_i$. 

\subsubsection{}\label{hypothesethetavarphi}
Suppose that the factors of $\theta_1$ and $\varphi_1$ satisfy
\begin{IEEEeqnarray}{l}
\text{for any }i\in\Lambda_1\sqcup\Lambda_2,\quad\alpha_i\ne 1\text{ or }\eta\circ N_{\mathbb{F}_{q^{d_i}}/\mathbb{F}_q},\\
\text{for any }i\in\Lambda_1,\quad\alpha_{i^{\ast}}=\alpha_i^{-1},\\
\text{for any }i\in\Lambda_2,\quad\alpha_i^{q^{d_i/2}}=\alpha_i^{-1},\\
\text{for any }i\in\Lambda_1,\quad\varphi_{i^{\ast}}=\varphi_i.
\end{IEEEeqnarray}
We choose $\tilde{\varphi}_1$, an extension of $\varphi_1$ to $W_{L_1}\lb F_w\rb$, in such a way that
\begin{equation}
\chi_1=R^{L_1}_{\varphi_1}\theta_1=|W_{L_1}|^{-1}\sum_{v\in W_{L_1}}\tilde{\varphi_1}(vF_w)R_{T_v}^{L_1}\theta_1,
\end{equation}
is an irreducible character of $L_1^{F_w}$.
\begin{Prop}\label{IrrLIstable}
Let $\chi_0$ be a quadratic-unipotent  character of $L_0^{F_0}$. Then, $\chi_1\boxtimes\chi_0$ is a $\sigma$-stable irreducible character of $L_I^{F_w}$. Identified with a character of $M^F_{I,w}$, its induction $R^G_{M_{I,w}}(\chi_1\boxtimes\chi_0)$ is a $\sigma$-stable irreducible character of $\GL_n(q)$.
\end{Prop}
\begin{proof}
By the hypothesis on $\theta_1$ and $\varphi_1$, $\chi_1$ is $\sigma$-stable, and so $\chi_1\boxtimes\chi_0$ is $\sigma$-stable. It follows from the definition of $R^G_{M_{I,w}}$ that if $\chi_1\boxtimes\chi_0$ is $\sigma$-stable, then $R^G_{M_{I,w}}(\chi_1\boxtimes\chi_0)$ is $\sigma$-stable.
\end{proof}

\subsubsection{}
The rest of this section is devoted to proving the following proposition.
\begin{Prop}\label{IrrMstable}
Every $\sigma$-stable irreducible character of $\GL_n(q)$ is of the form given by Proposition \ref{IrrLIstable}.
\end{Prop}

\subsubsection{}
We will need the following lemma.
\begin{Lem}
Let $M$ be a Levi subgroup of $G=\GL_n(k)$. Let $x\in G$ be such that $x^{-1}\sigma$ preserves $M$, i.e. $x^{-1}\sigma(M)x=M$. Then the action of $x^{-1}\sigma$ on $Z_M$ is the composition of $z\mapsto z^{-1}$ and an automorphism induced by an element of $N_{\GL_n}(M)$.
\end{Lem}
\begin{proof}
Let $L_J$ be a standard Levi subgroup and $g\in G$ be such that $M=gL_Jg^{-1}$. We have the maps $$L_J ~\stackrel{\ad g}{\longrightarrow}~ M ~
\begin{array}{c}
\stackrel{\ad x}{\longrightarrow}\\[-0.7em] \stackrel[\sigma]{}{\longrightarrow}
\end{array}~\sigma(M),$$ and it is equivalent to considering $Z_{L_J}$. The automorphism $$(\ad g)^{-1}\circ\sigma^{-1}\circ(\ad x)\circ(\ad g)(l),$$ can be written as $\tau^{-1}\circ \ad n$, with $\tau$ being the transpose-inverse automorphism and $n=\mathscr{J}^{-1}\sigma(g)^{-1}xg$. Since $L_J$ is automatically $\tau$-stable, we have $n\in N_G(L_J)$. Note that $\tau$ acts on the centre by inversion.
\end{proof}

Suppose that $\chi\in\Irr(\GL_n(q))$ is an irreducible character induced from the triple $(M,\varphi,\theta)$ following Theorem \ref{LS}. Let $L_J$ be a standard Levi subgroup that is $G$-conjugate to $M$. If the $G^F$-conjugacy class of $M$ is represented by $(n_1,d_1)\cdots(n_s,d_s)$, then $M\cong \prod_{i\in\Lambda}\GL_{n_i}(k)^{d_i}$, and $M^F\cong \prod_{i\in\Lambda}\GL_{n_i}(q^{d_i})$, where $\Lambda:=\Lambda_J(n)$ with $n\in N_G(L_J)$ following the notations in the proof of Proposition \ref{G^FLevi}. With respect to this decomposition, we write $\theta=(\alpha_i)_{i\in\Lambda}$ with $\alpha_i\in\Irr(\mathbb{F}^{\ast}_{q^{d_i}})$, where we have abbreviated $\alpha_i\circ\det$ as $\alpha_i$. Denote by $\hat{\alpha}_i$ the $F$-orbit of $\alpha_i$ for each $i$, and write $\hat{\alpha}_i^{-1}=\{a^{-1}\mid a\in\hat{\alpha}_i\}$. The semi-simple part of $\chi$ can then be represented by the sequence $(n_1,\hat{\alpha}_i)\cdots(n_s,\hat{\alpha}_s)$. 

\begin{Prop}\label{FsigOrbite}
In order for $\chi$ to be $\sigma$-stable, it is necessary that for each $i\in\Lambda$, there exists a unique $i^{\ast}\in\Lambda$ such that $n_i=n_{i^{\ast}}$ and $\hat{\alpha}_i^{-1}=\hat{\alpha}_{i^{\ast}}$.
\end{Prop}
\begin{proof}
Assume that $\chi$ is $\sigma$-stable. Theorem \ref{LS} implies that there exists $x\in G^F$ such that
\begin{eqnarray}
\ad x(M)&=&\sigma(M), \\
(\ad x)^{\ast}\sigma_{\ast}\theta&=&\theta,\label{sigmatheta=theta}
\end{eqnarray} 
so $x$ satisfies the assumption in the previous lemma. Put $\mathcal{Z}_M:=Z_M/[M,M]$ and regard $\theta$ as a character of $\mathcal{Z}_M^F$. Choose an isomorphism $\mathcal{Z}_M\cong (k^{\ast})^r$ for some $r$. Let $v\in\mathfrak{S}_r$ be such that the action of $x^{-1}\sigma$ on $\mathcal{Z}_M$ is the composition of $v$ and the inversion. Since $x$ and $\sigma$ are $F$-stable, the action of $v$ is compatible with the action of $F$ and so induces an action on $\mathcal{Z}_M^F$. If we express the action of $F$ as the composition of some permutation $\tau\in\mathfrak{S}_r$ and raising each entry to the $q$-th power, then $v$ commutes with $\tau$. The centraliser of $\tau$ is described by its orbits in $\{1,\ldots,r\}$: an element commuting with $\tau$ induces a permutation among the orbits of the same size and a cyclic permutation within each orbit. If we write $\theta=(\alpha_i)$, then the action of $x^{-1}\sigma$ on $\mathcal{Z}_M^F$ induces $(\alpha_i)\mapsto(\alpha_{i^{\ast}}^{-q^c})$ for some $0\le c< d_i$, with $i^{\ast}:=v^{-1}(i)$, whence the assertion. 

For any $i\in\Lambda$, the $i^{\ast}$ as claimed in the proposition is unique. Suppose $i^{\ast}$ and $i^{\ast\prime}$ with $i^{\ast}\ne i^{\ast\prime}$ both satisfy the assertion in the proposition, then $\hat{\alpha}_{i^{\ast}}=\hat{\alpha}_{i^{\ast\prime}}$, which contradicts the regularity of  $\theta$.
\end{proof}
\begin{Rem}
If $i^{\ast}=i$, then it is necessary that $d_i$ is an even number, and $\alpha_i^{-1}=\alpha_i^{q^{d_i/2}}$.
\end{Rem}
\begin{Rem}\label{pm}
There are at most two $i\in\Lambda$ such that $\alpha_i^2=1$ in order for $\theta$ to be regular. We denote them by $\pm$. It is necessary that $d_+=d_-=1$. Also denote by $\pm$ the corresponding two elements of $\Gamma$.
\end{Rem}

\subsubsection{}
The previous proposition allows us to define the sets
\begin{eqnarray}
\Lambda_1&=&\{i\in \Lambda |i^{\ast}\ne i\}/\sim\\
\Lambda_2&=&\{i\in \Lambda |i=i^{\ast}\}\setminus\{\pm\},
\end{eqnarray}
in a way similar to (\ref{Lambda1chi}) and (\ref{Lambda2chi}). Since the $G^F$-conjugacy class of $M$ only depends on the coset of $n$, we can assume that $n$ is a block permutation matrix. There exists a block permutation matrix $x$ such that $L_{J'}:=xL_Jx^{-1}$ is of the form
\[
  \renewcommand{\arraystretch}{1.3}
\left(\!
  \begin{array}{ccccccccc}
    \multicolumn{1}{c|}{L_{i_1}\!} & & & & & & & &\\
    \cline{1-1}
    & \ddots & & & & & & &\\
    \cline{3-3}
    & & \multicolumn{1}{|c|}{L_{i_s}\!} &  \multicolumn{2}{c}{} & & & &\\
    \cline{3-5}
    & & & \multicolumn{2}{|c|}{} & \multirow{2}*{} & & &\\
    & & & \multicolumn{2}{|c|}{\raisebox{.6\normalbaselineskip}[0pt][0pt]{~~~$L_+$~~~}} & & & &\\
    \cline{4-6}
    & & & & & \multicolumn{1}{|c|}{\!L_{-}\!\!} & & &\\
    \cline{6-7}
    & & & & & & \multicolumn{1}{|c|}{\!L_{i_s'}\!\!} & & \\
    \cline{7-7}
    & & & & & & & \ddots & \\
    \cline{9-9}
    & & & & & & & &\multicolumn{1}{|c}{\!L_{i_1'}\!\!}\\
     \end{array}
  \right).
\]
where $L_{i'{\!\!}_j}\cong L_{i_j}$ for any $j$. From this, we can recover the induction data of \S \ref{hypothesethetavarphi} as follows.

On $L_{J'}$, the Frobenius concerned is $F_{xnx^{-1}}$, which fixes $L_+$ and $L_-$ as $d_+=d_-=1$. We can further conjugate by a permutation matrix that normalises $L_{J'}$, say $y$, in such a way that the action of $F_{yxnx^{-1}y^{-1}}$ on $L_{J'}$ is built up from (\ref{ActionSymetrique}) and moreover the $\alpha_i$'s satisfy the hypothesis \S \ref{hypothesethetavarphi}. Such $y$ exists because a conjugacy class in $N_G(L_J)/L_J$ is uniquely determined by the $d_i$'s. Denote by $\theta_{J'}$ and $\varphi_{J'}$ the characters associated to $L_{J'}$ that are transferred from $\theta$ and $\varphi$ via $\ad g$, $\ad x$ and $\ad y$.

Put $v=yxnx^{-1}y^{-1}$, it is a $\sigma$-stable block permutation matrix by the choice of $y$. Let $h\in (G^{\sigma})^{\circ}$  be such that $h^{-1}F(h)=v$. Then $M_{J',v}=hL_{J'}h^{-1}$ is an $F$-stable Levi subgroup that is $G^F$-conjugate to $M$, because $nL_J$ is conjugate to $vL_{J'}$. If we regard $\theta_{J'}$ and $\varphi_{J'}$ as some characters associated to $M_{J',v}$ via the isomorphism $\ad h$, then $\chi$ is equal to the induction $R^G_{\varphi_{J'}}\theta_{J'}$ for a triple $(M_{J',v},\varphi_{J'},\theta_{J'})$.

Define $L_I$ to be the standard Levi subgroup of the form (\ref{LI}) such that $n_0=n_++n_-$ and that $L_I$ coincides with $L_{J'}$ away from $\GL_{n_0}$. We see that $M_{I,v}:=hL_Ih^{-1}$ is a $\sigma$-stable and $F$-stable Levi factor of a $\sigma$-stable parabolic subgroup. Moreover, it contains $M_{J',v}$ and $\sigma(M_{J',v})$. 

Note that $\leftidx{^{\sigma}\!}{\chi}$ is defined by the triple $(\sigma(M_{J',v}),\sigma_{\ast}\varphi_{J'},\sigma_{\ast}\theta_{J'})$ and $\sigma_{\ast}\theta_{J'}$ is already equal to $\theta_{J'}$ away from $\GL_{n_o}$ by construction. By Theorem \ref{LS}, in order for $\chi$ to be $\sigma$-stable, it is necessary that $\sigma_{\ast}\varphi_{J'}=\varphi_{J'}$ away from $\GL_{n_0}$. And it suffices since in the component $\GL_{n_0}$ we have a quadratic-unipotent character. This completes the proof of Proposition \ref{IrrMstable}.

\subsubsection{}\label{Tchi}
The \textit{type} of a $\sigma$-stable irreducible character consists of some non negative integers $n_{\pm}$, and some positive integers $n_i$, $d_i$, $n'_j$, and $d'_j$ parametrised by some possibly empty finite sets $\Lambda_1$ and $\Lambda_2$, denoted by
\begin{equation}
\mathfrak{t}=n_+n_-(n_i,d_i)_{i\in\Lambda_1}(n'_j,d'_j)_{j\in\Lambda_2},
\end{equation}
satisfying
\begin{equation}
n=n_++n_-+\sum_{i}2n_id_i+\sum_{j}2n_j'd_j'.
\end{equation}
If $\smash{\tilde{\mathfrak{t}}=\tilde{n}_+\tilde{n}_-(\tilde{n}_i,\tilde{d}_i)_{i\in\tilde{\Lambda}_1}(\tilde{n}'_j,\tilde{d}'_j)_{j\in\tilde{\Lambda}_2}}$ is another sequence of integers, we regard it as the same as $\mathfrak{t}$ if and only if there exist some bijections $\smash{\Lambda_1\cong\tilde{\Lambda}_1}$ and $\smash{\Lambda_2\cong\tilde{\Lambda}_2}$ such that the integers are matched (and $n_+=\tilde{n}_+$ and $n_-=\tilde{n}_-$). We denote by $\mathfrak{T}_{\chi}$ the set of the types of the $\sigma$-stable irreducible characters of $\GL_n(q)$.

Given $\mathfrak{t}\in \mathfrak{T}_{\chi}$, denote by $\bar{\mathfrak{T}}_{\chi}(\mathfrak{t})$ the set of the data
\begin{equation}
\bar{\mathfrak{t}}=\lambda_+\lambda_-(\lambda_i,\hat{\alpha}_i)_{i\in\Lambda_1}(\lambda'_j,\hat{\alpha}'_j)_{j\in\Lambda_2}
\end{equation}
satisfying
\begin{itemize}
\item[-]
$\lambda_{\pm}\in\mathcal{P}_{n_{\pm}}$, $\lambda_i\in\mathcal{P}_{n_i}$, $\lambda'_j\in\mathcal{P}_{n'_j}$, for the integers $n_+$, $n_-$, $n_i$ and $n'_j$ associated to $\mathfrak{t}$;
\item[-]
$\hat{\alpha}_i\subset\Irr(\bar{\mathbb{F}}_{\smash{q^{d_i}}}^{\ast})$ is an $F$-orbit of order $d_i$ that is not stable under inversion, with $\hat{\alpha}_i$ identified with $\hat{\alpha}^{-1}_i$;
\item[-]
$\hat{\alpha}'_j\subset\Irr(\bar{\mathbb{F}}^{\ast}_{\smash{q^{2d'_j}}})$ is an $F$-orbit of order $2d'_j$ that is stable under inversion;
\item[-]
$\hat{\alpha}_{i}\ne\hat{\alpha}^{\pm 1}_{i'}$ if $i\ne i'$ and $\hat{\alpha}'_{j}\ne\hat{\alpha}'_{j'}$ if $j\ne j'$.
\end{itemize}
If $\tilde{\lambda}_+\tilde{\lambda}_-(\tilde{\lambda}_i,\hat{\beta}_i)_{i\in\tilde{\Lambda}_1}(\tilde{\lambda}_j,\hat{\beta}'_j)_{j\in\tilde{\Lambda}_2}$ is another such datum, we regard it as the same as $\bar{\mathfrak{t}}$ if and only if $\tilde{\lambda}_+=\lambda_+$, $\tilde{\lambda}_-=\lambda_-$, and there exist bijetions $\Lambda_1\cong\tilde{\Lambda}_1$ and $\Lambda_2\cong\tilde{\Lambda}_2$ such that the partitions and orbits are matched.

Write $\bar{\mathfrak{T}}_{\chi}=\sqcup_{\mathfrak{t}\in\mathfrak{T}_{\chi}}\bar{\mathfrak{T}}_{\chi}(\mathfrak{t})$. By Proposition \ref{IrrLIstable} and Proposition \ref{IrrMstable}, the $\sigma$-stable irreducible characters are in bijection with $\bar{\mathfrak{T}}_{\chi}$. 

\subsubsection{}
Equivalently, the set $\bar{\mathfrak{T}}_{\chi}$ can be described as follows. Write $$k^{\vee}=\lim_{\longrightarrow}\Irr(\mathbb{F}_{q^n}^{\ast}),$$ as $n$ runs over positive integers. Denote by $\sigma$ the inversion $a\mapsto a^{-1}$ on this set. Denote by $\Phi^{\ast}$ the set of $\lb F\rb\times \lb\sigma\rb$-orbits in $k^{\vee}$. For any map of sets $f:\Phi^{\ast}\rightarrow\mathcal{P}$, write $$||f||=\sum_{x\in\Phi^{\ast}}|f(x)|\cdot|x|,$$where $|x|$ is the size of the orbit. Then $\bar{\mathfrak{T}}_{\chi}$ is identified with the set $$\{f:\Phi^{\ast}\mapsto\mathcal{P}; ||f||=n\}.$$ The $\lb F\rb\times \lb\sigma\rb$-orbits are either the single points $\{1\}$, $\{\eta\}$, or the union of two $F$-orbits that are not stable under inversion, or a single $F$-orbit that is stable under inversion. One easily sees how $\bar{\mathfrak{T}}_{\chi}$ is recovered from these maps.

%%%%%%%%%%%%%%%%%%%% Classification des Classes de Conjugaisons
\section{Parametrisation of Conjugacy Classes}\label{SecParaClass}
\subsection{$F$-Stable Quasi-Semi-Simples Classes}
A $G$-conjugacy class contains some $G^F$-conjugacy classes if and only if it is $F$-stable. We will first give the parametrisation of the $F$-stable quasi-semi-simple conjugacy classes in $G\sigma$. Recall that in $\leftidx{^o\!}{\bar{G}}^F$, we denote by $\sigma$ the element $t_0\sigma'$ (\textit{cf.} Convention \ref{convsigma}).
\subsubsection{}
We begin with the parametrisation of the quasi-semi-simple $G$-conjugacy classes. We take for $T$ the maximal torus consisting of the diagonal matrices, then $(T^{\sigma})^{\circ}$ consists of the matrices
\begin{eqnarray}
\diag(a_1,a_2,\ldots,a_m,a_m^{-1},\ldots,a_2^{-1},a_1^{-1}), \text{if $n=2m$,}\\
\diag(a_1,a_2,\ldots,a_m,1,a_m^{-1},\ldots,a_2^{-1},a_1^{-1}), \text{if $n=2m+1$,}
\end{eqnarray}
with $a_i\in k^{\ast}$ for all $i$, and the commutator $[T,\sigma]$ consists of the matrices
\begin{eqnarray}
\diag(b_1,b_2,\ldots,b_m,b_m,\ldots,b_2,b_1),\text{if $n=2m$,}\\
\diag(b_1,b_2,\ldots,b_m,b_{m+1},b_m,\ldots,b_2,b_1),\text{if $n=2m+1$,}
\end{eqnarray}
with $b_i\in k^{\ast}$ for all $i$. So the elements of  $S:=[T,\sigma]\cap(T^{\sigma})^{\circ}$ are the matrices 
\begin{eqnarray}
\diag(e_1,e_2,\ldots,e_m,e_m,\ldots,e_2,e_1),\text{if $n=2m$,}\\
\diag(e_1,e_2,\ldots,e_m,1,e_m,\ldots,e_2,e_1),\text{if $n=2m+1$,}
\end{eqnarray} 
with $e_i=\pm1$ for all $i$. We index the entries of a diagonal matrix by the set $$\{1,2,\ldots,m,-m,\ldots,-2,-1\}$$ or the set $$\{1,2,\ldots,m,0,-m,\ldots,-2,-1\}$$ according to the parity of $n$ so that every matrix in $(T^{\sigma})^{\circ}$ satisfies $a_{-i}=a_i^{-1}$ for all $i$. We will abbreviate an element of $(T^{\sigma})^{\circ}$ as $[a_1,\ldots,a_m]$ regardless of the parity of $n$.

We have the following proposition.
\begin{Prop}(\cite[Proposition 1.16]{DM18})\label{6.11}
The quasi-semi-simple classes in $G\sigma$ are in bijection with the $W^{\sigma}$-orbits in $T/[T,\sigma]\cong(T^{\sigma})^{\circ}/S$. 
\end{Prop}

That is to say, the class of $[a_1,\ldots,a_m]\sigma$ is invariant under the following operations,
\begin{itemize}
\item[-] Permutation of the $a_i$'s;
\item[-] $a_i\mapsto a_i^{-1}$, for any $i$;
\item[-] $a_i\mapsto -a_i$, for any $i$,
\end{itemize}
and $[b_1,\ldots,b_m]\sigma$ belongs to the same class if it only differs from $[a_1,\ldots,a_m]\sigma$ by these operations. For another description of these conjugacy classes, see also \cite[Example 7.3]{DM15}.
\begin{Rem}
The parametrisation given by this proposition relies on the choice of $\sigma$. In either $\leftidx{^s\!}{\bar{G}}$ or $\leftidx{^o\!}{\bar{G}}$, we will use the symplectic type quasi-central element $\sigma$.
\end{Rem}

\subsubsection{}
Denote by $\hat{k}$ the quotient of $k^{\ast}$ by the action of $\mathbb{Z}/2\mathbb{Z}\times\mathbb{Z}/2\mathbb{Z}$:
\begin{equation}
\begin{split}
(1,0):a\mapsto a^{-1},\quad
(0,1):a\mapsto -a.
\end{split}
\end{equation} 
For any $a\in k$, denote by $\hat{a}$ the set $\{a,-a,a^{-1},-a^{-1}\}$. Write $\mathbb{S}=(\mathbb{Z}/2\mathbb{Z}\times\mathbb{Z}/2\mathbb{Z})^m$. The torus $(T^{\sigma})^{\circ}$ is isomorphic to $(k^{\ast})^m$. Let $\mathbb{S}$ act on it componentwise, i.e. each component $\mathbb{Z}/2\mathbb{Z}\times\mathbb{Z}/2\mathbb{Z}$ acts on $k^{\ast}$ as above. Then the quasi-semi-simple conjugacy classes in $G\sigma$ are parametrised by the $\mathfrak{S}_m$-orbits in $\hat{k}^m$. The conjugacy class of $t\sigma=[a_1,\ldots,a_m]\sigma$ is determined by the multiset $\{\hat{a}_1,\ldots,\hat{a}_m\}$. We may regard its elements as the eigenvalues of $t\sigma$.

The action of $F$ on $k$ induces an action on $\hat{k}$, given by $$\hat{a}\mapsto \hat{a}^q:=\{a^q,-a^q,a^{-q},-a^{-q}\}.$$ Denote by $\Phi_{\ast}$ the set of $F$-orbits in $\hat{k}$. Let $\hat{\alpha}\in\Phi_{\ast}$ and take $\hat{a}\in\hat{\alpha}$. Write $d=|\hat{\alpha}|$. Denote by $e$ and $\epsilon$ the signs such that $\smash{a^{q^d}}=ea^{\epsilon}$, for any $a\in\hat{a}$. Note that $e$ and $\epsilon$ are independent of the choice of $\hat{a}\in\hat{\alpha}$ or the choice of $a\in\hat{a}$. We say that $\hat{\alpha}$ is an orbit of type $(d,\epsilon,e)$.

There is some ambiguity with the type thus defined. Let $\hat{\alpha}_1$ be an orbit of type $(d_1,\epsilon_1,e_1)$ and let $\hat{\alpha}_2$ be an orbit of type $(d_2,\epsilon_2,e_2)$. Obviously if $d_1> d_2$, then $\hat{\alpha}_1$ and $\hat{\alpha}_2$ are distinct orbits. Suppose $d_1=d_2=d$. If $x\in k^{\ast}$ satisfies both of the two equations $\smash{x^{q^{d_1}}}=e_1x^{\epsilon_1}$ and $\smash{x^{q^{d_2}}}=e_2x^{\epsilon_2}$, then $e_1e_2x^{\epsilon_1-\epsilon_2}=1$. In order for this equaiton to be solvable, either $e_1=e_2$, $\epsilon_1=\epsilon_2$, that is, $\hat{\alpha}_1$ and $\hat{\alpha}_2$ coincide, or $e_1=e_2$, $\epsilon_1=-\epsilon_2$, which gives $x^2=1$, or $e_1=-e_2$, $\epsilon_1=-\epsilon_2$, which gives $x^2=-1$. The latter ones are the orbits $\{1,-1\}$ and $\{\mathfrak{i},-\mathfrak{i}\}$, so $d=1$, and they are said to be of type $(1)$ and of type $(\mathfrak{i})$ respectively. Except these cases, orbits of different types are all distinct. In the following, these two orbits are treated separately and so there will be no confusion among types.

\subsubsection{}
Let us define some combinatorial data that parametrise the $F$-stable quasi-semi-simple conjugacy classes.

We call the type of an $F$-stable quasi-semi-simple conjugacy class the data consisting of some non negative integers $n_+$, $n_-$, with the parity of $n_+$ being that of $n$, some positive integers $n_i$, $d_i$ and some signs $e_i$ and $\epsilon_i$, parametrised by a possibly empty finite set $\Lambda$, denoted by
\begin{equation}
\mathfrak{t}=n_+n_-(n_i,d_i,\epsilon_i,e_i)_{i\in\Lambda},
\end{equation}
satisfying
\begin{equation}
\sum_{i\in\Lambda}2n_id_i+n_++n_-=n.
\end{equation}
If $\tilde{\mathfrak{t}}=\tilde{n}_+\tilde{n}_-(\tilde{n}_i,\tilde{d}_i,,\tilde{\epsilon}_i,\tilde{e}_i)_{i\in\tilde{\Lambda}}$ is another sequence of integers, we regard it as the same as $\mathfrak{t}$ if and only if there exists a bijection $\Lambda\isomLR\tilde{\Lambda}$ such that the integers and the signs are mathched, and moreover, $n_+=\tilde{n}_+$ and $n_-=\tilde{n}_-$. We will denote by $\mathfrak{T}_{C,s}$ the set of the types of the $F$-stable quasi-semi-simple conjugacy classes. 
 
Given $\mathfrak{t}\in \mathfrak{T}_{C,s}$, denote by $\bar{\mathfrak{T}}_{C,s}(\mathfrak{t})$ the set of the data
\begin{equation}
\bar{\mathfrak{t}}=n_+n_-(n_i,\hat{\alpha}_i)_{i\in\Lambda},
\end{equation}
satisfying
\begin{itemize}
\item[-]
for any $i$, $\hat{\alpha}_i$ is an orbit of type $(d_i,\epsilon_i,e_i)\ne(1)$ or $(\mathfrak{i})$;
\item[-]
if $i\ne i'$, then $\hat{\alpha}_i\ne\hat{\alpha}_{i'}$.
\end{itemize}
If $\bar{\mathfrak{s}}=m_+m_-(m_i,\hat{\beta}_i)_{i\in\tilde{\Lambda}}$ is another such datum, we regard it the same as $\bar{\mathfrak{t}}$ if and only if there exists a bijection $\Lambda\isomLR\tilde{\Lambda}$ such that the integers and the orbits are matched, and moreover, $n_+=m_+$ and $n_-=m_-$. We will denote by $\bar{\mathfrak{T}}_{C,s}$ the union of $\bar{\mathfrak{T}}_{C,s}(\mathfrak{t})$ as $\mathfrak{t}$ runs over $\mathfrak{T}_{C,s}$.

It will sometimes be convenient to distinguish between the $i$'s with $\epsilon_i=1$ and the $i$'s with $\epsilon_i=-1$. Put $\Lambda_1\subset\Lambda$ to be the subset of the $i$'s such that $\epsilon_i=1$ and put $\Lambda_2=\Lambda\setminus\Lambda_1$. The following notations will also be used,
\begin{equation}
\begin{split}
\mathfrak{t}&=n_+n_-(n_i,d_i,e_i)_{i\in\Lambda_1}(n'_j,d'_j,e'_j)_{j\in\Lambda_2},\\
\bar{\mathfrak{t}}&=n_+n_-(n_i,\hat{\alpha}_i)_{i\in\Lambda_1}(n'_j,\hat{\alpha}'_j)_{j\in\Lambda_2}.
\end{split}
\end{equation}

There is another equivalent description of $\bar{\mathfrak{T}}_{C,s}$. Write $\Phi_{\ast}^{\circ}=\Phi_{\ast}\setminus(\hat{1}\sqcup\hat{\mathfrak{i}})$. For any map of sets $\Phi_{\ast}\rightarrow\mathbb{Z}_{\ge 0}$, define $$||f||=f(\hat{1})+f(\hat{\mathfrak{i}})+\sum_{x\in\Phi^{\circ}_{\ast}}2f(x)\cdot|x|,$$ where $|x|$ is the size of the orbit. Then $\bar{\mathfrak{T}}_{C,s}$ can be identified with the set $$\{f:\Phi_{\ast}\rightarrow\mathbb{Z}_{\ge 0};||f||=n\}.$$ Under this identification, $n_+$ is the image of the orbit of type $(1)$, $n_-$ is the image of the orbit of type $(\mathfrak{i})$ and $\Lambda$ parametrises the inverse image of $\mathbb{Z}_{>0}$ in $\Phi^{\circ}_{\ast}$.

\begin{Prop}\label{TCs}
The $F$-stable quasi-semi-simple conjugacy classes in $G\sigma$ are in bijection with $\bar{\mathfrak{T}}_{C,s}$.
\end{Prop}
\begin{proof}
Define a map
\begin{equation}\label{psi}
\psi:\bar{\mathfrak{T}}_{C,s}\longrightarrow\{\text{$F$-stable quasi-semi-simple classes in $G\sigma$}\}\subset \hat{k}^m/\mathfrak{S}_m
\end{equation}
as follows.

Write $\bar{\mathfrak{t}}=n_+n_-(n_i,\hat{\alpha}_i)_{i\in\Lambda}$ with $\hat{\alpha}_i$ of type $(d_i,\epsilon_i,e_i)$. We are going to define an element of $\hat{k}^m$ from $\bar{\mathfrak{t}}$, regarding the elements of $\hat{\alpha}_i$'s as eigenvalues and the $n_i$'s as their multiplicities. We will write $\psi(\bar{\mathfrak{t}})=(\hat{a}_j)_{1\le j\le m}$.
 
(i). Take $[n_+/2]$ subsets of $\{1,\ldots,m\}$, each consisting of a point, which will be called of type $(1)$, and then $(n_-/2)$ subsets, each consisting of a point, which will be called of type $(\mathfrak{i})$, and take for each $i\in\Lambda$, $n_i$ subsets of cardinality $d_i$. These subsets, combined with $\{0\}$ if $n$ is odd (we require that $\{0\}$ is of type $(1)$.), form a partition of $\{1,\ldots,m\}$($\cup\{0\}$) and we denote it by $(I_r)_r$.  

(ii). Choose for each $r$ an identification $I_r\cong\mathbb{Z}/d_r\mathbb{Z}$, where $d_r:=d_i$ if $I_r$ comes from $i$. 

(iii). If $I_r$ comes from $i\in\Lambda$ by the procedure (i), and $\hat{\alpha}_i=\{\hat{a},\hat{a}^q,\ldots,\hat{a}^{q^{d_i-1}}\}$, define for all $k\in I_r$, $\hat{a}_k:=\hat{a}^{q^k}$, under the identification $I_r\cong\mathbb{Z}/d_r\mathbb{Z}$. If $I_r$ is of type $(1)$(resp. $(\mathfrak{i})$), we define the only entry of $t$ corresponding to $I_r$ to be $\hat{1}$(resp. $\hat{\mathfrak{i}}$). Thus, we have defined an element of $\hat{k}^m$, whence a quasi-semi-simple conjugacy class in $G\sigma$.

The class $\psi(\bar{\mathfrak{t}})$ does not depend on the choices of the subsets $I_r$ or the identifications $I_r\cong\mathbb{Z}/d_r\mathbb{Z}$ due to the conjugation by $\mathfrak{S}_m$. Observe that the class of $\psi(\bar{\mathfrak{t}})$ is $F$-stable if and only if $\{\hat{a}_j\}_{1\le j\le m}$ and $\{\hat{a}_j^q\}_{1\le j\le m}$ coincide as multisets. This is satisfied by $\psi(\bar{\mathfrak{t}})$ since to each $I_r$ is associated an $F$-orbit. The map $\psi$ is surjective because the map $\hat{a}_j\mapsto\hat{a}_j^q$ identifies $\{\hat{a}_j\}_{1\le j\le m}$ with itself, and so we can choose a bijection $f$ from $\{1,\ldots,m\}$ to itself such that $\hat{a}_{f(j)}=\hat{a}_{j}^q$. Then $f$ defines a permutation of $\{1,\ldots,m\}$, which can be expressed as a product of some cyclic permutations, so the above constructions can be reversed. Injectivity of $\psi$ is obvious.
\end{proof}

\subsection{Centralisers and $G^F$-Classes}
We will see that the centraliser of a quasi-semi-simple element is in general a product of a symplectic group, an orthogonal group and some linear groups.
\subsubsection{}\label{centgeo}
Let $C$ be an $F$-stable quasi-semi-simple conjugacy class  corresponding to $\bar{\mathfrak{t}}=n_+n_-(n_i,\hat{\alpha}_i)_{i\in\Lambda_1}(n_j,\hat{\alpha}_j)_{j\in\Lambda_2}$ following Proposition \ref{TCs}, and we denote its type by $\mathfrak{t}=n_+n_-(n_i,d_i,e_i)_{i\in\Lambda_1}(n'_j,d'_j,e'_j)_{j\in\Lambda_2}$.  Let $t\sigma$, with $t\in(T^{\sigma})^{\circ}$ be a representative of $C$ as defined in the proof of Proposition \ref{TCs}. 
\begin{Lem}
We have,
\begingroup
\allowdisplaybreaks
\begin{align*}
C_G(t\sigma)\cong\Sp_{n_+}(k)\times\Ort_{n_-}(k)\times\prod_{i\in\Lambda_1}\GL_{n_i}(k)\times\prod_{j\in\Lambda_2}\GL_{n_j'}(k),\quad &\text{if $n$ is even,}\\
C_G(t\sigma)\cong\Sp_{n_-}(k)\times\Ort_{n_+}(k)\times\prod_{i\in\Lambda_1}\GL_{n_i}(k)\times\prod_{j\in\Lambda_2}\GL_{n_j'}(k),\quad &\text{if $n$ is odd.}\\
\end{align*}
\endgroup
In particular, the isomorphism class of $C_G(t\sigma)$ only depends on the type of $C$.

\end{Lem}
The numbers $n_+$ and $n_-$ are exchanged only because we have chosen different types of $\sigma$ for even and odd $n$.\begin{proof}
If $z\in G$ commutes with $t\sigma$, then it commutes with $t\sigma t\sigma=t\sigma(t)\sigma^2$, with $\sigma^2=\pm 1$ being central. Let us calculate $C_{C_G(t\sigma(t))}(t\sigma)$. That the $\hat{\alpha}_i$'s are pairwise distinct means that for $a_i\in\alpha_i$, $a_j\in\alpha_j$, $i\ne j$, we have $\smash{a_i^{q^c}\ne\pm a_j^{\pm 1}}$, for all $c$, so $\smash{a_i^{2q^c}\ne a_j^{\pm 2}}$, for all $c$. Besides, the integers $n_+$ and $n_-$ become the multiplicities of $1$ and $-1$ in $t\sigma(t)$ respectively. Consequently, the centraliser of $t\sigma(t)$ is a Levi subgroup $\bar{L}_0:=C_G(t\sigma(t))$ isomorphic to
\begin{equation}
\prod_{i\in\Lambda_1}(\GL_{n_i}\times\GL_{n_i})^{d_i}\times\prod_{j\in\Lambda_2}(\GL_{n'_j}\times\GL_{n'_j})^{d'_j}\times\GL_{n_+}\times\GL_{n_-}
\end{equation}
with the action of $\sigma$ given by
\begin{equation}\label{actiondesig}
\begin{split}
\sigma:\GL_{n_i}\times\GL_{n_i}&\longrightarrow \GL_{n_i}\times\GL_{n_i}\\
(g,h)&\longmapsto (\sigma_0(h),\sigma_0(g)),
\end{split}
\end{equation}
for all $i\in\Lambda_1$, and similarly for $j\in\Lambda_2$, where $\sigma_0(g)=J\leftidx{^t}{g}{^{-1}}J^{-1}$, with $(J)_{ab}=\delta_{a,n_i+1-b}$, for any $i$ or $j$. 

The action induced by $t\sigma$ on each $\GL_{n_i}\times\GL_{n_i}$ coincides with that of $\sigma$. If $n$ is even, and $\bar{G}=\leftidx{^s\!}{\bar{G}}$, the action induced by $t\sigma$ on $\GL_{n_+}$ and $\GL_{n_-}$ are respectively the automorphisms associated to $\mathscr{J}_n$ or $\mathscr{J}'_n$ defined in \S \ref{autostand}. It follows that in $\bar{G}=\leftidx{^s\!}{\bar{G}}$,
\begin{equation}
L_0:=C_G(t\sigma)\cong\Sp_{n_+}(k)\times\Ort_{n_-}(k)\times\prod_{i\in\Lambda_1}\GL_{n_i}(k)\times\prod_{j\in\Lambda_2}\GL_{n_j'}(k).
\end{equation}
The case of $\bar{G}=\leftidx{^o\!}{\bar{G}}$ and that of odd $n$ are similar.
\end{proof}

\subsubsection{}\label{repdeclass}
Let us introduce a sign $\eta\in\{\pm 1\}$ that can be $-1$ only if
\begin{itemize}
\item[-] $n$ is even and $n_->0$, or
\item[-] $n$ is odd and $n_+>0$,
\end{itemize}
or rather, if the orthogonal factor of the centraliser is non trivial.

\begin{Prop}\label{qssrat}
The quasi-semi-simple conjugacy classes in $G^{F}\sigma$ are parametrised by the data
\begin{equation}
\{(\eta,\bar{\mathfrak{t}})\}\subset\{\pm 1\}\times\bar{\mathfrak{T}}_{C,s}.
\end{equation}
\end{Prop}
\begin{proof}
Clear.
\end{proof}
To simplify the notation, we will write $\eta\bar{\mathfrak{t}}$ instead of $(\eta,\bar{\mathfrak{t}})$. One should be careful however as to which $G^F$-class corresponds to $\eta=+1$ in this parametrisation. This is explained below.

If the centraliser of a semi-simple element has two connected components, then the corresponding two $G^F$-classes can be distinguished by the homomorphism 
\begin{equation}\label{epsilon}
\GL_n(q)\rtimes\lb\sigma\rb\stackrel{\det}{\longrightarrow} \mathbb{F}^{\ast}_q/(\mathbb{F}^{\ast}_q)^2\longrightarrow\bs\mu_2.
\end{equation}
The first map sends $g\in G(q)$ to $\det(g)\mod(\mathbb{F}^{\ast}_q)^2$ and sends $\sigma$ to $1$ (if $n$ is even, we require that it sends $\sigma'$ to $1$), and the second map is the nontrivial homomorphism. The value of $\eta$ is identified with the image of the corresponding $G^F$-class under this homomorphism. In fact, the above homomorphism extends the character $\eta\circ\det$ of $\GL_n(q)$, with $\eta$ being the order 2 irreducible character of $\mathbb{F}_q^{\ast}$. This explains the notation $\eta$.

To see that the above homomorphism can distinguish the two $G^F$-conjugacy classes contained in the same $G$-conjugacy class, we argue as follows. Let $t\sigma\in G^F\sigma$ be such that $C_G(t\sigma)$ has two connected component. Then according to our concrete description of $C_G(t\sigma)$, its two connected components are distinguished by the values of the determinant, which is $\pm 1$, corresponding to the two connected components of the orthogonal factors. Let $g\in G$ be such that $g^{-1}F(g)=z\in C_G(t\sigma)\setminus C_G(t\sigma)^{\circ}$, then $gs\sigma g^{-1}$ is a representative of another $G^F$-conjugacy class. Applying the determinant to the equality $g^{-1}F(g)=z$ gives $\det (g)^{q-1}=-1$, so that $\det(g)^2\in\mathbb{F}_q^{\ast}\setminus(\mathbb{F}_q^{\ast})^2$; applying the above homomorphism to the element $gs\sigma g^{-1}$ gives $\det(s)\det(g)^2$, whence the claim.

\subsubsection{}
The centraliser of each quasi-semi-simple element of $G^F\sigma$ is given as below. Let $t\sigma\in G^F\sigma$ be a quasi-semi-simple element corresponding to
\begin{equation}
\eta d_+d_-(n_i,\hat{\alpha}_i)_{i\in\Lambda_1}(n'_j,\hat{\beta}_j)_{j\in\Lambda_2}.
\end{equation}

If $n$ is even, then its centraliser in $G^{F}$ is
\begin{equation}
\Sp_{n_+}(q)\times\Ort^{\eta}_{n_-}(q)\times\prod_{i\in\Lambda_1}\GL_{n_i}(q^{d_i})\times\prod_{j\in\Lambda_2}\GL^{-}_{n'_j}(q^{d'_j}).
\end{equation}

If $n$ is odd, then its centraliser in $G^F$ is
\begin{equation}
\Sp_{n_-}(q)\times\Ort_{n_+}(q)\times\prod_{i\in\Lambda_1}\GL_{n_i}(q^{d_i})\times\prod_{j\in\Lambda_2}\GL^{-}_{n'_j}(q^{d'_j}).
\end{equation}
Note that for odd $n_+$, $\Ort^+_{n_+}(q)$ is isomorphic to $\Ort^-_{n_+}(q)$.

\subsubsection{}
We refer to \S \ref{unirat} for the parametrisation of the unipotent classes of finite classical groups. Let $C$ be a semi-simple $G^F$-conjugacy class corresponding to
\begin{equation}
\eta\bar{\mathfrak{t}}=\eta d_+d_-(n_i,\hat{\alpha}_i)_{i\in\Lambda_1}(n'_j,\hat{\alpha}'_j)_{j\in\Lambda_2}.
\end{equation}

For odd $n$, the $G^F$-classes which have $C$ as semi-simple parts are parametrised by
\begin{equation}
\Lambda_{n_-}^{s}\Lambda^o_{n_+}(\lambda_i)_{i\in\Lambda_1}(\lambda'_j)_{j\in\Lambda_2},
\end{equation}
where $\Lambda_{n_-}^s\in\bs\Psi^{sp}_{n_-}$, $\Lambda^o_{n_+}\in\tilde{\bs\Psi}{}^{ort,+}_{n_+}$, $\lambda_i\in\mathcal{P}_{n_i}$, $\lambda_j'\in\mathcal{P}_{n_j'}$.

For even $n$, the $G^F$-classes which have $C$ as semi-simple parts are parametrised by
\begin{equation}
\Lambda_{n_+}^{s}\Lambda^o_{n_-}(\lambda_i)_{i\in\Lambda_1}(\lambda'_j)_{j\in\Lambda_2},
\end{equation}
where $\Lambda_{n_+}^s\in\bs\Psi^{sp}_{n_+}$, $\Lambda^o_{n_-}\in\tilde{\bs\Psi}{}^{ort,\eta}_{n_-}$, $\lambda_i\in\mathcal{P}_{n_i}$, $\lambda_j'\in\mathcal{P}_{n_j'}$. 

These data can equivalently described as follows. Recall that each element of $\bs\Psi^{sp}$ (resp. $\tilde{\bs\Psi}{}^{ort,\pm}$) is associated with a partition of symplectic type (resp. orthogonal type). For any $\Lambda\in\bs\Psi^{sp}\sqcup\tilde{\bs\Psi}{}^{ort,+}\sqcup\tilde{\bs\Psi}{}^{ort,-}$, denote by $|\Lambda|$ the size of the corresponding partition. For any $\eta\in\{\pm\}$ and any triple $f=(\Lambda^s,\Lambda^o,f^{\circ})$ with $\Lambda^s\in\bs\Psi^{sp}$, $\Lambda^o\in\tilde{\bs\Psi}{}^{ort,\eta}$, and a map of sets $f^{\circ}:\Phi_{\ast}^{\circ}\rightarrow\mathcal{P}$, write $$||f||=|\Lambda^s|+|\Lambda^o|+\sum_{x\in\Phi_{\ast}^{\circ}}2|f(x)|\cdot|x|.$$ Finally, put $\bar{\mathfrak{T}}_C=\{(\eta,f);||f||=n\}$ (with $f$ and $-f$ identified if $|\Lambda^o|=0$). This set parametrises the set of $G^F$-conjugacy classes in $G^F\sigma$. Given any such 4-tuple $f=(\eta,\Lambda^s,\Lambda^o,f^{\circ})$, the semi-simple part of the corresponding conjugacy class, described by $\Phi_{\ast}\rightarrow\mathbb{Z}_{\ge 0}$, is obtained by taking the sizes of $\Lambda^s$, $\Lambda^o$ and the images of $f^{\circ}$. 

%%%%%%%%%%%%%%%%%%%%            La Descente de Shintani
\section{Shintani Descent}\label{SecShintani}
Now $G$ denote a connected reductive group over $k$.
\subsection{Eigenvalues of the Frobenius}
In this part, we collect some results on the eigenvalues of the Frobenius endomorphism acting on the intersection cohomology of the Deligne-Lusztig variety $X_w$. We will write $X_{w,F}$ if it is necessary to specify the Frobenius that is involved. Recall that $X_w$ is the subvariety of the flag variety $\mathcal{B}$ consisting of the Borel subgroups $B$ such that $(B,F(B))$ are conjugate to $(B_0,\dot{w}B_0\dot{w}^{-1})$ by $G$, where $\dot{w}\in G$ is a representative of $w\in W_G$ and $B_0$ is some fixed $F$-stable Borel subgroup.

\subsubsection{}
The Deligne-Lusztig character $R^G_{T_w}1$ is realised by the virtual representation $$\bigoplus_i(-1)^iH^i_c(X_w,\ladic).$$ Recall that the Lusztig series $\mathcal{E}(G^F,(1))$ consists of the irreducible representations that appear as a direct summand of some $R^G_{T_w}1$, or equivalently of a vector space $H^i_c(X_w,\ladic)$. Denote by $\mathbf{H}^i(\bar{X}_w,\ladic)$ the intersection cohomology of $X_w$. By \cite[Corollary 2.8]{L84b}, each element of $\mathcal{E}(G^F,(1))$ is also an irreducible $G^F$-subrepresentation of  $\mathbf{H}^i(\bar{X}_w,\ladic)$ for some $i$ and $w$. We write
\begin{equation}
\begin{split}
M_i(w,F)&:=\mathbf{H}^i(\bar{X}_{w,F},\ladic)\\
H_i(w,F)&:=\End_{G^{F}}(M_i(w,F))
\end{split}
\end{equation}
If $w=1$, $M_i(1,F)$ is just the $\ell$-adic cohomology of $X_1$ and its simple factors are the principal series representations, which are in bijection with the irreducible representations of $H_i(1,F)$. If $F$ is split, $H_i(1,F)\cong \ladic[W]$, and if $F$ is twisted by a graph automorphism $\sigma$, $H_i(1,F)\cong\ladic[W^{\sigma}]$. In what follows, we fix the Frobenius $F$ and $w\in W$, and write $M_i$ and $H_i$ instead of $M_i(w,F)$ and $H_i(w,F)$.

\subsubsection{}
Denote by $F_0$ a split Frobenius over $\mathbb{F}_q$, and denote by $F$ the Frobenius defining the $\mathbb{F}_q$-structure of $G$. Assume that some power of $F_0$ is a power of $F$. Let $b$ be the smallest integer such that $F_0^b$ is a power of $F$. In the case that interests us, $b=1$ or $2$. Note that $F$ and $F_0$ commute. The action of $F_0$ on $\mathcal{B}$ induces an isomorphism of $M_i$ as a vector space. Moreover $F_0$ induces by conjugation an algebra automorphism of $H_i$, still denoted by $F_0$, which is unipotent (\cite[Theorem 2.18]{L84b}). Let $\rho$ be an irreducible representation of $G^F$ that appears in $M_i$. Denote by $M_{i,\rho}$ the isotypic component corresponding to $\rho$. By \cite[Proposition 2.20]{L84b}, the action of $F_0$ respects the isotypic decomposition, i.e. $F_0(M_{i,\rho})=M_{i,\rho}$. The algebra $H_i$ is decomposed into some simple algebras $H_{i,\rho}=\End_{G^F}(M_{i,\rho})$. 

\subsubsection{}\label{Flaissedra}
Now assume that $\rho$ is fixed. Denote by $[\rho]$ the vector space on which $G^F$ acts by the representation $\rho$. There exists a $\ladic$-space $V$ such that $M_{i,\rho}\cong [\rho]\otimes V$ and that $H_{i,\rho}\cong\End_{\ladic}V$. Since $H_{i,\rho}$ is a simple algebra, we have $F_0=\phi_G\otimes\phi_H$ with $\phi_G\in\ladic[G^F]$ and $\phi_H\in H_{i,\rho}$, which are invertible as $F_0$ is. Consider the adjoint representation,
\begin{equation}
\begin{split}
\GL(V)&\longrightarrow\GL(H_{i,\rho})\\
\phi_H&\longmapsto \ad\phi_H.
\end{split}
\end{equation}
Since $\ad\phi_H=\ad F_0$, which is unipotent, we see that $\phi_H$ is a unipotent endomorphism up to a scalar. Modifying $\phi_H$ if necessary, we can assume that $\phi_H$ is unipotent. Choose a basis $\{e_1,\ldots,e_s\}$ of $V$ in such a way that in this basis $\phi_H$ is triangular. Each $M_{i,\rho,r}=[\rho]\otimes e_r$ provides the representation $\rho$. By \cite[Proposition 2.20]{L84b}, the representation $\rho:G^F\rightarrow \GL(M_{i,\rho,r})$ is $F_0$-stable and extends into a representation of $G^F\lb F_0\rb$, denoted by $\tilde{\rho}$,with $F_0^b$ acting trivially on $G^F$. The action of $F_0$, regarded as an element of $G^F\lb F_0\rb$, on $M_{i,\rho,r}$ is defined by $(\lambda'_{\rho})^{-1}q^{-i/2}\phi$, where $\lambda'_{\rho}$ is a root of unity. Another choice of $\tilde{\rho}$ corresponds to a multiple of $\lambda'_{\rho}$ by a $b$-th root of unity. The value of $\lambda'_{\rho}$ only depends on $\rho$ and a choice of $\tilde{\rho}$, and does not depend on $w$ or $i$. In other words, $F_0=\tilde{\rho}\otimes\varphi'$, where $\varphi'$ is a unipotent endomorphism multiplied by $\lambda'_{\rho}q^{i/2}\varphi$.

\subsubsection{}\label{compisononFst}
If we consider a Frobenius $F_0'$ that is not necessarily split, it may happen that the action of $F_0'$ does not respect the isotypic components of $M_i$. However, we can nevertheless consider those components that are preserved by $F_0'$. In fact, a component $M_{i,\rho}$ is $F_0'$-stable if and only if the character $\theta_{\rho}\in\Irr(H_{i})$ associated is $F_0'$-stable. Let $M_{i,\rho}$ be such a component, we still have $M_{i,\rho}\cong[\rho]\otimes V$ and $\rho$ extends into a representation of $G^F\lb F_0'\rb$.

\subsubsection{}\label{vpdeFsurHc}
Now we consider the action of $F_0$ on the $\ell$-adic cohomology. By \cite[Lemma 1.4]{S85}, the eigenvalues of $F_0$ on $H^i_c(X_w,\ladic)$ are $\lambda'_{\rho}$ times a power of $q^{b/2}$ which is not necessarily $q^{ib/2}$. Let $\mu=\lambda'_{\rho}q^{kb/2}$ be such an eigenvalue. Then, the subspace $M_{i,\rho,\mu}\subset M_{i,\rho}$ of eigenvalue $\mu$ is $F_0$-stable and there exists a decomposition $M_{i,\rho,\mu}\cong [\rho]\otimes V_{\mu}$ such that the action of $F_0$ on $M_{i,\rho,\mu}$ is decomposed as $\tilde{\rho}\otimes\varphi_{\mu}$ where $\varphi_{\mu}$ is $\lambda'_{\rho}q^{k/2}$ times a unipotent endomorphism of $V_{\mu}$. Once again, $\lambda'_{\rho}$ only depends on $\rho$ and a choice of $\tilde{\rho}$.

\subsubsection{}\label{vpdeF^2}
There are two particular cases that interest us.
\begin{Thm}(\cite[Theorem 3.34]{L77})
Let $i\in\mathbb{Z}$ and $w\in W_G$ be arbitrary. If $(G,F)$ is of type $A_n$, $n\ge 1$, then all of the eigenvalues of $F$ on $H_c^i(X_w,\ladic)$ are powers of $q$. If $(G,F)$ is of type $\leftidx{^2}{A_n}{}$, $n\ge 2$, then all of the eigenvalues of $F^2$ on $H_c^i(X_w,\ladic)$ are powers of $(-q)$.
\end{Thm}

\subsection{Shintani Descent}
\subsubsection{}
Let $F_1$ and $F_2$ be two commuting Frobenius endomorphism. Denote by $\mathcal{K}(G^{F_1}.F_2)$ the $F_2$-conjugacy classes of $G^{F_1}$ and by $\mathcal{K}(G^{F_2}.F_1)$ the $F_1$-conjugacy classes of $G^{F_2}$, and we denote by $\mathcal{C}(G^{F_2}.F_1)$ and $\mathcal{C}(G^{F_1}.F_2)$ the set of functions that are constant on the $F_1$-conjugacy classes of $G^{F_2}$ and the functions that are constant on the $F_2$-conjugacy classes of $G^{F_1}$ respectively.

Define a map $N_{F_1/F_2}:\mathcal{K}(G^{F_1},F_2)\rightarrow\mathcal{K}(G^{F_2}.F_1)$ as follows. For $g\in G^{F_1}$, there exists $x\in G$ such that $xF_2(x^{-1})=g$. Then $g':=x^{-1}F_1(x)\in G^{F_2}$, and its $F_1$-conjugacy class is well defined. This defines a bijection from $N_{F_1/F_2}:\mathcal{K}(G^{F_1},F_2)\isom\mathcal{K}(G^{F_2}.F_1)$. We write $g'=N_{F_1/F_2}(g)$ by abus of notation. Denote by $\Sh_{F_2/F_1}:\mathcal{C}(G^{F_2}.F_1)\rightarrow\mathcal{C}(G^{F_1}.F_2)$ the induced bijection. It is easy to check that $\Sh_{F_2/F_1}\circ\Sh_{F_1/F_2}=\Id$ and that $\Sh_{F/F}$ is an involution that may not be the identity.

\subsection{Action on the Irreducible Characters}
\subsubsection{}
Let $U$ be a unipotent character of $G^{F_1}$ that extends to $G^{F_1}\lb F_2\rb$. We denote the restriction to $G^{F_1}.F_2$ of its extension by $E_{F_2}(U)\in\mathcal{C}(G^{F_1}.F_2)$. The Shintani descent sends it into $\mathcal{C}(G^{F_2}.F_1)$. On the other hand, the unipotent irreducible characters of $G^{F_2}$ that extends to $G^{F_2}.F_1$, which we denote by $\mathcal{E}(G^{F_2},(1))^{F_1}$, have as extensions some elements of $\mathcal{C}(G^{F_2}.F_1)$. We will see that the functions $\Sh_{F_1/F_2}E_{F_2}(U)$ can be expressed as linear combinations of the extensions of the elements of $\mathcal{E}(G^{F_2},(1))^{F_1}$. 

\subsubsection{}
Let $B$ be an $F_1$-stable and $\sigma$-stable Borel subgroup. Put
\begin{equation}
H:=\End_{G^{F_1}}(\Ind^{G^{F_1}}_{B^{F_1}}1)\cong\End_{G^{F_1}}(\ladic[\mathcal{B}^{F_1}])\cong\ladic[W^{F_1}].
\end{equation}
The irreducible characters of $H$ are in bijection with the principal series representations of $G^{F_1}$. For $\psi\in\Irr(H)$, denote by $U_{\psi}\in\Irr(G^{F_1})$ the corresponding character. By \S \ref{compisononFst}, $U_{\psi}$ extends to $G^{F_1}\lb F_2\rb$ if $\psi$ is $F_2$-stable, in which case $\psi$ itself extends to $H\lb F_2\rb$ in such a way that the action of $F_2$ on $\Ind^{G^{F_1}}_{B^{F_1}}1$ is decomposed into $F_2=E_{F_2}(U_{\psi})(F_2)\otimes E_{F_2}(\psi)(F_2)$, where we denote by $E_{F_2}(U_{\psi})$ and $E_{F_2}(\psi)$ the extensions of $U_{\psi}$ and $\psi$ respectively.

\subsubsection{}
Let $\rho\in\Irr(G^{F_2})$ be unipotent. The $\rho$-isotypic component of $H^{i}_{c}(X_{w,F_2})$ is of the form $[\rho]\otimes V$. The action of the split Frobenius $F_0$ on this component can be written as $\tilde{\rho}(F_0)\otimes \varphi$, where $\varphi$ is $\lambda'_{\rho}$ times a power of $q^{1/2}$ and a unipotent endomorphism, according to \S \ref{vpdeFsurHc}. We denote by $\Omega_{F_2}$ the isomorphism of the space $\mathcal{C}(G^{F_2},(1))$ that multiplies $\rho$ by $\lambda'_{\rho}$. Denote by $E_{F_0}(\rho)$ the restriction of $\tilde{\rho}$ to $G^{ F_2}.F_0$.

\subsubsection{}
Fix the split Frobenius $F_0$ and the order 2 quasi-central automorphism $\sigma$. In what follows, we only consider $(F_1,F_2)=(\sigma_1F_0^m,\sigma_2F_0)$,where $m\in\mathbb{Z}_{>0}$ and $\sigma_i=1$ or $\sigma$. Take $\rho\in\mathcal{E}(G^{ F_2},(1))^{\sigma_1}$, i.e. a $\sigma_1$-stable representation, and denote by $E_{\sigma_1}(\rho)$ an extension of $\rho$ to $G^{ F_2}\lb\sigma_1\rb$. Since $F_0$ acts as $\sigma_2^{-1}$ on $G^{F_2}$, we can define the extension $E_{F_0}(\rho)(F_0)$ to be an extension $E_{\sigma_2}(\rho)(\sigma_2^{-1})$, which commutes with $E_{\sigma_1}(\rho)(\sigma_1)$ because either one of $\sigma_1$ and $\sigma_2$ is 1 or they are equal. This allows us to define an extension $E_{F_1}(\rho)$ of $\rho$ to $G^{F_2}.F_1$ by requiring $E_{F_1}(\rho)(\sigma_1F_0^m)=E_{\sigma_1}(\rho)(\sigma_1)E_{F_0}(\rho)(F_0^m)$. It is well defined. In addition, $E_{\sigma_1}$ defines an isomorphism of vector spaces
\begin{equation}
\begin{split}
\ladic[\mathcal{E}(G^{ F_2},(1))^{\sigma_1}]&\longrightarrow\mathcal{C}(G^{F_2}\sigma_1,(1))\\
\rho&\longmapsto E_{\sigma_1}(\rho).
\end{split}
\end{equation}
\subsubsection{}
The following theorem makes explicit the transition matrix.
\begin{Thm}(\cite[Th\'eor\`eme 5.6]{DM94})\label{5.6}
We keep the above notations. For any $m\in\mathbb{Z}_{>0}$, and any $\psi\in\Irr(W^{\sigma_1})^{\sigma_2}$ we have
\begin{equation}
\begin{split}
\Sh_{\sigma_1F_0^m/\sigma_2F_0}E_{\sigma_2F_0}(U_{\psi})&=\sum_{\rho\in\mathcal{E}(G^{\sigma_2F_0},1)^{\sigma_1}}\langle R^{G^{\sigma_2F}\sigma_1}_{\psi},E_{\sigma_1}(\rho)\rangle_{G^{\sigma_2F_0}\sigma_1}\lambda^{'m}_{\rho}E_{\sigma_1F_0^{m}}(\rho)\\
&=E_{\sigma_1\sigma_2^{-m}}(\Omega^m_{\sigma_2F_0}E_{\sigma_1^{-1}}(R^{G^{\sigma_2F_0}\sigma_1}_{\psi})).
\end{split}
\end{equation}
(See \cite[D\'efinition 5.1]{DM94} for the definition of $R^{G^{\sigma_2F_0}\sigma_1}_{\psi}$ or more generally in \ref{Rphi}.)
\end{Thm}
\subsection{Commutation with the Deligne-Lusztig Induction}
\subsubsection{} 
The following proposition due to Digne shows how the Deligne-Lusztig induction commutes with Shintani Descent.
\begin{Prop}(\cite[Proposition 1.1]{Di})\label{Di1.1}
Let $G$ be a connected reductive group defined over $\mathbb{F}_q$, equipped with the Frobenius endomorphism $F$ and let $\sigma$ be a quasi-central automorphism of $G$. Let $L\subset P$ be an $F$-stable and $\sigma$-stable Levi factor of a $\sigma$-stable parabolic subgroup. Then
\begin{equation}
\Sh_{\sigma F/F}\circ R^{G^{\sigma F}\sigma^{-1}}_{L^{\sigma F}\sigma^{-1}}=R^{G^F\sigma}_{L^F\sigma}\circ\Sh_{\sigma F/F}.
\end{equation}
\end{Prop}

%%%%%%%%%%%%%%%%%% Faisceaux de Caract\`eres
\section{Character Sheaves}\label{SecCharSheaf}
In this section, $G$ denotes a not necessarily connected reductive group. By local system, we mean a local system of $\ladic$-vector spaces. If $X$ is a variety over $k$, we denote by $\mathcal{D}(X)$ the bounded derived category of constructible $\ladic$-sheaves on $X$. For any $g\in G$, denote by $g_sg_u$ the Jordan decomposition of $g$, with $g_s$ being semi-simple and $g_u$ unipotent.

\subsection{Character Sheaves on Groups Not Necessarily Connected}
\subsubsection{}
If $G^1$ is a connected component of $G$, define $$Z_{G^{\circ},G^1}^{\circ}:=C_{Z_{G^{\circ}}}(g)^{\circ},\quad\text{for any $g\in G^1$}.$$ It does not depend on the choice of $g\in G^1$. An \emph{isolated stratum} of $G^1$ is an orbit of isolated elements under the action of $Z_{G^{\circ},G^1}^{\circ}\times G^{\circ}$ given by $$(z,x):g\longmapsto zxgx^{-1}.$$(See \cite[\S 1.21 (d), \S 3.3 (a)]{L03I})
\begin{Eg}\label{isostm}
For the group $\bar{G}$ defined in \S \ref{barG}, we have $Z_{G,G\sigma}^{\circ}=\{1\}$, and so an isolated stratum of $G\sigma$ is an isolated $G$-conjugacy class.
\end{Eg}

Given an isolated stratum $S$, denote by $\mathcal{S}(S)$ the category of local systems on $S$ invariant under the action of $Z_{G^{\circ},G^1}^{\circ}\times G^{\circ}$ given by
$$(z,x):g\longmapsto z^nxgx^{-1},$$ for some integer $n>0$. We refer to \cite[\S 6]{L03I} for the definition of \emph{cuspidal local system (for $G$)}. If $\mathcal{E}\in\mathcal{S}(S)$ is cuspidal, we say that $(S,\mathcal{E})$ is a \emph{cuspidal pair (for $G$)}.

\subsubsection{}
Let $S$ be an isolated stratum of $G$ contained in some connected component $G^1$, let $\mathcal{E}\in\mathcal{S}(S)$, and let $S_s$ be the set of semi-simple parts of the elements of $S$. We assume that $Z_{G^{\circ},G^1}^{\circ}=\{1\}$ until the lemma below.

Let $s\in S_s$ be an arbitrary element and put $$U_s:=\{u\in C_G(s)\mid u \text{ is unipotent such that }su\in S\}.$$Then any $C_G(s)^{\circ}$-conjugacy class in $U_s$ is a stratum of $C_G(s)$ (\textit{cf.} \cite[\S 6.5]{L03I}). Note that the map $U_s\rightarrow S$, $u\mapsto su$, defines an embedding.
\begin{Lem}(\cite[Lemma 6.6]{L03I})\label{L036.6}
The pair $(S,\mathcal{E})$ is cuspidal for $G$ if and only if for any $C_G(s)^{\circ}$-conjugacy class $S'\subset U_s$, the pair $(S',\mathcal{E}|_{S'})$ is cuspidal for $C_G(s)$.
\end{Lem}

\subsubsection{}\label{YLS}
Let $L$ be a Levi subgroup of $G^{\circ}$ and let $S$ be an isolated stratum of $N_G(L,P)$ for a parabolic subgroup $P$ with Levi factor $L$. (See \cite[\S 2.2 (a), \S 3.5]{L03I}) Define, $$S_{reg}=\{g\in S\mid C_G(g_s)^{\circ}\subset L\}.$$ Define $$Y_{L.S}=\bigcup_{x\in G^{\circ}}xS_{reg}x^{-1},$$ and $$\tilde{Y}_{L,S}=\{(g,xL)\in G\times G^{\circ}/L\mid x^{-1}gx\in S_{reg}\},$$ equipped with the action by $G^{\circ}$, $h:(g,xL)\mapsto(hgh^{-1},hxL)$, and $$\hat{Y}_{L.S}=\{(g,x)\in G\times G^{\circ}\mid x^{-1}gx\in S_{reg}\},$$equipped with the action by $G^{\circ}\times L$, $(h,l):(g,x)\mapsto(hgh^{-1},hxl^{-1})$. Consider the morphisms $$S\stackrel{\alpha}{\longleftarrow}\hat{Y}_{L,S}\stackrel{\beta}{\longrightarrow}\tilde{Y}_{L,S}\stackrel{\pi}{\longrightarrow }Y_{L,S},$$ where $\alpha(g,x)=x^{-1}gx$, $\beta(g,x)=(g,xL)$ and $\pi(g,xL)=g$. Put $$\tilde{\mathcal{W}}_S:=\{n\in N_{G^{\circ}}(L)\mid nSn^{-1}=S\},$$and $\mathcal{W}_S=\tilde{\mathcal{W}}_S/L$. It is a finite group. Then $\beta$ is a principal  $L$-bundle and $\pi$ is a principal $\mathcal{W}_S$-bundle. (See \cite[\S 3.13]{L03I}) If $\mathcal{E}\in\mathcal{S}(S)$ is irreducible and cuspidal for $N_G(L,P)$, put $$\tilde{\mathcal{W}}_{\mathcal{E}}:=\{n\in\tilde{\mathcal{W}}_S\mid\ad(n)^{\ast}\mathcal{E}\cong\mathcal{E}\},$$ and $\mathcal{W}_{\mathcal{E}}=\tilde{\mathcal{W}}_{\mathcal{E}}/L$. 
\begin{Eg}\label{Eg-W_E}
Let $G=\GL_n(k)$ and $\bar{G}=G\rtimes\lb\sigma\rb$ as in \S \ref{barG}. Let $T$ be the maximal torus consisting of diagonal matrices and $B$ the Borel subgroup consisting of upper triangular matrices. Let $P$ be a $\sigma$-stable standard parabolic subgroup with respect to $B$, such that the unique Levi factor $L$ containing $T$ is isomorphic to $\GL_m(k)\times(k^{\ast})^{2N}$ for some non negative integers $m$ and $N$. Put $\bar{L}=N_{\bar{G}}(L,P)=L\sqcup L\sigma$ and let $S\subset L\sigma$ be an isolated stratum. Then $N_G(L)/L\cong \mathfrak{S}_{2N}$ and $\mathcal{W}_S\cong\mathfrak{W}_N$. We consider $\mathcal{W}_{\mathcal{E}}$ for some particular $\mathcal{E}$ in the following.

Let $S'=S/Z_{L,L\sigma}^{\circ}$ and let $\pi':S\rightarrow S'$ be the projection. Let $M=L/[L,L]$ and $\bar{M}=\bar{L}/[L.L]$. Then $\ad\sigma$ induces an action on $M$. Define the map $\mathcal{L}_{\sigma}:M\rightarrow M$, $m\mapsto m^{-1}\sigma(m)$. Then put $\bar{L}''=\bar{M}/\Ima(\mathcal{L}_{\sigma})$ and let $\pi'':\bar{L}\rightarrow \bar{L}''$ be the projection. Let $S''$ be the image of $S$ under $\pi''$. Suppose that $\mathcal{E}'\in\mathcal{S}(S')$ is irreducible and $\mathscr{L}\in\mathcal{S}(S'')$ satisfies $\mathscr{L}^{\otimes 2}\cong\ladic$. Consider the local system $\mathcal{E}:=\pi^{\prime\ast}\mathcal{E}'\otimes\pi^{\prime\prime\ast}\mathscr{L}$ on $S$, which lies in $\mathcal{S}(S)$ by \cite[\S 5.3]{L03I}.  Note that $S''$ is a connected component of $\bar{L}''$, and since the identity component $L''$ is central, the equivariant local systems on $S''$ can be identified with those on $L''$, which is a torus isomorphic to $(k^{\ast})^N$. The action of $\mathcal{W}_S$ on $L$ induces an action on $L''$. Let $T_+\subset L''$ be the largest $\mathcal{W}_{\mathcal{E}}$-stable subtorus such that $\mathscr{L}|_{T_+}\cong\ladic$. Let $T_-\subset L''$ be the unique $\mathcal{W}_{\mathcal{E}}$-stable subtorus such that $L''\cong T_+\times T_-$. Let $N_+=\dim T_+$ and $N_-=\dim T_-$, then $N_++N_-=N$. Since the factor $\pi^{\prime\ast}\mathcal{E}'$ is automatically $\mathcal{W}_S$-invariant, we have $\mathcal{W}_{\mathcal{E}}\cong\mathfrak{W}_{N_+}\times\mathfrak{W}_{N_-}$.
\end{Eg}

\subsubsection{}\label{CharSheaf}
Fix $\mathcal{E}\in\mathcal{S}(S)$. There exists a $G^{\circ}$-equivariant local system $\tilde{\mathcal{E}}$ on $\tilde{Y}_{L,S}$ such that $\beta^{\ast}\tilde{\mathcal{E}}\cong\alpha^{\ast}\mathcal{E}$. Denote by $\mathbf{E}=\End(\pi_!\tilde{\mathcal{E}})$, the endomorphism algebra of $\pi_!\tilde{\mathcal{E}}$. We have a canonical decomposition (\cite[\S 7.10 (a); IV, \S 21.6]{L03II}) $$\mathbf{E}=\bigoplus_{w\in\mathcal{W}_{\mathcal{E}}}\mathbf{E}_w,$$where the factors $\mathbf{E}_w:=\Hom(\ad(n_w)^{\ast}\mathcal{E},\mathcal{E})$, each one defined by some representative  $n_w$ of $w$, are of dimension 1 and satisfy $\mathbf{E}_w\mathbf{E}_v=\mathbf{E}_{wv}$. Choose base $\{b_w\mid w\in\mathcal{W}_{\mathcal{E}}\}$ of $\mathbf{E}$ with $b_w\in\mathbf{E}_w$ for any $w$.

Define $$K=\IC(\bar{Y}_{L,S},\pi_{!}\tilde{\mathcal{E}}),$$
where $\bar{Y}_{L,S}$ is the closure of $Y_{L,S}$ in $G$. There exists a canonical isomorphism $\mathbf{E}\cong\End(K)$. Let $\Lambda'$ be a finite set parametrising the isomorphism classes of the irreducible representations of $\mathbf{E}$ and for each  $i\in\Lambda'$, we denote by $V_i$ a corresponding representation. Then, we have the canonical decompositions
$$\pi_!\tilde{\mathcal{E}}\cong\bigoplus_{i\in\Lambda'}V_i\otimes(\pi_!\tilde{\mathcal{E}})_i,\quad K\cong\bigoplus_{i\in\Lambda'}V_i\otimes K_i,$$
where $$ (\pi_!\tilde{\mathcal{E}})_i=\Hom_{\mathbf{E}}(V_i,\pi_!\tilde{\mathcal{E}}),\quad K_i=\Hom_{\mathbf{E}}(V_i,K)$$ are the simple factors. Moreover, $K_i\cong \IC(\bar{Y}_{L,S},(\pi_!\tilde{\mathcal{E}})_i)$.

\subsubsection{}\label{FoncCara}
Assume that $F(L)=L$, $F(S)=S$ and $F^{\ast}\mathcal{E}\isom\mathcal{E}$, where $F$ is the Frobenius of $G$. We fix an isomorphism $\phi_0:F^{\ast}\mathcal{E}\isom\mathcal{E}$. It induces an isomorphism $\tilde{\phi}:F^{\ast}\pi_!\tilde{\mathcal{E}}\isom\pi_!\tilde{\mathcal{E}}$ and an isomorphisme $\phi:F^{\ast}K\isom K$. Recall that, given a variety $X/k$ equipped with the  Frobenius $F$, a complex $A\in\mathscr{D}(X)$ and an isomorphism $\phi:F^{\ast}A\isom A$, the \textit{characteristic function} $\chi_{A,\phi}:X^F\rightarrow\ladic$ is defined by
\begin{equation}
\chi_{A,\phi}(x)=\sum_{i\in\mathbb{Z}}(-1)^i\tr(\phi,\mathcal{H}^i_xA),
\end{equation}
where $\mathcal{H}^i_xA$ is the stalk at $x$ of the cohomology sheaf in degree $i$ of $A$.

\begin{Thm}(\cite[Theorem 16.14, \S 16.5, \S 16.13]{L03III})\label{16.14}
Let $s$ and $u\in G^F$ be a semi-simple element and a unipotent element such that $su=us\in\bar{Y}_{L,S}$. Then,
\begin{equation}
\chi_{K,\phi}(su)=\sum_{\substack{h\in G^{\circ F};\\h^{-1}sh\in S_s}}\frac{|L^F_h|}{|C_G(s)^{\circ F}||L^F|}Q_{L_h,C_G(s)^{\circ},\mathfrak{c}_h,\mathcal{F}_h,\phi_h}(u),
\end{equation}
where $S_s$ is the set of the semi-simple parts of the elements of  $S$, and $Q_{L_h,C_G(s)^{\circ},\mathfrak{c}_h,\mathcal{F}_h,\phi_h}$ is the \textit{generalised Green function} (See \S \ref{FoncGrGe} below) associated to the data $L_h$, $C_G(s)^{\circ}$, $\mathfrak{c}_h$, $\mathcal{F}_h$, $\phi_h$ defined by
\begin{itemize}
\item[-] $L_h:=hLh^{-1}\cap C_G(s)^{\circ}$;
\item[-] $\mathfrak{c}_h:=\{v\in C_G(s)\mid v\text{ unipotent, }h^{-1}svh\in S\}$;
\item[-] $\mathcal{F}_h$: inverse image of $\mathcal{E}$ under the embedding $\mathfrak{c}_h\rightarrow S$, $v\mapsto h^{-1}svh$;
\item[-] $\phi_h:F^{\ast}\mathcal{F}_h\isom\mathcal{F}_h$, an isomorphism induced from $\phi_0$ under the above embedding.
\end{itemize}
\end{Thm}

Denote by $L^1$ (resp. $G^1$) the connected component of $N_G(L)$ (resp. $G$) containing $S$. If we define
\begin{itemize}
\item[-]
$\Sigma_h:=hZ^{\circ}_{L,L^1}h^{-1}\mathfrak{c}_h$;
\item[-]
$\mathcal{E}_h$: inverse image of $\mathcal{E}$ by the embedding $\Sigma_h\rightarrow S$, $v\mapsto h^{-1}svh$,
\item[-]
$\phi'_h:F^{\ast}\mathcal{E}_h\isom\mathcal{E}_h$ an isomorphism induced from $\phi_0$ under the above embedding,
\end{itemize}
then the above embedding $\mathfrak{c}_h\rightarrow S$ factors through the inclusion $\mathfrak{c}_h\rightarrow\Sigma_h$, $\mathcal{F}_h$ is the inverse image of $\mathcal{E}_h$ and $\phi_h$ is induced from $\phi'_h$ under this inclusion. The point is that, $\Sigma_h$ is a finite union of isolated strata, which has $\mathfrak{c}_h$ as the subset of unipotent elements, so that these data fit into the following definition of generalised Green functions.

\subsubsection{Generalised Green Functions}\label{FoncGrGe}
Given 
\begin{itemize}
\item[-] $G$ a reductive algebraic group,
\item[-] $L\subset G^{\circ}$ an $F$-stable Levi subgroup,
\item[-] $\Sigma^u$ the set of the unipotent elements of a finite union of isolated strata $\Sigma$ of $N_G(L)$ satisfying $F(\Sigma)=\Sigma$, $F(\Sigma^u)=\Sigma^u$,
\item[-] $\mathcal{F}$ an $L$-equivariant local system on $\Sigma^u$, and
\item[-] $\phi_1:F^{\ast}\mathcal{F}\isom\mathcal{F}$ an isomorphism,
\end{itemize}
we choose a local system $\mathcal{E}$ on $\Sigma$ that restricts to $\mathcal{F}$ under the inclusion $\Sigma^u\rightarrow\Sigma$ and an isomorphism (which always exists) $\phi'_1:F^{\ast}\mathcal{E}\isom\mathcal{E}$ that induces $\phi_1$. Define $K=\IC(\bar{Y}_{L,\Sigma},\pi_!\tilde{\mathcal{E}})$, where $Y_{L,\Sigma}$, $\tilde{Y}_{L,\Sigma}$ and $\pi:\tilde{Y}_{L,\Sigma}\rightarrow Y_{L,\Sigma}$ are defined by the procedure \S \ref{YLS}, and denote by $\phi:F^{\ast}K\isom K$ the isomorphism induced from $\phi_1'$.

The generalised Green function associated to $G$, $L$, $\Sigma^u$, $\mathcal{F}$ and $\phi_1$, denoted by $Q_{L,G,\Sigma^u,\mathcal{F},\phi_1}$, is defined by (\cite[\S 15.12]{L03III})
\begin{equation}
\begin{split}
G^F_u=\{\text{unipotent elements of $G^F$}\}&\longrightarrow\ladic\\
u&\longmapsto\chi_{K,\phi}(u).
\end{split}
\end{equation}
It does not depend on the choice of $\mathcal{E}$ and $\phi_1'$.

\subsubsection{}
Let $G$, $L$, $\Sigma^u$, $\Sigma$, $\mathcal{F}$, $\mathcal{E}$, $\phi_1$ and $\phi_1'$ be as in \S \ref{FoncGrGe}, and let $Q_L^G(-,-)$ be the two variable Green function as in \S \ref{2-var-Green}.
\begin{Thm}(\cite[Theorem 1.14]{L90})
There exists a constant $q_0>1$ which only depends on the Dynkin diagram of $G$, such that if $q>q_0$, then for any $u\in G_u^F$, $$Q_{L,G,\Sigma^u,\mathcal{F},\phi_1}(u)=(-1)^{\dim\Sigma}|L^F|^{-1}\sum_{v\in L^F_u}Q^G_L(u,v)\chi_{\mathcal{E},\phi_1'}(v).$$
\end{Thm}

In the context of Theorem \ref{16.14}, the isomorphism $\phi_h$ is induced from $\phi_0$ via the embedding $\mathfrak{c}_h\rightarrow S$, $v\mapsto h^{-1}svh$, and so $\chi_{\mathcal{E}_h,\phi_h'}(v)=\chi_{\mathcal{E},\phi_0}(h^{-1}svh)$. The above theorem then says $$Q_{L_h,C_G(s)^{\circ},\mathfrak{c}_h,\mathcal{F}_h,\phi_h}(u)=(-1)^{\dim\Sigma_h}|L_h^F|^{-1}\sum_{v\in (L_h)^F_u}Q^{C_G(s)^{\circ}}_{L_h}(u,v)\chi_{\mathcal{E},\phi_0}(h^{-1}svh).$$ Comparing Proposition \ref{2.10} and Theorem \ref{16.14}, this shows that $\chi_{K,\phi}$ is equal to $R^{G^1}_{L^1}\chi_{\mathcal{E},\phi_0}$ up to a sign.

\subsubsection{}
The isomorphism $\phi$ of \S \ref{FoncCara} induces an algebra isomorphism $\iota:\mathbf{E}\isom\mathbf{E}$. There exists a subset $\Lambda\subset\Lambda'$ and some isomorphisms $\iota_i:V_i\isom V_i$, $\phi_i:F^{\ast}K_i\isom K_i$, for $i\in\Lambda$, such that the isomorphism $b_w\phi:F^{\ast}K\isom K$, with respect to the decomposition $K=\bigoplus_{i\in\Lambda'}V_i\otimes K_i$ is of the form
\begin{itemize}
\item[-]
$b_w\iota_i\otimes\phi_i$, if $i\in\Lambda$;
\item[-]
$V_i\otimes F^{\ast}K_i\rightarrow V_j\otimes K_j$, with $j\ne i$, if $i\notin\Lambda$.
\end{itemize}
(See \cite[\S 20.3, \S 21.6]{L03IV}) Consequently,
\begin{equation}\label{KbwKiphii}
\chi_{K,b_w\phi}=\sum_{i\in\Lambda}\tr(b_w\iota_i,V_i)\chi_{K_i,\phi_i}.
\end{equation}

\subsubsection{}
Take a representative $n_w\in\tilde{\mathcal{W}}_{\mathcal{E}}$ of $w$, and an element $g_w\in G^{\circ}$ such that $g_w^{-1}F(g_{w})=n_w$. Define $L_w=g_wLg_w^{-1}$, $S_w=g_wSg_w^{-1}$ and $\mathcal{E}_w=\ad(g_w^{-1})^{\ast}\mathcal{E}$. Then $L_w$ and $S_w$ are $F$-stable and the isomorphism $\phi_0:F^{\ast}\mathcal{E}\isom\mathcal{E}$ induces an isomorphism $\phi_{0,w}:F^{\ast}\mathcal{E}_w\isom\mathcal{E}_w$. These allow us to define $Y_{L^w,S^w}$, $\tilde{Y}_{L^w,S^w}$, $\pi_w:\tilde{Y}_{L^w,S^w}\rightarrow Y_{L^w,S^w}$, $\tilde{\mathcal{E}}_w$, $K_w$ and $\phi_w:F^{\ast}K_w\isom K_w$ by the same procedure. It can be checked that (See \cite[\S 21.6]{L03IV})
\begin{equation}\label{Kwphiw}
\chi_{K,b_w\phi}=\chi_{K_w,\phi_w}.
\end{equation}

\subsubsection{}
Write $\mathcal{W}=\mathcal{W}_{\mathcal{E}}$, denote by $\bar{\mathcal{W}}$ a set of representatives of the \emph{effective} $F$-conjugacy classes (\cite[\S 20.4]{L03IV}), and write $\mathcal{W}_w=\{v\in\mathcal{W}\mid F^{-1}(v)wv^{-1}=w\}$. If $w$ is not in some effective $F$-conjugacy class and $i\in\Lambda$, then $\tr(b_w\iota_i,V_i)=0$ (\cite[\S 20.4 (a)]{L03IV}). We have for all $i$, $j\in\Lambda$ (\cite[\S 20.4 (c)]{L03IV})
\begin{equation}
\sum_{w\in \bar{\mathcal{W}}}|\mathcal{W}_w|^{-1}\tr(b_w\iota_i,V_i)\tr(\iota_j^{-1}b_w^{-1},V_j)=\delta_{ij},
\end{equation}
where $\delta_{ij}=1$ if $i=j$ and $\delta_{ij}=0$ otherwise. In fact, $|\bar{\mathcal{W}}|=|\Lambda|$ and $(\tr(b_w\iota_i,V_i))_{i\in\Lambda,w\in\bar{\mathcal{W}}}$ is an invertible square matrix (\cite[\S 20.4 (e), (f), (g)]{L03IV}). 

This, combined with the equalities (\ref{KbwKiphii}) and (\ref{Kwphiw}), gives
\begin{equation}\label{dec-char-sheaf}
\chi_{K_i,\phi_i}=\frac{1}{|\mathcal{W}_{\mathcal{E}}|}\sum_{\substack{w\in\mathcal{W}_{\mathcal{E}}\\w\text{ effective}}}\tr(b_w\iota_i,V^{\vee}_i)\chi_{K_w,\phi_w},
\end{equation}
where $V_i^{\vee}$ is the contragradient representation.
\begin{Eg}\label{quad-uni-sheaf}
Let $\bar{G}$, $L$, $S$, and $\mathcal{E}$ be as in Example \ref{Eg-W_E}. Let $K_i$ be the character sheaf corresponding to the irreducible representation $V_i^{\vee}$ and let $\phi_i$ be chosen in such a way that (\ref{dec-char-sheaf}) holds. Since the action of $F$ on $\mathcal{W}_E$ is trivial, we may require that $\iota_i=\Id$ and $b_w$ is the element $w$ in the group algebra of $\mathcal{W}_E$. Then $$\chi_{K_i,\phi_i}=\frac{1}{|\mathfrak{W}_{N_+}\times\mathfrak{W}_{N_-}|}\sum_{w=(w_+,w_-)\in\mathfrak{W}_{N_+}\times\mathfrak{W}_{N_-}}\tr(w,V^{\vee}_i)R^{G\sigma}_{L_w\sigma}\chi_{\mathcal{E}_w,\phi_{0,w}}.$$
\end{Eg}

%%%%%%%%%%%%%%%%%%% Extensions des Caract\`eres sigma-Stables I
\section{Extensions of $\sigma$-Stable Characters}
\subsection{Some Elementary Lemmas}
When dealing with a Levi subgroup $L$ of $\GL_n$, one often regards it as a direct product of smaller $\GL_{n'}$'s and reduces the problem to these direct factors. However, if $\sigma$ is an automorphism of $L$, then $L\rtimes\lb\sigma\rb$ is not actually the direct product of groups of the form $\GL_{n'}\rtimes\lb\sigma'\rb$. We give some lemmas that allow us to apply arguments in the same spirit. Let $H$ denote a finite group in this part, which could either be a finite group of Lie type or a Weyl group. Let $\sigma$ be an automorphism of $H$, which could be induced from the automorphism of an algebraic group or the Frobenius. Denote by $H\lb\sigma\rb$ the semi-direct product of $H$ and the cyclic group generated by $\sigma$, with the generator acting as $\sigma$ on $H$.
\begin{Lem}\label{LemdeRed}
Let $H=H_1\times\cdots\times H_s$ be a product of finite groups and let $\sigma=\sigma_1\times\cdots\times\sigma_s$ be the product of some automorphisms of the direct factors. Then $H\lb\sigma\rb$ is a subgroup of $\prod_iH_i\lb\sigma_i\rb$. Moreover, if the characters $\chi_i\in\Irr(H_i)$ extend to $\tilde{\chi}_i\in\Irr(H_i\lb\sigma_i\rb)$, then $\bar{\chi}:=(\boxtimes_i\tilde{\chi}_i)|_{H\lb\sigma\rb}$ is irreducible and restricts to $\boxtimes_i\chi_i\in\Irr(H)$.
\end{Lem}
\begin{proof}
We define a map by
\begin{equation}
\begin{split}
H\lb\sigma\rb&\longrightarrow \prod_iH_i\lb\sigma_i\rb\\
(h_1,\ldots,h_s)\sigma^i&\longmapsto (h_1\sigma_1^i,\ldots,h_s\sigma_s^i),
\end{split}
\end{equation}
which is obviously a homomorphism and is injective. Then the assertion on $\bar{\chi}$ is immediate.
\end{proof}

We define an exterior tensor product that is "twisted" by $\sigma$. 
\begin{Defn}\label{tildebox}
Let $H=H_1\times H_2$ be a product of finite groups and $\sigma=\sigma_1\times\sigma_2$ a product of automorphisms. For $i=1$ and $2$, let $f_i$ be a function on $H_i\sigma_i$ that is invariant under the conjugation by $H_i$. The function  $f_1\widetilde{\boxtimes} f_2$ on $H\sigma$ is defined as the restriction of $$f_1\boxtimes f_2~(=\pr_1^{\ast}f_1\cdot\pr_2^{\ast}f_2)\in \mathcal{C}(H_1\lb\sigma_1\rb\times H_2\lb\sigma_2\rb)$$ to $H\lb\sigma\rb$.
\end{Defn}

\begin{Lem}\label{extdecirc}
Let $H=K\times\cdots\times K$ be the direct product of $d$ copies of a finite group $K$. Let $\psi$ be an automorphism of $K$, let $\zeta=(i_1,\ldots,i_d)\in\mathfrak{S}_d$ be a circular permutation, and let $(n_1,\ldots,n_d)$ be a $d$-tuple of integers. With these data, we can define an automorphism $\Psi$ of $H$ by
\begin{equation}
\begin{split}
\Psi:H&\longrightarrow H\\
(k_1,\ldots,k_d)&\longmapsto(\psi^{n_1}(k_{\zeta(1)}),\ldots,\psi^{n_d}(k_{\zeta(d)})).
\end{split}
\end{equation}
Denote by $\mathscr{H}$ the direct product $K\rtimes\lb\psi\rb\times\cdots\times K\rtimes\lb\psi\rb$, and let $\zeta$ act by permuting the components: $$\zeta:(k_1,\ldots,k_d)\longmapsto(k_{\zeta(1)},\ldots,k_{\zeta(d)}),\quad k_i\in K\rtimes\lb\psi\rb,$$ Let $\chi$ be a $\psi$-stable irreducible character of $K$ and denote by $\tilde{\chi}$ an extension of $\chi$ to $K\rtimes\lb\psi\rb$. Then,
\begin{itemize}
\item[(i)] $H\rtimes\lb\Psi\rb$ is a subgroup of $\mathscr{H}$;
\item[(ii)] The character $\tilde{\chi}\otimes\cdots\otimes\tilde{\chi}$ of $\mathscr{H}$ extends to a character of $\mathscr{H}\rtimes\lb\zeta\rb$. Its restriction $\bar{\chi}$ to $H\rtimes\lb\Psi\rb$ is irreducible;
\item[(iii)] For all $h=(k_1,\ldots,k_d)\in H$, we have $$\bar{\chi}(h\Psi)=\tilde{\chi}(k_{i_1}\psi^{n_{i_1}} k_{i_2}^{n_{i_2}}\psi\cdots k_{i_d}\psi^{n_{i_d}}).$$
\end{itemize}
\end{Lem}
\begin{proof}
For each $i\in\{1,\ldots,d\}$ and $r\in\mathbb{Z}_{>0}$, put $n_i(r)=\sum_{0\le p\le r-1}n_{\zeta^p(i)}$. Define a map
\begin{equation}
\begin{split}
H\lb \Psi\rb&\longrightarrow (K\rtimes\lb\psi\rb\times\cdots\times K\rtimes\lb\psi\rb)\rtimes\lb\zeta\rb\\
(k_1,\ldots,k_d)&\longmapsto(k_1,\ldots,k_d)\\
(k_1,\ldots,k_d)\Psi^r&\longmapsto(k_1\psi^{n_1(r)},\ldots,k_d\psi^{n_d(r)})\zeta^r.
\end{split}
\end{equation}
One can verify that it is an injective group homomorphism. The character $\bar{\chi}$ is irreducible as its restriction to $H$ is so.

Let us compute the value of $\bar{\chi}$. Let $\rho:K\rightarrow\GL(V)$ be a representation that realises the character $\chi$. Let $\tilde{\rho}$ denote its extension to $K\lb\psi\rb$. Then $V^{\otimes d}$ is a representation of $\mathscr{H}$, defining the action of $\zeta$ by $v_1\otimes\cdots\otimes v_d\mapsto v_{\zeta(1)}\otimes\cdots\otimes v_{\zeta(d)}$. We use an argument of linear algebra. Take $A^{(1)},\ldots,A^{(d)}\in\GL(V)$ and let $\zeta$ act on $V^{\otimes d}$ as above. Then we have
\begin{equation}
\tr(A^{(1)}\otimes\cdots\otimes A^{(d)}\circ\zeta|V\otimes\cdots\otimes V)=\tr(A^{(i_1)}\cdots A^{(i_d)}|V).
\end{equation}
We conclude the proof by taking $\tilde{\rho}(k_i\psi^{n_i})$ for $A^{(i)}$.
\end{proof}
%%%%%%%%%%%%%%%%%
\begin{comment}
\begin{Lem}\label{idenacirc}
On garde les notations du lemme pr\'ec\'edent. Supposons $\psi$ d'ordre $d$. Alors $H\lb F\rb$ est isomorphe \`a $H\lb\zeta\rb$, o\`u $\zeta$ envoie $(k_1,\ldots,k_d)$ sur $(k_{\zeta(1)},\ldots,k_{\zeta(d)})$.
\end{Lem}
\begin{proof}
Pour simplifier, on suppose que $\zeta=(d,d-1,\ldots,1)$. On d\'efinit l'isomorphisme par
\begin{equation}
\begin{split}
(k_1,\ldots,k_d)&\longmapsto (k_1,\psi^{-1}(k_2),\ldots,\psi^{-d+1}(k_d))\\
(k_1,\ldots,k_d)F&\longmapsto (k_1,\psi^{-1}(k_2),\ldots,\psi^{-d+1}(k_d))\zeta
\end{split}
\end{equation}
On v\'erifie ais\'ement que c'est un homomorphisme de groupe et bijectif en notant que $F$ a le m\^eme ordre que $\zeta$.
\end{proof}
\end{comment}
%%%%%%%%%%%%%%%%%

\subsection{Uniform Extensions}
\subsubsection{}
Denote by $L$ the algebraic group defined over $\mathbb{F}_q$ 
\begin{equation}
L=\prod_{i\in\Lambda_1}(\GL_{n_i}\times\GL_{n_i})^{d_i}\times\prod_{i\in\Lambda_2}(\GL_{n_i}\times\GL_{n_i})^{d_i}
\end{equation}
for some finite sets $\Lambda_1$ and $\Lambda_2$, endowed with the Frobenius $F$ acting on it as $F_w$ in (\ref{Flineaire}) and (\ref{Funitaire}), and with the automorphism $\sigma$ acting on it as in (\ref{sigmatransposition}). Let $T\subset L$ be an $F$-stable and $\sigma$-stable maximal torus and we wirte $W_L:=W_{L}(T)$. For all $i\in\Lambda_1\cup\Lambda_2$, we write $L_{i}:=(\GL_{n_i}\times\GL_{n_i})^{d_i}$, and denote by $T_{i}$ the corresponding direct factor of $T$. Write $W_{i}:=W_{L_{i}}(T_{i})$. Then we have
\begin{equation}
W_{i}\cong(\mathfrak{S}_{n_i}\times\mathfrak{S}_{n_i})^{d_i},\quad W_{i}^{\sigma}\cong \mathfrak{S}_{n_i}^{d_i} 
\end{equation}
for all $i\in\Lambda_1\cup\Lambda_2$. These are some direct factors of $W_L$ and of $W_L^{\sigma}$ that are stable under $F$.

Define an injection 
\begin{equation}\label{mapbracket}
\begin{split}
\Irr(W^{\sigma}_{L})^F&\hooklongrightarrow\Irr(W_{L})^F\\
\varphi&\longmapsto[\varphi]
\end{split}
\end{equation}
in the following manner.

For each $i\in\Lambda_1$, we have the bijections
\begin{equation}
\Irr(W_{i})^F\cong\mathcal{P}_{n_i}\times\mathcal{P}_{n_i},\quad \Irr(W_{i}^{\sigma})^F\cong \mathcal{P}_{n_i}.
\end{equation}
We define $\Irr(W^{\sigma}_{i})^F\rightarrow\Irr(W_{i})^F$ to be sending $\varphi$ to $(\varphi,\varphi)$. 

For each $i\in\Lambda_2$, we have
\begin{equation}
\Irr(W_{i})^F\cong\mathcal{P}_{n_i},\quad \Irr(W_{i}^{\sigma})^F\cong \mathcal{P}_{n_i}.
\end{equation}
We define $\Irr(W\smash{^{\sigma}_{i}})^F\rightarrow\Irr(W_{i})^F$ to be sending $\varphi$ to $\varphi$. 

For any $\varphi\in\Irr(W^{\sigma}_{L})^F$, we denote by $\tilde{\varphi}$ an extension of $\varphi$ to $W^{\sigma}_{L}\rtimes\lb F\rb$ and denote by $\smash{\widetilde{[\varphi]}}$ an extension of $[\varphi]$ to $W_{L}\rtimes\lb F\rb$, where $F$ is regarded as an automorphism of finite order of the Weyl group. Eventually, these choices need to be specified.

Denote by $\smash{\Irr^{\sigma}(L^F)}$ the set of the $\sigma$-stable linear characters of $L^F$. For any $\theta$ in $\smash{\Irr^{\sigma}(L^F)}$, we denote by $\tilde{\theta}$ its trivial extension to $L^F\lb\sigma\rb$. We also denote by the same letter the restriction of $\tilde{\theta}$ to $L^F\sigma$.

\subsubsection{}
Given $\varphi\in\Irr(W^{\sigma}_{L})^F$ and $\theta\in\smash{\Irr^{\sigma}(L^F)}$, by Theorem \ref{LS}, there is a particular choice of the extension $\widetilde{[\varphi]}$ such that 
\begin{equation}
\chi_1:=R^{L}_{[\varphi]}\theta=|W_{L}|^{-1}\sum_{w\in W_{L}}\widetilde{[\varphi]}(wF)R_{T_w}^{L}\theta.
\end{equation}
is a character. Obviously, it is a $\sigma$-stable character of $L^F$. Denote by $\tilde{\chi}_1$ an extension of  $\chi_1$ to $L^F\lb\sigma\rb$.

For any choice of the extension $\tilde{\varphi}$, put
\begin{equation}\label{Rphi}
R^{L\sigma}_{\varphi}\tilde{\theta}=|W_{L}^{\sigma}|^{-1}\sum_{w\in W_{L}^{\sigma}}\tilde{\varphi}(wF)R_{T_w\sigma}^{L\sigma}\tilde{\theta}.
\end{equation}
It is an $L^F$-invariant function on $L^F\sigma$.

\begin{Thm}\label{ExtL_1}
For a particular choice of the extension $\tilde{\varphi}$, we have
\begin{equation}
\tilde{\chi}_1|_{L^{F}\sigma}=\pm R^{L\sigma}_{\varphi}\tilde{\theta}.
\end{equation}
\end{Thm}

We will prove this theorem in the following section.

\subsection{The Proof}
The proof is to reduce the problem to smaller and smaller factors of $L$, until we can apply the known results on $\GL^{\pm}_{n'}(q)$, for various $n'$. The choice of the extension $\tilde{\varphi}$ will also be reduced to the smaller components until the choices are clear.
\subsubsection{Reduction to the Unipotent Characters}
Let $\chi_1=R^{L}_{[\varphi]}1$ be an irreducible character of $L^F$, which is necessarily $\sigma$-stable. Denote by $\tilde{\chi}_1\in\Irr(L^F\lb\sigma\rb)$ such that $\tilde{\chi}_1|_{L^F}=R^{L}_{[\varphi]}1$. Assume that for some choice of $\tilde{\varphi}$,
\begin{equation}
\tilde{\chi}_1|_{L^F\sigma}=R^{L\sigma}_{\varphi}1,
\end{equation}
where we denote by $1$ the trivial extension of the trivial character. 
Since $(R^{L}_{[\varphi]}1)\otimes\theta=R^{L}_{[\varphi]}\theta$ and $(R^{L\sigma}_{\varphi})1\otimes\tilde{\theta}=R^{L\sigma}_{\varphi}\tilde{\theta}$, we have
\begin{equation}
(\tilde{\chi}_1\otimes\tilde{\theta})|_{L^F}=R^{L}_{[\varphi]}\theta,\quad
(\tilde{\chi}_1\otimes\tilde{\theta})|_{L^F\sigma}=R^{L\sigma}_{\varphi}\tilde{\theta}.
\end{equation}
So it suffices to prove the theorem for the unipotent characters.

\subsubsection{Reduction with Respect to the Action of $F$ and $\sigma$}
We have decomposed $L$ into a product of the $L_{i}$'s for $i\in\Lambda_1\sqcup\Lambda_2$, each one being $F$-stable and $\sigma$-stable. Let us show that it suffices to prove the theorem for the $L_{i}$'s.

Write $F=(F_{i})_{i\in\Lambda_1\sqcup\Lambda_2}$ and $\sigma=(\sigma_{i})_{i\in\Lambda_1\sqcup\Lambda_2}$, where for each $i$, $F_{i}$ and $\sigma_{i}$ are respectively a Frobenius and an automorphism of the corresponding direct factor. The given $\varphi$ can be written as $(\varphi_i)_{i\in\Lambda_1\sqcup\Lambda_2}$, with $\varphi_{i}\in\Irr(W_{i}^{\sigma_{i}})^{F_{i}}$, then $[\varphi]=([\varphi_i])_i$. Suppose that $R^{L_{i}}_{[\varphi_i]}1$ are some irreducible characters, denoted by $\chi_i$, each one being $\sigma$-stable, and they extend to $\tilde{\chi}_i\in\Irr(L_{i}^{F_{i}}\lb\sigma_{i}\rb)$. We will show that if for some choices of the extensions $\tilde{\varphi}_i\in\Irr(W_{i}^{\sigma_{i}}\lb F_{i}\rb)$, the following equality holds $$\tilde{\chi}_i|_{L_{i}^{F_{i}}\sigma_{i}}=R_{\varphi_i}^{L_{i}\sigma_i}1,$$ then there is some choice of $\tilde{\varphi}\in\Irr(W_L^{\sigma}\lb F\rb)$ such that $$\tilde{\chi}_1|_{L^F\sigma}=R_{\varphi}^{L\sigma}1.$$

Given the extensions of the factors $\tilde{\chi}_i$, we can obtain an extension $\tilde{\chi}_1$ following Lemma  \ref{LemdeRed}. By definition, for any $l\sigma\in L^F\sigma$ which can be identified with $\prod_il_i\sigma_{i}\in\prod_iL^{F_{i}}_{i}\lb\sigma_{i}\rb$, we have,
\begin{equation}
\tilde{\chi}_1(l\sigma)=\prod_iR_{\varphi_i}^{L_{i}\sigma_{i}}1(l_i\sigma_{i}).
\end{equation}

On the one hand,
\begin{equation}
\begin{split}
\prod_iR_{T_{w_i}\sigma_{i}}^{L_{i}\sigma_{i}}1(l_i\sigma_{i})
&=\prod_i\tr(l_i\sigma_{i}|H^{\ast}_c(X_{w_i}))\\
&=\tr(\prod_il_i\sigma_{i}|\bigotimes_i H^{\ast}_c(X_{w_i}))\\
&=\tr(l\sigma|H^{\ast}_c(X_w))=R^{L\sigma}_{T_w\sigma}1(l\sigma).
\end{split}
\end{equation}
where the $T_{w_i}$ and $T_w$ are defined with respect to $T_{i}$ and $T$.

On the other hand, applying Lemma \ref{LemdeRed} to the Weyl groups and the Frobenius, we obtain an extension $\tilde{\varphi}$, such that for any $w=\prod_iw_i\in W^{\sigma}_L$, we have,
\begin{equation}
\tilde{\varphi}(wF)=(\prod_i\tilde{\varphi}_i)|_{W_L^{\sigma}.F}(\prod_i w_iF_{i})=\prod_i\tilde{\varphi}_i(w_iF_{i}).
\end{equation}

Consequently, $\tilde{\chi}_1|_{L^F\sigma}=R_{\varphi}^{L\sigma}1$. (One may also check that $\boxtimes_iR^{L_{i}}_{[\varphi_i]}1=R^{L}_{[\varphi]}1$ with some similar but simpler arguments.)

\subsubsection{The Linear Part, I}\label{LinearI}
In this part we fix $i\in\Lambda_1$. Write $M=\GL_{n_i}\times\GL_{n_i}$, equipped with the Frobenius $$F_M:(g,h)\mapsto(F_0(g),F_0(h))$$ with $F_0$ being the standard Frobenius of $\GL_n$ and the automorphism $$\tau(g,h)=(\sigma_0(h),\sigma_0(g))$$ of the form \S \ref{2facQC}, with $\sigma_0$ commuting with $F_0$. Then $L_{i}=M\times\cdots\times M$ is a direct product of $d_i$ copies of $M$ equipped with the automorphism $\sigma=\tau\times\cdots\times\tau$ consisting of $d_i$ copies of $\tau$. We will fix $i$ and write $d=d_i$. We may assume that the maximal torus $T_{i}\subset L_{i}$, which is $F_{i}$-stable and $\sigma_{i}$-stable, is of the form $T_M\times\cdots\times T_M$, where $T_M\subset M$ is an $F_M$-stable and $\tau$-stable maximal torus. Note that $F_M$ acts trivially on $W_M:=W_M(T_M)$. The Frobenius $F_{i}$ acts on $L_{i}$ in the following manner
\begin{equation}
\begin{split}
M\times\cdots\times M&\longrightarrow M\times\cdots\times M\\
(m_1,\ldots,m_d)&\longmapsto (F_M(m_d),F_M(m_1),\ldots,F_M(m_{d-1})).
\end{split}
\end{equation}
We have a natural commutative diagram
\begin{equation}
\begin{CD}
\Irr(W^{\tau}_M) @>>>\Irr(W^{\sigma_{i}}_{i})^{F_{i}}\\
@VVV @VVV\\
\Irr(W_M) @>>> \Irr(W_{i})^{F_{i}}
\end{CD}
\end{equation}
where the upper horizontal bijective map $\varphi_M\longmapsto\varphi_i=(\varphi_M,\ldots,\varphi_M)$ 
identifies each element of $\Irr(W^{\sigma_{i}}_{i})^{F_{i}}$ with $d$ identical copies of an element of $\Irr(W^{\tau}_M)$. Denote by $[\varphi_M]$ and $[\varphi_i]$ the images of the vertical maps defined as in (\ref{mapbracket}), which are matched under the lower horizontal map.
 
Endow $M$ with the Frobenius $F_M^d$. Suppose that for some choice of $\tilde{\varphi}_M$,
\begin{equation}\label{LinearM}
R^{M.\tau}_{\varphi_M}1=|W^{\tau}_M|^{-1}\sum_{w\in W^{\tau}_M}\tilde{\varphi}_M(wF^d_M)R^{M.\tau}_{T_w.\tau}1
\end{equation}
where $T_w$ is defined with respect to $T_M$, is an extension of the irreducible character $R^{M}_{[\varphi_M]}1$ of $M^{F_M^d}$. Let us show that
\begin{equation}\label{LinearL}
R^{L_{i}\sigma_{i}}_{\varphi_i}1=|W^{\sigma_{i}}_{i}|^{-1}\sum_{w\in W^{\sigma_{i}}_{i}}\tilde{\varphi}_i(wF_{i})R^{L_{i}\sigma_{i}}_{T_w\sigma}1
\end{equation}
is an extension of the irreducible character $R^{L_{i}}_{[\varphi_i]}1$ of $L_{i}^{F_{i}}$. In fact, there is a natural isomorphism $\smash{M^{F_M^d}\cong L_{i}^{F_{i}}}$ compatible with the action of $\tau$ and $\sigma$. We are going to show that $\smash{R^{M.\tau}_{\varphi_M}1}$ coincides with $\smash{R^{L_{i}\sigma_{i}}_{\varphi_i}1}$ under this isomorphism and they are an extension of the same character of $\smash{M^{F_M^d}=L_{i}^{F_{i}}}$ corresponding to $[\varphi_M]=[\varphi_i]$.

Applying Lemma \ref{extdecirc} to $K=W^{\tau}_M$, $H=W^{\sigma_{i}}_{i}$, $\psi=F_M$, and $\zeta=(d,\ldots,2,1)\in\mathfrak{S}_d$, we deduce from $\tilde{\varphi}_M$ an extension $\tilde{\varphi}_i$ such that for $w=(w_1,\ldots,w_d)\in (W^{\tau}_M)^d\cong W^{\sigma_{i}}_{i}$, 
\begin{equation}
\begin{split}
\tilde{\varphi}_i(wF_{i})&=\tilde{\varphi}_M(w_dF_Mw_{d-1}F_M\ldots w_1F_M)\\
&=\tilde{\varphi}_M(w_dw_{d-1}\ldots w_1F_M^d).
\end{split}
\end{equation}
Write $w_1'=w_dw_{d-1}\ldots w_1$. Then, $w$ is $F_{i}$-conjugate to $w'=(w_1',1,\ldots,1)$, so for any $l\in L_{i}^{F_{i}}$, we have
$$R^{L_{i}\sigma_{i}}_{T_w\sigma_{i}}1(l\sigma_{i})=R^{L_{i}\sigma_{i}}_{T_{w'}\sigma_{i}}1(l\sigma_{i})=\tr(l\sigma_{i}|H^{\ast}_c(X_{w'}))$$ We can write $l=(m,F_M(m),\ldots,F_M^{d-1}(m))$ with $m$ satisfying $F^d_M(m)=m$. Since the two varieties
\begin{equation*}
\begin{split}
X_{w'}&=\{B\in\mathcal{B}_{L_{i}}|(B,F_{i}(B))\in\mathcal{O}(w')\},\\
 X_{w_1'}&=
 \{B\in\mathcal{B}_M|(B,F_M^d(B))\in\mathcal{O}(w_1')\}
\end{split}
\end{equation*}
are isomorphic, and the actions of $l\sigma_{i}$ and of $m\tau$ on the two varieties are compatible, we have
\begin{equation*}
\tr(l\sigma_{i}|H^{\ast}_c(X_{w'}))
=\tr(m\tau|H^{\ast}_c(X_{w'_1}))
=R^{M.\tau}_{T_{w'_1}.\tau}1(m\tau),
\end{equation*}
Consequently, the value of $\smash{\tilde{\varphi}_i(wF_{i})R^{L_{i}\sigma_{i}}_{T_w\sigma_{i}}1}$ only depends on $w_1'\in W_M^{\tau}$ and is equal to $\tilde{\varphi}_M(w_1'F_M^d)R^{M.\tau}_{T_{w'_1}.\tau}1$, This, together with the fact that $|W_{i}^{\sigma_{i}}|=|W^{\tau}_M|^d$, shows that $\smash{R^{M.\tau}_{\varphi_M}1=R^{L_{i}\sigma_{i}}_{\varphi_i}1}$. (Similar arguments show that $\smash{R^{M}_{[\varphi_M]}1=R^{L_{i}}_{[\varphi_i]}1}$.)

\subsubsection{The Unitary Part, I} 
In this part we require that $i\in\Lambda_2$. We keep the same notations as above except that $F_{i}$ acts on $L_{i}$ in the following manner
\begin{equation}
\begin{split}
M\times\cdots\times M&\longrightarrow M\times\cdots\times M\\
(m_1,\ldots,m_d)&\longmapsto (F'_M(m_d),F_M(m_1),\ldots,F_M(m_{d-1})),
\end{split}
\end{equation}
where $$F_M':(g,h)\mapsto(F_0(h),F_0(g)).$$ Denote by $F_M'':M\rightarrow M$ the Frobenius $$F'_MF^{d-1}_M:(g,h)\mapsto (F_0^d(h),F^d_0(g)).$$ We still have a natural identification 
\begin{equation}
\begin{split}
\Irr(W^{\tau}_M)^{F''_M}&\longrightarrow\Irr(W^{\sigma_{i}}_{i})^{F_{i}}\\
\varphi&\longmapsto\varphi_L=(\varphi,\ldots,\varphi)
\end{split}
\end{equation}
($F_M'$ acts trivially on $\Irr(W^{\tau}_M)$ since $\sigma_0$ induces an inner automorphism on the Weyl group) and an isomorphism $\smash{M^{F''_M}\cong L_{i}^{F_{i}}}$. In a way similar to \S \ref{LinearI}, from the equality (\ref{LinearM}), with $F_M^d$ replaced by $F_M''$, one deduces the equality (\ref{LinearL}), with $F_i$ defined in the present setting.

\subsubsection{}
From now on we write $M=G\times G$, $G=\GL_n^{\epsilon}(q)$ with $\epsilon=\pm$. Denote by $F_0$ the split Frobenius of $G$ and by $F'_0$ the Frobenius of $G$ corresponding to $\epsilon$, and denote by $\sigma_0$ an order 2 automorphism of $G$, which commutes with the Frobenius endomorphisms. %By \cite[Theoreme 3.20, Theoreme 3.22 (i)]{L4}, every unipotent character of $\GL_n^{+}(\mathbb{F}_q)$ is uniform while $\GL_n^{-}(\mathbb{F}_q)$ has a unique unipotent cuspidal representation if and only if $n=1/2(s^2+s)$ for some integer $s\ge 1$, and has no unipotent cuspidal representation otherwise.

\subsubsection{The Linear Part, II}\label{caslin}
It is essential that we allow $G$ to be $\GL_n^-(q)$, which will be applied to the unitary part later. Define
\begin{IEEEeqnarray*}{rlcrl}
\sigma:M&\longrightarrow M  &\quad\quad &F:M&\longrightarrow M\\
(g,h)&\longmapsto (\sigma_0(h),\sigma_0(g))  &\quad\quad &(g,h)&\longmapsto (F'_0(g),F'_0(h)).
\end{IEEEeqnarray*}
Let $\chi_G$ be a unipotent irreducible character of $G^{F'_0}$ corresponding to some $\varphi\in\Irr(W_G)$, it defines a character $\chi_M=\chi_G\otimes\chi_G\in\Irr(M^F)$ which is invariant under the action of $\sigma$ and so extends to $M^F\lb\sigma\rb$, denoted by $\tilde{\chi}_M$. Every $\sigma$-stable irreducible unipotent character of $M^F$ is of the form $\chi_G\otimes\chi_G$. Regarding $\varphi$ as a character of $W_M^{\sigma}$, we show that up to a sign,
\begin{eqnarray}
\tilde{\chi}_M|_{M^F\sigma}=R^{M\sigma}_{\varphi}1:=|W_{M}^{\sigma}|^{-1}\sum_{w\in W_{M}^{\sigma}}\tilde{\varphi}(wF)R^{M\sigma}_{T_{w}\sigma}1,
\end{eqnarray}
for some choice of the extension $\tilde{\varphi}$.

We apply Lemma \ref{extdecirc} by taking $G^{F'_0}$ for $K$ and obtain $$\tilde{\chi}_M((g,h)\sigma)=\chi_G(g\sigma_0(h))$$ for any $(g,h)\in M^F$. By Theorem \ref{LS}, the irreducible unipotent character of $G^{F'_0}$ can be expressed as $$|W_G|^{-1}\sum_{w\in W_G}\tilde{\varphi}_G(wF'_0)R^G_{T_w}1,$$ for some choice of $\tilde{\varphi}_G$. The extension $\tilde{\varphi}$ is then defined by $\tilde{\varphi}_G$ under the isomorphism $W_G\cong W_M^{\sigma}$, noticing that the action of  $F$ is compatible with the action of $F'_0$ under this isomorphism. Comparing this expression with $R^{M\sigma}_{\varphi}1$, we are reduced to show that for any $\smash{(g,h)\in M^F}$
\begin{equation}\label{caselinearequality}
R^{M\sigma}_{T_{w_M}\sigma}1((g,h)\sigma)=R^G_{T_w}1(g\sigma_0(h))
\end{equation}
with $w_M=(w,\sigma_0(w))\in W^{\sigma}_M$. Observe that $(g,h)\sigma\mapsto R^G_{T_w}1(g\sigma_0(h))$ defines a function on $M^F\sigma$ invariant under the conjugation by $M^F$.

\subsubsection{}
Let us prove a more general assertion. Let $I$ be an $F'_0$-stable Levi subgroup of $G$. Then $J:=I\times\sigma_0(I)$ is a $\sigma$-stable Levi factor of a $\sigma$-stable parabolic subgroup of $M$, which justifies the functor $R^{M\lb\sigma\rb}_{J\lb\sigma\rb}$. Let $\chi_I$ be an irreducible character of $I^{F'_0}$, it defines a character $\chi_J=\chi_I\otimes\sigma_0({\chi}_I)\in\Irr(J^F)$ which is invariant under the action of $\sigma$ and thus extends to $J^F\lb\sigma\rb$. Denoted by $\tilde{\chi}_J$ a choice of such extension.  The following lemma with $I=T_w$ and $\chi_I=1$ proves the above assertion.
\begin{Lem}\label{lemcaslin}
We keep the notations as above. Assume that for any $(g',h')\in J^F$, we have $\tilde{\chi}_J((g',h')\sigma)=\chi_I(g'\sigma_0(h'))$. Then, for any $(g,h)\in M^F$, we have
\begin{equation}
R^{M\lb\sigma\rb}_{J\lb\sigma\rb}\tilde{\chi}_J((g,h)\sigma)=R^G_I\chi_I(g\sigma_0(h)).
\end{equation}
\end{Lem}
\begin{proof}
Let $(g,h)\sigma=\zeta\mu$ be the Jordan decomposition, with $\zeta$ semi-simple and $\mu$ unipotent, and we write $\mu=(u,w)$ and $\zeta=(s,t)\sigma$. Beware that neither $s$ nor $t$ is necessarily semi-simple. Also let $g\sigma_0(h)=\bar{s}\bar{u}$ be the Jordan decomposition, with $\bar{s}$ semi-simple and $\bar{u}$ unipotent. The proof will simply be comparing the following two formulas term by term.
\begin{IEEEeqnarray}{l}
R^{M\lb\sigma\rb}_{J\lb\sigma\rb}\tilde{\chi}_J((g,h)\sigma)=|J^F|^{-1}|C_M(\zeta)^{\circ F}|^{-1}\sum_{\substack{\{x\in M^F|\\x\zeta x^{-1}\in J\sigma\}}}\sum_{v\in C_{x^{-1}\bar{J}x}(\zeta)_u^{F}}Q^{C_M(\zeta)^{\circ}}_{C_{x^{-1}Jx}(\zeta)^{\circ}}(\mu,v^{-1})\leftidx{^x}{\tilde{\chi}_J}(\zeta v)\nonumber\\
R^G_I\chi_I(g\sigma_0(h))=|I^{F'_0}|^{-1}|C_G(\bar{s})^{\circ F'_0}|^{-1}\sum_{\substack{\{y\in G^{F'_0}|\\y\bar{s}y^{-1}\in I\}}}\sum_{v_1\in C_{y^{-1}Iy}(\bar{s})\smash{^{\circ F'_0}_u}}Q^{C_G(\bar{s})^{\circ}}_{C_{y^{-1}Iy}(\bar{s})^{\circ}}(\bar{u},v_1^{-1})\leftidx{^y}{\chi_I}(\bar{s} v_1).\nonumber
\end{IEEEeqnarray}
An element $(z_1,z_2)\in M$ commutes with $(s,t)\sigma$ if and only if $z_2=t\sigma_0(z_1)t^{-1}$ and $z_1=s\sigma_0(z_2)s^{-1}$, if and only if $z_1\in C_G(s\sigma_0(t))$ and $z_2=t\sigma_0(z_1)t^{-1}$, whence an isomorphism $C_M((s,t)\sigma)\cong C_G(s\sigma_0(t))$. An element of $M\lb\sigma\rb$ is semi-simple if and only if its square is semi-simple since the characteristic is odd. The equality $((s,t)\sigma)^2=(s\sigma_0(t),t\sigma_0(s))$ shows that $s\sigma_0(t)$ is semi-simple. The unipotent part $(u,w)$ commutes with $(s,t)\sigma$, so $u\in C_G(s\sigma_0(t))$. Considering the equality $$(g\sigma_0(h),h\sigma_0(g))=((g,h)\sigma)^2=((s,t)\sigma)^2(u,w)^2=(s\sigma_0(t),t\sigma_(s))(u^2,w^2),$$ we have $g\sigma_0(h)=s\sigma_0(t)u^2$. This gives the Jordan decomposition of $g\sigma_0(h)$ because $u$ commutes with $s\sigma_0(t)$ and $u^2$ is unipotent. Therefore, $\bar{s}=s\sigma_0(t)$, $\bar{u}=u^2$ and $|C_M(\zeta)^{\circ F}|=|C_G(\bar{s})^{\circ F'_0}|$.

Write $x=(x_1,x_2)\in M^F$. The condition $x\zeta x^{-1}\in J\sigma$ means that $x_1s\sigma_0(x_2^{-1})\in I$ and that $x_2t\sigma_0(x_1^{-1})\in\sigma_0(I)$, which implies that $x_1s\sigma_0(t)x^{-1}_1\in I$. Fix $(x_1,x_2)$ satisfying $x\zeta x^{-1}\in J\sigma$, then every element of $(x_1,\sigma_0(I^{F'_0})x_2)$ also satisfies it. Let $(x_1,x_2)$ and $(x_1,x_2')$ be two elements satisfying $x\zeta x^{-1}\in J\sigma$, then the conditions $x_1s\sigma_0(x_2^{-1})\in I$ and $x_1s\sigma_0(x_2^{'-1})\in I$ implies that $x_2'x_2^{-1}\in \sigma_0(I^{F'_0})$. We thus obtain a bijection of sets $$\{x\in M^F|x\zeta x^{-1}\in J\sigma\}\cong\{y\in G^{F'_0}|y\bar{s}y^{-1}\in I\}\times I^{F'_0},$$ with $y$ corresponding to the factor $x_1$ of $x$. We will see that the sum over $x$ in the character formula is invariant under multiplying $x_2$ by an element of $\sigma_0(I^{F'_0})$ on the left, which cancels a factor $|I^{F'_0}|$ from $|J^F|=|I^{F'_0}|^2$.

The isomorphism $C_M(\zeta)^{\circ}\cong C_G(\bar{s})^{\circ}$ restricts to an isomorphism $C_{x^{-1}Jx}(\zeta)^{\circ}\cong C_{y^{-1}Iy}(\bar{s})^{\circ}$. For characteristic reason, the unipotent elements of $C_{x^{-1}J x}(\zeta)$ are contained in  $C_{x^{-1}J x}(\zeta)^{\circ}$ by \cite[Remarque 2.7]{DM94} and \cite[Th\'eor\`eme 1.8 (i)]{DM94}, whence a bijection $C_{x^{-1}Jx}(\zeta)^{\circ F}_u\cong C_{y^{-1}Iy}(s\sigma_0(t))\smash{^{\circ F'_0}_u}$, by which $v=(v_1,v_2)$ is sent to $v_1$. 

Now we compare the characters $\leftidx{^x}{\tilde{\chi}_J}(\zeta v)$ et $\leftidx{^y}{\chi_I}(\bar{s} v_1)$. Write $x=(x_1,x_2)$ and $v=(v_1,v_2)$. Then 
\begin{equation}
x\zeta vx^{-1}=(x_1s\sigma_0(v_2)\sigma_0(x_2)^{-1},x_2t\sigma_0(v_1)\sigma_0(x_1)^{-1})\sigma.
\end{equation}
Taking into account the equality $v_2=t\sigma_0(v_1)t^{-1}$, we have
\begin{equation}
\big(x_1s\sigma_0(v_2)\sigma_0(x_2)^{-1}\big)\sigma_0\big(x_2t\sigma_0(v_1)\sigma_0(x_1)^{-1})\big)=x_1s\sigma_0(t)v_1^2x_1^{-1}.
\end{equation}
By assumption, $\tilde{\chi}_J((g',h')\sigma)=\chi_I(g'\sigma_0(h'))$, for any $(g',h')\in J^F$, whence 
\begin{equation}
\leftidx{^x}{\tilde{\chi}_J}(\zeta v)=\chi_I(x_1s\sigma_0(t)v_1^2x_1^{-1})=\leftidx{^{x_1}}{\chi_I}(\bar{s} v_1^2),
\end{equation}
where we also see that multiplying $x_2$ by an element of $\sigma_0(I^{F_0})$ on the left does not change the value. Since $v\mapsto v^2$ defines a bijection of $C_{y^{-1}Iy}(\bar{s})_u^{\smash{\circ F'_0}}$ into itself, it only remains to show the first of the following equalities of Green functions 
\begin{equation}
Q^{C_G(\bar{s})^{\circ}}_{C_{y^{-1}Iy}(\bar{s})^{\circ}}(u^2,v_1^{-2})=Q^{C_G(\bar{s})^{\circ}}_{C_{y^{-1}Iy}(\bar{s})^{\circ}}(u,v_1^{-1})=Q^{C_M(\zeta)^{\circ}}_{C_{x^{-1}Jx}(\zeta)^{\circ}}(\mu,v^{-1}),
\end{equation}
which follows from the fact that the value of the Green function only depends on the associated partition and a power prime to $p$ does not change the Jordan blocks of a unipotent matrix.
\end{proof}

%%%%%%%%%%%%%%%%%%%
\begin{comment}
\subsubsection{}
We deduce from the previous lemma that the extensions of the $\sigma$-stable characters of $L^F$ are uniform. By Lemma \ref{extdecirc} and Lemma \ref{idenacirc}, if $\chi_G$ is an arbitrary character of $G^{F_0'}$ and if $\tilde{\chi}_L$ is an extension of $\chi_L:=\chi_G\otimes\sigma_0(\chi_G)$, then for any $(g,h)\in L^F$, we have $\tilde{\chi}_L|_{L^F\sigma}((g,h)\sigma)=\chi_G(g\sigma_0(h))$. There exists a Levi subgroup $M\subset G$, $\varphi\in\Irr(W_G)^{F'_0}$ and $\theta\in\Irr_{reg}(M^{F'_0})$ such that up to a sign,
\begin{equation}
\chi_G=\frac{1}{|W_M|}\sum_{w\in W_M}\tilde{\varphi}(wF_0)R^G_{T_w}\theta.
\end{equation} 

Write $\theta_J=\theta\otimes\sigma_0(\theta)\in\Irr_{reg}(J^F)$, where $J$ is as above. Applying Lemma \ref{extdecirc} and Lemma \ref{idenacirc} to the extension $\tilde{\theta}_J$, we see that $\theta$ and $\tilde{\theta}_J$ satisfy the assumption of Lemma \ref{lemcaslin}, which shows that $R^G_{T_w}\theta(g\sigma_0(h))=R^{L\sigma}_{T_{w_L}\sigma}\tilde{\theta}((g,h)\sigma)$, whence
\begin{equation}
\tilde{\chi}_L|_{L^F\sigma}=\pm\frac{1}{|W^{\sigma}_J|}\sum_{w\in W^{\sigma}_J}\tilde{\varphi}(wF_0)R^{L\sigma}_{T_{w}\sigma}\tilde{\theta}.
\end{equation}
\end{comment}
%%%%%%%%%%%%%%%%%%%%

\subsubsection{The Unitary Part, II}
Define
\begin{IEEEeqnarray*}{rlcrl}
\sigma:M&\longrightarrow M  &\quad\quad &F:M&\longrightarrow M\\
(g,h)&\longmapsto (\sigma_0(h),\sigma_0(g))  &\quad\quad &(g,h)&\longmapsto (F_0(h),F_0(g)).
\end{IEEEeqnarray*}
Now, $M^F$ is isomorphic to $G^{F_0^2}$ under the map $(g,F_0(g))\mapsto g$, and $M^{\sigma}$ is isomorphic to $G$ under the map $(g,\sigma_0(g))\mapsto g$. The Frobenius $F$ acts on $M^{\sigma}\cong G$ by $g\mapsto \sigma_0F_0(g)$. The automorphism $\sigma$ acts on $M^F$ by $$(g,F_0(g))\mapsto(\sigma_0F_0(g),\sigma_0(g))=(\sigma_0F_0(g),\sigma_0F_0^2(g)),$$ or in other words, $\sigma$ acts on  $G^{F_0^2}$ as $\sigma_0F_0$. 

Let $\chi_M$ be a unipotent irreducible character of $M^F$, and let $\tilde{\chi}_M$ be an irreducible character of $M^F\lb\sigma\rb$ that extends $\chi_M$. We are going to show that up to a sign,
\begin{eqnarray}
\tilde{\chi}_M|_{M^F\sigma}=R^{M\sigma}_{\varphi}1=|W_{M}^{\sigma}|^{-1}\sum_{w\in W_{M}^{\sigma}}\tilde{\varphi}(wF)R^{M\sigma}_{T_{w}\sigma}1
\end{eqnarray}
for some choice of $\tilde{\varphi}$. Under the isomorphism $W_M^{\sigma}\cong W_G$, the Frobenius $F$ acts as $\sigma_0$ on $W_G$, and so an $F$-stable character of $W_M^{\sigma}$ is just a character of $W_G$. We are reduced to show that if $\chi_G$ is an irreducible unipotent character of $G^{F_0^2}$ corresponding to $\varphi\in\Irr(W_G)$, then its extension $\tilde{\chi}_G$ to $G^{F_0^2}\lb\sigma_0F_0\rb$ is given by the above formula up to a sign.

We need the Shintani descent. Suppose that $\chi_G\in\Irr(G^{F_0^2})^{\sigma_0F_0}$ is the unipotent character corresponding to $\varphi\in\Irr(W_G)^{\sigma_0}$. We apply Theorem \ref{5.6} with $(\sigma_1F_0^m,\sigma_2F_0)=(F_0^2,\sigma_0F_0)$, i.e. $m=2$, $\sigma_1=1$ and $\sigma_2=\sigma_0$, and deduce that
\begin{equation}
\begin{split}
\tilde{\chi}_G&=E_{\sigma_0F_0}(\chi_G)=\Sh_{\sigma_0F_0/F_0^2}\Omega^2_{\sigma_0F_0}R^{G^{\sigma_0F_0}}_{\varphi}1\\
&=\pm\Sh_{\sigma_0F_0/F_0^2}R^{G^{\sigma_0F_0}}_{\varphi}1=\pm\Sh_{\sigma_0F_0/F_0^2}|W_G|^{-1}\sum_{w\in W_G}\tilde{\varphi}(w\sigma_0F_0)R^{G^{\sigma_0F_0}}_{T_w}1,
\end{split}
\end{equation}
since $\Omega^2_{\sigma_0F_0}=\pm 1$ because $R^{G^{\sigma_0F_0}}_{\varphi}1$ is an irreducible unipotent character $G^{\sigma_0F_0}$ on which $\Omega^2_{\sigma_0F_0}$ acts as a scalar, whose value is given by \S \ref{vpdeF^2}. For example, $\Omega_{\sigma_0F_0}=1$ on principal series representations and $\Omega^2_{\sigma_0F_0}=-1$ on cuspidal unipotent characters according to (\cite[Table I]{L77}). The sign $(\pm 1)$ does not matter since the two extensions of $\chi_G$ only differ by a sign. It remains to show that $\Sh_{\sigma_0F_0/F_0^2}R^{G^{\sigma_0F_0}}_{T_w}1=R^{M\sigma}_{T_{w_M}\sigma}1$, where $w_M$ is as in (\ref{caselinearequality}) and $M$ is equipped with the Frobenius $F$. The function $R^{G^{\sigma_0F_0}}_{T_w}1$ is invariant under $F_0^2$-conjugation as $F_0^2$ acts trivially on $G^{\sigma_0F_0}$, which justifies $\smash{\Sh_{\sigma_0F_0/F_0^2}\circ R^{G^{\sigma_0F_0}}_{T_w}1}$.

Proposition  \ref{Di1.1} gives \begin{equation}
\Sh_{\sigma F/F}\circ R^{M^{\sigma F}\sigma}_{T_{w_M}\sigma}1=R^{M^F\sigma}_{T_{w_M}\sigma}1.
\end{equation}
(One checks that with respect to a fixed $F$-stable and $\sigma$-stable maximal torus $T\subset M$, the maximal torus $T_{w_M}$ of type $w_M$ with respect to $F$ is also of type $w_M$ with respect to $\sigma F$, using the fact that for $\sigma$ quasi-central, $w_M$ has a representative in $M^{\sigma}$.) Since $F$ acts as $\sigma$ on $M^{\sigma F}$, the function $\smash{R^{M^{\sigma F}\sigma}_{T_{w_M}\sigma}1}$ is invariant under the $F$-conjugation of $M^{\sigma F}$, and its Shintani descent $\smash{\Sh_{\sigma F/F}\circ R^{M^{\sigma F}\sigma}_{T_{w_M}\sigma}1}$ belongs to $\mathcal{C}(M^F\sigma F)$. There is a natural bijection $\mathcal{C}(M^F\sigma F)\isomLR\mathcal{C}(G^{F_0^2}\sigma_0F_0)$. Let us show that $$\Sh_{\sigma_0F_0/F_0^2}\circ R^{G^{\sigma_0F_0}}_{T_w}1=\Sh_{\sigma F/F}\circ R^{M^{\sigma F}\sigma}_{T_{w_M}\sigma}1,$$ which concludes the proof.

For $g\in G^{F_0^2}$, there exists $x\in G$ such that $x\sigma_0F_0(x)^{-1}=g$, and so $$N_{F_0^2/\sigma_0F_0}(g)=x^{-1}F^2_0(x)\in G^{\sigma_0F_0}.$$ We also have $$(x,F_0(x))\sigma F(x^{-1},F_0(x)^{-1})=(g,F_0(g)),$$ and so $$N_{F/\sigma F}((g,F_0(g))=(x^{-1},F_0(x)^{-1})F(x,F_0(x))=(x^{-1}F_0^2(x),1)\in M^{\sigma F}.$$ Therefore, 
\begin{equation}
\begin{split}
&\Sh_{\sigma F/F}\circ R^{M^{\sigma F}\sigma}_{T_{w_M}\sigma}1((g,F_0(g))\stackrel{\circled{1}}{=}R^{M^{\sigma F}\sigma}_{T_{w_M}\sigma}1((x^{-1}F_0^2(x),1))\\&\stackrel{\circled{2}}{=}R^{G^{\sigma_0F_0}}_{T_w}1(x^{-1}F_0^2(x))\stackrel{\circled{3}}{=}\Sh_{\sigma_0F_0/F_0^2}\circ R^{G^{\sigma_0F_0}}_{T_w}1(g).
\end{split}
\end{equation}
Equality \circled{1} is the definition of Shintani descent and we have identified the functions on $M^{\sigma F}\sigma$ to the functions on $M^{\sigma F}$ invariant under the $F$-conjugation. We have equality \circled{2} by (\S \ref{caselinearequality}) with the automorphism $\sigma(g,h)=(\sigma_0(h),\sigma_0(g))$ and the Frobenius $\sigma F(g,h)=(\sigma_0F_0(g),\sigma_0F_0(h))$. Equality \circled{3} is again the definition of Shintani descent.

%%%%%%%%%%%%%%% Extensions des Caract\`eres Quadratic-Unipotents
\subsection{Extensions of Quadratic-Unipotent Characters}\label{ExtQuaUni}
We now focus on $L_0\cong \GL_{n_0}(k)$. This subsection is devoted to the statement of the main theorem of \cite{W1}.

Let $\sigma_0$ and $\sigma_0'$ be the automorphisms defined for $\GL_{n_0}(k)$ in the same way $\sigma$ and $\sigma'$ are defined for $\GL_n(k)$, and let $F_0$ be the Frobenius that sends each entry to its q-th power. Now the semi-direct product of $\GL_{n_0}(k)$ by $\sigma_0$ (resp. by $\sigma_0'$) is denoted by $\leftidx{^s\!}{\bar{G}}_0$ (resp. $\leftidx{^o\!}{\bar{G}}_0$). We may regard $\sigma_0$ also as an element of $\leftidx{^o\!}{\bar{G}}_0$, acting as $\sigma_0$ on $\GL_{n_0}(k)$ but satisfying $\sigma_0^2=-1$. The point is that, we want to fix a quasi-central element in $\leftidx{^o\!}{\bar{G}}_0$ to work with, and $\sigma_0$ is a convenient choice. 

\subsubsection{}\label{Lw}
Let $m$, $N_+$ and $N_-$ be non-negative integers such that $n_0=m+2N_++2N_-$. We consider the data $(h_1,h_2,w_+,w_-)$, with $(h_1,h_2)\in\mathbb{N}\times\mathbb{Z}$, $w_+\in \mathfrak{W}_{N_+}$ and $w_-\in \mathfrak{W}_{N_-}$. To simplify, we write $\mathfrak{W}_+=\mathfrak{W}_{N_+}$, $\mathfrak{W}_-=\mathfrak{W}_{N_-}$ and $\mathbf{w}=(h_1,h_2,w_+,w_-)$ instead. To each $\mathbf{w}$ is associated an $F_0$-stable and $\sigma_0$-stable Levi factor $L_{\mathbf{w}}$ of a $\sigma_0$-stable parabolic subgroup, isomorphic to $T_{w_+}\times T_{w_-}\times\GL_m(k)$, where $T_{w_{\pm}}$ is isomorphic to $(k^{\ast})^{2N_{\pm}}$ equipped with the Frobenius twisted by $w_{\pm}$. Each of its factors being preserved by $\sigma_0$ and $F_0$, we write $\sigma_0=\sigma_+\times\sigma_-\times\sigma_{00}$ and $F_0=F_+\times F_-\times F_{00}$ with respect to this decomposition. 

\subsubsection{}\label{cuspGL_m}
The cuspidal local systems on $\GL_m(k)\sigma_{00}$ are described as follows. 

Let $(h_1,h_2)\in\mathbb{N}\times\mathbb{Z}$. By \S \ref{isolelts}, there is a unique semi-simple isolated conjugacy class on $\GL_m(k)\sigma_{00}$ with connected centraliser isomorphic to $\Sp_{h_1(h_1+1)}(k)\times\SO_{h_2^2}(k)$. Let $s\sigma_{00}$ be an $F$-stable element representing this conjugacy class. By Theorem \ref{cusploc}, there is a unique cuspidal local system on $\Sp_{h_1(h_1+1)}(k)\times\SO_{h_2^2}(k)$, which is supported on the unipotent conjugacy class whose symplectic (resp. orthogonal) component corresponds to the symplectic partition $\lambda_1:=(2h_1,2h_1-2,\ldots,2)$ (resp. orthogonal partition $\lambda_2:=(2|h_2|-1,2|h_2|-3,\ldots,1)$). Let $u=(u_1,u_2)\in\Sp_{h_1(h_1+1)}(k)\times\SO_{h_2^2}(k)$ be an $F$-stable element representing this unipotent conjugacy class. 

By Proposition \ref{IV18.2} and Lemma \ref{L036.6}, there exist cuspidal local systems supported on the conjugacy class of $s\sigma_{00}u$, which itself is an isolated stratum according to Example \ref{isostm}. The irreducible equivariant local systems on this conjugacy class are parametrised by $((\mathbb{Z}/2\mathbb{Z})^{\bs{\kappa}(\lambda_1)})^{\vee}\times((\mathbb{Z}/2\mathbb{Z})^{\bs{\kappa}(\lambda_2)})^{\vee}$ since the component group of $C_{\GL_m(k)}(s\sigma_{00}u)$ is isomorphic to $(\mathbb{Z}/2\mathbb{Z})^{\bs{\kappa}(\lambda_1)}\times(\mathbb{Z}/2\mathbb{Z})^{\bs{\kappa}(\lambda_2)}$ according to \S \ref{uni-centraliser}. The irreducible equivariant local systems on the conjugacy class of $u_1\in \Sp_{h_1(h_1+1)}(k)$ (resp. $u_2\in\SO_{h_2^2}(k)$) are parametrised by $((\mathbb{Z}/2\mathbb{Z})^{\bs{\kappa}(\lambda_1)})^{\vee}$ (resp. $((\mathbb{Z}/2\mathbb{Z})^{\bs{\kappa}(\lambda_2)})^{\vee}/\Delta$) for similar reasons, where $\Delta\subset((\mathbb{Z}/2\mathbb{Z})^{\bs{\kappa}(\lambda_2)})^{\vee}$ is the subgroup generated by $(\eta,\ldots,\eta)$. The restriction map that sends a local system on the class of $s\sigma_{00}u$ to the class of $u$ in $C_{\GL_m(k)}(s\sigma_{00})^{\circ}$ is given by the natural quotient map 
\begin{equation}\label{Delta-quot}
((\mathbb{Z}/2\mathbb{Z})^{\bs{\kappa}(\lambda_1)})^{\vee}\times((\mathbb{Z}/2\mathbb{Z})^{\bs{\kappa}(\lambda_2)})^{\vee}\longrightarrow ((\mathbb{Z}/2\mathbb{Z})^{\bs{\kappa}(\lambda_1)})^{\vee}\times((\mathbb{Z}/2\mathbb{Z})^{\bs{\kappa}(\lambda_2)})^{\vee}/\Delta.
\end{equation}
Therefore by Lemma \ref{L036.6} again, there are two cuspidal local systems supported on the class of $s\sigma_{00}u$, corresponding to the fibre of the above quotient map over the unique cuspidal local system on $\Sp_{h_1(h_1+1)}(k)\times\SO_{h_2^2}(k)$. These exploit all cuspidal local systems on $\GL_m(k)\sigma_{00}$.

\subsubsection{}\label{cuspSpSO}
We may choose $s\sigma_{00}$ in such a way that $C_{\GL_m(q)}(s\sigma_{00})^F$ is split. Denote by $C_1\subset\Sp_{h_1(h_1+1)}(k)$ (resp. $C_2\subset\SO_{h_2^2}(k)$) the conjugacy class of $u_1$ (resp. $u_2$), and by $\mathcal{E}_1$ (resp. $\mathcal{E}_2$) the unique cuspidal local system supported on $C_1$ (resp $C_2$), which is necessarily $F$-stable, i.e. there is an isomorphism $\psi_1:F^{\ast}\mathcal{E}_1\isom\mathcal{E}_1$ (resp. an isomorphism $\psi_2:F^{\ast}\mathcal{E}_2\isom\mathcal{E}_2$). In the following, we write $\mathbb{Z}/2\mathbb{Z}$ in the multiplicative form.

The $\Sp_{h_1(h_1+1)}(q)$-conjugacy classes contained in $C_1^F$ are parametrised by the elements of $(\bs\mu_2)^{\bs{\kappa}(\lambda_1)}$ as in \S \ref{cent-uni}. For any $a=(e_i)_{i\in\bs{\kappa}(\lambda_1)}\in (\bs\mu_2)^{\bs{\kappa}(\lambda_1)}$, denote by $C_a$ the corresponding conjugacy class.  Following \cite[\S II.4]{W2}, we normalise $\psi_1$ in such a way that the characteristic function $\phi_1$ of $\mathcal{E}_1$ is given by  $$\phi_1(C_a)=\prod_{i\in\{1,\ldots,h_1\},~i\text{ is odd}}e_{2i}.$$

The $\SO_{h_2^2}(q)$-conjugacy classes contained in $C_1^F$ are parametrised by the elements of $(\bs\mu_2)^{\bs{\kappa}(\lambda_2)}$ satisfying (\ref{eta-constraint}) as in \S \ref{cent-uni}. For any such $a=(e_i)_{i\in\bs{\kappa}(\lambda_2)}\in(\bs\mu_2)^{\bs{\kappa}(\lambda_2)}$, denote by $C_a$ the corresponding conjugacy class. Following \cite[\S II.5]{W2}, we normalise $\psi_2$ in such a way that the characteristic function $\phi_2$ of $\mathcal{E}_2$ is given by  $$\phi_2(C_a)=\prod_{i\in\{1,\ldots,|h_2|\},~i+1\equiv h_2\!\!\!\!\mod 2}e_{2i-1}.$$  

Denote by $C$ the conjugacy class of $s\sigma_{00}u$. Let $\mathcal{E}$ be a cuspidal local system on $C$ and let $\psi:F^{\ast}\mathcal{E}\isom\mathcal{E}$ be an isomorphism. The $\GL_m(q)$-conjugacy classes contained in $C^F$ are parametrised by $(a,b)=((e_i),(e'_j))\in(\mathbb{Z}/2\mathbb{Z})^{\bs{\kappa}(\lambda_1)}\times(\mathbb{Z}/2\mathbb{Z})^{\bs{\kappa}(\lambda_2)}$, and we denote by $C_{a,b}$ the corresponding conjugacy class. We can normalise $\psi$ in such a way that the characteristic function $\phi$ of $\mathcal{E}$ is either $$\phi(C_{a,b})=\prod_{i\in\{1,\ldots,h_1\},~i\text{ is odd}}e_{2i}\prod_{j\in\{1,\ldots,|h_2|\},~j+1\equiv h_2\!\!\!\!\mod 2}e'_{2j-1},$$ or$$\phi(C_{a,b})=\prod_{i\in\{1,\ldots,h_1\},~i\text{ is odd}}e_{2i}\prod_{j\in\{1,\ldots,|h_2|\},~j\equiv h_2\!\!\!\!\mod 2}e'_{2j-1},$$
since the two elements in the fibre of (\ref{Delta-quot}) differ by $\Delta=(\eta,\ldots,\eta)$. We call these functions cuspidal functions.

\subsubsection{}
To each $(h_1,h_2)\in\mathbb{N}\times\mathbb{Z}$ is associated a unique cuspidal function $\phi(h_1,h_2)$ supportd on $C^F$, which is defined below. Put
\begin{equation*}
s(h_2)=
\begin{cases}
0, & \text{if $h_2\ge 0$},\\
1, & \text{if $h_2<0$},
\end{cases}
\end{equation*}
and put 
\begin{equation*}
\begin{split}
&\delta(h_1)=\dim C_1=\frac{|h_2^3-h_2|}{3},\\&\delta(h_2)=\dim C_2=\frac{h_1(2h_1+1)(h_1+1)}{6},\\
&\delta(h_1,h_2)=\delta(h_1)+\delta(h_2).
\end{split}
\end{equation*}

Then we define
\begin{equation}\label{phih1h2}
\phi(h_1,h_2)(C_{a,b})=q^{\delta(h_1,h_2)/2}\prod_{\substack{i\in\{1,\ldots,h_1\}\\i\text{ is odd}}}e_{2i}\prod_{\substack{j\in\{1,\ldots,|h_2|\}\\j\equiv h_1+h_2+1+s(h_2)\!\!\mod 2}}e'_{2j-1},
\end{equation}
with $(a,b)=((e_i),(e'_j))$, following \cite[\S 5]{W1}. It is a function on $\GL_m(q)\sigma_{00}$ invariant under the conjugation by $\GL_m(q)$.
\begin{Rem}
We may regard $\GL_m(q)\sigma_{00}$ either as $\GL_m(q)\rtimes\lb\sigma_{00}\rb\setminus\GL_m(q)$ or as $\GL_m(q)\rtimes\lb\sigma'_{00}\rb\setminus\GL_m(q)$ with $\sigma_{00}'$ being the orthogonal type automorphism, i.e. the above construction works for both $\leftidx{^s\!}{\bar{G}}$ and $\leftidx{^o\!}{\bar{G}}$.
\end{Rem}

\subsubsection{}\label{phi(rho)}
As in \S \ref{Lw}, we fix an isomorphism $L_{\mathbf{w}}\cong T_{w_+}\times T_{w_-}\times\GL_m(k)$ and write $\sigma=\sigma_+\times\sigma_-\times\sigma_{00}$. Let $\tC$ be the trivial character of $T_{w_+}^{F_+}$. It trivially extends to $T_{w_+}^{F_+}\lb\sigma_+\rb$, and so we can regard $\tC$ as a function on $T_{w_+}^{F_+}\sigma_+$. Let $\eta$ be the order 2 irreducible character of $\mathbb{F}_q^{\ast}$. Composing $\eta$ with the homomorphism $T^{F_-}_{w_-}\rightarrow\mathbb{F}_q^{\ast}$ defined by the product of norm maps, we can regard $\eta$ as an invariant function on $T_{w_-}^{F_-}\sigma_-$, whose value at $\sigma_-$ is equal to $1$. 

Suppose that $n$ is odd or $\bar{G}=\leftidx{^s\!}{\bar{G}}$ if $n$ is even. Then by Definition \ref{tildebox}, $1\widetilde{\boxtimes}\eta\widetilde{\boxtimes}\phi(h_1,h_2)$ is an invariant function on $L^F_{\mathbf{w}}\sigma_0$, which we will denote by $\phi_{\mathbf{w}}$, where $\mathbf{w}$ is as in \S \ref{Lw}. 

Suppose $\bar{G}=\leftidx{^o\!}{\bar{G}}_0$. The element $\sigma_0$ satisfies $\sigma_0^2=-1$ and so does each of its components: $\sigma_+$, $\sigma_-$ and $\sigma_{00}$. Now $L_{\mathbf{w}}\rtimes\lb\sigma_0\rb\ne L_{\mathbf{w}}\sqcup L_{\mathbf{w}}\sigma_0$ since $\sigma_0$ is not an order 2 element. However, $L_{\mathbf{w}}\sqcup L_{\mathbf{w}}\sigma_0$ is nevertheless a group. If we replace $T_{w_+}\rtimes\lb\sigma_+\rb$ by $T_{w_+}\sqcup T_{w_+}\sigma_+$, $T_{w_-}\rtimes\lb\sigma_-\rb$ by $T_{w_-}\sqcup T_{w_-}\sigma_-$ and $\GL_m\rtimes\lb\sigma_{00}\rb$ by $\GL_m\sqcup\GL_m\sigma_{00}$, then the arguments of Lemma \ref{LemdeRed} give an inclusion: $$L_{\mathbf{w}}\sqcup L_{\mathbf{w}}\sigma_0\subset(T_{w_+}\sqcup T_{w_+}\sigma_+)\times(T_{w_-}\sqcup T_{w_-}\sigma_-)\times(\GL_m\sqcup\GL_m\sigma_{00}).$$

Again, we can regard $\tC$ as a function on $T_{w_+}^{F_+}\sigma_+$. We can also extend $\eta$ to a character of $T_{w_-}^{F_-}\sqcup T_{w_-}^{F_-}\sigma_-$ in such a way that its value at $\sigma_-$ is equal to 1, because the value of $\eta$, regarded as a character of $T_{w_-}^{F_-}$, is always equal to $1$ at $-1\in T_{w_-}^{F_-}$. Consequently, we can define functions $\tC\widetilde{\boxtimes}\eta\widetilde{\boxtimes}\phi(h_1,h_2)$ on $L_{\mathbf{w}}^F\sigma_0\subset\leftidx{^o\!}{\bar{G}}_0$, which we will also denote by $\phi_{\mathbf{w}}$. By definition, for any $g\in\GL_m(q)$, and 
\begin{equation}\label{l=11g}
l=(1,1,g)\in T_{w_+}\times T_{w_-}\times\GL_m(q),
\end{equation} we have 
\begin{equation}\label{phiw(ls)}
\phi_{\mathbf{w}}(l\sigma_0)=\phi(h_1,h_2)(g\sigma_{00}).
\end{equation} 

Waldspurger introduced some generalised Deligne-Lusztig characters, denoted by $\phi(\mathbf{w})$ in \cite[\S 5]{W1}, which is equal to $R^{L_0\sigma_0}_{L_{\mathbf{w}}\sigma_0}\phi'_{\mathbf{w}}$ for some function $\phi'_{\mathbf{w}}$ on $L_{\mathbf{w}}^F\sigma_0$. (The function $\phi'_{\mathbf{w}}$ does not appear explicitly in \cite{W1}.) The two functions $\phi_{\mathbf{w}}$ and $\phi'_{\mathbf{w}}$ are extensions of the same character of $L_{\mathbf{w}}^F$ to $L_{\mathbf{w}}^F\sigma_0$, and so may differ by a sign. Besides, Waldspurger's $\phi'_{\mathbf{w}}$ may not satisfy (\ref{phiw(ls)}). We need to determine exactly when these two functions differ by a sign. It suffices to compare them at an element of $L_{\mathbf{w}}^F\sigma_0$ where both of them are non zero.

\subsubsection{}
In \cite[\S 3]{W1}, using linear algebra, Waldspurger defined an element $\theta_L\in L_{\mathbf{w}}^F\sigma_0$, satisfying:
\begin{enumerate}
\item[(Wi)] \label{Wi} $\theta_L^2=1$; (See \cite[\S 3]{W1}.)
\item[(Wii)] \label{Wii} $\theta_Ll\theta_L^{-1}=(\sigma_+(t_+),\sigma_-(t_-),\bs\theta_m(g))$, for any $l=(t_+,t_-,g)\in T_{w_+}\times T_{w_-}\times\GL_m(k)$, where $\bs\theta_m$ is a given automorphism of $\GL_m$ of orthogonal type; (See \cite[\S 3]{W1}.)
\item[(Wiii)] \label{Wiii} $\eta\circ\det(\theta_L)=\epsilon$. (See the proof of \cite[\S4 Lemme]{W1}.)
\end{enumerate}
Condition (Wiii) deserves some comments. First of all, our $\det$ (see (\ref{epsilon})) satisfies $\det(\sigma'_0)=1$ while Waldspurger's $\mathbf{det}$ satisfies $\mathbf{det}(\tau)=1$, where $\tau$ is the transpose-inverse automorphism. So they differ by $\det(\mathscr{J}_{n_0})$ (See \S \ref{autostand}), whence the absence of the term $\eta(-1)^d$ as in the proof of \cite[\S 4 Lemme]{W1}. Secondly, in \cite[\S 4 Lemme]{W1}, Waldspurger assumed that only one of $w_+$ and $w_-$ is non trivial and its conjugacy class can only be $((d),\varnothing)$, or $(\varnothing,(d))$. The sign $\epsilon$ is equal to $-1$ precisely when the non trivial 2-partition is $(\varnothing,(d))$. In general, $\epsilon$ should be replaced by $\sgn((w_+,w_-))$, where $\sgn$ is the character $$\sgn:\mathfrak{W}_+\times\mathfrak{W}_-\subset \mathfrak{W}_{N_++N_-}\longrightarrow\{\pm1\}.$$ (See \S \ref{W^D}.) If not for (iii), we could have taken $\theta_L=\sigma_0$ when $n$ is odd.

\begin{Lem}\label{thetaL}
Suppose that $n$ is odd. If $\sgn((w_+,w_-))=1$, then $\theta_L$ can be taken to be $\sigma_0$; Otherwise, $\theta_L$ can be taken to be $l_0\sigma_0$, with $$l_0=(1,1,\lambda)\in T_{w_+}\times T_{w_-}\times\GL_m(k),$$where $\lambda\in\mathbb{F}_q^{\ast}\setminus(\mathbb{F}_q^{\ast})^2$ and is regarded as a scalar matrix in $\GL_m$.
\end{Lem}
\begin{proof}
Obviously the $\theta_L$ thus defined satisfies conditions (Wi), (Wii) and (Wiii) above.
\end{proof}

\begin{Lem}\label{thetaL-even}
Suppose that $n$ is even and $\bar{G}=\leftidx{^o\!}{\bar{G}}_0$, and that only one of $w_+$ and $w_-$ is non trivial. Denote by $w$ the non trivial one among them. Suppose that $w$ lies in the conjugacy class corresponding to the 2-partition $((d),\varnothing)$, or $(\varnothing,(d))$. We may assume that $w=((1,\ldots,1,1),(12\cdots d))$ or $w=((1,\ldots,1,-1),(12\cdots d))$ using the notation of \S \ref{couronne}, where $(12\cdots d)$ is the circular permutation. Denote by $\dot{w}\in\GL_n$ a $\sigma_0$-stable element normalising $L_0\cong T\times\GL_m$, which acts on $T$ according to this permutation, and let $x\in\GL_n$ be a $\sigma_0$-stable element satisfying $x^{-1}F(x)=\dot{w}$. With these notations, if $\sgn(w)=1$, then $\theta_L$ can be taken to be $x\sigma'_0x^{-1}$; Otherwise, $\theta_L$ can be taken to be $xl_0\sigma'_0x^{-1}$, where $$l_0=(t,\Id)\in T\times\GL_m(k),$$and $t=(a_1,\ldots,a_{d-1},a_d,a_d,a_{d-1},\ldots a_1)\in (k^{\ast})^{2d}$, with $a_i^q=a_{i+1}$ for any $1\le i\le d-1$ and $a_d^q=-a_1$.
\end{Lem}
\begin{proof}
Note that in the group $\leftidx{^o\!}{\bar{G}}_0$ we have $\sigma_0^2=-\Id$ and $\sigma_0^{\prime 2}=\Id$. Also note that $xL_0x^{-1}=L_{\mathbf{w}}$. Obviously, if $\sgn(w)=1$, then $x\sigma'_0x^{-1}$ satisfies all three conditions imposed on $\theta_L$.

Now suppose $\sgn(w)=-1$. It is easy to see that $l_0\sigma'_0$ is an order 2 element, and that $l_0$ does not affect the action of $\sigma'_0$ on $L_0$. These properties are preserved under conjugation by $x$. Note that $\sigma'_0=t_0\sigma_0$, where $t_0$ is as in \S \ref{autostand} with $n$ replaced by $n_0$. That $xl_0\sigma'_0x^{-1}$ is an $F$-stable element: $$F(xl_0\sigma'_0x^{-1})=xl_0\sigma'_0x^{-1},$$is equivalent to $\dot{w}F(l_0)\sigma'_0(\dot{w}^{-1})=l_0$. Note that $\dot{w}$ is not $\sigma'_0$-stable. Write $l_0=(t,\Id)$ and $t=(a_1,\ldots,a_{d-1},a_d,a_d,a_{d-1},\ldots a_1)\in (k^{\ast})^{2d}$, we obtain the equations in the $a_i$'s.

We compute $\eta\circ\det(xl_0\sigma'_0x^{-1})$. Since $\sigma_0(x)=x$, we have $\sigma'_0(x)=t_0xt_0^{-1}$, and so $xl_0\sigma'_0x^{-1}=xl_0t_0x^{-1}t_0^{-1}\sigma'_0$. By definition, $\eta\circ\det(xl_0t_0x^{-1}t_0^{-1}\sigma'_0)=\eta\circ\det(l_0)$. We have $$\det(l_0)=a_1^2a_2^2\cdots a_d^2=(a_1^{1+q+\cdots+q^{d-1}})^2.$$ Using the equations in the $a_i$'s, we can see that this element lies in $\mathbb{F}_q$ while $a_1^{1+q+\cdots+q^{d-1}}$ is not an element of $\mathbb{F}_q$. Therefore $\eta\circ\det(l_0)=-1$.
\end{proof}
\begin{Rem}\label{thetaL-even-Rem}
In the general circumstance, $(w_+,w_-)$ is a product of signed cycles as in \S \ref{couronne} and \S \ref{wreath}, and we may apply the same arguments to each signed cycle $w$. We may write $L_0=T\times\GL_m$ and $T\cong(k^{\ast})^{2N}$ with $2N+m=n_0$. Let $t=(a_1,\ldots,a_N,a_N,\ldots,a_1)\in T$ and $l_0=(t,\Id)\in L_0$. If $\{i_1,\ldots,i_d\}$ is a subset of $\{1,\ldots,N\}$ such that there is a signed cycle $w$ circularly permuting its elements, then we may define $(a_{i_j})_{1\le j\le d}$ as in the lemma. If $x$ is a $\sigma_0$-stable element of $\GL_m$ such that $xL_0x^{-1}=L_{\mathbf{w}}$, then $xl_0\sigma'_0x^{-1}$ is the desired $\theta_L$.
\end{Rem}

\begin{Rem}
Other choices of $\theta_L$, including Waldspurger's, are essentially the same. Let $l'_0=(t',g'_0)\in L_0$ define an $F$-stable $\theta'_L$ satisfying (Wi), (Wii) and (Wiii), in the way $l_0$ defines $\theta_L$. In order for $\theta'_L$ to satisfy (Wii), $g'_0$ must be a scalar matrix (so that it induces trivial automorphism). When $n$ is even, an $F$-stable scalar matrix does not affect (Wiii) and so we may simply take $g'_0=\Id$ in Lemma \ref{thetaL-even}. When $n$ is odd, $g'_0$ is necessarily of the form $\lambda$ in Lemma \ref{thetaL} in order for (Wiii) to be true. To determine $t'$, let us suppose that $w$ is as in Lemma \ref{thetaL-even}. In order for $\theta'_L$ to satisfy (Wi), $t'$ is necessarily of the form $$t'=(b_1,\ldots,b_{d-1},b_d,b_d,b_{d-1},\ldots b_1)\in (k^{\ast})^{2d}.$$And in order for $\theta'_L$ to be $F$-stable, the $b_i$'s must satisfy the equations
\begingroup
\allowdisplaybreaks
\begin{align*}
&\text{$b_i^q=b_{i+1}$ for any $1\le i\le d-1$ and $b_d^q=-b_1$, if $n$ is even,}\\
&\text{$b_i^q=b_{i+1}$ for any $1\le i\le d-1$ and $b_d^q=b_1$, if $n$ is odd.}
\end{align*}
\endgroup
(If $n$ is odd, then $\sigma'_0$ should be replaced by $\sigma_0$ and the $\dot{w}$ in the proof of Lemma \ref{thetaL-even} is $\sigma_0$-stable.) We deduce that the norm of $t't^{-1}$ must lie in $(\mathbb{F}_q^{\ast})^2$, where $t$ is as in Lemma \ref{thetaL} or Remark \ref{thetaL-even-Rem} depending on the parity of $n$. Because of this, the term $t't^{-1}$ will play no role in what follows.
\end{Rem}

\subsubsection{}
In \cite[\S 3]{W1}, Waldspurger constructed a representation $(\rho'_{\chi},E'_{\chi})$ of the group $L_{\mathbf{w}}^F\rtimes\lb\sigma_0\rb$. Then the function $\phi'_{\mathbf{w}}$ mentioned above is the character of this representation restricted to $L_{\mathbf{w}}^F\sigma_0$. By the construction of $\rho'_{\chi}$, for any $l$ as in (\ref{l=11g}), we have
\begin{equation}\label{phi'w(ls)}
\phi'_{\mathbf{w}}(l\theta_L)=\phi(h_1,h_2)(g\sigma_{00}),\text{ if $n$ is odd},
\end{equation} 
and
\begin{equation}\label{phi'w(ls)even}
\phi'_{\mathbf{w}}(l\theta_L)=\phi(h_1,h_2)(g\sigma'_{00}),\text{ if $n$ is even}.
\end{equation} 

\begin{Prop}\label{Choice-h2}
Suppose that $n$ is odd. For any $(h_1,|h_2|)\in\mathbb{N}^2$, there is precisely one value of $h_2$ such that for any $(w_+,w_-)\in\mathfrak{W}_+\times\mathfrak{W}_-$, we have $\phi_{\mathbf{w}}=\phi'_{\mathbf{w}}$ as functions on $L^F_{\mathbf{w}}\sigma_0$. The choice of such $h_2$ is indicated in the table below. If $h_2$ is the negative of these values, then $\phi_{\mathbf{w}}=\sgn((w_+,w_-))\phi'_{\mathbf{w}}$. 
\begin{center}
\begin{tabular}{ |p{3cm}|p{3cm}|p{3cm}| }
\hline
 & $|h_2|\equiv1\mod4$ & $|h_2|\equiv3\mod4$ \\
\hline
$h_1/2\text{ is even}$ & $h_2>0$ & $h_2<0$\\
\hline
$h_1/2\text{ is odd}$ &$h_2<0$ &$h_2>0$\\
\hline
$h_1\equiv1\mod4$ &$h_2>0$ &$h_2<0$\\
\hline
$h_1\equiv3\mod4$ &$h_2<0$ &$h_2>0$\\
\hline
\end{tabular}
\end{center}
\end{Prop}
\begin{proof}
Note that $h_2$ must be odd, since $n$ is odd. We first consider $(w_+,w_-)$ with $\sgn((w_+,w_-))=1$. By Lemma \ref{thetaL}, we have $\theta_L=\sigma_0$. By (\ref{phiw(ls)}) and (\ref{phi'w(ls)}), we have $\phi_{\mathbf{w}}=\phi'_{\mathbf{w}}$, regardless of the value of $h_2$.

Now suppose that $\sgn((w_+,w_-))=-1$ and so $\theta_L=l_0\sigma_0$ as in Lemma \ref{thetaL}. To compare $\phi_{\mathbf{w}}$ and $\phi'_{\mathbf{w}}$ using (\ref{phiw(ls)}) and (\ref{phi'w(ls)}), we first fix an element $l\in L_{\mathbf{w}}^F$ such that they have non zero value. Suppose that $l$ is of the form (\ref{l=11g}). Let $s\in\GL_m(q)$ be such that $s\sigma_{00}$ is semi-simple in $\GL_m\sigma_{00}$, and that $C_{\GL_m}(s\sigma_{00})^{\circ}\cong\Sp_{h_1(h_1+1)}\times\SO_{h_2^2}$. Let $u=(u_1,u_2)\in C_{\GL_m}(s\sigma_{00})^{\circ F}$ be a unipotent element as in \S \ref{cuspGL_m}. Put $g=us$, then $\phi(h_1,h_2)(g\sigma_{00})\ne 0$. We have $\phi_{\mathbf{w}}(l\theta_L)=\phi(h_1,h_2)(\lambda g\sigma_{00})$, where $l$ is as in (\ref{phiw(ls)}) and $\lambda$ is as in Lemma \ref{thetaL}. It remains to compare $\phi(h_1,h_2)(g\sigma_{00})$ and $\phi(h_1,h_2)(\lambda g\sigma_{00})$.

Note that $\lambda g\sigma_{00}$ lies in the $\GL_m(k)$-conjugacy class of $g\sigma_{00}$. We want to identify the $\GL_m(q)$-conjugacy class of $\lambda g\sigma_{00}$ relative to $g\sigma_{00}$. For any $x\in k^{\ast}$ such that $x^2=\lambda$, we have $xg\sigma_{00}x^{-1}=\lambda g\sigma_{00}$, where $x$ is regarded as a scalar matrix. Then $x^{-1}F(x)=x^{q-1}=-\Id\in\GL_m$. On the other hand, we know that $x^{-1}F(x)$ must lie in $C_{\GL_m}(us\sigma_{00})$, whose component group is isomorphic to $(\mathbb{Z}/2\mathbb{Z})^{\bs{\kappa}(\lambda_1)}\times(\mathbb{Z}/2\mathbb{Z})^{\bs{\kappa}(\lambda_2)}$ according to \S \ref{cuspGL_m}. We need to locate $-\Id$ in this group.

The $m$-dimensional $k$-vector space $V$ on which $\GL_m$ acts can be decomposed into $V_1\oplus V_2$ in such a way that $\Sp_{h_1(h_1+1)}\subset\GL(V_1)$ and $\SO_{h_2^2}\subset\GL(V_2)$. According to \S \ref{cent-uni}, for $i=1$, $2$, we have the following decomposition $$V_i\cong (\oplus_dU_{i,d}\otimes W_{i,d})\bigoplus(\oplus_{\delta}U'_{i,\delta}\otimes W'_{i,\delta}).$$ The matrix $-\Id_{V_i}$ can be identified as acting as $-\Id$ on each $U_{i,d}$ and $U'_{i,\delta}$, and as $\Id$ on each $W_{i,d}$ and $W'_{i,\delta}$. We can thus regard $-\Id_{V_i}$ as an element of the reductive group $R$ in \S \ref{cent-uni}, whose orthogonal components are $-\Id_{U_{i,d}}$. Now $-\Id_{U_{i,d}}$ has positive determinant if $U_{i,d}$ has even dimension, and negative determinant otherwise. The location of $-\Id$ in the component group $(\mathbb{Z}/2\mathbb{Z})^{\bs{\kappa}(\lambda_1)}\times(\mathbb{Z}/2\mathbb{Z})^{\bs{\kappa}(\lambda_2)}$ is determined by these determinants. By \S \ref{cuspGL_m}, the uniponent element $u_1$ (resp. $u_2$) corresponds to the symplectic partition $\lambda_1:=(2h_1,2h_1-2,\ldots,2)$ (resp. orthogonal partition $\lambda_2:=(2|h_2|-1,2|h_2|-3,\ldots,1)$). In these partitions, each part has multiplicity one, and so by the construction of \S \ref{cent-uni}, the vector spaces $U_{i,d}$ are one-dimensional. Therefore, $-\Id_{V_i}$ gives an element of $(\mathbb{Z}/2\mathbb{Z})^{\bs{\kappa}(\lambda_i)}$ all of whose components are non trivial.

We deduce that if $\phi(h_1,h_2)(g\sigma_{00})$ is given by (\ref{phih1h2}), then replacing $g\sigma_{00}$ by $\lambda g\sigma_{00}$ amounts to replacing every $e_{2i}$ and $e'_{2j-1}$ by its negative. The effect on the value of $\phi(h_1,h_2)$ is determined by the parity of the number of factors in the products over $i$ and $j$ in (\ref{phih1h2}). Suppose that $h_1/2$ is an even integer and $|h_2|\equiv1\mod4$. The product over $i$ has an even number of terms, and so replacing $e_{2i}$ by $-e_{2i}$ does not affect the value of $\phi(h_1,h_2)$. If $h_2>0$, then $h_1+h_2+1+s(h_2)$ is even. Since $|h_2|\equiv1\mod4$, the number of even numbers in $\{1,\ldots,|h_2|\}$ is even. We conclude that in this case $\phi_{\mathbf{w}}=\phi'_{\mathbf{w}}$. If $h_2<0$, then $h_1+h_2+1+s(h_2)$ is odd. The number of odd numbers in $\{1,\ldots,|h_2|\}$ is odd, and so $\phi_{\mathbf{w}}=-\phi'_{\mathbf{w}}$. We have thus proved the item in the first row and the first column in the table. Other items are completely similar.
\end{proof}

\begin{Prop}\label{phi'weven}
Suppose that $n$ is even and $\bar{G}=\leftidx{^o\!}{\bar{G}}_0$. Then for any $\mathbf{w}=(h_1,h_2,w_+,w_-)$, we have $\phi_{\mathbf{w}}=\sgn(w_-)\phi'_{\mathbf{w}}$ as functions on $L^F_{\mathbf{w}}\sigma_0$.
\end{Prop}
\begin{proof}
Let $g\in\GL_m(q)$ be such that $\phi(h_1,h_2)(g\sigma'_{00})\ne0$, let $l$ be of the form (\ref{l=11g}) and let $l_0$ be as defined in Remark \ref{thetaL-even-Rem}. Write $l_0=(t_+,t_-,\Id)\in T_{w_+}\times T_{w_-}\times\GL_m$. We have $$\phi_{\mathbf{w}}(l\theta_L)=\phi_{\mathbf{w}}(lxl_0x^{-1}x\sigma'_0x^{-1})=\eta(t_-)\phi(h_1,h_2)(g\sigma'_{00}),$$and $\eta(t_-)=\sgn(w_-)$ by the construction of Lemma \ref{thetaL-even} and Remark \ref{thetaL-even-Rem}. The proposition follows from (\ref{phi'w(ls)even}).
\end{proof}

\subsubsection{}\label{h_1h_2}
We keep the notations as above and assume that $p\ne 2$ and $q>n_0$. Let $\chi_{(\mu_+,\mu_-)}$ be a quadratic-unipotent character, which extends to a character $\tilde{\chi}_{(\mu_+,\mu_-)}\in\Irr(L_0^{F_0}\lb\sigma_0\rb)$. 

Let $(\mu_+,\mu_-)$ be a $2$-partition of $n_0$ and write $n_+=|\mu_+|$ and $n_-=|\mu_-|$. Let $m_+$ and $m_-$ be some non negative integers such that $(m_+,m_+-1,\ldots,1,0)$ and $(m_-,m_--1,\ldots,1,0)$ are the 2-cores of $\mu_+$ and of $\mu_-$ respectively, and write $N_{\pm}=(n_{\pm}-m_{\pm}(m_{\pm}+1)/2)/2$. There exists a unique pair $(h_1,h_2)\in\mathbb{N}\times\mathbb{Z}$ such that 
\begin{equation}
\begin{split}
m_+&=\sup\{h_1+h_2,-h_1-h_2-1\},\\
m_-&=\sup\{h_1-h_2,h_2-h_1-1\}.
\end{split}
\end{equation}
Note that exchanging $\mu_+$ and $\mu_-$ changes $(h_1,h_2)$ into $(h_1,-h_2)$. Write $$m=m_+(m_++1)/2+m_-(m_-+1)/2,$$ and so $n_0=m+2N_++2N_-$. We have $m=h_1(h_1+1)+h_2^2$. Fix some integers $r_+>l(\mu_+)$ and $r_->l(\mu_-)$ satisfying:
\begin{quote}
\textbf{Assumption.} 
\begin{itemize}
\item $r_-\equiv h_2\mod 2$, $r_+\equiv h_2+1\mod 2$, if $n$ is odd and $(h_1,h_2)$ satisfies the table in Proposition \ref{Choice-h2};\\
\item $r_-\equiv h_2+1\mod 2$, $r_+\equiv h_2\mod 2$, if $n$ is odd and $(h_1,h_2)$ does not satisfy the table in Proposition \ref{Choice-h2};\\
\item $r_+\equiv r_-\equiv h_2+1\mod 2$, if $n$ is even.
\end{itemize}
\end{quote}
Let $(\alpha_+,\beta_+)_{r_+}$ and $(\alpha_-,\beta_-)_{r_-}$ be the $2$-partitions associated to $\mu_+$ and to $\mu_-$ respectively (See \S \ref{quotcore}), such that the unordered 2-partitions $(\alpha_+,\beta_+)$ and $(\alpha_-,\beta_-)$ are the corresponding 2-quotients. Each of the two $2$-partitions $(\alpha_{\pm},\beta_{\pm})_{r_{\pm}}$ determines an irreducible character of $\mathfrak{W}_{\pm}$ respectively, denoted by $\varphi_+$ and $\varphi_-$. Then with the fixed $r_+$ and $r_-$, the $2$-partitions $(\mu^+,\mu^-)$ are in bijection with the data $(h_1,h_2,\varphi_+,\varphi_-)$.
\begin{Rem}
In the case of $h_1=0$ and $|h_2|\le1$, the assumptions read:
\begin{itemize}
\item $r_+$ and $r_-$ are odd, if $h_2=0$;\\
\item $r_+$ is even if and only if $n_+$ is odd, and $r_-$ is even if and only if $n_-$ is odd, if $|h_2|=1$.
\end{itemize}
We only need to notice that for $(h_1,h_2)\in\{0\}\times\{\pm1\}$, the pair $(h_1,h_2)$ satisfies the table in Proposition \ref{Choice-h2} if and only if $n_+$ is odd.
\end{Rem}

For each $(h_1,h_2,\varphi_+,\varphi_-)$, we have the invariant function on $L_0^{F_0}\sigma_0$ defined by
\begin{equation}
R^{L_0\sigma_0}_{\varphi}:=\frac{1}{|\mathfrak{W}_+|}\frac{1}{|\mathfrak{W}_-|}\sum_{\substack{w_+\in \mathfrak{W}_+\\ w_-\in \mathfrak{W}_-}}\varphi_+(w_+)\varphi_-(w_-)R^{L_0\sigma_0}_{L_{\mathbf{w}}\sigma_0}\phi_{\mathbf{w}}.
\end{equation}
This is the characteristic function of character sheaf in Example \ref{quad-uni-sheaf}. Again, this definition makes sense in $\leftidx{^s\!}{\bar{G}}_0$ and in $\leftidx{^o\!}{\bar{G}}_0$.
\begin{Lem}
The assumption on $(r_+,r_-)$ above is equivalent to the assumption at the end of \cite[\S 15]{W1}, in the sense that they define the same function $R^{L_0\sigma_0}_{\varphi}$ (denoted by $\phi(\bs\rho)$ in \cite{W1}).
\end{Lem}
\begin{proof}
Suppose that $n$ is odd. By Proposition \ref{Choice-h2}, if $(h_1,h_2)$ satisfies the table therein, then $\phi_{\mathbf{w}}=\phi'_{\mathbf{w}}$ for any $(w_+,w_-)$, and so $R^{L_0\sigma_0}_{\varphi}$ coincides with the $\phi(\bs\rho)$ in \cite{W1}. Suppose that $(h_1,h_2)$ does not satisfy the table in Proposition \ref{Choice-h2}, then $\phi_{\mathbf{w}}=\sgn((w_+,w_-))\phi'_{\mathbf{w}}$. Note that $\sgn((w_+,w_-))=\sgn(w_+)\sgn(w_-)$, and that multiplying $\varphi_{\pm}$ by $\sgn$ amounts to switching $\alpha_{\pm}$ with $\beta_{\pm}$ at the level of 2-partitions. Recall that changing the parity of $r_{\pm}$ results in permuting $\alpha_{\pm}$ with $\beta_{\pm}$. (See \S \ref{quotcore}.) Our assumption on $r_{\pm}$ reflects this change of parity.

Suppose that $n$ is even. By Proposition \ref{phi'weven}, we have an extra term $\sgn(w_-)$ when replacing $\phi'_{\mathbf{w}}$ by $\phi_{\mathbf{w}}$. Therefore the parity of $r_-$ is different from Waldspurger's.
\end{proof}

\begin{Eg}
We give some examples to show how $R^{L_0\sigma_0}_{\varphi}$ is affected by the parity of $r_+$ and $r_-$. Recall \S \ref{quotcore} that $r=r_{\pm}$ controls the 2-quotient of a given partition $\lambda$. If $\lambda=(2)$ and $r$ is odd, then the 2-quotient is $((1),\varnothing)$. If $\lambda=(3)$ and $r$ is even, then the 2-quotient is $((1),\varnothing)$. In terms of representation theory, these are the correct 2-quotients, in the sense that if $\lambda$ represents a trivial character then its 2-quotient represents a trivial character. Changing the parity of $r$ results in $(\varnothing,(1))$.

Suppose that $n=2$. In the following table, we consider four quadratic-unipotent characters, we give the corresponding 2-partition $(\mu_+,\mu_-)$, the value of $(h_1,h_2)$, and the correct $\varphi_+$ and $\varphi_-$ in terms of 2-partitions $(\alpha_+,\beta_+)$ and $(\alpha_-,\beta_-)$.
\begin{center}
\begin{tabular}{ |p{2.5cm}|p{2.5cm}|p{2.5cm}|p{2.5cm}| }
\hline
$(\mu_+,\mu_-)$ & $(h_1,h_2)$ & $(\alpha_+,\beta_+)$ & $(\alpha_-,\beta_-)$\\
\hline
$((2),\varnothing)$ & $(0,0)$ & $((1),\varnothing)$ & $(\varnothing,\varnothing)$ \\
\hline
$(\varnothing,(2))$ & $(0,0)$ & $(\varnothing,\varnothing)$ & $((1),\varnothing)$ \\
\hline
$((1^2),\varnothing)$ & $(0,0)$ & $(\varnothing,(1))$ & $(\varnothing,\varnothing)$ \\
\hline
$(\varnothing,(1^2))$ & $(0,0)$ & $(\varnothing,\varnothing)$ & $(\varnothing,(1))$ \\
\hline
\end{tabular}
\end{center}
\end{Eg}
In the first row, we have $(\alpha_+,\beta_+)=((1),\varnothing)$. This means that $\varphi_+$ is the trivial character of $\mathfrak{S}_2$, and so the extension of the identity character to $\GL_2\sigma$ is the sum of two Deligne-Lusztig characters. In the third row, we have $(\alpha_+,\beta_+)=(\varnothing,(1))$. This means that $\varphi_+$ is the sign character of $\mathfrak{S}_2$, and so the extension of the Steinberg character to $\GL_2\sigma$ is the difference of two Deligne-Lusztig characters.

If we follow the assumption at the end of \cite[\S 15]{W1}, then $r_+$ would be odd and $r_-$ even. Following \S \ref{quotcore}, the 2-partition $(\mu_+,\mu_-)=((2),\varnothing)$ would produce $(\alpha_+,\beta_+)_{r_+}=((1),\varnothing)$ as expected. However, the 2-partition $(\mu_+,\mu_-)=(\varnothing,(2))$ would produce $(\alpha_-,\beta_-)_{r_-}=(\varnothing,(1))$, which is wrong. This is precisely because our $R^{\GL_2\sigma}_{T_w\sigma}\eta$ differs from Waldspurger's by a sign when $w\ne 1$.

Suppose that $n=3$. We consider the following four quadratic-unipotent characters.
\begin{center}
\begin{tabular}{ |p{2.5cm}|p{2.5cm}|p{2.5cm}|p{2.5cm}| }
\hline
$(\mu_+,\mu_-)$ & $(h_1,h_2)$ & $(\alpha_+,\beta_+)$ & $(\alpha_-,\beta_-)$\\
\hline
$((3),\varnothing)$ & $(0,1)$ & $((1),\varnothing)$ & $(\varnothing,\varnothing)$ \\
\hline
$(\varnothing,(3))$ & $(0,-1)$ & $(\varnothing,\varnothing)$ & $((1),\varnothing)$ \\
\hline
$((1),(2))$ & $(0,1)$ & $(\varnothing,\varnothing)$ & $((1),\varnothing)$ \\
\hline
$((2),(1))$ & $(0,-1)$ & $((1),\varnothing)$ & $(\varnothing,\varnothing)$ \\
\hline
\end{tabular}
\end{center}
For example, the third row corresponds to the character induced from the character $(\tC,\eta\circ\det)$ of $\mathbb{F}_q^{\ast}\times\GL_2(q)$, and we have $(\alpha_-,\beta_-)=((1),\varnothing)$. This means that the extension of this character to $\GL_3\sigma$ is the sum of two Deligne-Lusztig characters.

Following the assumption at the end of \cite[\S 15]{W1}, we have that $r_+$ is even and $r_-$ odd. Therefore the 2-partition $(\mu_+,\mu_-)=((1),(2))$ produces $(\alpha_-,\beta_-)_{r_-}=((1),\varnothing)$ as expected. However, the 2-partition $(\mu_+,\mu_-)=((2),(1))$ produces $(\alpha_+,\beta_+)_{r_+}=(\varnothing,(1))$, which is wrong. Similarly, $(\alpha_+,\beta_+)_{r_+}$ is correct for $(\mu_+,\mu_-)=((3),\varnothing)$ and $(\alpha_-,\beta_-)_{r_-}$ is wrong for $(\mu_+,\mu_-)=(\varnothing,(3))$. Again, this is because some of our Deligne-Lusztig characters differ from Waldspurger's by a sign. From another point of view, for $(\mu_+,\mu_-)=((3),\varnothing)$ and $((1),(2))$, the pair $(h_1,h_2)$ satisfies the table in Proposition \ref{Choice-h2}, whereas for the other two values of $(\mu_+,\mu_-)$, the pair $(h_1,h_2)$ does not. Therefore in the latter case we must replace Waldspurger's assumption on $(r_+,r_-)$ by ours.

\begin{Thm}(\cite[\S 17]{W1})\label{ExtL_0}
Suppose that $n$ is odd or $L_0^{F_0}\lb\sigma_0\rb=\leftidx{^o\!}{\bar{G}}_0$ if $n_0$ is even. Then for any $(\mu_+,\mu_-)\in\mathcal{P}_{n_0}(2)$, we have,
\begin{equation}\label{L0S0eq1}
\tilde{\chi}_{(\mu_+,\mu_-)}|_{L_0^{F_0}\sigma_0}=\pm R^{L_0\sigma_0}_{\varphi}.
\end{equation}
\end{Thm}
Given a quadratic-unipotent character $\chi$ of $\GL_{n_0}(q)$, let $\rho:\GL_{n_0}(q)\rightarrow\GL(V)$ be a representation that realises it. Then $\rho(-1)=\pm\Id_V$, with $\rho(-1)=-\Id_V$ exactly when $\chi(-1)=-\chi(1)$. Define the indicator
$$
\bs\gamma_{\chi}=
\begin{cases}
\mathfrak{i} & \text{if $\chi(-1)=-\chi(1)$;} \\
1 & otherwise.
\end{cases}
$$
\begin{Cor}
Suppose that $n_0$ is even and $L_0^{F_0}\lb\sigma_0\rb=\leftidx{^s\!}{\bar{G}}_0$. Then for any $(\mu_+,\mu_-)\in\mathcal{P}_{n_0}(2)$, we have,
\begin{equation}\label{L0S0eq2}
\tilde{\chi}_{(\mu_+,\mu_-)}|_{L_0^{F_0}\sigma_0}=\pm \bs{\gamma}_{\chi}R^{L_0\sigma_0}_{\varphi}.
\end{equation}
\end{Cor}
\begin{proof}
We compare the two sides of (\ref{L0S0eq1}) and (\ref{L0S0eq2}) under the bijection of conjugacy classes described in \S \ref{GsGo}. According to \S \ref{GsGo}, the left hand sides of the two equations differ by $\bs{\gamma}_{\chi}$. As for the right hand sides, one sees from Proposition \ref{2.10} that the induction $R^{L_0\sigma_0}_{L_{\mathbf{w}}\sigma_0}\phi_{\mathbf{w}}$ does not depend on whether we work with $\leftidx{^s\!}{\bar{G}}_0$ or with $\leftidx{^o\!}{\bar{G}}_0$.
\end{proof}

%%%%%%%%%%%%%%%%%%%%%
\section{Computation of the Character Formula}\label{ComputeA}
In order to determine the character table of $\GL_n(q)\lb\sigma\rb$, we need to explicitly calculate the induced functions of the form $R^{G\sigma}_{M\sigma}\phi_M$ for some invariant function $\phi_M$ defined on some $\sigma$-stable and $F$-stable Levi factor $M$ of some $\sigma$-stable parabolic subgroup.
\subsection{Connected Groups}\label{splitconn}
\subsubsection{}\label{splitconn}
Let us recall how this is done for a split connected group $G$. Fix a split maximal torus $T_0$ and let $W_G$ denote the Weyl group defined by $T_0$. Let us  simplify the situation by assuming that $M=T_{\tau}$, $\tau\in W_G$, is a maximal torus. The character formula (\textit{cf.} \S \ref{formucarconn}) reads
\begin{equation}
R^{G}_{T_{\tau}}\theta(g)=|T_{\tau}^F|^{-1}|C_G(s)^{\circ F}|^{-1}\sum_{\{h\in G^F|s\in \leftidx{^h}{T_{\tau}}{} \}}Q^{C^{\circ}_G(s)}_{C_{\prescript{h}{}T_{\tau}}^{\circ}(s)}(u)\prescript{h}{}\theta(s),
\end{equation}
for the Jordan decomposition $g=su$. Assume that $s\in T^F_{\tau}$, and put
\begin{equation}
A(s,\tau):=\{h\in G\mid hsh^{-1}\in T_{\tau}\}.
\end{equation}
We will have to determine the set
\begin{equation}
A^F(s,\tau):=\{h\in A(s,\tau)\mid F(h)=h\}.
\end{equation}

Write $L=C_G(s)^{\circ}$. Define 
\begin{equation}
\begin{split}
B(s,\tau):=&\{\text{The $L^F$-conjugacy classes of the $F$-stable maximal tori of $L$}\\ &\text{that are $G^F$-conjugate to $T_{\tau}$.}\}
\end{split}
\end{equation}
We fix $s$ and $\tau$ and write $A$, $A^F$ and $B$ in what follows. There is a surjective map $A^F\rightarrow B$ which sends $h$ to the class of $h^{-1}T_{\tau}h$. It factors through $\iota:A^F/L^F\rightarrow B$. The Green function associated to $h$ in fact only depends on $\iota(h)$ while $\prescript{h}{}\theta(s)$ is constant on each right $L^F$-coset of $A^F$. The calculation is eventually reduced to evaluating $\prescript{h}{}\theta(s)$ on the fiber of $\iota$ over an element $\bar{\nu}\in B$. We may regard $\bar{\nu}$ as the $F$-conjugacy class of some $\nu\in W_L(T_{\tau})$.

\subsubsection{}
We have $A=N_G(T_{\tau}).L$, since for each $h$, there exists $l\in L$ such that $h^{-1}T_{\tau}h=lT_{\tau}l^{-1}$. We deduce from it an isomorphism
\begin{equation}
A^F/L^F\cong (A/L)^F\cong (N_G(T_{\tau})/N_L(T_{\tau}))^F\cong (W_G(T_{\tau})/W_L(T_{\tau}))^F,
\end{equation}
which sends $h=nl$ to the class of $n$. This does not depend on the choice of the $n$ and $l$ such that $h=nl$. We choose some $g\in G$  such that $T_{\tau}=gT_0g^{-1}$, and put $L_0=g^{-1}Lg$. We can further identify the above set to $(W_G(T_0)/W_{L_0}(T_0))^{\tau}$. Write $W_{L_0}=W_{L_0}(T_0)$. The conjugation by $\tau$ preserves $L_0$ since $L$ is $F$-stable. Now, a coset $wW_{L_0}$ is $\tau$-stable if and only if 
\begin{equation}
w^{-1}\tau w\tau^{-1}\in W_{L_0}
\end{equation}
and $\iota(wW_{L_0})\in\bar{\nu}$ if and only if 
\begin{equation}
w^{-1}\tau w\tau^{-1}\in\bar{\nu},
\end{equation}
regarding $\bar{\nu}$ as a $\tau$-conjugacy class of $W_{L_0}(T_0)$.
Indeed, if $hL^F$ corresponds to $wW_L$, then the $L^F$-conjugacy class of $h^{-1}T_{\tau}h=l^{-1}T_{\tau}l$ is represented by $lF(l)^{-1}=n^{-1}F(n)$, which is none other than $w^{-1}\tau w\tau^{-1}$ via the isomorphism $\ad g^{-1}$. The computation of the $w$'s is completely combinatorial. 

\subsubsection{}
To summarise, once the Green functions have been computed (see the introduction), the calculation of the character formula goes as follows.
\begin{itemize}
\item[(i)] Describe combinatorially the sets $A^F/L^F$ and $B$;
\item[(ii)] Specify a surjection $\iota:A^F/L^F\rightarrow B$ and calculate the fibres of $\iota$;
\item[(iii)] For each $\bar{\nu}\in B$ and each $hL^F\in\iota^{-1}(\bar{\nu})$, evaluate the character $\theta(hsh^{-1})$.
\end{itemize}
The summation in the character formula is decomposed into one summation over $B$ and then one summation over the fibre of $\iota$. Notice that the summation of the $\prescript{h}{}\theta(s)$'s over each fibre of $\iota$ is just permuting the "eigenvalues" of $s$.

\subsection{The Case of $\GL_n(q)\lb\sigma\rb$}
\subsubsection{}\label{NewNotation}
We will change our notations in what follows. Let $N$ and $m$ be some non negative integers such that $2N+m=n$. Write $G=\GL_n(k)$, and let $M_0$ be the $\sigma$-stable standard Levi subgroup of $G$ isomorphic to $\GL_m(k)\times (k^{\ast})^{2N}$ (with respect to upper triangular matrices and diagonal matrices), then $W^{\sigma}:=N_G(M_0\sigma)/M_0\cong\mathfrak{W}_N$ and $F$ acts trivially on $W^{\sigma}$. Let $Q_0$ be the standard parabolic subgroup which has $M_0$ as its Levi factor. For any $\tau\in W^{\sigma}$, let $M_{\tau}\cong\GL_m(k)\times (k^{\ast})^{2N}$ be an $F$-stable and $\sigma$-stable Levi factor of some $\sigma$-stable parabolic subgroup of $G$, such that the $G^F$-conjugacy class of $M_{\tau}\lb\sigma\rb$ is parametrised by the conjugacy class of $\tau$. We can assume that $M_{\tau}=g_{\tau}M_0g_{\tau}^{-1}$ for some $g_{\tau}\in (G^{\sigma})^{\circ}$ such that $g_{\tau}^{-1}F(g_{\tau})=\dot{\tau}\in N_G(M_0)$ is a representative of $\tau$. We will fix such a $\tau$ in what follows.

\subsubsection{}
Let $s\sigma\in G^{F}\sigma$ be a semi-simple element and let $(h_1,h_2)\in\mathbb{N}\times\mathbb{Z}$ be such that $h_2^2+h_1(h_1+1)=m$. Define 
\begin{equation*}
\begin{split}
A(s\sigma,\tau,h_1,h_2)=&\{h\in G\mid hs\sigma h^{-1}\in M_{\tau}\sigma \text{ is isolated with $C_{M_{\tau}}(hs\sigma h^{-1})$ isomorphic }\\&\text{ to the product of $\Sp_{h_1(h_1+1)}(k)\times\Ort_{h^2_2}(k)$ and a torus}\}.
\end{split}
\end{equation*}
Define $$A^F(s\sigma,\tau,h_1,h_2)=\{h\in A(s\sigma,\tau,h_1,h_2)|F(h)=h\}.$$ We will give a combinatorial description of this set.

Write $L'=C_G(s\sigma)^{\circ}$. If $K'\subset L'$ is an $F$-stable Levi subgroup, put $K=C_G(Z^{\circ}_{K'})$. By Proposition \ref{K'etK}, it is the smallest $F$-stable and $s\sigma$-stable Levi subgroup of $G$ such that $(K\cap L')^{\circ}=K'$, which is a Levi factor of an $s\sigma$-stable parabolic subgroup, say $Q$. So $N_{\bar{G}}(K)\cap N_{\bar{G}}(Q)=K\lb s\sigma\rb$, where $\bar{G}=G\sqcup G\sigma$.  Moreover, from Proposition \ref{L3S2.2}~(i), we deduce that $s\sigma$ is isolated in $K\lb s\sigma\rb$. 
\begin{Lem}\label{Mmu}
Suppose that $s\sigma$ is an $F$-stable semi-simple element with connected centraliser isomorphic to $\Sp_{n_s}(k)\times\SO_{n_o}(k)\times\prod_{i}\GL_{n_i}(k)$ for some non negative integers $n_s$, $n_o$ and $n_i$'s. If $h_1(h_1+1)\le n_s$ and $h_2^2\le n_o$, then there exists some $\mu\in W^{\sigma}$ such that $s\sigma$ is $G^F$-conjuate to an element of $M_{\mu}\sigma$ such that its connected centraliser in $M_{\mu}$ is isomorphic to the product of $\Sp_{h_1(h_1+1)}(k)\times\SO_{h_2^2}(k)$ and a torus.
\end{Lem}
\begin{proof}
We use the notations preceding the lemma. By assumption, there exists an $F$-stable Levi subgroup $K'\subset L'$ that is isomorphic to the product of $\Sp_{h_1(h_1+1)}(k)\times\SO_{h^2_2}(k)$ and a torus. Then $K$ is isomorphic to the product of $\GL_m(k)$ and a torus, therefore $G$-conjugate to $M_0$. Since $N_{\bar{G}}(Q)$ meets $G\sigma$, there is $g\in G$ such that $gQg^{-1}=Q_0$. We may then assume that $gKg^{-1}=M_0$. Now $g(K\lb s\sigma\rb)g^{-1}=g(M_{0}\lb\sigma\rb)g^{-1}$. By Proposition \ref{1.40}, $K\lb s\sigma\rb$ is $G^F$-conjugate to $M_{\mu}\lb\sigma\rb$ for some $\mu\in W^{\sigma}$.
\end{proof}

Define 
\begin{equation*}
\begin{split}
B(s\sigma,\tau,h_1,h_2)=&\{\text{ The $L^{'F}$-conjugacy classes of the $F$-stable Levi subgroups $K'\subset L'$ }\\
&\text{ isomorphic to the product of $\Sp_{h_1(h_1+1)}(k)\times\SO_{h^2_2}(k)$ and a torus}\\
&\text{ such that $K\lb s\sigma\rb$ is $G^F$-conjugate to $M_{\tau}\lb\sigma\rb$.}\}
\end{split}
\end{equation*}
In what follows, we fix $s\sigma$, $\tau$, $h_1$ and $h_2$, then $A(s\sigma,\tau,h_1,h_2)$ and $B(s\sigma,\tau,h_1,h_2)$ will be denoted by $A$ and $B$.

\subsubsection{}
We will assume that $s\sigma$ satisfies the assumption in the previous lemma and that it lies in $M_{\mu}\sigma$ for some $\mu$, so in particular, it has the prescribed centraliser in $M_{\mu}$. Write $s_0\sigma=g_{\mu}^{-1}s\sigma g_{\mu}$, $L_0'=C_G(s_0\sigma)^{\circ}$ and $L'=C_G(s\sigma)^{\circ}$. Then $s_0\sigma\in M_0\sigma$ is isolated with connected centraliser isomorphic to the product of $\Sp_{h_1(h_1+1)}(k)\times\SO_{h_2^2}(k)$ and a torus. 

\begin{Lem}\label{A=NL}
There is a natural bijection: 
\begin{eqnarray}
A&\lisomLR &N_G(M_0\sigma).L_0'\\
h&\longleftrightarrow &g_{\tau}^{-1}hg_{\mu}.
\end{eqnarray}
\end{Lem}
\begin{proof}
For $h\in A$, write $x=g_{\tau}^{-1}hg_{\mu}$. By definition, $xs_0\sigma x^{-1}$ lies in $M_0\sigma$ with the prescribed centraliser in $M_0$. There exists $l\in L'_0$ such that $$x^{-1}M_{0}x\cap L'_0=l(M_{0}\cap L_0')l^{-1}=lM_{0}l^{-1}\cap L_0',$$ since the left hand side is a Levi subgroup of $L_0'$ of a given isomorphism type. By assumption, $s_0\sigma$ is isolated in $x^{-1}M_{0}\lb\sigma\rb x$ and in $M_{0}\lb\sigma\rb$, and so Remark \ref{K'etKisole} implies that $x^{-1}M_{0} x=lM_{0}l^{-1}$. So there exists $n\in N_G(M_0)$ such that $x=nl^{-1}$. From the fact $xs_0\sigma x^{-1}\in M_0\sigma$ we see that $n\sigma(n)^{-1}\in M_0$. 

Let us determine the set $N_G(M_0\sigma)$. If $n$ normalises $M_0\sigma$, then it normalises $M_0\sigma.M_0\sigma$, so it normalises $M_0$. Then, $nM_0\sigma n^{-1}=nM_0n^{-1}n\sigma(n^{-1})\sigma\in M_0\sigma$, so $n\sigma(n^{-1})\in M_0$. We see that $N_G(M_0\sigma)$ consists of those components of $N_G(M_0)$ that are $\sigma$-stable. Finally, we note that $g_{\tau}nlg_{\mu}^{-1}$ belongs to $A$, for any $n\in N_G(M_0\sigma)$ and $l\in L_0'$. 
\end{proof}
Since $F(h)=F(g_{\tau}xg_{\mu}^{-1})=g_{\tau}\dot{\tau}F(x)\dot{\mu}^{-1}g_{\mu}^{-1}$, the Frobenius on $A$ is transferred to the map $F_{\tau,\mu}:x\mapsto\dot{\tau}F(x)\dot{\mu}^{-1}$ via the above bijection. The set $N_G(M_0\sigma).L_0'$ is mapped into itself by $F_{\tau,\mu}$.

Write $M'_0=C_{M_0}(s_0\sigma)^{\circ}$. It is a Levi subgroup of $L'_0$ isomorphic to the product of $\Sp_{h_1(h_1+1)}(k)\times\SO_{h^2_2}(k)$ and a torus. Similarly, for any $v\in W^{\sigma}$ such that $s\sigma\in M_v\sigma$, write $M'_v=C_{M_v}(s\sigma)^{\circ}$.

\begin{Cor}\label{NsigmaN''}
There is a natural bijection:
\begin{equation}\label{ALNsigmaN''}
A^F/L^{\prime F}\cong(N_G(M_0\sigma)/N_{L_0'}(M'_0))^{F_{\tau,\mu}}.
\end{equation}
\end{Cor}
Remark \ref{L'Lnorm} implies that $N_{L_0'}(M'_0)$ is indeed a subgroup of $N_G(M_0\sigma)$.
\begin{proof}
Since $L'_0$ is connected, $A^F/L^{\prime F}\cong (A/L')^F$. The bijection is then induced from the bijection of Lemma \ref{A=NL}.
\end{proof}

\subsubsection{}
Let us point out that the identity component of $N_G(M_0\sigma)$ is $M_0$ whereas that of $N_{L_0'}(M'_0)$ is $M'_0$, so we cannot directly reduce the problem to a purely combinatorial one as in \S \ref{splitconn}.

Write $N'_0=N_{L'_0}(M'_0)$ and $\bar{N}'_0:=N'_0.M_0$. The latter is the union of the connected components of $N_G(M_0\sigma)$ that meet $L'_0$. 
\begin{Lem}
Each connected component of $\bar{N}'_0$ contains exactly one connected component of $N'_0$.
\end{Lem}
\begin{proof}
It suffices to consider the identity components of the two groups. Note that the connected component of an element of $N'_0$ is determined by its action on $Z_{M_0'}^{\circ}$. An element of the identity component of $\bar{N}'_0$ must induce trivial action on this torus because $M_0=C_G(Z_{M_0'}^{\circ})$.
\end{proof}
We deduce from this lemma an isomorphism $N'_0/M_0'\cong\bar{N}'_0/M_0$. We can then regard $W'_0:=N_{L_0'}(M_0')/M_0'$ as a subgroup of $W^{\sigma}$.

\begin{Lem}\label{1eq-for-h}
Let $n\in N_G(M_0\sigma)$ and let $w$ be the class of $n$ in $W^{\sigma}$. Then
\begin{itemize}
\item[(i)] The coset $nN'_0$ is $F_{\tau,\mu}$-stable if and only if 
\begin{equation}\label{eqn-for-n}
n^{-1}\dot{\tau}F(n)\dot{\mu}^{-1}\in N'_0;
\end{equation}
\item[(ii)] In the case of (i), we have
\begin{equation}\label{eqn-for-w}
w^{-1}\tau w\mu^{-1}\in W'_0.
\end{equation}
\end{itemize}
\end{Lem}
\begin{proof}
Obvious.
\end{proof}

\begin{Lem}\label{w-implies-n}
Suppose that there exists some $w\in W^{\sigma}$ satisfying (\ref{eqn-for-w}). Then there exists some $n\in N_G(M_0\sigma)$ that represents $w$ and satisfies (\ref{eqn-for-n}).
\end{Lem}
\begin{proof}
Let $\dot{w}\in N_G(M_0\sigma)$ be an $F$-stable element representing $w$ then we want to find an element $t\in M_0$ such that $$\dot{w}^{-1}t^{-1}\dot{\tau}F(t)\dot{w}\dot{\mu}^{-1}\in N'_0.$$ Write $$\dot{w}^{-1}t^{-1}\dot{\tau}F(t)\dot{w}\dot{\mu}^{-1}=\big(\dot{w}^{-1}t^{-1}F_{\tau}(t)\dot{w}\big)\big(\dot{w}^{-1}\dot{\tau}\dot{w}\dot{\mu}^{-1}\big).$$ Since $t\mapsto t^{-1}F_{\tau}(t)$ is surjective onto $M_0$, the desired $t$ exists and $t\dot{w}$ is the sought-for $n$.
\end{proof}

\begin{Prop}
Let $s\sigma$ be an $F$-stable semi-simple element with connected centraliser isomorphic to $\Sp_{n_s}(k)\times\SO_{n_o}(k)\times\prod_{i}\GL_{n_i}(k)$ for some non negative integers $n_s$, $n_o$ and $n_i$'s. Then the set $A^F(s\sigma,\tau,h_1,h_2)$ is non-empty if and only if 
\begin{itemize}
\item[(i)] $h_1(h_1+1)\le n_s$, $h_2^2\le n_o$ and,
\item[(ii)] the solution set of equation (\ref{eqn-for-w}) is non-empty.
\end{itemize}
\end{Prop}
\begin{proof}
This is a combination of Corollary \ref{NsigmaN''}, Lemma \ref{Mmu}, Lemma \ref{1eq-for-h} and Lemma \ref{w-implies-n}.
\end{proof}
\begin{Rem}
Condition (ii) in the above Proposition is actually independent of the choice of $\mu$. Indeed, from the proof of Lemma \ref{Mmu}, one sees that the choice of $M_{\mu}$ comes from a choice of an $F$-stable Levi subgroup $K'\subset L'$ and a different choice of $K'$ results in multiplying $\mu$ on the left by an element of $W_0'$.
\end{Rem}

\subsubsection{}
We assume from now on that $A^F$ is non-empty. 
\begin{Lem}
The map
\begin{equation}
\begin{split}
\iota:A^F/L^{'F}&\longrightarrow B\\
hL^{'F}&\longmapsto \text{ the class of } C_{h^{-1}M_{\tau}h}(s\sigma)^{\circ}.
\end{split}
\end{equation}
is well defined and surjective.
\end{Lem}
\begin{proof}
If $h\in A^F$, then $hs\sigma h^{-1}$ normalises $M_{\tau}$ and a parabolic subgroup containing it, so $s\sigma$ normalises $h^{-1}M_{\tau}h$ and a parabolic subgroup containing it. It follows that $K':=C_{h^{-1}M_{\tau}h}^{\circ}(s\sigma)$ is an $F$-stable Levi subgroup of $L'$. We then obtain the Levi subgroup $K$ as above. From the fact that $C_{h^{-1}M_{\tau}h}(s\sigma)=h^{-1}C_{M_{\tau}}(hs\sigma h^{-1})h$ and from the assumption on $hs\sigma h^{-1}$, we deduce that $s\sigma\in K\lb s\sigma\rb$ is isolated with centraliser isomorphic to the product of $\Sp_{h_1(h_1+1)}(k)\times\Ort_{h^2_2}(k)$ and a torus. Since $s\sigma$ is also isolated in $h^{-1}M_{\tau}\lb\sigma\rb h$, by Remark \ref{K'etKisole}, we have $K=h^{-1}M_{\tau}h$, and so $K\lb s\sigma\rb=h^{-1}M_{\tau}\lb\sigma\rb h$. We see that the $L^{'F}$-class of $K'$ indeed belongs to $B$. Obviously this map factors through the quotient $A^F/L^{'F}$. Given $K'\in B$ with $h\in G^F$ such that $hK\lb s\sigma\rb h^{-1}=M_{\tau}\lb\sigma\rb$, the same argument shows that $h\in A^F$, whence surjectivity.
\end{proof}

\subsubsection{}
We can in fact further assume that $s\sigma\in M_{\tau}\sigma$, i.e. $\mu=\tau$ as long as $A^F$ is non-empty. The equation (\ref{eqn-for-w}) then becomes $$w^{-1}\tau w\tau^{-1}\in W'_0,$$ with $W'_0$ being stable under conjugation by $\tau$, and we will denote by $\mathbb{W}(s\sigma,\tau)$ the set of those $w\in W^{\sigma}$ satisfying this condition. Then $B$ is the set of $\tau$-conjugacy classes of $W'_0$ that meet $\{w^{-1}\tau w\tau^{-1}\mid w\in\mathbb{W}(s\sigma,\tau)\}$.
\begin{Rem}
With $\tau=\mu$, the condition (\ref{eqn-for-w}) simply means that the $\tau$-conjugacy class of $1$ meets $W'_0$. This is what one would guess without any computation. But it is not useful in determining whether $A^F$ is empty or not, because one would have to first find a $G$-conjugate $s_0\sigma$ of $s\sigma$ such that the associated $W'_0$ is $\tau$-stable, which is an equally difficult problem. In practice, it is easier to find some $\mu\in W^{\sigma}$ and some $s_0\sigma$ such that $W'_0$ is $\mu$-stable, then we can see if $A^F$ is empty using (\ref{eqn-for-w}).
\end{Rem}

Let $\bar{\nu}$ be a $\tau$-conjugacy class of $W_0'$. 
\begin{Lem}\label{11eq-for-h}
Let $n\in N_G(M_0\sigma)$ and let $w$ be the class of $n$ in $W^{\sigma}$. Suppose that $n$ satisfies (\ref{eqn-for-n}). Then under the bijection (\ref{ALNsigmaN''}) between the classes of $h$ and the classes of $n$, the $F_{\dot{\tau}}$-stable Levi subgroup $C_{h^{-1}M_{\tau}h}(s\sigma)^{\circ}\subset L'$ lies in the class $\bar{\nu}$ if and only if  
\begin{equation}\label{eqn-for-nu}
w^{-1}\tau w\tau^{-1}\in\bar{\nu}.
\end{equation}
\end{Lem}
\begin{proof}
Write $h=nl$ with $n\in N_G(M_{\tau}\sigma)$ and $l\in L'$. The $L^{'F}$-class of the Levi subgroup $C_{h^{-1}M_{\tau}h}(s\sigma)^{\circ}=l^{-1}M'_{\tau}l\subset L'$ is given by $lF(l)^{-1}=n^{-1}F(n)\in N_{L'}(M'_{\tau})$, or rather $n_0^{-1}\dot{\tau}F(n_0)\dot{\tau}^{-1}\in N_{L'_0}(M'_0)$, with $n_0=g_{\tau}^{-1}ng_{\tau}$. Then the class of $n_0$ is $w$.
\end{proof}

Denote by $\mathbb{W}_{\tau,\nu}$ the subset of $\mathbb{W}(s\sigma,\tau)$ consisting of elements $w$ satisfying (\ref{eqn-for-nu}).
\begin{Lem}\label{structure-W}
We have:
\begin{itemize}
\item[(i)] $\mathbb{W}_{\tau,\nu}$ is a union of $W'_0$-cosets;
\item[(ii)] $\mathbb{W}_{\tau,\nu}$ surjects onto $\bar{\nu}$ via $w\mapsto w^{-1}\tau w\tau^{-1}$;
\item[(iii)] The fibre in $\mathbb{W}_{\tau,\nu}$ over each element of $\bar{\nu}$ is in bijection with $C_{W^{\sigma}}(\tau)$.
\end{itemize}
\end{Lem}
\begin{proof}
(i) and (ii): If $w^{-1}\tau w\tau^{-1}=\nu$ and $w_0\in W_0'$, then $(ww_0)^{-1}\tau ww_0\tau^{-1}=w_0^{-1}\nu(\tau w_0\tau^{-1})$. (iii): If $w_1$ and $w_2$ are mapped to the same element of $\bar{\nu}$ under the map of (ii), then $w_2\in C_{W^{\sigma}}(\tau)w_1$.
\end{proof}
The combinatorial aspect of $A^F$, i.e. the structure of $\mathbb{W}(s\sigma,\tau)$, is now completely understood.

\subsubsection{}
We now come to the last step in evaluating the induction $R_{M_{\tau}}^{G}\phi$. An element $w\in W^{\sigma}$ lies in $\mathbb{W}(s\sigma,\tau)$ if and only if it defines a $\tau$-stable coset in $W^{\sigma}/W'$. Below we consider a map $A^F\rightarrow(W^{\sigma}/W')^{\tau}$, and show that as long as $h$ lies in the fibre over a $W'$-coset, then $\phi(hs\sigma h^{-1})$ has constant value. This reduces the evaluation of the character formula to the computation of $\phi(hs\sigma h^{-1})$.
\begin{Prop}\label{last-computation}
Fix $s\sigma$, $h_1$, $h_2$, $\tau$ and $\nu$. Write $z_{\tau}=|C_{W^{\sigma}}(\tau)|$ and $z_{\nu}=|C_{W'}(\nu)|$, where $C_{W^{\sigma}}(\tau)$ is the usual centraliser and $C_{W'}(\nu)$ is the stabiliser of $\nu$ for the $\tau$-conjugation. Let $$\pi:A^F\longrightarrow (N_{G}(M_{\tau}\sigma).L'/L')^F\longrightarrow (W^{\sigma}/W')^{\tau}$$ be the map induced by (\ref{ALNsigmaN''}). Then,
\begin{itemize}
\item[(i)]
The map $\pi$ is surjective.
\item[(ii)]
The inverse image under $\pi$ of $\mathbb{W}_{\tau,\nu}$ is $N_{G^F}(M_{\tau}\sigma)hL^{\prime F}$, for some $h\in A^F$. Its cardinality is equal to $|M_{\tau}^F|\cdot|L^{\prime F}|\cdot|C_{M_{\tau}}(hs\sigma h^{-1})^{\circ F}|^{-1}z_{\tau}z_{\nu}^{-1}$.
\item[(iii)]
The inverse image under $\pi$  of a $W'$-coset in $\mathbb{W}_{\tau,\nu}$ is $M_{\tau}^FhL^{\prime F}$, for some $h\in A^F$. Its cardinality is equal to $|M_{\tau}^F|\cdot|L^{\prime F}|\cdot|C_{M_{\tau}}(hs\sigma h^{-1})^{\circ F}|^{-1}$.
\end{itemize}
\end{Prop}
\begin{Rem}
The surjection $A^F\rightarrow B$ factors through $\pi$.
\end{Rem}
\begin{proof}
(i). The surjectivity of the first map is obvious according to previous discussions. The surjectivity of the second map follows from Lemma \ref{w-implies-n}. (ii). Suppose that $h_1$, $h_2\in A^F$ are mapped to $\mathbb{W}_{\tau,\nu}$. By definition, $C_{h_1^{-1}M_{\tau}h_1}(s\sigma)^{\circ}$ and $C_{h_2^{-1}M_{\tau}h_2}(s\sigma)^{\circ}$ are $L^{\prime F}$-conjugate, i.e. there exists $l\in L^{\prime F}$ such that $$lC_{h_1^{-1}M_{\tau}h_1}(s\sigma)^{\circ}l^{-1}=C_{h_2^{-1}M_{\tau}h_2}(s\sigma)^{\circ}.$$Note that $lC_{h_1^{-1}M_{\tau}h_1}(s\sigma)^{\circ}l^{-1}=C_{lh_1^{-1}M_{\tau}h_1l^{-1}}(s\sigma)^{\circ}$. By the same argument as in the proof of Lemma \ref{A=NL}, we have $$lh_1^{-1}M_{\tau}h_1l^{-1}=h_2^{-1}M_{\tau}h_2.$$We deduce that $n:=h_1l^{-1}h_2^{-1}$ lies in $N_{G^F}(M_{\tau}\sigma)$. We conclude that $h_1\in N_{G^F}(M_{\tau}\sigma)h_2 L^{\prime F}$. The cardinality of $N_{G^F}(M_{\tau}\sigma)h_2 L^{\prime F}$ is equal to $$|N_{G^F}(M_{\tau}\sigma)|\cdot|L^{\prime F}|\cdot|N_{G^F}(M_{\tau}\sigma)\cap h_2 L^{\prime F}h_2^{-1}|^{-1}.$$Now $N_{G^F}(M_{\tau}\sigma)\cap h_2 L^{\prime F}h_2^{-1}=N_{h_2 L^{\prime F}h_2^{-1}}(M_{\tau}')$, where $M'_{\tau}=C_{M_{\tau}}(h_2s\sigma h_2^{-1})^{\circ}$. We have $|N_{G^F}(M_{\tau}\sigma)|=|M_{\tau}^F|z_{\tau}$ and $N_{h_2 L^{\prime F}h_2^{-1}}(M_{\tau}')=|M_{\tau}^{\prime F}|z_{\nu}$. Thus we get the cardinality $$N_{G^F}(M_{\tau}\sigma)h_2 L^{\prime F}=|M_{\tau}^F|\cdot|L^{\prime F}|\cdot|C_{M_{\tau}}(hs\sigma h^{-1})^{\circ F}|^{-1}z_{\tau}z_{\nu}^{-1}.$$ (iii). The number of $W'$-cosets in $\mathbb{W}_{\tau,\nu}$ is equal to $z_{\tau}z_{\nu}^{-1}$ according to Lemma \ref{structure-W}. We will argue by controlling the size of the fibres of $\pi$. Let $h\in A^F$, then by the definition of $\pi$, we have $\pi^{-1}(\pi(h))=(M_{\tau}hL')^F$. Clearly, $M_{\tau}^FhL^{\prime F}$ is contained in $(M_{\tau}hL')^F$. We have $|M_{\tau}^FhL^{\prime F}|=|M_{\tau}^F|\cdot|L^{\prime F}|\cdot|M_{\tau}^F\cap hL^{\prime F}h^{-1}|^{-1}$. Note that $M_{\tau}^F\cap hL^{\prime F}h^{-1}=(M_{\tau}\cap hL^{\prime F}h^{-1})^F=C_{M_{\tau}}(hs\sigma h^{-1})^{\circ F}$. Therefore the size of each fibre of $\pi$ is bounded below by $|M_{\tau}^F|\cdot|L^{\prime F}|\cdot|C_{M_{\tau}}(hs\sigma h^{-1})^{\circ F}|^{-1}$. It must not be strictly larger than this number since we know the size of $\pi^{-1}(\mathbb{W}_{\tau,\nu})$. We conclude that each fibre of $\pi$ is necessarily of the form $M_{\tau}^FhL^{\prime F}$.
\end{proof}

\begin{Rem}
If $M_{\tau}=T_{\tau}$ is a torus, then $
C_{M_{\tau}}(hs\sigma h^{-1})^{\circ F}\cong T_{\tau}^{\prime F}$ with $T'_{\tau}:=C_{T_{\tau}}(\sigma)^{\circ}$.
\end{Rem}

%%%%%%%%%%%%%%%%%%% 
\section{The Formula}
\subsection{Decomposition into Deligne-Lusztig Inductions}
\subsubsection{}
By Proposition \ref{IrrLIstable} and Proposition \ref{IrrMstable}, every $\sigma$-stable irreducible character $\chi$ of $\GL_n(q)$ is of the form $R^G_{M_{I,w}}(\chi_{1}\boxtimes\chi_0)$, for an $F$-stable Levi subgroup $M_{I,w}$ isomorphic to the $\sigma$-stable standard Levi subgroup $L_I$ of the form (\ref{LI}) equipped with the Frobenius $F_w$ given by (\ref{Fn_0}), (\ref{Flineaire}) and (\ref{Funitaire}) and with the action of $\sigma$ given by (\ref{sigman_0}) and (\ref{sigmatransposition}). Decomposing $L_I$ into $L_1\times L_0$ following \S \ref{constIrr1}, then $\chi_1$ and $\chi_0$ are identified with some $\sigma$-stable characters of $L_1^{F_w}$ and $L_0^{F_0}$ respectively, where we also denote by the same letter the restriction of $F_w$ to $L_1$ and by $F_0$ its restriction to $L_0$. We decompose $\sigma$ into $(\sigma_1,\sigma_0)$ with respect to $L_1\times L_0$. Recall that $\chi_0$ is defined by a 2-partition $(\mu_+,\mu_-)$ and that $\chi_1$ is defined by $[\varphi_1]\in\Irr(W_{L_1})^{F_w}$ and $\theta_1\in\Irr_{reg}(L_1^{F_w})^{\sigma_1}$ satisfying the assumptions of \S \ref{hypothesethetavarphi}, where $\varphi_1\in\Irr(W^{\sigma_1}_{L_1})^{F_w}$.

By Lemma \ref{IndRes}, an extension of $\chi$ to $G^F\sigma$ is obtained by first extending $\smash{\chi_{1}\boxtimes\chi_0}$ to $\smash{M_{I,w}^F\sigma}$ and then taking the induction $\smash{R^{G\sigma}_{M_{I,w}\sigma}}$. One can equally extend $\smash{\chi_{1}\boxtimes\chi_0}$, regarded as a character of  $L_I^{F_w}$, to $\smash{L_I^{F_w}\sigma}$. The extension of $\chi_1$ to $L_1^{F_w}\sigma_1$ is given by Theorem \ref{ExtL_1} and the extension of $\chi_0$ to $L_0^{F_0}\sigma_0$ is given by Theorem \ref{ExtL_0}. Explicitly, we have
\begingroup
\allowdisplaybreaks
\begin{align*}
\tilde{\chi}_1|_{L_1^{F_w}\sigma_1}&=|W_{L_1}^{\sigma_1}|^{-1}\sum_{w_1\in W_{L_1}^{\sigma_1}}\tilde{\varphi}_1(w_1F_w)R_{T_{w_1}\sigma_1}^{L_1\sigma_1}\tilde{\theta}_1,
\\%%
\tilde{\chi}_0|_{L_0^{F_0}\sigma_0}&=\frac{1}{|\mathfrak{W}_+|}\frac{1}{|\mathfrak{W}_-|}\sum_{\substack{w_+\in \mathfrak{W}_+\\ w_-\in \mathfrak{W}_-}}\varphi_+(w_+)\varphi_-(w_-)R^{L_0\sigma_0}_{L_{\mathbf{w}}\sigma_0}\phi_{\mathbf{w}},
\end{align*}
\endgroup
and in the second equality there is an extra $\bs{\gamma}_{\chi_0}$ if $\sigma_0$ is of symplectic type. Now put $$\bs{\gamma}_{\chi}:=
\begin{cases}
\bs{\gamma}_{\chi_0} & \text{if $\sigma$ is of symplectic type},\\
1 & \text{otherwise}.
\end{cases}$$

Write $L_{w_1,\mathbf{w}}=T_{w_1}\times L_{\mathbf{w}}$. It is an  $F_w$-stable and $\sigma$-stable Levi factor of a $\sigma$-stable parabolic subgroup of $L_I$. Combining the above two formulas gives
\begin{equation*}
\tilde{\chi}_1\widetilde{\boxtimes}\tilde{\chi}_0|_{L_I^{F_w}\sigma}=\bs{\gamma}_{\chi}|W_{L_1}^{\sigma_1}\times\mathfrak{W}_+\times\mathfrak{W}_-|^{-1}\!\!\!\!\!\!\!\!\sum_{\substack{(w_1,w_+,w_-)\\\in W_{L_1}^{\sigma_1}\times\mathfrak{W}_+\times\mathfrak{W}_-}}\!\!\!\!\!\!\!\!\tilde{\varphi}_1(w_1F_w)\varphi_+(w_+)\varphi_-(w_-)R^{L_I\sigma}_{L_{w_1,\mathbf{w}}}(\tilde{\theta}_1\widetilde{\boxtimes}\phi_{\mathbf{w}}),
\end{equation*}
Suppose that under the isomorphism $L_I\lb\sigma\rb\cong M_{I,w}\lb\sigma\rb$, which is compatible with the action of $\sigma$, the subgroups $L_{w_1,\mathbf{w}}$ become $M_{w_1,\mathbf{w}}$, and that the decomposition $L_1\times L_0$ induces the decomposition $M_{I,w}\cong M_1\times M_0$. We deduce 
\begin{Thm}\label{decIrr}
The extension of $\chi$ is given by the following formula:
\begin{equation*}
\tilde{\chi}|_{G^F\sigma}=\bs{\gamma}_{\chi}|W_{M_1}^{\sigma_1}\times\mathfrak{W}_+\times\mathfrak{W}_-|^{-1}\!\!\!\!\!\!\!\!\sum_{\substack{(w_1,w_+,w_-)\\\in W_{M_1}^{\sigma_1}\times\mathfrak{W}_+\times\mathfrak{W}_-}}\!\!\!\!\!\!\!\!\tilde{\varphi}_1(w_1F)\varphi_+(w_+)\varphi_-(w_-)R^{G\sigma}_{M_{w_1,\mathbf{w}}}(\tilde{\theta}_1\widetilde{\boxtimes}\phi_{\mathbf{w}}).
\end{equation*}
\end{Thm}

\subsubsection{}
Recall \S \ref{Tchi} that the $\sigma$-stable irreducible characters of $\GL_n(q)$ are parametrised by $\bar{\mathfrak{T}}_{\chi}$. Denote by $\bar{\mathfrak{T}}^0_{\chi}\subset\bar{\mathfrak{T}}_{\chi}$ the subsets of the elements $$\lambda_+\lambda_-(\lambda_i,\alpha_i)_{i\in\Lambda_1}(\lambda'_j,\alpha'_j)_{j\in\Lambda_2}$$ in which at most one of $|\lambda_+|$ and $|\lambda_-|$ is odd, and $\lambda_{\pm}$ is a partition with trivial 2-core or with 2-core $(1)$ according to the parity.  
\begin{Cor}\label{base-of-uniform}
The $\sigma$-stable irreducible characters of $\GL_n(q)$ that extends to uniform functions on $\GL_n(q)\sigma$ are in bijection with $\bar{\mathfrak{T}}^0_{\chi}$, and the extensions of these characters constitute a base (identifying two extensions of the same character) of the space of the uniform functions on $\GL_n(q)\sigma$.
\end{Cor}
\begin{proof}
We have seen in Theorem \ref{decIrr} that the extension of a general $\sigma$-stable irreducible character is decomposed into a linear combination of cuspidal functions induced from $$M^F_{w_1,\mathbf{w}}\cong T^F_{w_1}\times T^F_{w_+}\times T^F_{w_-}\times\GL_m(q),$$
for various $w_1$ and $\mathbf{w}$. Cuspidal functions induced from $M^F_{w_1,\mathbf{w}}$ with $m>1$ can not be uniform (see \S \ref{nonconnunif} for the definition of uniform functions). Now the condition $m\le 1$ is equivalent to the condition that the sum of the 2-cores of $\lambda_+$ and $\lambda_-$ is either empty or $(1)$. We see that $\lambda_+$ and $\lambda_-$ satisfy the assumption in the definition of $\bar{\mathfrak{T}}^0_{\chi}$, whence the first assertion.

For each $\bar{\mathfrak{t}}\in\bar{\mathfrak{T}}^0_{\chi}$, denote by $\chi_{\bar{\mathfrak{t}}}$ the corresponding character, and choose an extension $\tilde{\chi}_{\bar{\mathfrak{t}}}\in\mathcal{C}(\GL_n(q)\sigma)$. Then, $\{\tilde{\chi}_{\bar{\mathfrak{t}}}\mid \bar{\mathfrak{t}}\in\bar{\mathfrak{T}}^0_{\chi}\}$ is a set consisting of functions orthogonal to each other. Theorem \ref{decIrr} gives a transition matrix between the set of (generalised) Deligne-Lusztig characters and that of the $\tilde{\chi}_{\bar{\mathfrak{t}}}$'s.
\end{proof}
\begin{Rem}
The extension of an irreducible character is then either uniform, ot orthogonal to all uniform functions. Since the characteristic function of a quasi-semi-simple conjugacy class is uniform, the extension of a character corresponding to an element of $\bar{\mathfrak{T}}_{\chi}\setminus\bar{\mathfrak{T}}^0_{\chi}$ vanishes on every quasi-semi-simple element. 
\end{Rem}
\begin{Rem}
The two groups $\leftidx{^s\!}{\bar{G}}^F$ and $\leftidx{^o\!}{\bar{G}}^F$ have differences only in those characters satisfying $\chi(-1)=-\chi(1)$. From the above corollary, we deduce that if $\chi$ extends to a uniform function on $G^F\sigma$, then $\chi(-1)=\chi(1)$. This is due to the fact that for such a character, $\eta$ must have even "multiplicities" in the semi-simple part.
\end{Rem}

%%%%%%%%%%%%%%%%%%
\section{Examples}
We give the character tables of $\GL_2\rtimes\lb\sigma\rb$ and $\GL_3\rtimes\lb\sigma\rb$. We assume that $q\equiv 1\mod 4$ and so $\sqrt{-1}\in\mathbb{F}_q$. The trivial character of a group will be denoted by $\tC$. The $\sigma$-stable irreducible characters are specified in terms of $(L,\varphi,\theta)$ as in Theorem \ref{LS}. Denote by $T$ the maximal torus consisting of the diagonal matrices and denote by $w$ the unique nontrivial element of $W_G(T)$, either for $\GL_2(k)$ or $\GL_3(k)$, that is fixed by $\sigma$. Denote by $T_w$ a $\sigma$-stable and $F$-stable maximal torus corresponding to the conjugacy class of $w$. In the following, we will freely use the formulas in \S \ref{formules}.

\subsection{$\sigma$-stable Irreducible Characters}
We specify the $\sigma$-stable irreducible characters of $\GL_2(q)$ and $\GL_3(q)$, and compute the numbers of these characters and of the quadratic-unipotent characters of $\GL_4(q)$ and $\GL_5(q)$.
\subsubsection{}
Suppose $G=\GL_2(k)$. There are $q+3$ $\sigma$-stable irreducible characters of $\GL_2(q)$, and 5 of them are quadratic-unipotent, among which one extends into a non-uniform function.

The quadratic-unipotent characters induced from $L=G$ are the following.
\begin{center}
\begin{tabular}{cccc}
$\tC$ & $\eta\tC$ & $\St$ & $\eta\St$\\
\hline
\end{tabular}
\end{center}

The only one quadratic-unipotent character induced from $L=T$ is the following.
\begin{center}
\begin{tabular}{c}
$R^G_T(\tC,\eta)$\\
\hline
\end{tabular}
\end{center}
It is the unique $\sigma$-stable irreducible character with non-uniform extension.

Other $\sigma$-stable irreducible characters are either of the form,
\begin{center}
\begin{tabular}{cc}
$R^G_T(\alpha,\alpha^{-1})$ \\
\hline
\end{tabular}
\end{center}
with $\alpha\in\Irr(\mathbb{F}_q^{\ast})$ satisfying $\alpha^q=\alpha$, $\alpha\ne\tC$ or $\eta$. There are $(q-3)/2$ of them; or of the form,
\begin{center}
\begin{tabular}{cc}
$R^G_{T_w}(\omega)$ \\
\hline
\end{tabular}
\end{center}
with $\omega\in\Irr(\mathbb{F}_{q^2}^{\ast})$ satisfying $\omega^q=\omega^{-1}$, $\omega\ne\tC$ or $\eta$. There are $(q-1)/2$ of them.

\subsubsection{}
Suppose $G=\GL_3(k)$. There are $2q+6$ $\sigma$-stable irreducible characters of $\GL_3(q)$, and 10 of them are quadratic-unipotent, among which two have non-uniform extensions.

The quadratic-unipotent characters induced from $L=G$ are the following.
\begin{center}
\begin{tabular}{cccccc}
$\tC$ & $\eta\tC$ & $\chi_2$ & $\eta\chi_2$ $\chi_3$ & $\eta\chi_3$\\
\hline
\end{tabular}
\end{center}
The characters $\chi_2$ and $\eta\chi_2$ are associated to the sign character of $\mathfrak{S}_3$. The characters $\chi_3$ and $\eta\chi_3$ are associated to the degree 2 character of $\mathfrak{S}_3$, and these two characters have non-uniform extensions.

The quadratic-unipotent characters induced from $L\cong \GL_2(k)\times k^{\ast}$ are the following.
\begin{center}
\begin{tabular}{cccc}
$R^G_L(\tC,\eta)$ & $R^G_L(\eta\tC,\Id)$ & $R^G_L(\St,\eta)$ & $R^G_L(\eta\St,\Id)$\\
\hline
\end{tabular}
\end{center}

Other $\sigma$-stable irreducible characters are either of the form,
\begin{center}
\begin{tabular}{cc}
$R^G_T(\alpha,\tC,\alpha^{-1})$ & $R^G_T(\alpha,\eta,\alpha^{-1})$ \\
\hline
\end{tabular}
\end{center}
with $\alpha\in\Irr(\mathbb{F}_q^{\ast})$ satisfying $\alpha^q=\alpha$, $\alpha\ne\tC$ or $\eta$. There are $q-3$ of them; or of the form,
\begin{center}
\begin{tabular}{cc}
$R^G_{T_w}(\omega,\tC,\omega^{-1})$ & $R^G_{T_w}(\omega,\eta,\omega^{-1})$\\
\hline
\end{tabular}
\end{center}
with $\omega\in\Irr(\mathbb{F}_{q^2}^{\ast})$ satisfying $\omega^q=\omega^{-1}$, $\omega\ne\tC$ or $\eta$. There are $q-1$ of them.

\subsubsection{}
Suppose $G=\GL_4(k)$. There are $$\frac{1}{2}(q-2)(q-3)+7(q-2)+20$$ $\sigma$-stable irreducible characters of $\GL_4(k)$, and 20 of them are quadratic-unipotent.

The Levi subgroups $L=\GL_4(k)$, $L=\GL_3(k)\times k^{\ast}$ and $L=\GL_2(k)\times\GL_2(k)$ give rise to $(5+3+2)\times 2$ quadratic-unipotent characters, knowing that $|\Irr(\mathfrak{S}_4)|=5$.

The Levi subgroup $L\cong \GL_2(k)\times(k^{\ast}\times k^{\ast})$ gives $$|\text{Quad. Unip. of } \GL_2|\times (q-2)=5(q-2)$$ noticing that $q-2=(q-3)/2+(q-1)/2$ as in the case of $G=\GL_2(k)$.

The Levi subgroup $L\cong\GL_2(k)\times\GL_2(k)$ gives $(q-2)\times 2$ with $2=|\Irr(\mathfrak{S}_2)|$.

The maximal torus $(k^{\ast})^2\times(k^{\ast})^2$ gives $\frac{1}{2}(q-2)(q-3)$.

\subsubsection{}
Suppose $G=\GL_5(k)$. There are $$(q-2)(q-3)+14(q-2)+36$$ $\sigma$-stable irreducible characters, and 36 of them are quadratic-unipotent.

The Levi subgroups $L=\GL_5(k)$, $L\cong \GL_4(k)\times k^{\ast}$ and $L\cong \GL_3(k)\times\GL_2(k)$ give $(7+5+3\times 2)\times 2$ quadratic-unipotent characters, knowing that $|\Irr(\mathfrak{S}_5)|=7$.

The Levi subgroup $L\cong\GL_3(k)\times(k^{\ast})^2$ gives $$|\text{Quad. Unip. of }\GL_3|\times (q-2)=10(q-2).$$

The Levi subgroup $L\cong k^{\ast}\times(\GL_2(k)\times\GL_2(k))$ gives $2\times 2\times (q-2)$, with one factor $2=|\Irr(\mathfrak{S}_2)|$ and the other factor $2=|\{1,\eta\}|$.

The maximal torus $L\cong k^{\ast}\times(k^{\ast})^2\times(k^{\ast})^2$ gives $2\times\frac{1}{2}(q-2)(q-3)$.

\subsection{Conjugacy Classes}
We present the conjugacy classes of $\GL_2(q)\sigma$ and of $\GL_3(q)\sigma$, and count the isolated classes of $\GL_4(q)\sigma$ and of $\GL_5(q)\sigma$. Denote by $su$ the Jordan decomposition of an element of the conjugacy class concerned. Note that $\Ort_1(k)\cong \bs\mu_2$ and $\Ort_2(k)\cong k^{\ast}\rtimes\lb\tau\rb$ with $\tau(x)=x^{-1}$.

\subsubsection{}
Suppose $G=\GL_2(k)$. There are $q+3$ conjugacy classes, and 5 of them are isolated. 

\begin{itemize}
\item[-] $s=(1,1)\sigma$, $C_G(s)=\SL_2(k)$.\\
The unipotent parts are given by the partitions defined by Jordan blocks. Then the centralisers and the $G^F$-classes are specified accordingly as below,
\begin{center}
\begin{tabular}{ccc}
$(1^2)$ & \multicolumn{2}{c}{$(2)$}\\
\hline
$\SL_2(k)$ & \multicolumn{2}{c}{$\Ort_1(k)V$}\\
\hline
$C_1$ & $C_2$ & $C_3$
\end{tabular}
\end{center}
where $V\cong\mathbb{A}^1$ is the unipotent radical. If we use the unit element of a root subgroup of $\SL_2(k)$ to represent $(2)$, then $C_2$ corresponds to the identity component of the centraliser. The two components of $\Ort_1(k)V$ have as representatives the scalars $\pm \Id$.

\item[-] $s=(\mathfrak{i},-\mathfrak{i})\sigma$, $C_G(s)=\Ort_2(k)$. \\
Denote by $C_4$ the $G(q)$-class corresponding to the identity component, and $C_5$ the other class. The two components of $\Ort_2(k)$ have as representatives 
$$
\left(
\begin{array}{cc}
1 & 0\\
0 & 1
\end{array}
\right)
\text{ and }
\left(
\begin{array}{cc}
0 & 1\\
1 & 0
\end{array}
\right)
$$ respectively, and so induce the Frobenius $x\mapsto x^{q}$ and $x\mapsto x^{-q}$ respectively. In other words, the centralisers of $C_4$ and $C_5$ are $\Ort^+_2(q)$ and $\Ort^-_2(q)$ respectively.

\item[-] $s=(a,a^{-1})\sigma$, $C_G(s)=k^{\ast}$.\\
For any value of $a$, the corresponding $G$-class contains a unique $G(q)$-class. The classes are as follows.
\begin{center}
\begin{tabular}{cccc}
$a^{q-1}=1$ & $a^{q+1}=1$ & $a^{q-1}=-1$ & $a^{q+1}=-1$\\
\hline
$C_6(a)$ & $C_7(a)$ & $C_8(a)$ & $C_9(a)$\\
\end{tabular}
\end{center}
The Frobenius on $C_G(s)\cong k^{\ast}$ with $s\in C_6$ or $C_8$ is $x\mapsto x^q$, while the Frobenius on $C_G(s)\cong k^{\ast}$ with $s\in C_7$ or $C_9$ is $x\mapsto x^{-q}$.
\end{itemize}

We have
\begin{itemize}
\item[-] $|C_1|=|G(q)|/|\SL_2(q)|=q-1$; 
\item[-] $|C_2|=|C_3|=|G(q)|/2|V(q)|=\frac{1}{2}(q-1)^2(q+1)$;
\item[-] $|C_4|=|G(q)|/|\Ort_2^+(q)|=\frac{1}{2}q(q+1)(q-1)$; $|C_5|=|G(q)|/|\Ort_2^-(q)|=\frac{1}{2}q(q-1)^2$;
\item[-] $|C_6|=|C_8|=|G(q)|/(q-1)=q(q+1)(q-1)$;
\item[-] $|C_7|=|C_9|=|G(q)|/(q+1)=q(q-1)^2$.
\end{itemize}

\subsubsection{}
Suppose $G=\GL_3(k)$. There are $2q+6$ conjugacy classes, and 10 of them are isolated. Now each semi-simple $G$-conjugacy class contains two $G(q)$-conjugacy classes, distinguished by the sign $\eta$ (\textit{cf.} (\ref{epsilon})). Depending on the value of $\eta$, we will write $C^+$ or $C^-$ to represent the corresponding conjugacy class contained in a given $G$-conjugacy class.
\begin{Nota}
In what follows, we write $\epsilon$ instead of $\eta$ to avoid clashing with the character of $\mathbb{F}_q^{\ast}$.
\end{Nota}

\begin{itemize}
\item[-] $s=(1,1,1)\sigma$, $C_G(s)=\Ort_3(k)$.\\
The unipotent parts are given by the partitions defined by Jordan blocks. Then the centralisers and the $G^F$-classes are specified accordingly as below,
\begin{center}
\begin{tabular}{cccc}
\multicolumn{2}{c}{$(1^3)$} & \multicolumn{2}{c}{$(3)$}\\
\hline
\multicolumn{2}{c}{$\Ort_3(k)$} & \multicolumn{2}{c}{$\Ort_1(k).V$}\\
\hline
$C_1^+$ & $C_1^-$ & $C_2^+$ & $C_2^-$
\end{tabular}
\end{center}
where $V\cong\mathbb{A}^1$ is the unipotent radical.

\item[-] $s=(\mathfrak{i},1,-\mathfrak{i})\sigma$, $C_G(s)\cong \SL_2(k)\times\Ort_1(k)$.\\
The unipotent parts are given by the partitions defined by Jordan blocks. Then the centralisers and the $G^F$-classes are specified accordingly as below,
\begin{center}
\begin{tabular}{cccccc}
\multicolumn{2}{c}{$(1^2)$} & \multicolumn{4}{c}{$(2)$}\\
\hline
\multicolumn{2}{c}{$\SL_2(k)\times\Ort_1(k)$} & \multicolumn{4}{c}{$\Ort_1(k)\times\Ort_1(k).V$}\\
\hline
$C_3^+$ & $C_3^-$ & $C_4^+$ & $C_5^+$ & $C_4^-$ & $C_5^-$ 
\end{tabular}
\end{center}
where $V\cong\mathbb{A}^1$ is the unipotent radical. If we use the unit element of a root subgroup of $\SL_2(k)$ to represent $(2)$, then the correspondence between the classes $C_4^+$, $C_5^+$, $C_4^-$, $C_5^-$ and the connected components of $C_G(s)$ is given as follows. 
\begin{center}
\begin{tabular}{cccc}
$C_4^+$ & $C_5^+$ & $C_4^-$ & $C_5^-$ \\
\hline
$\left(
\begin{array}{ccc}
1 & 0 & 0\\
0 & 1 & 0\\
0 & 0 & 1
\end{array}
\right)$ & 
$\left(
\begin{array}{ccc}
-1 & 0 & 0\\
0 & 1 & 0\\
0 & 0 & -1
\end{array}
\right)$ &
$\left(
\begin{array}{ccc}
1 & 0 & 0\\
0 & -1 & 0\\
0 & 0 & 1
\end{array}
\right)$ &
$\left(
\begin{array}{ccc}
-1 & 0 & 0\\
0 & -1 & 0\\
0 & 0 & -1
\end{array}
\right)$
\end{tabular}
\end{center}

\item[-] $s=(a,1,a^{-1})\sigma$, $C_G(s)=k^{\ast}\times\bs\mu_2$, identified with $\{\diag(x,\pm 1, x^{-1}); x\in k^{\ast}\}$.\\
The conjugacy classes are as follows.
\begin{center}
\begin{tabular}{cccccccc}
\multicolumn{2}{c}{$a^{q-1}=1$} & \multicolumn{2}{c}{$a^{q+1}=1$} & 
\multicolumn{2}{c}{$a^{q-1}=-1$} & \multicolumn{2}{c}{$a^{q+1}=-1$}\\
\hline
$C_{6}^+(a)$ & $C_{6}^-(a)$ & $C_{7}^+(a)$ & $C_{7}^-(a)$
& $C_{8}^+(a)$ & $C_{8}^-(a)$ & $C_{9}^+(a)$ & $C_{9}^-(a)$\\\end{tabular}
\end{center}
The Frobenius on $C_G(s)\cong k^{\ast}$ with $s\in C_{6}^{\pm}$ or $C_{8}^{\pm}$ is $x\mapsto x^q$, while the Frobenius on $C_G(s)\cong k^{\ast}$ with $s\in C_{7}^{\pm}$ or $C_{9}^{\pm}$ is $x\mapsto x^{-q}$.
\end{itemize}

We have
\begin{itemize}
\item[-] $|C_1^+|=|C_1^-|=|G(q)|/|\Ort_3(q)|$;
\item[-] $|C_2^+|=|C_2^-|=|G(q)|/2|V(q)|$;
\item[-] $|C_3^+|=|C_3^-|=|G(q)|/2|\SL_2(q)|$;
\item[-] $|C_4^+|=|C_4^-|=|C_5^+|=|C_5^-|=|G(q)|/4|V(q)|$;
\item[-] $|C_{6}^{\pm}|=|C_{8}^{\pm}|=|G(q)|/2(q-1)$;
\item[-] $|C_{7}^{\pm}|=|C_{9}^{\pm}|=|G(q)|/2(q+1)$.
\end{itemize}

\subsubsection{}
Suppose  $G=\GL_4(k)$. There are 20 isolated conjugacy classes.

\begin{itemize}
\item[-] $s=(1,1,1,1)\sigma$, $C_G(s)=\Sp_4(k)$.\\
The unipotent parts are given by the partitions defined by Jordan blocks. Then the reductive parts of the centralisers are specified accordingly as below,
\begin{center}
\begin{tabular}{cccc}
$(1^4)$ & $(1^22)$ & $(2^2)$ & $(4)$\\
\hline
$\Sp_4(k)$ & $\SL_2(k)\times\Ort_1(k)$ & $\Ort_2(k)$ & $\Ort_1(k)$
\end{tabular}
\end{center}
This gives 7 classes.

\item[-] $s=(\mathfrak{i},1,1,-\mathfrak{i})\sigma$, $C_G(s)=\SL_2(k)\times\Ort_2(k)$.\\
The unipotent parts are given by the partitions defined by Jordan blocks. Then the reductive parts of the centralisers are specified accordingly as below,
\begin{center}
\begin{tabular}{cc}
$(1^2)$ & $(2)$\\
\hline
$\SL_2(k)\times\Ort_2(k)$ & $\Ort_1(k)\times\Ort_2(k)$ 
\end{tabular}
\end{center}
This gives 6 classes.

\item[-] $s=(\mathfrak{i},\mathfrak{i},-\mathfrak{i},-\mathfrak{i})\sigma$, $C_G(s)=\Ort_4(k)$.\\
The unipotent parts are given by the partitions defined by Jordan blocks. Then the reductive parts of the centralisers are specified accordingly as below,
\begin{center}
\begin{tabular}{ccc}
$(1^4)$ & $(13)$ & $(2^2)$\\
\hline
$\Ort_4(k)$ & $\Ort_1(k)\times\Ort_1(k)$ & $\SL_2(k)$
\end{tabular}
\end{center}
This gives 7 classes.
\end{itemize}

\subsubsection{}
Suppose  $G=\GL_5(k)$. There are 36 isolated conjugacy classes.

\begin{itemize}
\item[-] $s=(1,1,1,1,1)\sigma$, $C_G(s)=\Ort_5(k)$.\\
The unipotent parts are given by the partitions defined by Jordan blocks. Then the reductive parts of the centralisers are specified accordingly as below,
\begin{center}
\begin{tabular}{cccc}
$(1^5)$ & $(1^23)$ & $(12^2)$ & $(5)$\\
\hline
$\Ort_5(k)$ & $\Ort_2(k)\times\Ort_1(k)$ & $\Ort_1(k)\times\SL_2(k)$ & $\Ort_1(k)$
\end{tabular}
\end{center}
This gives 10 classes.

\item[-] $s=(\mathfrak{i},1,1,1,-\mathfrak{i})\sigma$, $C_G(s)=\Ort_3(k)\times\SL_2(k)$.\\
The unipotent parts are given by the partitions defined by Jordan blocks. Then the reductive parts of the centralisers are specified accordingly as below,
\begin{center}
\begin{tabular}{cc}
$(1^3)$ & $(3)$ \\
\hline
$\Ort_3(k)$ & $\Ort_1(k)$
\end{tabular}
$\times$
\begin{tabular}{cc}
$(1^2)$ & $(2)$ \\
\hline
$\SL_2(k)$ & $\Ort_1(k)$
\end{tabular}

\end{center}
This gives $(2+2)\times (1+2)=12$ classes.

\item[-] $s=(\mathfrak{i},\mathfrak{i},1,-\mathfrak{i},-\mathfrak{i})\sigma$, $C_G(s)=\Sp_4(k)\times\Ort_1(k)$.\\
The unipotent parts are given by the partitions defined by Jordan blocks. Then the reductive parts of the centralisers are specified accordingly as below,
\begin{center}
$\Ort_1(k)\times$
\begin{tabular}{cccc}
$(1^4)$ & $(1^22)$ & $(2^2)$ & $(4)$\\
\hline
$\Sp_4(k)$ & $\SL_2(k)\times\Ort_1(k)$ & $\Ort_2(k)$ & $\Ort_1(k)$
\end{tabular}
\end{center}
This gives $2\times 7=14$ classes.
\end{itemize}

\subsection{The Tables}
The calculation of the values of the uniform characters is reduced to the determination of the sets $$A=A(s\sigma, T_w)=\{h\in G^F\mid hs\sigma h^{-1}\in T_w\sigma\}$$ for various $G^F$-conjugacy classes of $F$-stable and $\sigma$-stable maximal tori $T_w$ contained in some $\sigma$-stable Borel subgroups, and semi-simple $G^F$-conjugacy classes of elements $s\sigma$.

\subsubsection{}
The procedure (\textit{cf.} \ref{ComputeA}) for computing $A$ can be summarised as follows.

Suppose $s\sigma$ is an $F$-fixed element contained in $T\sigma$, and $T_w$ can be written as $gTg^{\circ}$ for some $g\in C_G(\sigma)^{\circ}$. If $h\in G^F$ conjugates $s\sigma$ into $T_w\sigma$, then there exists some $l\in C_G(s\sigma)^{\circ}$ such that $n:=g^{-1}hl$ lies in $N_G(T\sigma)$. Recall that $N_G(T\sigma)\subset N_G(T)$ consists of the connected components that are stable under $\sigma$. Then $g^{-1}hs\sigma h^{-1}g=ns\sigma n^{-1}$ is an $F_w$-fixed element of $T\sigma$. If $s\sigma$ is an $F$-fixed element contained in $T_w\sigma$, then similar arguments show that $g^{-1}hs\sigma h^{-1}g=ng^{-1}s\sigma gn^{-1}$ with $gng^{-1}=hl$. The conjugation by $n$ can be separated into a permutation of the "eigenvalues" and a conjugation by an element of $T$. For each $s\sigma$ and $T_w$, we will first find some $t\in T$ such that $ts\sigma t^{-1}$ (or $tg^{-1}s\sigma gt^{-1}$ if we start with some $s\sigma\in T_w\sigma$) is fixed by $F_w$, then evaluate any character of $T_w^F\lb\sigma\rb$ under the isomorphism $T_w^F\lb\sigma\rb\cong T^{F_w}\lb\sigma\rb$. The value does not depend on the choice of $t$, and the permutation of the "eigenvalues" is simple.

We will use the following observation.
Let $\smash{\omega\in\Irr(\mathbb{F}_{q^2}^{\ast})}$ be such that $\omega^q=\omega^{-1}$, and let $a\in\mathbb{F}_{q^2}^{\ast}$ be such that $a^q=a$. Then $a=b^{q+1}$ for some $\smash{b\in\mathbb{F}_{q^2}^{\ast}}$, so $$\omega(a)=\omega(b^{q+1})=\omega^{q+1}(b)=1.$$

\subsubsection{}
Suppose $G=\GL_2(k)$. Consider the characters $R^G_T(\tC,\eta)$, $R^G_T(\alpha,\alpha^{-1})$ and $R^G_{T_w}(\omega)$.

The calculation of the extensions of $R^G_T(\tC,\eta)$ is a direct application of the theorem of Waldspurger. Following the notations of (\ref{h_1h_2}), we have $(\mu_+,\mu_-)=((1),(1))$. The 2-cores are $(1)$ and $(1)$, and so $m_+=m_-=1$. We deduce that $h_1=1$ and $h_2=0$. So the cuspidal function is supported on the class of $su$ with $C_G(s)^{\circ}\cong\SL_2(k)$ and $u$ corresponding to the partition $(2)$. We find $\delta(h_1,h_2)=1$. So the values of this character on its support are $\pm\sqrt{q}$ and vanish on all other classes.

If $s\sigma=\sigma$ and so $C_G(s\sigma)=\SL_2(k)$, then $h\sigma h^{-1}=\det (h)\sigma$ (regarding $\det(h)$ as a scalar matrix), which belongs to $T\sigma$ or $T_w\sigma$ for any $h$. So $A=G^F$ and $\tilde{\theta}(hs\sigma h^{-1})=\theta(\sigma)=1$ for any $h$ as $\theta$ has trivial value on the scalars.

If $s\sigma=(\mathfrak{i},-\mathfrak{i})\sigma$ and so $C_G(s\sigma)=\Ort_2(k)$, then the elements of $A$ are exactly those $h\in G^F$ such that $(h^{-1}Th\cap C_G(s\sigma))^{\circ}$ is a maximal torus of $C_G(s\sigma)^{\circ}=\SO_2(k)$ which itself is a torus whose centraliser in $G$ is $T$ or $T_w$ according to whether $s\sigma\in C_4$ or $s\sigma\in C_5$. Consequently, $A(C_5,T)=A(C_4,T_w)=\varnothing$, while $A(C_4,T)$ and $A(C_5,T_w)$ are the normalisers of $T$ and $T_w$ respectively. It is easy to check that $$\tilde{\theta}(hs\sigma h^{-1})=\alpha(\mathfrak{i})\alpha^{-1}(-\mathfrak{i})=\alpha(-1)$$ if $s\sigma\in C_4$. If $s\sigma\in C_5$, then we use the method at the beginning of this section. It suffices to find some $t\in T$ such that $ts\sigma t^{-1}$ is fixed by $F_w$. Indeed, we can take $t=\diag(\lambda,1)$ with $\lambda^q=-\lambda$ so that $(\mathfrak{i}\lambda)^q=-\mathfrak{i}\lambda$. We get $\theta(s\sigma)=\omega(\mathfrak{i}\lambda)$. The value is independent of the choice of $\lambda$. We can also do it directly and explicitly, and obtain the same result. The elements of $\Ort_2(k)\setminus\SO_2(k)$ are of the form 
$$
\left(
\begin{array}{cc}
0 & x\\
x^{-1} & 0
\end{array}
\right),$$ so they do not belong to $\SL_2(k)$. Let us describe  $T_s$ explicitly by choosing $g_2\in\SL_2(k)$ such that $$g_2^{-1}F(g_2)=
\left(
\begin{array}{cc}
0 & 1\\
-1 & 0
\end{array}
\right)$$ and putting $T_w=g_2Tg_2^{-1}$. We choose $\lambda\in k^{\ast}$ such that $\lambda^q=-\lambda$. Put $$g_1=g_2
\left(
\begin{array}{cc}
\lambda & 0\\
0 & 1
\end{array}
\right),$$ such that $g_1^{-1}F(g_1)\in\Ort_2(k)$. Then the representative $s\sigma\in C_5$ is given by $$g_1\left(
\begin{array}{cc}
\mathfrak{i} & 0\\
0 & -\mathfrak{i}
\end{array}
\right)\sigma g^{-1}_1=g_2
\left(
\begin{array}{cc}
\mathfrak{i}\lambda & 0\\
0 & -\mathfrak{i}\lambda
\end{array}
\right)g_2^{-1}\sigma\in T_w\sigma,$$ and so $\tilde{\theta}(s\sigma)=\omega(\mathfrak{i}\lambda)$. If $\omega=\eta$, then taking the norm gives $\lambda^2$ and evaluating $\eta$ gives $-1$.

If $s\sigma=(a,a^{-1})\sigma$, and so $C_G(s\sigma)=T$ or $T_w$ according to whether $a^q=\pm a$ or $a^q=\pm a^{-1}$, then $A$ is equal to the normaliser of $T$ or $T_w$ or empty according to the  $G^F$-class of $C_G(s\sigma)$. If $s\sigma\in C_8$, the $F$-stable conjugate of $s\sigma$ in $T\sigma$ is given by $\diag(a\lambda,a^{-1}\lambda)\sigma$ with $\lambda^q=-\lambda$. If $\theta=\eta\circ\det|_{T}$, then $\tilde{\theta}((a\lambda,a^{-1}\lambda)\sigma)=\eta(\lambda^{2})=-1$. If $s\sigma\in C_9$, the representative of $C_9(a)$ is given by $$g
\left(
\begin{array}{cc}
a\lambda & 0\\
0 & a^{-1}\lambda
\end{array}
\right)g^{-1}\sigma\in T_s^F\sigma.$$ Again, if $\theta=\eta\circ\det|_{T_s}$, then $\tilde{\theta}((a\lambda,a^{-1}\lambda)\sigma)=\eta(\lambda^{2})=-1$.                                      

\subsubsection{}
Suppose $G=\GL_3(k)$. Consider the characters $\chi_3$, $R^G_T(\alpha,\tC,\alpha^{-1})$, $R^G_T(\omega,\tC,\omega^{-1})$, $R^G_T(\eta\chi_3)$, $R^G_T(\alpha,\eta,\alpha^{-1})$ and $R^G_T(\omega,\eta,\omega^{-1})$.

For $\chi_3$, we use the theorem of Waldspurger. We have $(\mu_+,\mu_-)=((1^3),\varnothing)$. The 2-cores are $(2,1)$ and $\varnothing$, and so $m_+=2$ and $m_-=0$. We deduce that $h_1=1$ and $h_2=1$. So the cuspidal function is supported on the class of $su$ with $C_G(s)^{\circ}\cong\SL_2(k)\times \Ort_1(k)$ and $u$ corresponding to the partition $(2)$. We find $\delta(h_1,h_2)=1$ and so the values of this character are $\pm\sqrt{q}$.

If $s\sigma=(1,1,1)\sigma$ and so $C_G(s\sigma)=\Ort_3(k)$, then one has to understand the set $A^F=(N_G(T\sigma).L')^F$, with the notations of Lemma \ref{A=NL}. If $h=nl\in A^F$, then the $L^{'F}$-conjugacy class of $h^{-1}Th\cap L'$ corresponds to the $F$-class of $n^{-1}F(n)\in N_{L'}(T\cap L')$. But $N_G(T\sigma)\cong W_G(T)^{\sigma}\cong\mathfrak{S}_2$, so $n^{-1}F(n)$ necessarily belongs to $T\cap L'=(T^{\sigma})^{\circ}$, i.e. $h^{-1}Th\cap L'$ is always $L^{'F}$-conjugate to $T\cap L'$ and the only Green function that appears in the formula of $R_T^G\theta(s\sigma u)$ is $Q^{\SO_3(k)}_{(T^{\sigma})^{\circ}}(u)$. We also have a similar result for $T_w$. Expressing the elements $h=nl$ as some explicit matrices, we find that $\tilde{\theta}(hs\sigma h^{-1})$ does not depend on $h$. It remains to calculate
\begin{equation}
\begin{split}
|N_{G^F}(T\sigma)L^{'F}|=&|L^{'F}||N_{G^F}(T\sigma)||N_{L^{'F}}(C_{T}(\sigma )^{\circ})|^{-1}\\
=&|\SO_3(q)|\cdot 2(q-1)^3\cdot 2(q-1),
\end{split}
\end{equation}
For $T_w$, we have
\begin{equation}
|N_{G^F}(T_w\sigma)L^{'F}|=|\SO_3(q)|\cdot 2(q-1)(q^2-1)\cdot 2(q+1).
\end{equation}
The other $G(q)$-class contained in the $G$-class of $\sigma$ has as representative $(1,\lambda^2,1)\sigma$ with $\lambda^q=-\lambda$, so for example the value of $R^G_T(\alpha,\eta,\alpha^{-1})(C_2)$ differs from $R^G_T(\alpha,1,\alpha^{-1})(C_2)$ by a sign. 

The main difference between $\GL_3(q)$ and $\GL_2(q)$ is the class $(\mathfrak{i},1,-\mathfrak{i})\sigma$ (as opposed to $(\mathfrak{i},-\mathfrak{i})\sigma$ for $\GL_2(k)$). We have $C_G(s\sigma)\cong \SL_2(k)\times\Ort_1(k)$ so in particular it contains representatives of each element of $W_G(T)^{\sigma}$. Therefore, the sets $A^F$ are not empty either for $T$ or for $T_w$. Suppose that $s\sigma$ represents $C_3^{\epsilon}$ and we want to evaluate $R^G_T(\alpha,\eta,\alpha^{-1})$ at $s\sigma$. Let $t\in T$ be such that $ts\sigma t^{-1}$ is fixed by $F$. Then $ts\sigma t^{-1}$ can be written as $\diag(\mathfrak{i}x,y,-\mathfrak{i}x)\sigma$. It is necessary that $x^q=x$ and $y^q=y$. So $\alpha(\mathfrak{i})\alpha^{-1}(-\mathfrak{i})=\alpha(-1)$. Applying $\epsilon$ gives $\eta(y)=\epsilon(C_3^{\epsilon})$. Therefore $(\alpha,\eta,\alpha^{-1})(hs\sigma h^{-1})$ evaluates $\epsilon\alpha(-1)$. Now we evaluate $R^G_T(\omega,\eta,\omega^{-1})$ at $s\sigma$. Again we write $ts\sigma t^{-1}$ as $\diag(\mathfrak{i}x,y,-\mathfrak{i}x)\sigma$, but which is $F_w$-stable. It is necessary that $x^q=-x$ and $y^q=y$. Applying $\epsilon$ gives $\eta(y)=-\epsilon(C_3^{\epsilon})$ since $x^2\notin (\mathbb{F}_q^{\ast})^2$. Therefore $(\omega,\eta,\omega^{-1})(hs\sigma h^{-1})$ evaluates $-\epsilon\omega(\mathfrak{i}\lambda)$. If $s\sigma\in C_8^{\pm}$, then an $F$-stable element $\diag(ax,y,a^{-1}x)\sigma=ts\sigma t^{-1}$ satisfies $x^q=-x$. and $y^q=y$. Therefore $(\alpha,\eta,\alpha^{-1})(hs\sigma h^{-1})$ evaluates $-\epsilon\alpha(a^{\pm 2})$, where the $\pm 2$ power is due to permutation of "eigenvalues". For $(\omega,\eta,\omega^{-1})$ at $C_9^{\pm}$ the calculation is similar.

%%%%%%%%%%%%%%%% Table 1
\newpage
\renewcommand{\arraystretch}{3.5}
\scriptsize

\rotatebox{270}{
\vbox{
\captionof{table}{The Character Table of $\GL_2(q)\lb\sigma\rb$}
\begin{tabular}{ |p{1.5cm}|p{1cm}p{1cm}p{1cm}|p{1cm}p{1cm}|p{2.5cm}p{2.5cm}p{2.5cm}p{2.5cm}|  }
\hline
 & & & & & & & & & \\[3.5\dimexpr -\normalbaselineskip]
 & \multicolumn{3}{c}{$(1,1)\sigma$} &  \multicolumn{2}{c}{$(\mathfrak{i},-\mathfrak{i})\sigma$} &  \multicolumn{4}{c}{$(a,a^{-1})\sigma$} \\
\hline
 & \multicolumn{1}{c|}{$(1^2)$} & \multicolumn{2}{|c|}{$(2)$} & \multicolumn{2}{c|}{$1$} & \multicolumn{1}{c|}{$a^{q-1}=1$} & \multicolumn{1}{c|}{$a^{q+1}=1$} & \multicolumn{1}{c|}{$a^{q-1}=-1$} & \multicolumn{1}{c|}{$a^{q+1}=-1$}\\
\hline
 &\multicolumn{1}{c|}{$C_1$} &\multicolumn{1}{c|}{$C_2$} &\multicolumn{1}{c|}{$C_3$} &\multicolumn{1}{c|}{$C_4$} &\multicolumn{1}{c|}{$C_5$} &\multicolumn{1}{c|}{$C_6(a)$} &\multicolumn{1}{c|}{$C_7(a)$} &\multicolumn{1}{c|}{$C_8(a)$} &\multicolumn{1}{c|}{$C_9(a)$}\\
 \hline
 \multicolumn{1}{|c|}{$R^G_T(\tC,\eta)$} &\multicolumn{1}{c|}{$0$} &\multicolumn{1}{c|}{$-\sqrt{q}$} &\multicolumn{1}{c|}{$\sqrt{q}$} &\multicolumn{1}{c|}{$0$} &\multicolumn{1}{c|}{$0$} &\multicolumn{1}{c|}{$0$} &\multicolumn{1}{c|}{$0$} &\multicolumn{1}{c|}{$0$} &\multicolumn{1}{c|}{$0$}\\
 \hline
\multicolumn{1}{|c|}{$\tC$} &\multicolumn{1}{c|}{$1$} &\multicolumn{1}{c|}{$1$} &\multicolumn{1}{c|}{$1$} &\multicolumn{1}{c|}{$1$} &\multicolumn{1}{c|}{$1$} &\multicolumn{1}{c|}{$1$} &\multicolumn{1}{c|}{$1$} &\multicolumn{1}{c|}{$1$} &\multicolumn{1}{c|}{$1$}\\
 \hline
\multicolumn{1}{|c|}{$\eta\tC$} &\multicolumn{1}{c|}{$1$} &\multicolumn{1}{c|}{$1$} &\multicolumn{1}{c|}{$1$} &\multicolumn{1}{c|}{$1$} &\multicolumn{1}{c|}{$-1$} &\multicolumn{1}{c|}{$1$} &\multicolumn{1}{c|}{$1$} &\multicolumn{1}{c|}{$-1$} &\multicolumn{1}{c|}{$-1$}\\
 \hline
\multicolumn{1}{|c|}{$\St$} &\multicolumn{1}{c|}{$q$} &\multicolumn{1}{c|}{$0$} &\multicolumn{1}{c|}{$0$} &\multicolumn{1}{c|}{$1$} &\multicolumn{1}{c|}{$-1$} &\multicolumn{1}{c|}{$1$} &\multicolumn{1}{c|}{$-1$} &\multicolumn{1}{c|}{$1$} &\multicolumn{1}{c|}{$-1$}\\
 \hline
\multicolumn{1}{|c|}{$\eta\St$} &\multicolumn{1}{c|}{$q$} &\multicolumn{1}{c|}{$0$} &\multicolumn{1}{c|}{$0$} &\multicolumn{1}{c|}{$1$} &\multicolumn{1}{c|}{$1$} &\multicolumn{1}{c|}{$1$} &\multicolumn{1}{c|}{$-1$} &\multicolumn{1}{c|}{$-1$} &\multicolumn{1}{c|}{$1$}\\
 \hline
\multicolumn{1}{|c|}{$R^G_T(\alpha,\alpha^{-1})$} &\multicolumn{1}{c|}{$q+1$} &\multicolumn{1}{c|}{$1$} &\multicolumn{1}{c|}{$1$} &\multicolumn{1}{c|}{$2\alpha(-1)$} &\multicolumn{1}{c|}{$0$} &\multicolumn{1}{c|}{$\alpha(a^2)+\alpha(a^{-2})$} &\multicolumn{1}{c|}{$0$} &\multicolumn{1}{c|}{$\alpha(a^2)+\alpha(a^{-2})$} &\multicolumn{1}{c|}{$0$}\\
 \hline
\multicolumn{1}{|c|}{$R^G_{T_w}(\omega)$} &\multicolumn{1}{c|}{$1-q$} &\multicolumn{1}{c|}{$1$} &\multicolumn{1}{c|}{$1$} &\multicolumn{1}{c|}{$0$} &\multicolumn{1}{c|}{$2\omega(\mathfrak{i}\lambda)$} &\multicolumn{1}{c|}{$0$} &\multicolumn{1}{c|}{$\omega(a)+\omega(a^{-1})$} &\multicolumn{1}{c|}{$0$} &\multicolumn{1}{c|}{$\omega(a\lambda)+\omega(a^{-1}\lambda)$}\\
 \hline
\end{tabular}
}}

%%%%%%%%%%%%%%%%%% Table 2
%%%%%%% Table 2.1
\newpage
\rotatebox{270}{
\vbox{
\captionof{table}{The Character Table of $\GL_3(q)\lb\sigma\rb$, (i)}
\begin{tabular}{ |p{2cm}|p{1.4cm}p{1.4cm}p{1.4cm}p{1.4cm}|p{1.4cm}p{1.4cm}p{1.4cm}p{1.4cm}p{1.4cm}p{1.4cm}|  }
\hline
 & & & & & & & & & & \\[3.5\dimexpr -\normalbaselineskip]
 
 & \multicolumn{4}{c}{$(1,1,1)\sigma$} &  \multicolumn{6}{c|}{$(\mathfrak{i},1,-\mathfrak{i})\sigma$} \\
\hline
 & \multicolumn{2}{c|}{$(1^3)$} & \multicolumn{2}{c|}{$(3)$} & \multicolumn{2}{c|}{$(1^2)$} & \multicolumn{4}{c|}{$(2)$} \\
\hline
 &\multicolumn{1}{c|}{$C_1^+$} &\multicolumn{1}{c|}{$C_1^-$} &\multicolumn{1}{c|}{$C_2^+$} &\multicolumn{1}{c|}{$C_2^-$} &\multicolumn{1}{c|}{$C_3^+$} &\multicolumn{1}{c|}{$C_3^-$} &\multicolumn{1}{c|}{$C_4^+$} &\multicolumn{1}{c|}{$C_5^+$} &\multicolumn{1}{c|}{$C_4^-$} & \multicolumn{1}{c|}{$C_5^-$}\\
\hline
&  \multicolumn{8}{|c|}{$(a,1,a^{-1})\sigma$} & &\\
\hline
& \multicolumn{2}{c|}{$a^{q-1}=1$} & \multicolumn{2}{c|}{$a^{q+1}=1$} & \multicolumn{2}{c|}{$a^{q-1}=-1$} & \multicolumn{2}{c|}{$a^{q+1}=-1$} & &\\
\hline
&\multicolumn{1}{c|}{$C_{6}^+$} & \multicolumn{1}{c|}{$C_{6}^-$} 
&\multicolumn{1}{c|}{$C_{7}^+$} & \multicolumn{1}{c|}{$C_{7}^-$} 
&\multicolumn{1}{c|}{$C_{8}^+$} & \multicolumn{1}{c|}{$C_{8}^-$} 
&\multicolumn{1}{c|}{$C_{9}^+$} & \multicolumn{1}{c|}{$C_{9}^-$} & &\\
\hline
%theta
\multicolumn{1}{|c|}{$\chi_3$} & \multicolumn{1}{c|}{$0$} &\multicolumn{1}{c|}{$0$} &\multicolumn{1}{c|}{$0$} &\multicolumn{1}{c|}{$0$} &\multicolumn{1}{c|}{$0$} &\multicolumn{1}{c|}{$0$} &\multicolumn{1}{c|}{$\sqrt{q}$} &\multicolumn{1}{c|}{$-\sqrt{q}$} &\multicolumn{1}{c|}{$-\sqrt{q}$} & \multicolumn{1}{c|}{$\sqrt{q}$} \\
\cline{2-11}
& \multicolumn{2}{c|}{$0$} & \multicolumn{2}{c|}{$0$} & \multicolumn{2}{c|}{$0$} & \multicolumn{2}{c|}{$0$} & &\\
\hline
%etatheta
\multicolumn{1}{|c|}{$\eta\chi_3$} & \multicolumn{1}{c|}{$0$} &\multicolumn{1}{c|}{$0$} &\multicolumn{1}{c|}{$0$} &\multicolumn{1}{c|}{$0$} &\multicolumn{1}{c|}{$0$} &\multicolumn{1}{c|}{$0$} &\multicolumn{1}{c|}{$\sqrt{q}$} &\multicolumn{1}{c|}{$-\sqrt{q}$} &\multicolumn{1}{c|}{$\sqrt{q}$} & \multicolumn{1}{c|}{$-\sqrt{q}$} \\
\cline{2-11}
& \multicolumn{2}{c|}{$0$} & \multicolumn{2}{c|}{$0$} & \multicolumn{2}{c|}{$0$} & \multicolumn{2}{c|}{$0$} & &\\
\hline
\end{tabular}
}}

%%%%%% Table 2.2
\newpage
\rotatebox{270}{
\vbox{
\captionof{table}{The Character Table of $\GL_3(q)\lb\sigma\rb$, (ii)}
\begin{tabular}{ |p{2cm}|p{1.4cm}p{1.4cm}p{1.4cm}p{1.4cm}|p{1.4cm}p{1.4cm}p{1.4cm}p{1.4cm}p{1.4cm}p{1.4cm}|  }
\hline
 & & & & & & & & & & \\[3.5\dimexpr -\normalbaselineskip]
 
 &\multicolumn{1}{c|}{$C_1^+$} &\multicolumn{1}{c|}{$C_1^-$} &\multicolumn{1}{c|}{$C_2^+$} &\multicolumn{1}{c|}{$C_2^-$} &\multicolumn{1}{c|}{$C_3^+$} &\multicolumn{1}{c|}{$C_3^-$} &\multicolumn{1}{c|}{$C_4^+$} &\multicolumn{1}{c|}{$C_5^+$} &\multicolumn{1}{c|}{$C_4^-$} & \multicolumn{1}{c|}{$C_5^-$}\\
\hline
&\multicolumn{1}{c|}{$C_{6}^+$} & \multicolumn{1}{c|}{$C_{6}^-$} 
&\multicolumn{1}{c|}{$C_{7}^+$} & \multicolumn{1}{c|}{$C_{7}^-$} 
&\multicolumn{1}{c|}{$C_{8}^+$} & \multicolumn{1}{c|}{$C_{8}^-$} 
&\multicolumn{1}{c|}{$C_{9}^+$} & \multicolumn{1}{c|}{$C_{9}^-$} & &\\
\hline
%Id
\multicolumn{1}{|c|}{$\tC$} & \multicolumn{1}{c|}{$1$} &\multicolumn{1}{c|}{$1$} &\multicolumn{1}{c|}{$1$} &\multicolumn{1}{c|}{$1$} &\multicolumn{1}{c|}{$1$} &\multicolumn{1}{c|}{$1$} &\multicolumn{1}{c|}{$1$} &\multicolumn{1}{c|}{$1$} &\multicolumn{1}{c|}{$1$} & \multicolumn{1}{c|}{$1$} \\
\cline{2-11}
& \multicolumn{2}{c|}{$1$} & \multicolumn{2}{c|}{$1$} & \multicolumn{2}{c|}{$1$} & \multicolumn{2}{c|}{$1$} & &\\
\hline
%etaId
\multicolumn{1}{|c|}{$\eta\tC$} 
& \multicolumn{1}{c|}{$1$} &\multicolumn{1}{c|}{$-1$} &\multicolumn{1}{c|}{$1$} &\multicolumn{1}{c|}{$-1$} &\multicolumn{1}{c|}{$1$} &\multicolumn{1}{c|}{$-1$} &\multicolumn{1}{c|}{$1$} &\multicolumn{1}{c|}{$1$} &\multicolumn{1}{c|}{$-1$} & \multicolumn{1}{c|}{$-1$} \\
\cline{2-11}
& \multicolumn{2}{c|}{$\epsilon$} & \multicolumn{2}{c|}{$\epsilon$} & \multicolumn{2}{c|}{$\epsilon$} & \multicolumn{2}{c|}{$\epsilon$} & &\\
\hline
%chi2
\multicolumn{1}{|c|}{$\chi_2$} 
& \multicolumn{1}{c|}{$q$} &\multicolumn{1}{c|}{$q$} &\multicolumn{1}{c|}{$0$} &\multicolumn{1}{c|}{$0$} &\multicolumn{1}{c|}{$q$} &\multicolumn{1}{c|}{$q$} &\multicolumn{1}{c|}{$0$} &\multicolumn{1}{c|}{$0$} &\multicolumn{1}{c|}{$0$} & \multicolumn{1}{c|}{$0$} \\
\cline{2-11}
& \multicolumn{2}{c|}{$1$} & \multicolumn{2}{c|}{$-1$} & \multicolumn{2}{c|}{$1$} & \multicolumn{2}{c|}{$-1$} & &\\
\hline
%etachi2
\multicolumn{1}{|c|}{$\eta\chi_2$} 
& \multicolumn{1}{c|}{$q$} &\multicolumn{1}{c|}{$-q$} &\multicolumn{1}{c|}{$0$} &\multicolumn{1}{c|}{$0$} &\multicolumn{1}{c|}{$q$} &\multicolumn{1}{c|}{$-q$} &\multicolumn{1}{c|}{$0$} &\multicolumn{1}{c|}{$0$} &\multicolumn{1}{c|}{$0$} & \multicolumn{1}{c|}{$0$} \\
\cline{2-11}
& \multicolumn{2}{c|}{$\epsilon$} & \multicolumn{2}{c|}{$-\epsilon$} & \multicolumn{2}{c|}{$\epsilon$} & \multicolumn{2}{c|}{$-\epsilon$} & &\\
\hline
\end{tabular}
}}

%%%%% Table 2.3
\newpage
\rotatebox{270}{
\vbox{
\captionof{table}{The Character Table of $\GL_3(q)\lb\sigma\rb$, (iii)}
\begin{tabular}{ |p{2cm}|p{1.4cm}p{1.4cm}p{1.4cm}p{1.4cm}|p{1.4cm}p{1.4cm}p{1.4cm}p{1.4cm}p{1.4cm}p{1.4cm}|  }
\hline
 & & & & & & & & & & \\[3.5\dimexpr -\normalbaselineskip]
 
 &\multicolumn{1}{c|}{$C_1^+$} &\multicolumn{1}{c|}{$C_1^-$} &\multicolumn{1}{c|}{$C_2^+$} &\multicolumn{1}{c|}{$C_2^-$} &\multicolumn{1}{c|}{$C_3^+$} &\multicolumn{1}{c|}{$C_3^-$} &\multicolumn{1}{c|}{$C_4^+$} &\multicolumn{1}{c|}{$C_5^+$} &\multicolumn{1}{c|}{$C_4^-$} & \multicolumn{1}{c|}{$C_5^-$}\\
\hline
&\multicolumn{1}{c|}{$C_{6}^+$} & \multicolumn{1}{c|}{$C_{6}^-$} 
&\multicolumn{1}{c|}{$C_{7}^+$} & \multicolumn{1}{c|}{$C_{7}^-$} 
&\multicolumn{1}{c|}{$C_{8}^+$} & \multicolumn{1}{c|}{$C_{8}^-$} 
&\multicolumn{1}{c|}{$C_{9}^+$} & \multicolumn{1}{c|}{$C_{9}^-$} & &\\
\hline
%Id2eta
\multicolumn{1}{|c|}{$R^G_L(\tC_2,\eta)$} & \multicolumn{1}{c|}{$1$} &\multicolumn{1}{c|}{$-1$} &\multicolumn{1}{c|}{$1$} &\multicolumn{1}{c|}{$-1$} &\multicolumn{1}{c|}{$q$} &\multicolumn{1}{c|}{$-q$} &\multicolumn{1}{c|}{$0$} &\multicolumn{1}{c|}{$0$} &\multicolumn{1}{c|}{$0$} & \multicolumn{1}{c|}{$0$} \\
\cline{2-11}
& \multicolumn{2}{c|}{$\epsilon$} & \multicolumn{2}{c|}{$\epsilon$} & \multicolumn{2}{c|}{$-\epsilon$} & \multicolumn{2}{c|}{$-\epsilon$} & &\\
\hline
%eta2Id
\multicolumn{1}{|c|}{$R^G_L(\eta\tC_2,\tC)$} 
& \multicolumn{1}{c|}{$1$} &\multicolumn{1}{c|}{$1$} &\multicolumn{1}{c|}{$1$} &\multicolumn{1}{c|}{$1$} &\multicolumn{1}{c|}{$q$} &\multicolumn{1}{c|}{$q$} &\multicolumn{1}{c|}{$0$} &\multicolumn{1}{c|}{$0$} &\multicolumn{1}{c|}{$0$} & \multicolumn{1}{c|}{$0$} \\
\cline{2-11}
& \multicolumn{2}{c|}{$1$} & \multicolumn{2}{c|}{$1$} & \multicolumn{2}{c|}{$-1$} & \multicolumn{2}{c|}{$-1$} & &\\
\hline
%Steta
\multicolumn{1}{|c|}{$R^G_L(\St,\eta)$} 
& \multicolumn{1}{c|}{$q$} &\multicolumn{1}{c|}{$-q$} &\multicolumn{1}{c|}{$0$} &\multicolumn{1}{c|}{$0$} &\multicolumn{1}{c|}{$1$} &\multicolumn{1}{c|}{$-1$} &\multicolumn{1}{c|}{$1$} &\multicolumn{1}{c|}{$1$} &\multicolumn{1}{c|}{$-1$} & \multicolumn{1}{c|}{$-1$} \\
\cline{2-11}
& \multicolumn{2}{c|}{$\epsilon$} & \multicolumn{2}{c|}{$-\epsilon$} & \multicolumn{2}{c|}{$-\epsilon$} & \multicolumn{2}{c|}{$\epsilon$} & &\\
\hline
%etaStId
\multicolumn{1}{|c|}{$R^G_L(\eta\St,\tC)$} 
& \multicolumn{1}{c|}{$q$} &\multicolumn{1}{c|}{$q$} &\multicolumn{1}{c|}{$0$} &\multicolumn{1}{c|}{$0$} &\multicolumn{1}{c|}{$1$} &\multicolumn{1}{c|}{$1$} &\multicolumn{1}{c|}{$1$} &\multicolumn{1}{c|}{$1$} &\multicolumn{1}{c|}{$1$} & \multicolumn{1}{c|}{$1$} \\
\cline{2-11}
& \multicolumn{2}{c|}{$1$} & \multicolumn{2}{c|}{$-1$} & \multicolumn{2}{c|}{$-1$} & \multicolumn{2}{c|}{$1$} & &\\
\hline
\end{tabular}
}}

%%%%%% Table 2.4
\newpage
\rotatebox{270}{
\vbox{
\captionof{table}{The Character Table of $\GL_3(q)\lb\sigma\rb$, (iv)}
\begin{tabular}{ |p{2cm}|p{1.2cm}p{1.2cm}p{1.2cm}p{1.2cm}|p{1.5cm}p{1.5cm}p{1.2cm}p{1.2cm}p{1.2cm}p{1.2cm}|  }
\hline
 & & & & & & & & & & \\[3.5\dimexpr -\normalbaselineskip]
 
 &\multicolumn{1}{c|}{$C_1^+$} &\multicolumn{1}{c|}{$C_1^-$} &\multicolumn{1}{c|}{$C_2^+$} &\multicolumn{1}{c|}{$C_2^-$} &\multicolumn{1}{c|}{$C_3^+$} &\multicolumn{1}{c|}{$C_3^-$} &\multicolumn{1}{c|}{$C_4^+$} &\multicolumn{1}{c|}{$C_5^+$} &\multicolumn{1}{c|}{$C_4^-$} & \multicolumn{1}{c|}{$C_5^-$}\\
\hline
&\multicolumn{1}{c|}{$C_{6}^+$} & \multicolumn{1}{c|}{$C_{6}^-$} 
&\multicolumn{1}{c|}{$C_{7}^+$} & \multicolumn{1}{c|}{$C_{7}^-$} 
&\multicolumn{1}{c|}{$C_{8}^+$} & \multicolumn{1}{c|}{$C_{8}^-$} 
&\multicolumn{1}{c|}{$C_{9}^+$} & \multicolumn{1}{c|}{$C_{9}^-$} & &\\
\hline
%al1al
\multicolumn{1}{|c|}{$(\alpha,\tC,\alpha^{-1})$} 
& \multicolumn{1}{c|}{$q+1$} &\multicolumn{1}{c|}{$q+1$} &\multicolumn{1}{c|}{$1$} &\multicolumn{1}{c|}{$1$} &\multicolumn{1}{c|}{$(q+1)\alpha(-1)$} &\multicolumn{1}{c|}{$(q+1)\alpha(-1)$} &\multicolumn{1}{c|}{$\alpha(-1)$} &\multicolumn{1}{c|}{$\alpha(-1)$} &\multicolumn{1}{c|}{$\alpha(-1)$} & \multicolumn{1}{c|}{$\alpha(-1)$} \\
\cline{2-11}
& \multicolumn{2}{c|}{$\alpha(a^2)+\alpha(a^{-2})$} & \multicolumn{2}{c|}{$0$} & \multicolumn{2}{c|}{$\alpha(a^2)+\alpha(a^{-2})$} & \multicolumn{2}{c|}{$0$} & &\\
\hline
%al-1al
\multicolumn{1}{|c|}{$(\alpha,\eta,\alpha^{-1})$} 
& \multicolumn{1}{c|}{$q+1$} &\multicolumn{1}{c|}{$-q-1$} &\multicolumn{1}{c|}{$1$} &\multicolumn{1}{c|}{$-1$} &\multicolumn{1}{c|}{$(q+1)\alpha(-1)$} &\multicolumn{1}{c|}{$-(q+1)\alpha(-1)$} &\multicolumn{1}{c|}{$\alpha(-1)$} &\multicolumn{1}{c|}{$\alpha(-1)$} &\multicolumn{1}{c|}{$-\alpha(-1)$} & \multicolumn{1}{c|}{$-\alpha(-1)$} \\
\cline{2-11}
& \multicolumn{2}{c|}{$\epsilon(\alpha(a^2)+\alpha(a^{-2}))$} & \multicolumn{2}{c|}{$0$} & \multicolumn{2}{c|}{$-\epsilon(\alpha(a^2)+\alpha(a^{-2}))$} & \multicolumn{2}{c|}{$0$} & &\\
\hline
%w1w
\multicolumn{1}{|c|}{$(\omega,\tC,\omega^{-1})$} 
& \multicolumn{1}{c|}{$-q+1$} &\multicolumn{1}{c|}{$-q+1$} &\multicolumn{1}{c|}{$1$} &\multicolumn{1}{c|}{$1$} &\multicolumn{1}{c|}{$(-q+1)\omega(\mathfrak{i}\lambda)$} &\multicolumn{1}{c|}{$(-q+1)\omega(\mathfrak{i}\lambda)$} &\multicolumn{1}{c|}{$\omega(\mathfrak{i}\lambda)$} &\multicolumn{1}{c|}{$\omega(\mathfrak{i}\lambda)$} &\multicolumn{1}{c|}{$\omega(\mathfrak{i}\lambda)$} & \multicolumn{1}{c|}{$\omega(\mathfrak{i}\lambda)$} \\
\cline{2-11}
& \multicolumn{2}{c|}{$0$} & \multicolumn{2}{c|}{$\omega(a)+\omega(a^{-1})$} & \multicolumn{2}{c|}{$0$} & \multicolumn{2}{c|}{$\omega(a\lambda)+\omega(a^{-1}\lambda)$} & &\\
\hline
%w-1w
\multicolumn{1}{|c|}{$(\omega,\eta,\omega^{-1})$} 
& \multicolumn{1}{c|}{$-q+1$} &\multicolumn{1}{c|}{$q-1$} &\multicolumn{1}{c|}{$1$} &\multicolumn{1}{c|}{$-1$} &\multicolumn{1}{c|}{$(q-1)\omega(\mathfrak{i}\lambda)$} &\multicolumn{1}{c|}{$(-q+1)\omega(\mathfrak{i}\lambda)$} &\multicolumn{1}{c|}{$-\omega(\mathfrak{i}\lambda)$} &\multicolumn{1}{c|}{$-\omega(\mathfrak{i}\lambda)$} &\multicolumn{1}{c|}{$\omega(\mathfrak{i}\lambda)$} & \multicolumn{1}{c|}{$\omega(\mathfrak{i}\lambda)$} \\
\cline{2-11}
& \multicolumn{2}{c|}{$0$} & \multicolumn{2}{c|}{$\epsilon(\omega(a)+\omega(a^{-1}))$} & \multicolumn{2}{c|}{$0$} & \multicolumn{2}{c|}{$-\epsilon(\omega(a\lambda)+\omega(a^{-1}\lambda))$} & &\\
\hline
\end{tabular}
}}

%%%%%%%%%%%%%%%%
\normalsize

%%%%%%%%%%%%%%%%%%%%
\newpage
\addtocontents{toc}{\protect\setcounter{tocdepth}{-1}}
\bibliographystyle{alpha}
\bibliography{BIB}
\end{document}